\documentclass[10pt,twoside,reqno]{amsart}
\usepackage{amssymb,amsthm,amsmath}
\usepackage{mathabx}
\usepackage[usenames,dvipsnames]{color}
\usepackage[utf8]{inputenc}    
\usepackage[T1]{fontenc}       
\usepackage[margin=1in]{geometry}
\usepackage{bbm}
\usepackage{enumerate}

\usepackage{tikz}

\usepackage{xcolor}

\usepackage[hyphens]{url}
\usepackage[hidelinks, colorlinks, allcolors = black, breaklinks = true]{hyperref}
\makeatletter
\def\subsubsection{\@startsection{subsubsection}{3}%
  \z@\z@{-\fontdimen2\font}%
  {\normalfont\bfseries}}
\makeatother

\makeatletter
\def\paragraph{\@startsection{paragraph}{4}%
  \z@\z@{-\fontdimen2\font}%
  {\normalfont\bfseries}}
\makeatother

\makeatletter
\newcommand{\proofstep}[1]{%
  \par
  \addvspace{\medskipamount}
  \textit{#1\@addpunct{.}}\enspace\ignorespaces
}
\makeatother

\xdefinecolor{lightblue}{RGB}{160,220,255}
\xdefinecolor{lightgray}{RGB}{220,220,220}

\renewcommand{\d}{\ensuremath{\mathrm{d}}}
\newcommand{\R}{\ensuremath{\mathbb{R}}}

\newcommand{\N}{\ensuremath{\mathbb{N}}}

\newcommand{\NN}{\ensuremath{\mathcal N}}

\newcommand{\vertiii}[1]{{\left\lvert\kern-0.25ex\left\lvert\kern-0.25ex\left\lvert #1
    \right\rvert\kern-0.25ex\right\rvert\kern-0.25ex\right\rvert}}

\renewcommand{\d}{\ensuremath{{\rm d}}}

\newcommand{\floor}[1]{\left\lfloor#1\right\rfloor}

\newcommand{\ignore}[1]{}

\newtheorem{theorem}{Theorem}[section]
\newtheorem{definition}[theorem]{Definition}
\newtheorem{proposition}[theorem]{Proposition}

\newtheorem{lemma}[theorem]{Lemma}
\newtheorem{corollary}[theorem]{Corollary}
\newtheorem{assumptions}[theorem]{Assumptions}

\numberwithin{equation}{section}
\newcommand{\skliaustask}{\left\lvert\kern-0.25ex\left\lvert\kern-0.25ex\left\lvert}
\newcommand{\skliaustasd}{\right\rvert\kern-0.25ex\right\rvert\kern-0.25ex\right\rvert}
\newcommand{\reftext}[1]{#1}
\newcommand{\eqncr}{\\}
\newcommand{\noxml}[1]{#1}
\newcommand{\mathbh}[1]{\mathbbm{#1}}
\newcommand{\nobreakpostdisplay}{}

\newcommand{\sep}{, }

\begin{document}
\title[The Navier-slip thin-film equation for 3D fluid films]{The Navier-slip thin-film equation for 3D fluid films: existence and uniqueness}
\keywords{free boundary\sep classical solutions\sep moving contact line\sep fourth-order degenerate-parabolic equations\sep lubrication approximation\sep Navier slip}
\subjclass[2010]{35R35\sep 35A09\sep 35K65\sep 35K25\sep 35K55\sep 76A20\sep 76D08}
\thanks{The authors thank Slim Ibrahim, Herbert Koch, Nader Masmoudi, J\"urgen Saal, and Christian Seis for discussions. MVG received funding from the Fields Institute for Research in Mathematical Sciences in Toronto, the University of Michigan at Ann Arbor, the Max Planck Institute for Mathematics in the Sciences in Leipzig, as well as from the US National Science Foundation (grant NSF DMS-1054115) and the Deutsche Forschungsgemeinschaft (grant GN 109/1-1). MP received funding from the Fields Institute for Research in Mathematical Sciences in Toronto, from an FSMP fellowship, and from an EPDI fellowship, at different stages of the preparation of this work. MVG is grateful to the Hausdorff Center for Mathematics in Bonn for its kind hospitality. MP appreciates the kind hospitality of the Technical University of Munich, of ETH Z\"urich, and of the ICERM at Brown University. The authors are grateful to the organizers of the Fields Institute 2014 Thematic Program on \emph{Variational Problems in Physics, Economics and Geometry}, where this research was initiated.}
\date{\today}
\author[Manuel V.~Gnann]{Manuel~V.~Gnann}
\address[Manuel V.~Gnann]{Institute for Applied Mathematics, Faculty of Mathematics and Computer Sciences,
Heidelberg University, Im Neuenheimer Feld 205, 69120 Heidelberg, Germany}
\email{gnann@ma.tum.de}
\author[Mircea Petrache]{Mircea~Petrache}
\address[Mircea Petrache]{Pontificia Catolica Universidad de Chile, Avda. Vicuna Mackenna 4860, 6904441 Santiago, Chile}
\email{decostruttivismo@gmail.com}
\begin{abstract}
We consider the thin-film equation $\partial_t h + \nabla \cdot \left(h^2 \nabla \Delta h\right) = 0$ in physical space dimensions (i.e., one dimension in time $t$ and two lateral dimensions with $h$ denoting the height of the film in the third spatial dimension), which corresponds to the lubrication approximation of the Navier-Stokes equations of a three-dimensional viscous thin fluid film with Navier-slip at the substrate. This equation can have a free boundary (the contact line), moving with finite speed, at which we assume a zero contact angle condition (complete-wetting regime). Previous results have focused on the $1+1$-dimensional version, where it has been found that solutions are not smooth as a function of the distance to the free boundary. In particular, a well-posedness and regularity theory is more intricate than for the second-order counterpart, the porous-medium equation, or the thin-film equation with linear mobility (corresponding to Darcy dynamics in the Hele-Shaw cell). Here, we prove existence and uniqueness of classical solutions that are perturbations of an asymptotically stable traveling-wave profile. This leads to control on the free boundary and in particular its velocity.
\end{abstract}
\maketitle
%

\bigskip

\tableofcontents
\section{The setting}
\subsection{Formulation as a free-boundary problem}

We study the thin-film equation
\begin{subequations}\label{tfe_free}
\begin{equation}\label{tfe_higher}
\partial_t h + \nabla\cdot\left(h^2 \, \nabla\Delta h\right) = 0
\quad\mbox{for } \, t > 0 \, \mbox{ and } \, (y,z) \in\{h > 0\}
\end{equation}
in $1+2$ dimensions. Here, $t$ denotes the time variable
and $(y,z) \in \mathbb{R}^2$ the base point of the fluid film
with height $h = h(t,y,z)$
(cf.~\reftext{Fig.~\ref{fig:3dliquid}}).

\begin{figure}[ht]
\centering
\begin{tikzpicture}[scale=1]
\path[fill=lightblue] (1.5,7) to [out=0,in=195] (8.5,10.2) -- (8.5,7) -- (1.5,7);
\path[fill=lightgray] (-1,7) -- (8.5,7) -- (8.5,6.3) -- (-1,6.3);

\draw [very thick,->] (-1,7) -- (8.3,7);
\draw [very thick,->] (0,6) -- (0,11);

\draw[very thick,blue] (1.5,7) to [out=0,in=195] (8.5,10.2);

\draw [gray,dashed] (5.15,6.8) -- (5.15,8.4);
\draw [gray,dashed] (0,8.4) -- (5.15,8.4);
\draw [thick,blue,dashed,<->] (0,6.8) -- (5.15,6.8);
\draw [blue] (2.5,6.8) node[anchor=north] {$z$};


\draw (2,9.5) node[anchor=east] {gas};

\draw [blue] (7.7,8.5) node[anchor=east] {liquid};

\draw [thick,violet,->] (1.8,7.6) -- (1.5,7);
\draw [violet] (1.8,7.6) node[anchor=south] {triple junction};

\draw (0,10.5) node[anchor=east] {film height $h$};
\draw [blue] (0,8.4) node[anchor=east] {$h(t,y,z)$};
\draw (7.7,7) node[anchor=north] {$z$};

%
%

\draw [very thick,red,dashed, <-] 
(-1.5,3) 
to [out=135, in=-90] (-2.5,5)
to [out=90, in=-135] (-1.5,7)
;

\path[fill=lightblue] 
(1,-1) 
to [out=90,in=-90] (1.2,0.5)
to [out=90,in=-90] (1,1.5)
to [out=90,in=-90] (1.5,3)
to [out=90,in=-90] (1,5)
-- (8.5,5) -- (8.5,-1);
\path[fill=lightgray] 
(1,-1) 
to [out=90,in=-90] (1.2,0.5)
to [out=90,in=-90] (1,1.5)
to [out=90,in=-90] (1.5,3)
to [out=90,in=-90] (1,5)
-- (-1,5) -- (-1,-1);

\draw [very thick,->] (-1,0) -- (8.3,0);
\draw [very thick] (0,-1) -- (0,.5);
\draw [very thick,->] (0,1.7) -- (0,4.8);

\draw[very thick,violet] (1,-1) 
to [out=90,in=-90] (1.2,0.5)
to [out=90,in=-90] (1,1.5)
to [out=90,in=-90] (1.5,3)
to [out=90,in=-90] (1,5);

%

\draw [very thick,red,dashed] (-1,3) -- (8.5,3);

\draw (-.1,1.7) node[anchor=north] {$\begin{array}{c}\text{unwetted}\\ \text{region}\end{array}$};

\draw [blue] (3,1) node[anchor=east] {liquid};

\draw [thick,violet,->] (3,2) -- (1.5,2.5);
\draw [violet] (3,2) node[anchor=west] {triple junction};

\draw (0,4.5) node[anchor=east] {$y$};

\draw (7.7,0) node[anchor=north] {$z$};

\end{tikzpicture}
\caption{Schematic of a liquid thin film in the plane $y,z$, and in the plane $z,h$ at fixed coordinate $y$.}
\label{fig:3dliquid}
\end{figure}
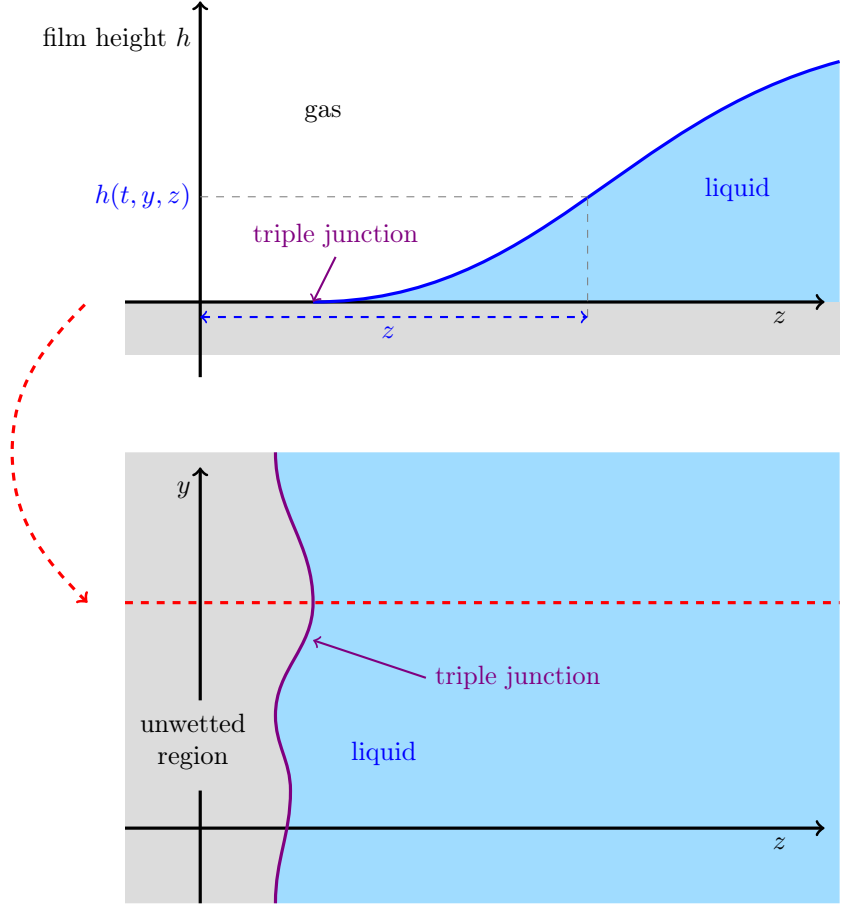

The differential operators in \reftext{\eqref{tfe_higher}} read
\begin{equation*}
\nabla:=
\begin{pmatrix} \partial_y \\ \partial_z
\end{pmatrix}
\quad\mbox{and} \quad\Delta:= \partial_y^2 + \partial_z^2.
\end{equation*}
It is known that \reftext{\eqref{tfe_free}} allows for solutions evolving with
finite speed of propagation (cf. \cite{b.1996,b.1996.2,hs.1998,g.2002,g.2003}), that is, a free
boundary $\partial\{h > 0\}$ (the \emph{contact line}) will appear.
This is a moving line in three-dimensional physical space, which
evolves in time and forms the \emph{triple junction} between the three
phases liquid, gas, and solid. Here, we assume a zero contact angle at
the contact line (the angle between the interfaces liquid-gas and
liquid-solid), that is,
%
%
\begin{equation}\label{contact1}
\left(\nu, \nabla h\right) = 0 \quad\mbox{at } \, \partial\{h >
0\},
\end{equation}
where $\nu= \nu(t,y,z) \in\mathbb{R}^2$ denotes the (inner) unit
normal of
the free boundary $\partial\{h > 0\}$. Reformulating equation~\reftext{\eqref{tfe_higher}} in divergence form as
\begin{equation}\label{def_V}
\partial_t h + \nabla\cdot\left(h V\right) = 0 \quad\mbox{for }
\, t
> 0 \, \mbox{ and } \, (y,z) \in\{h > 0\},
\end{equation}
we read off the transport velocity $V = h \nabla\Delta h$ of the film
height, for which the boundary value on $\partial\{h > 0\}$ has to
equal the velocity $V_0=V_{|\partial\{h > 0\}}$ of the free boundary. Thus we impose
another condition at the contact line, which reads
\begin{equation}\label{contact2}
h \nabla\Delta h = V_0 \quad\mbox{at } \, \partial\{h > 0\}.
\end{equation}
\end{subequations}
We note that the problem is invariant under changing the
function $h$ by a lateral translation. Hence, our subsequent
results are valid up to translating $h$ appropriately.

\subsection{The thin-film equation with general mobility}\label{sec:tfe_gen}

In fact, equation~\reftext{\eqref{tfe_higher}} is a special case
of the general thin-film equation
\begin{equation}\label{tfe_general}
\partial_t h + \nabla\cdot\left(h^n \nabla\Delta h\right) = 0
\quad
\mbox{in} \quad\{h > 0\},
\end{equation}
where $n \in(0,3)$ determines the mobility $h^n$ entering \reftext{\eqref{tfe_general}}. Note that the upper and lower bounds on $n$ are due to
the following observations:
\begin{enumerate}[(i)]
\item[(i)] For $n \le0$, equation~\reftext{\eqref{tfe_general}} is not degenerate
anymore and therefore has infinite speed of propagation. Secondly,
non-negativity of solutions (cf. \cite{bf.1990}) is not ensured
anymore. Therefore, in these cases solutions to \reftext{\eqref{tfe_general}}
bear no physical interpretation as fluid films.
\item[(ii)] For $n = 3$, equation~\reftext{\eqref{tfe_general}} is the lubrication
approximation of the Navier--Stokes equations with no slip at the
substrate. Then the variables $h$ and $z$ in the spatial part of \reftext{\eqref{tfe_general}} have a critical scaling leading to a singularity of $h$
at the contact line. Furthermore, the triple junction is fixed for all
times as a moving contact line would lead to infinite dissipation
(cf. \cite{dd.1974,hs.1971,m.1964}). As the degeneracy increases
with $n$, the free boundary cannot move for all $n \ge3$, that is, a
physical interpretation ceases to be valid.
\end{enumerate}
In fact, the integer cases $n = 1$ and $n = 2$ carry most physics. The
case of linear mobility $n = 1$ can be interpreted as the lubrication
approximation of the Darcy flow in the Hele-Shaw cell
(cf. \cite{go.2003,km.2013,km.2015}). In this case the fluid is trapped
between two narrow walls, so that the flow field is in good
approximation laminar and parallel to the walls and the dependence on
the coordinate perpendicular to the walls is a Poiseuille-type
parabola. Hence, it appears that the $1+1$-dimensional case is the
physically most relevant one for linear mobility $n = 1$.

In a physical $1+2$-dimensional lubrication model, linear mobilities
can be
reproduced by assuming a nonlinear slip condition in which the slip length
diverges like $h^{-1}$ as $h \searrow0$ (see~Greenspan~\cite{g.1978}).
However, in this case it may be
argued that a more natural choice (in line with the original work of
Navier \cite{n.1823}, where the slip condition together with the Navier--Stokes
equations has been proposed first) is to use a quadratic mobility $n = 2$,
corresponding to an $h$-independent slip length. This is in fact the
lubrication approximation of the Navier--Stokes equations with Navier
slip at
the substrate. That is the case on which we chose to concentrate in the
present work.

Note that the dissipation functional in the Navier-slip case is given
by the (scaled) sum of the dissipation functionals for inner friction
only (no-slip case) and purely outer friction (Darcy), so that Navier
slip merely is a balance of these two contributions. We refer to the
reviews by
Bonn, Eggers, Indekeu, Meunier,
and Rolley \cite{beimr.2009}, de Gennes \cite{d.1985}, and Oron, Davis, and
Bankoff \cite{odb.1997}
for (non-rigorous)
derivations of \reftext{\eqref{tfe_higher}} starting from the Navier--Stokes
equations with Navier slip at the liquid-solid interface and to
J{\"a}ger and Mikeli{\'c} \cite{jm.2001} for a rigorous derivation of the Navier-slip condition due to
a rough liquid-solid interface.

\subsection{Weak solutions to the thin-film equation}

We emphasize that a well-established global existence theory of weak
solutions to \reftext{\eqref{tfe_general}} has been developed, starting with
Bernis and Friedman \cite{bf.1990} and later on upgraded to the stronger \emph
{entropy-weak} solutions by
Beretta, Bertsch, and Dal Passo \cite{bbd.1995}, and independently by
Bertozzi and Pugh \cite{bp.1996}, which also exist
in higher dimensions
(cf. \cite{dgg.1998,g.2004.2}). An alternative
gradient-flow approach leading to \emph{generalized minimizing-movement
solutions} (that are weak solutions as well) is due to
Loibl, Matthes, and Zinsl \cite{lmz.2016},
Matthes, McCann, and Savar{\'e} \cite{mms.2009}, and Otto \cite{o.1998}.
Qualitative properties of weak solutions
have been the subject of for instance the works of
Bernis \cite{b.1996.2,b.1996}, Gr{\"u}n \cite{g.2002,g.2003}, and
Hulshof and Shishkov \cite{hs.1998},
where finite speed of
propagation has been proved,
Bertsch, {Dal Passo}, Garcke, and  Gr{\"u}n \cite{bdgg.1998},
{Dal Passo}, Giacomelli, and Gr{\"u}n \cite{dgg.2001},
Fischer \cite{f.2013,f.2014,f.2016},
Giacomelli and Gr{\"u}n \cite{g.2006}, and
Gr{\"u}n \cite{g.2004}, where
waiting-time phenomena have been considered, or
Carlen and Ulusoy \cite{cu.2014},
Carrillo and Toscani \cite{ct.2002}, and
Matthes, McCann, and Savar{\'e} \cite{mms.2009}, where the intermediate asymptotics of \reftext{\eqref{tfe_general}} have been investigated. Partial-wetting boundary
conditions have been considered by
Bertsch, Giacomelli, and Karali \cite{bgk.2005},
Esselborn \cite{e.2016},
Mellet \cite{m.2015}, and
Otto \cite{o.1998}. We refer to
Ansini and Giacomelli \cite{ag.2004},
Bertozzi \cite{b.1998}, and
Giacomelli and Shishkov \cite{gs.2005} for detailed reviews.

Nevertheless, unlike in the porous-medium case \reftext{\eqref{pme}}, this theory
does neither give uniqueness of solutions nor enough control at the
free boundary to give an expression like \reftext{\eqref{contact2}} a classical
meaning. Furthermore, the regularity of the free boundary as a
sub-manifold of $(0,\infty) \times\mathbb{R}^2$ appears to be inaccessible
within this theory. This explains the interest in a well-posedness and
regularity theory of classical solutions to \reftext{\eqref{tfe_free}}.

\subsection{Well-posedness and classical solutions}\label{ssec:wellposed}

Well-posedness and regularity for zero contact angles in the Hele-Shaw
case (equation~\reftext{\eqref{tfe_general}} with $n = 1$) have been treated by
Bringmann, Giacomelli, Kn{\"u}pfer, and Otto \cite{bgko.2016,gko.2008},
Giacomelli and Kn{\"u}pfer \cite{gk.2010},
the first author \cite{g.2015}, and
the first author, Ibrahim, and Masmoudi \cite{gim.2017} in $1+1$ dimensions
and by
John \cite{j.2015} and Seis \cite{s.2017} in any number of spatial dimensions while
nonzero contact angles have been the subject of the works of
Kn{\"u}pfer and Masmoudi \cite{km.2013,km.2015} for $1+1$ dimensions only. The remarkable result is
that solutions are smooth functions in the distance to the free
boundary only for the \emph{linear-mobility} case, i.e., for the case
$n = 1$. The reason for this feature is the strong analogy to the
porous-medium equation
%
%
\begin{equation}\label{pme}
\partial_t h - \Delta h^m = 0 \quad\mbox{in} \quad\{h > 0\},
\end{equation}
where $m > 1$, which is also degenerate-parabolic, but additionally
satisfies a comparison principle. In fact, the spatial part of the
linearizations in the works of
Giacomelli, Kn{\"u}pfer, and Otto \cite{gko.2008},
the first author \cite{g.2015}, John \cite{j.2015}, and Seis \cite{s.2017} is
nothing else but the square (or a second-order polynomial) of the spatial part of the corresponding
linearization of \reftext{\eqref{pme}}. For \reftext{\eqref{pme}} a well-established
well-posedness and regularity theory (giving smooth solutions) is
available (cf. \cite{a.1988,dh.1998,k.1999,k.2016,s.2015}), which in
fact transfers to \reftext{\eqref{tfe_general}} with $n = 1$ and zero contact
angle on $\partial\{h > 0\}$. In what follows, we give evidence for the
insight that the strong analogy to the porous-medium equation \reftext{\eqref{pme}} is lost when passing to general mobilities $n \in(0,3) \setminus
\{1\}$, in particular Navier slip with $n=2$:

Well-posedness and regularity for the $1+1$-dimensional counterpart of \reftext{\eqref{tfe_free}}, i.e.,
%
%
\begin{equation}\label{tfe_1d}
\partial_t h + \partial_z \left(h^2 \partial_z^3 h\right) = 0
\quad
\mbox{for} \quad t > 0 \quad\mbox{and} \quad z \in\{h > 0\},
\end{equation}
subject to complete-wetting boundary conditions, have already been
obtained by
Giacomelli, the first author, Kn{\"u}pfer, and Otto \cite{ggko.2014} and subsequently improved by the first author \cite{g.2016}.
There, perturbations of traveling waves
%
%
\begin{equation}\label{tw}
h =
\begin{cases} x^{\frac3 2} & \mbox{for } \, x > 0, \\ 0 & \mbox{for }
\, x \le0
\end{cases}
\quad\mbox{with} \quad x := z - \frac3 8 t,
\end{equation}
have been investigated (the wave velocity is without loss of generality
normalized to $\frac3 8$ and the contact point is without loss of
generality at $z = 0$ at initial time $t = 0$). The result \cite{ggko.2014} establishes well-posedness of solutions for sufficiently
regular initial data and a partial regularity result, which has been
upgraded in \cite{g.2016} to obtain full regularity in form of
%
%
\begin{align}\label{exp_1d}
h = x^{\frac3 2} \left(1 + \sum_{\beta\le j + \beta k < N} a_{jk}(t)
x^{j+\beta k} + O\left(x^N\right)\right) \quad\mbox{as } \, x
\searrow
0, \quad\mbox{where} \quad x := z - Z_0(t),
\end{align}
the variable $Z_0(t)$ denotes the contact point, and where the number
%
%
\begin{equation}\label{beta}
\beta:= \frac{\sqrt{13}-1}{4} \in \left(\frac 1 2, 1\right)
\end{equation}
has been introduced. Here, the $a_{jk}(t)$ are (at least) continuous
functions of time $t$ and the order $N \in\mathbb{N}$ of
expansion~\reftext{\eqref{exp_1d}} can be chosen arbitrarily large. Note that
expansion~\reftext{\eqref{exp_1d}} is in line with the findings of
Giacomelli, the first author, and Otto \cite{ggo.2013}, where it has
been found that source-type self-similar solutions $h(t,z) = t^{-\frac
1 6} H(x)$ with $x := t^{-\frac1 6} z$, have the form
\begin{equation*}
H(x) = C  {\left\lvert x \pm1 \right\rvert}^{\frac3 2}
\left(1 + v\left( {\left\lvert x \pm1 \right\rvert
}, {\left\lvert x \pm1 \right\rvert}^\beta\right
)\right) \quad\mbox{as} \quad(-1,1)
\owns x \to\mp1,
\end{equation*}
where $v = v(x_1,x_2)$ is analytic around $(x_1,x_2) = (0,0)$ with
$v(0,0) = 0$ (cf.
\cite[Appendix, for the generalization to higher dimensions]{bgk.2016}).
The analysis in \cite{g.2016} heavily uses the
fact that spatial and temporal regularity can only be treated jointly,
so that expansion~\reftext{\eqref{exp_1d}} also implies higher regularity in
time, in particular of the contact point $Z_0 = Z_0(t)$. This was also
used in the work of Kn{\"u}pfer \cite{k.2011}, where
partial-wetting boundary conditions
(i.e., a fixed nonzero contact angle at the triple junction)
has been treated. For the case of general mobilities, we refer to the works of
Belgacem, the first author, and Kuehn \cite{bgk.2016}, as well as
Giacomelli, the first author, and Otto \cite{ggo.2013}
for source-type self-similar solutions with
complete- and partial-wetting boundary conditions (partially also in
higher dimensions), to Kn{\"u}pfer \cite{k.2015,k.2017} for general solutions in
the $1+1$-dimensional case and partial-wetting boundary conditions, and
to Degtyarev \cite{d.2017} for general solutions in higher dimensions with
partial-wetting boundary conditions, where similar features can be observed.

This work is concerned with developing a well-posedness and stability
analysis for the Navier-slip thin-film equation with complete-wetting
boundary conditions in physical dimensions $1+2$. We are aware of only
three works (which have been mentioned afore) establishing
well-posedness and regularity of the free boundary in more than $1+1$
dimensions. The first two, John \cite{j.2015} and Seis \cite{s.2017}, treat the linear
mobility case in arbitrary dimensions, showing that level sets (and
therefore also the free boundary) are analytic manifolds of time and
space. The third paper due to Degtyarev \cite{d.2017} treats the quadratic
mobility case in arbitrary dimensions under partial wetting
conditions, i.e., condition \reftext{\eqref{contact1}} is replaced by the less
degenerate condition $\left(\nu, \nabla h\right) = g$ for a function
satisfying point-wise bounds $g \ge \varepsilon > 0$. In this case local-in-time
existence of a unique smooth solution (implying smoothness of the free
boundary as well) has been shown.

\subsection{Link to fractional Laplacian free boundary problems} \label{sec:fraclap}

We further point out some parallels to other free boundary problems,
such as the obstacle problems for the Laplacian and the extension method of
Caffarelli and Silvestre \cite{cs.2007} for the fractional Laplacian.

To start with, compare the operator $\nabla\cdot\left(h^2 \, \nabla
\Delta h\right)$ from \reftext{\eqref{tfe_higher}} to the operators $\nabla
\cdot
\left( {\left\lvert y \right\rvert}^\gamma\nabla
h\right)$ or $\nabla\cdot\left( {\left\lvert y \right
\rvert}^\gamma\nabla\Delta h\right)$, appearing in the formulation
of the
fractional Laplacian $(-\Delta)^s$, for \mbox{$0<s<1$}, and for $1<s<2$
(described e.g. in Yang \cite{y.2013}), respectively. We may interpret
$ {\left\lvert y \right\rvert}^\gamma$ as a simplified
version of the more nonlinear
behavior of the height $h^2$ near the free boundary.

In the case of the obstacle problem for the fractional Laplacian with
$0<s<1$, which is the most classical free boundary problem in this
case, it was found in Caffarelli, Salsa, and Silvestre \cite{css.2008}
and Caffarelli, Ros-Oton, and Serra \cite{crs.2017} that a
power-like expansion in terms of the distance to the free boundary
analogous to \reftext{\eqref{exp_1d}} is available at regular points. Similar
studies for higher order fractional Laplacians which are more closely
related to our problem, e.g. for $s=3/2$ and with a boundary condition
such as \reftext{\eqref{contact1}}, seem not to have been developed yet.

In Caffarelli and V{\'a}zquez \cite{cv.1995}, as well as
Serfaty and Serra \cite{ss.2017} the obstacle problem with an evolving free
boundary, motivated by statistical physics problems, was studied using
a hodograph transform adapted to that problem for the case of the
Laplacian, and we refer to Caffarelli and Vasseur \cite{cv.2010}, as well as
Kiselev, Nazarov, and Volberg \cite{knv.2007} for parabolic
versions of the fractional Laplacian evolution, with applications in
fluid dynamics. It would be interesting to push the link between our
problem and such higher-order fractional Laplacian problems further, in
the hope of gaining a more geometric understanding of equation \reftext{\eqref{tfe_free}}.

\subsection{Perturbations of traveling waves}\label{sec:perturb}

Due to the choice of complete-wetting boundary conditions (cf.~\reftext{\eqref{contact1}}), the generic situation is the one in which the thin fluid
film will ultimately cover the whole surface. Close to the contact
line, assumed here to be almost straight, one may model this situation
by a traveling wave solution to \reftext{\eqref{tfe_higher}}, which we assume for
convenience to move in the $z$-direction:
%
%
\begin{subequations}\label{tw_profile}
%
%
\begin{equation}\label{tw_profile_1}
h_\mathrm{TW}(t,y,z) := H(x), \quad\mbox{where} \quad x = z - V t
\end{equation}
and $H = H(x)$ is a fixed profile moving with velocity $V$. Using \reftext{\eqref{tw_profile}} in \reftext{\eqref{tfe_higher}} gives the fourth-order ordinary
differential equation (ODE)
\begin{equation*}
- V \frac{\mathrm{d} H}{\mathrm{d} x} + \frac{\mathrm{d}}{\mathrm
{d} x} \left(H^2 \frac{\mathrm{d}^3 H}{\mathrm{d}
x^3}\right) = 0 \quad\mbox{for} \quad x > 0.
\end{equation*}
With help of \reftext{\eqref{contact1}}, we obtain
$\frac{\mathrm{d}}{\mathrm{d} x}H(0)=0$,
and up to changing the variable $x$ by a translation
we have $\partial\{h > 0\}=\{z=Vt\}=\{x=0\}$ in this case,
which gives the condition $H(0)=0$.
Condition~\reftext{\eqref{contact2}} gives $\left(H \frac{\mathrm{d}^3
H}{\mathrm{d} x^3}\right
)(0) = V$ and by an easy integration we find
\begin{equation*}
H \frac{\mathrm{d}^3 H}{\mathrm{d} x^3} = V \quad\mbox{for} \quad
x > 0.
\end{equation*}
A traveling-wave profile $H$ with support $[0,\infty)$ is given by
%
%
\begin{equation}\label{tw_profile_2}
H(x) =
\begin{cases} x^{\frac3 2} & \mbox{ for } \; x > 0, \\ 0 & \mbox{ for
} \; x \le0,
\end{cases}
\end{equation}
\end{subequations}
where $V = - \frac3 8$ is the rescaled velocity of the wave.

%
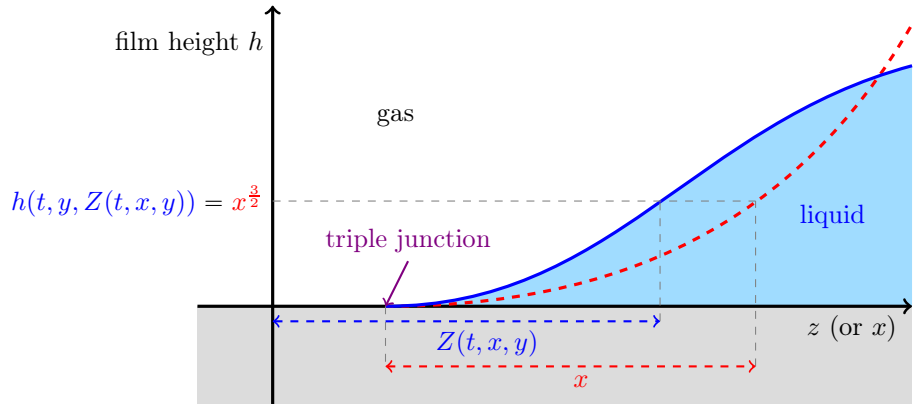
\begin{figure}[htp]
\centering
\begin{tikzpicture}[scale=1]
\path[fill=lightblue] (1.5,0) to [out=0,in=195] (8.5,3.2) -- (8.5,0) -- (1.5,0);
\path[fill=lightgray] (-1,0) -- (8.5,0) -- (8.5,-1.3) -- (-1,-1.3);

\draw [very thick,->] (-1,0) -- (8.5,0);
\draw [very thick,->] (0,-1.3) -- (0,4);

\draw[very thick,red,dashed] (1.5,0) to [out=0,in=240] (8.5,3.75);
\draw[very thick,blue] (1.5,0) to [out=0,in=195] (8.5,3.2);

\draw [gray,dashed] (5.15,-.2) -- (5.15,1.4);
\draw [gray,dashed] (6.42,-.8) -- (6.42,1.4);
\draw [gray,dashed] (1.5,0) -- (1.5,-.8);
\draw [gray,dashed] (0,1.4) -- (6.42,1.4);

\draw [thick,blue,dashed,<->] (0,-.2) -- (5.15,-.2);
\draw [blue] (2.85,-.2) node[anchor=north] {$Z(t,x,y)$};

\draw [thick,red,dashed,<->] (1.5,-.8) -- (6.42,-.8);
\draw [red] (4.1,-.8) node[anchor=north] {$x$};

\draw (0,1.4) node[anchor=east] {$\color{blue}h(t,y,Z(t,x,y))\color{black} = \color{red}x^{\frac 32}\color{black}$};

\draw (2,2.5) node[anchor=east] {gas};

\draw [blue] (8,1.2) node[anchor=east] {liquid};

\draw [thick,violet,->] (1.8,.6) -- (1.5,0);
\draw [violet] (1.8,.6) node[anchor=south] {triple junction};

\draw (0,3.5) node[anchor=east] {film height $h$};

\draw (7.7,0) node[anchor=north] {$z$ (or $x$)};

\end{tikzpicture}
\caption{Schematic of a liquid thin film and the von Mises transform \eqref{hodograph} at fixed coordinate $y$.}
\label{fig:hodograph}
\end{figure}

Our main result will establish well-posedness and stability of
perturbations of the traveling-wave profile \reftext{\eqref{tw_profile}}. As a
starting point, we transform equation~\reftext{\eqref{tfe_higher}} for $h$ to an
equation for the perturbation of $h_\mathrm{TW}$. Therefore, we define
new coordinates depending on the function $h$ and denoted by $(t,x,y)$,
which are related to the original coordinates $(t,y,z)$ via the
property that (cf. \reftext{Fig.~\ref{fig:hodograph}})
%
%
\begin{equation}\label{hodograph}
h\left(t,y,Z(t,x,y)\right) = x^{\frac3 2} \quad\mbox{for } \, t, x >
0 \quad\mbox{and} \quad y \in\mathbb{R}.
\end{equation}
The transformation interchanges dependent and independent variables and
we refer to it as the \emph{hodograph} or \emph{von Mises transform}. Since the active
coordinate is the variable $z$, under the assumption of a profile that
is strictly monotone in $z$ the transformation \reftext{\eqref{hodograph}} is
indeed well-defined, and we find that condition \reftext{\eqref{contact1}} is
automatically verified. Under this change of coordinates, the boundary
$\partial\{h > 0\}$ is transformed into the coordinate hyperplane $\{x
= 0\}$. Furthermore, in case of the traveling-wave solution
$h=h_\mathrm
{TW}$ given in \reftext{\eqref{tw_profile_1}} for $V=-\frac3 8$, we directly
find by comparison to \reftext{\eqref{tw_profile_2}} that \reftext{\eqref{hodograph}} is
satisfied by
%
%
\begin{equation}\label{ztw}
Z = Z_\mathrm{TW}(t,x,y) := x - \frac3 8 t.
\end{equation}

In Appendix~\ref{app:trafo_fix} we provide details on how the
free-boundary problem \reftext{\eqref{tfe_free}} transforms under the change of
coordinates \reftext{\eqref{hodograph}} into
%
%
\begin{equation}\label{tfe_transformed}
\begin{aligned}
& Z_t + F^{-1} \Big( D_y^2 - D_y G \left(D_x - \tfrac1 2\right) - G
D_y \left(D_x + \tfrac3 2\right) + G \left(D_x + \tfrac3 2\right) G
\left(D_x - \tfrac1 2\right) \\
& \qquad\qquad+ F \left(D_x + \tfrac3 2\right) F \left(D_x -
\tfrac12\right)\Big) \Big(D_y G - G \left(D_x + \tfrac1 2\right) G - F
\left
(D_x + \tfrac1 2\right) F\Big) = 0
\end{aligned}
\end{equation}
for $(t,x,y) \in(0,\infty)^2 \times\mathbb{R}$, where
%
%
\begin{equation}\label{def_fg}
F := Z_x^{-1} \qquad\mbox{and} \qquad G:= Z_x^{-1} Z_y
\end{equation}
and where we have introduced the operators
%
%
\begin{equation}\label{log_der}
D_x := x \partial_x = \partial_s \quad\mbox{with} \quad s := \ln x
\quad\mbox{and} \quad D_y := x \partial_y.
\end{equation}
Note that $D_x$ and $D_y$ do \emph{not} commute. For later purpose, we
also define
%
%
\begin{equation}\label{der_scale}
D := \left(D_x, D_y\right) \stackrel{\text{\reftext{\eqref{log_der}}}}{=} \left(x
\partial_x, x \partial_y\right) \quad\mbox{and} \quad D^\ell:=
D_y^{\ell_y}D_x^{\ell_x}, \quad\mbox{with} \quad\ell:= \left
(\ell
_x,\ell_y\right) \in\mathbb{N}_0^2,
\end{equation}
where the ordering of operators is crucial.

\subsection{The nonlinear Cauchy problem}

As a next step, we linearize equation~\reftext{\eqref{tfe_transformed}} around
the traveling-wave solution \reftext{\eqref{ztw}} by setting
%
%
\begin{equation}\label{def_v}
v := Z - Z_\mathrm{TW} \quad\mbox{with} \quad Z_\mathrm{TW}(t,x,y)
= x
- \frac3 8 t.
\end{equation}
Thus $Z_t = - \frac3 8 + v_t$ and by \reftext{\eqref{def_fg}} we observe that
\begin{equation*}
F^{-1} = 1+ v_x, \quad F = 1 - v_x + \mathrm{h.o.t.}, \quad\mbox
{and} \quad G =
v_y + \mathrm{h.o.t.},
\end{equation*}
where $\mathrm{h.o.t.}$ denotes terms of higher order (super-linear in
$\{v_x,
v_y\}$ or containing a term of the form $w \partial$, where $w \in\{
v_x, v_{y_j}\}$ and $\partial\in\{\partial_x,\partial_y\}$). We use
this in \reftext{\eqref{tfe_transformed}} and obtain
\begin{equation*}
D_y G - G \left(D_x + \tfrac1 2\right) G - F \left(D_x +
\tfrac12\right) F = - \frac1 2 + x^{-1} \left(D_x^2 + D_y^2\right) v +
\mathrm{h.o.t.}
\end{equation*}
Furthermore, we also have the operator identity
\begin{align*}
& D_y^2 - D_y G \left(D_x - \tfrac1 2\right) - G D_y \left(D_x +
\tfrac3 2\right) + G \left(D_x + \tfrac3 2\right) G \left(D_x -
\tfrac1 2\right) + F \left(D_x + \tfrac3 2\right) F \left(D_x -
\tfrac1 2\right) \\
& \quad= D_y^2 + \left(D_x + \tfrac3 2\right) \left(D_x - \tfrac12\right) + \frac1 2 x^{-1} \left(D_y^2 + D_x + 2\right) v + \mathrm{h.o.t.}
\end{align*}
Utilizing this in \reftext{\eqref{tfe_transformed}}, we obtain the \emph
{nonlinear Cauchy problem}
%
%
\begin{subequations}\label{nonlin_cauchy}
%
%
\begin{align}
x \partial_t v + q(D_x) v + D_y^2 r(D_x) v + D_y^4 v &= {\mathcal
N}(v) && \mbox
{for} \quad(t,x,y) \in(0,\infty)^2 \times\mathbb{R},\quad\label{eq_transformed}
\eqncr v_{|t = 0} &= v^{(0)} && \mbox{for} \quad(x,y) \in(0,\infty) \times\mathbb{R},
\end{align}
\end{subequations}
for given initial data $v^{(0)} = v^{(0)}(x,y) : \, (0,\infty) \times
\mathbb{R}\to\mathbb{R}$, where we have introduced the polynomial symbols
%
%
\begin{subequations}\label{poly_qr}
%
%
\begin{align}
q(\zeta) &:= \left(\zeta+ \tfrac1 2\right) \zeta\left(\zeta
^2-\tfrac
{3}{2}\zeta-\tfrac1 4\right)= \left(\zeta+ \tfrac1 2\right)
\left
(\zeta+ \beta- \tfrac1 2\right) \zeta\left(\zeta- \beta-
1\right
), \label{poly_q}
\eqncr
r(\zeta) &:= 2 \left(\zeta+ 1\right) \left(\zeta+ \tfrac12\right
),\label{poly_r}
\end{align}
\end{subequations}
with the irrational root
%
%
\begin{equation}\label{def_beta}
\beta= \frac{\sqrt{13}-1}{4} \in\left(\tfrac1 2,1\right).
\end{equation}
The nonlinearity ${\mathcal N}(v)$ is given by
%
%
\begin{equation}\label{def_nonlinearity}
\begin{aligned}
{\mathcal N}(v) := \, & - x F^{-1} \Big( D_y^2 - D_y G \left(D_x -
\tfrac12\right) - G D_y \left(D_x + \tfrac3 2\right) + G \left(D_x +
\tfrac32\right) G \left(D_x - \tfrac1 2\right) \\
& + F \left(D_x + \tfrac3 2\right) F \left(D_x - \tfrac1 2\right
)\Big
) \Big(D_y G - G \left(D_x + \tfrac1 2\right) G - F \left(D_x +
\tfrac
1 2\right) F\Big) \\
& + \frac3 8 x + q(D_x) v + D_y^2 r(D_x) v + D_y^4 v
\end{aligned}
\end{equation}
and is super-linear in $D^\ell v$, where $\ell\in\mathbb{N}_0^2$
with $1 \le
 {\left\lvert\ell\right\rvert} \le4$. We postpone a
precise characterization of its
algebraic structure to later sections (cf.~\S \ref{sec:nonlinear}) and
first concentrate on the characterization of the linear operator
%
%
\begin{equation}\label{def_l}
x \partial_t + {\mathcal L}= x \partial_t + {\mathcal L}\left
(x,\partial_x,\partial
_y\right) := x \partial_t + q(D_x) + D_y^2 r(D_x) + D_y^4.
\end{equation}
%

\subsection{The linearized evolution and loss of regularity at the free
boundary}\label{sec:lin_evol}

In this section, we develop a heuristic understanding of the properties
of the linear equation $\left(x \partial_t + {\mathcal L}\right) u =
f$, that is, we study the inhomogeneous linear problem (cf.~\reftext{\eqref{def_l}})
%
%
\begin{subequations}\label{lin_cauchy}
%
%
\begin{align}
x \partial_t v + q(D_x) v + D_y^2 r(D_x) v + D_y^4 v &= f && \mbox{for}
\quad(t,x,y) \in(0,\infty)^2 \times\mathbb{R},\quad \label{eq_lin}
\eqncr
v_{|t = 0} &= v^{(0)} && \mbox{for} \quad(x,y) \in(0,\infty) \times
\mathbb{R}
\end{align}
\end{subequations}
for a given right-hand side $f$. It appears to be convenient to apply
the Fourier transform in the variable $y$, that is, we set
%
%
\begin{equation}\label{ft_y}
\hat v(t,x,\eta) := \frac{1}{\sqrt{2 \pi}} \int_{-\infty}^\infty
e^{- i
y \eta} v(t,x,y) \, \mathrm{d} y.
\end{equation}
Thus, problem~\reftext{\eqref{lin_cauchy}} is transformed into
%
%
\begin{subequations}\label{lin_cauchy_f}
%
%
\begin{align}
x \partial_t \hat v + q(D_x) \hat v - x^2 \eta^2 r(D_x) \hat v + x^4
\eta^4 \hat v &= \hat f && \mbox{for} \quad(t,x,\eta) \in(0,\infty)^2
\times\mathbb{R},\qquad \label{lin_pde}
\eqncr
\hat v_{|t = 0} &= \hat v^{(0)} && \mbox{for} \quad(x,\eta) \in
(0,\infty) \times\mathbb{R}.
\end{align}
\end{subequations}
For sufficiently regular functions $\hat v$ and $\hat f$ at the contact
line, the terms $x \partial_t \hat v$, $x^2 \eta^2 r(D_x) \hat v$, and
$x^4 \eta^4 \hat v$ are higher-order corrections and the equation is
mainly dominated by the linear operator $q(D_x)$. The kernel of this
operator is given by
\begin{equation*}
\ker q(D_x) = \mbox{span}\{x^{-\gamma}: \, \gamma\mbox{ is a root
of }
q(D_x)\} \stackrel{\text{\reftext{\eqref{poly_q}}}}{=} \mbox{span}\left\{x^{-\frac1 2},
x^{\frac1 2 - \beta}, x^0, x^{1+\beta}\right\},
\end{equation*}
so that one may expect respective powers to also appear in the
solution. Yet, note that from \reftext{\eqref{poly_q}}, since $\beta\stackrel
{\text{\reftext{\eqref{def_beta}}}}{=} \frac{\sqrt{13}-1}{4} \in\left(\frac12,1\right
)$, the roots $-\frac12$ and $\frac1 2 - \beta$ of $q(\zeta)$ are
negative and therefore the powers $x^{-\frac1 2}, x^{\frac1 2 - \beta
}$ cannot appear in the expansion of the solution near $x=0$
(otherwise, because of $Z \stackrel{\text{\reftext{\eqref{def_v}}}}{=} Z_\mathrm{TW} +
v$, the contact line would be undefined). Therefore, a linear
combination of $x^0$ and $x^{1+\beta}$ is necessary and because of the
addend $x \partial_t \hat v$ a fixed-point iteration (and subsequently
undoing the Fourier transform in $y$, cf.~\reftext{\eqref{ft_y}}) will give
%
%
\begin{equation}\label{expansion_v}
D^\ell v(t,x,y) = D^\ell\left(v_0(t,y) + v_1(t,y) x + v_{1+\beta}(t,y)
x^{1 + \beta} + v_2(t,y)x^2\right) + o\left(x^{2+\delta}\right)
\end{equation}
as $x\searrow0$, where we have defined the boundary values (or \emph{traces})
%
%
\begin{subequations}\label{def_coeff}
%
%
\begin{eqnarray}
v_0(t,y) &:=& \lim_{x \searrow0} v(t,x,y),
\eqncr
v_1(t,y) &:=& \lim_{x \searrow0} x^{-1}\left(v(t,x,y) -
v_0(t,y)\right),
\eqncr
v_{1+\beta}(t,y) &:=& \lim_{x \searrow0} x^{-1-\beta}\left(v(t,x,y) -
v_0(t,y) - v_1(t,y) x\right),
\eqncr
v_2(t,y) &:=& \lim_{x \searrow0} x^{-2}\left(v(t,x,y) - v_0(t,y) -
v_1(t,y) x - v_{1+\beta}(t,y)x^{1+\beta}\right),
\end{eqnarray}
\end{subequations}
and where $\delta\in\left(0,2\beta-1\right)$ will be fixed later,
$\ell= (\ell_x,\ell_y) \in\mathbb{N}_0^2$ with $ {\left
\lvert\ell\right\rvert} := \ell_x + \ell
_y \le L+4$ and $\ell_y \le L_y+4$ for some fixed $L, L_y \in\mathbb{N}_0$,
and where we use the convention $D^\ell:= D_y^{\ell_y} D_x^{\ell_x}$.
The function $v_0 = v_0(t,y)$ determines the position of the contact
line, i.e., $Z_0(t,y) := Z(t,0,y) = - \frac3 8 t + v_0(t,y)$ is the
$z$-coordinate of the contact line, so that
\begin{equation*}
\mathbb{R}\owns y \mapsto
\begin{pmatrix} y \\ - \frac3 8 t + v_0(t,y)
\end{pmatrix}
\in\mathbb{R}^2
\end{equation*}
is the graph of the free boundary. In the analysis for the
$1+1$-dimensional counterpart \reftext{\eqref{tfe_1d}} in \cite{ggko.2014,g.2016} a slightly different transformation has been used,
that is, perturbations $u := F-1$, where $F = Z_x^{-1}$, have been
studied. The equation for the derivative in $x$ removes the position of
the contact line and appears to be convenient to prove appropriate
parabolic estimates. However, our problem cannot be transformed in a
closed problem in terms of $F$ as can be seen from the reasoning in \S \ref{sec:perturb}: the derivative $Z_y$ and thus the function $G =
Z_x^{-1} Z_y = F Z_y$, defined in \reftext{\eqref{def_fg}} and appearing in the
nonlinear equation \reftext{\eqref{tfe_transformed}}, cannot be extracted from
$F$ only. One possibility to make a stronger connection with the
setting of \cite{ggko.2014,g.2016}
is to apply the derivative $\partial_x$ to \reftext{\eqref{lin_pde}} divided by~$x$, so that we have with help of the
operator identity $x^{-1} D_x = (D_x+1) x^{-1}$
\begin{eqnarray*}
\partial_x (\partial_t + x^{-1} q(D_x)) &\stackrel{\text{\reftext{\eqref{poly_q}}}}{=}&
\partial_t \partial_x + x^{-1} D_x x^{-1} \left(D_x + \tfrac12\right)
\left(D_x + \beta- \tfrac1 2\right) \left(D_x - \beta- 1\right) D_x
\\
&\stackrel{\text{\reftext{\eqref{poly_q}}}}{=}& \partial_t \partial_x + x^{-1} D_x
\left
(D_x + \tfrac3 2\right) \left(D_x + \beta+ \tfrac1 2\right) \left
(D_x - \beta\right) \partial_x \\
&=& x^{-1} \left(x \partial_t + p(D_x)\right) \partial_x,
\end{eqnarray*}
where $p(\zeta) = \left(\zeta+ \tfrac3 2\right) \left(\zeta+
\beta+
\tfrac1 2\right) \zeta\left(\zeta-\beta\right)$ is the same
polynomial as in the works \cite{ggko.2014,g.2016}. However, again
the equation cannot be formulated in $v_x$ only, despite that we expect
in view of \reftext{\eqref{expansion_v}}
%
%
\begin{align}\label{expansion_u}
D^\ell u(t,x,y) &:= D^\ell\partial_x v(t,x,y) \nonumber\\
&= D^\ell\left(u_0(t,y) +
u_\beta(t,y) x^\beta+u_1(t,y)x\right) + o\left(x^{1+\delta}\right)
\quad\mbox{as } \, x \searrow0,
\end{align}
with some $\delta\in\left(0,2\beta-1\right)$, $ {\left
\lvert\ell\right\rvert} \le L+3$
and $\ell_y \le L_y+4$, where $u_0 = v_1$, $u_\beta= (1+\beta)^{-1}
v_{1+\beta}$, and $u_1=\frac12 v_2$. Expansion~\reftext{\eqref{expansion_u}} is
in line with the findings of \cite{ggko.2014,g.2016} for the 
$1+1$-dimensional case.

We further notice that by inserting expansion~\reftext{\eqref{expansion_v}} into \reftext{\eqref{lin_pde}}, we obtain the following expansion for the right-hand
side $f$:
%
%
\begin{equation}\label{expansion_f}
D^\ell f(t,x,y) = D^\ell\left(f_1(t,y) x+f_2(t,y)x^2\right) + o\left
(x^{2+\delta}\right) \quad\mbox{as} \quad x \searrow0,
\end{equation}
where $\delta\in\left(0,2\beta-1\right)$, $\ell\in\mathbb
{N}_0^2$ with
$ {\left\lvert\ell\right\rvert} \le L$ and $\ell_y
\le L_y$. Linear estimates will depend
on condition~\reftext{\eqref{expansion_f}}. Note that it is nontrivial to see
that the nonlinearity ${\mathcal N}(v)$ (cf.~\reftext{\eqref{def_nonlinearity}}) meets
this constraint.

In what follows we will not distinguish anymore in the notation between
Fourier transformed quantities $\hat f = \hat f(t,x,\eta)$ and
functions $f = f(t,x,y)$ since this will be apparent from the context
and from the choice of letters $y$ versus $\eta$.

\section{Main results and outline}

\subsection{Norms}\label{sec:norms}

Our main results concern existence, uniqueness, and stability of
solutions $v$ to the nonlinear Cauchy problem \reftext{\eqref{nonlin_cauchy}} and
control of the free boundary $\partial\{h > 0\}$ for initial data
$v^{(0)}$ with small norm $ {\left\lVert v^{(0)} \right
\rVert}_\mathrm{init}$, where we define
%
%
\begin{eqnarray}\label{norm_init}
 {\left\lVert v^{(0)} \right\rVert}_\mathrm{init}^2 &:=&
 {\left\lVert v^{(0)} \right\rVert}_{k,-1-\delta}^2 + {\left\lVert D_x v^{(0)} \right\rVert}_{\tilde k,
-\delta}^2 + {\left\lVert\tilde q(D_x) D_x v^{(0)}
\right\rVert}_{\tilde k, \delta
}^2+ {\left\lVert\tilde q(D_x)D_xv^{(0)} \right\rVert
}_{\check k,-\delta+1}^2\nonumber\\
&&+ {\left\lVert(D_x-3)(D_x-2)\tilde q(D_x)D_xv^{(0)}
\right\rVert}_{\check k,\delta
+1}^2+ {\left\lVert D_y \tilde q(D_x) D_x v^{(0)} \right\rVert}_{\breve
k,-\delta+2}^2
\end{eqnarray}
on the space of all locally integrable $v^{(0)}: \, (0,\infty) \times \mathbb R \to \mathbb R$ with ${\left\lVert v^{(0)} \right\rVert}_\mathrm{init} < \infty$.
Here,
\begin{equation*}
\tilde q(\zeta) := \left(\zeta+ \frac1 2\right) \left
(\zeta+ \beta- \frac1 2\right) \left(\zeta- 1\right)
\left(\zeta- \beta- 1\right)
\end{equation*}
is a fourth-order polynomial,
the integer parameters $k$, $\tilde k$, $\check k$, and $\breve k$
and the positive parameter $\delta$ will be specified later (see~\reftext{Assumptions~\ref{ass:parameters}}), and we have introduced the norm $\left\lVert\cdot\right\rVert_{k,\alpha}$, where
\begin{equation}\label{norm_2d_1}
 {\left\lVert v \right\rVert}_{k,\alpha}^2 := \sum_{0
\le j + j^\prime\le k} \int_{-\infty
}^\infty\int_0^\infty x^{-2 \alpha}\! \left(D_y^j D_x^{j^\prime}
v\right
)^2 x^{-2} \, \mathrm{d} x \, \mathrm{d} y \stackrel{\text{\reftext{\eqref{der_scale}}}}{=} \sum_{0 \le
{\left\lvert\ell\right\rvert} \le k} \int_{-\infty
}^\infty\int_0^\infty x^{-2 \alpha}
\!\left(D^\ell v\right)^2 x^{-2} \, \mathrm{d} x \, \mathrm{d} y
\end{equation}
with $k \in\mathbb{N}_0$ and $\alpha\in\mathbb{R}$. We will also
use the shorthand
$ {\left\lVert v \right\rVert}_\alpha:=
{\left\lVert v \right\rVert}_{0,\alpha}$. The norm ${\skliaustask
\cdot\skliaustasd}_\mathrm{Sol}$ for the solution $v$ is more
involved and given by
\begin{eqnarray}\label{norm_sol_simple}
{\skliaustask v\skliaustasd}_\mathrm{Sol}^2 &:=& \int_I \left
({\left\lVert\partial_t v \right\rVert}_{k - 2, - \delta- \frac32}^2 
+ {\left\lVert\partial_t D_x v \right\rVert}_{\tilde k
- 2, - \delta- \frac1 2}^2 \right) \mathrm{d}
t\nonumber\\
&& + \int_I \left( {\left\lVert\partial_t
\tilde q(D_x) D_x v \right\rVert}_{\tilde k - 2, \delta- \frac12}^2 + {\left\lVert\partial_t \tilde q(D_x) D_x v \right
\rVert}_{\check k - 2, - \delta+ \frac1 2}^2 \right) \mathrm{d} t
\nonumber\\
&& + \int_I \left( {\left\lVert\partial_t (D_x-3)
(D_x-2) \tilde q(D_x) D_x v \right\rVert}_{\check k - 2, \delta+
\frac1 2}^2 + {\left\lVert\partial_t D_y \tilde q(D_x) D_x v \right
\rVert}_{\breve k - 2, - \delta+ \frac3 2}^2 \right) \mathrm{d} t \nonumber\\
&& + \int_I \left( {\left\lVert v \right\rVert
}_{k + 2, - \delta- \frac1 2}^2
+ {\left\lVert D_x v \right\rVert
}_{\tilde k + 2, - \delta+ \frac1 2}^2
+ {\left\lVert\tilde q(D_x) D_x v \right\rVert}_{\tilde
k + 2, \delta+ \frac1 2}^2 \right) \mathrm{d} t \nonumber\\
&& + \int_I \left( {\left\lVert\tilde q(D_x) D_x v \right\rVert}_{\check k
+ 2, - \delta+ \frac32}^2 + {\left\lVert(D_x-3) (D_x-2) \tilde
q(D_x) D_x v \right\rVert}_{\check
k + 2, \delta+ \frac3 2}^2 \right) \mathrm{d} t \nonumber\\
&& + \int_I {\left\lVert D_y \tilde q(D_x) D_x v \right
\rVert}_{\breve k+2, - \delta+
\frac5 2}^2 \mathrm{d} t,
\end{eqnarray}
where $I \subseteq[0,\infty)$ is the time interval on which
problem~\reftext{\eqref{nonlin_cauchy}} is solved and 
the integer parameters
$k\ge 2$, $\tilde k\ge2$, $\check k\ge2$, $\breve k \ge 2$,
and the positive parameter $\delta$ will be determined
in what follows (see~\reftext{Assumptions~\ref{ass:parameters}}). Similarly to the case of the norm
${\left\lVert v^{(0)} \right\rVert}_\mathrm{init}$,
the quantity ${\skliaustask \cdot\skliaustasd}_\mathrm{Sol}$ gives a norm
on locally integrable functions $v: \, I \times (0,\infty) \times \mathbb R \to \mathbb R$ with ${\skliaustask v\skliaustasd}_\mathrm{Sol} < \infty$.

The significance of these norms will become clear from
the outline in \S \ref{sec:outline} and the heuristic discussions in
\S \ref{sec:coerc} and \S \ref{sec:lin_heur} while a rigorous justification
of the corresponding estimates is the subject of \S \ref{sec:lin_rig}.
Here, we just briefly motivate their choice:

First notice that for locally integrable $v$ with $ {\left
\lVert v \right\rVert}_\alpha<
\infty$ for some $\alpha\in\mathbb{R}$ we necessarily have $v =
o\left
(x^{\alpha+\frac12}\right)$ as $x \searrow0$ almost everywhere. By a standard embedding, $ {\left
\lVert v \right\rVert}_{2,\alpha
} < \infty$ implies $v = o\left(x^{\alpha}\right)$ as $x \searrow0$
classically. In this sense, the larger $\alpha$, the better the decay
of $v$ as $x \searrow0$, that is, $\alpha$ is linked to the regularity
of $v$ at the free boundary $\{x = 0\}$. On the other hand, just
increasing $k$ only leads to more regularity in the bulk $\{x > 0\}$
but this regularity is lost as $x \searrow0$ because each derivative
$D_x \stackrel{\text{\reftext{\eqref{der_scale}}}}{=} x \partial_x$ or $D_y \stackrel
{\text{\reftext{\eqref{der_scale}}}}{=} x \partial_y$ carries the factor $x$. Since we
expect $v$ and therefore also generic initial data $v^{(0)}$ to have an
expansion of the form \reftext{\eqref{expansion_v}}, our initial-data norm
$ {\left\lVert\cdot\right\rVert}_\mathrm{init}$
cannot have contributions $ {\left\lVert v^{(0)} \right
\rVert}_{k,\alpha}$ with $\alpha\ge0$ because $v^{(0)}_0 = \lim_{x
\searrow0} v^{(0)} \ne0$ in general. Nevertheless, we need to use
larger weights $\alpha$ as otherwise scaling-wise a control of the
Lipschitz constant of $v^{(0)}$ would be impossible. This control is
necessary because only for $v^{(0)}$ with small Lipschitz constant the
function $Z = Z(t,x,y)$ (cf.~\reftext{\eqref{def_v}}) is strictly increasing in
$x$ and thus the transformation \reftext{\eqref{hodograph}} is well-defined, as
reflected by the occurrence of factors $(1+v_x)^{-1}$ in the
nonlinearity ${\mathcal N}(v)$ (cf.~\S \ref{sec:nonlinear_struct}).

A way to allow for norms $ {\left\lVert\cdot\right
\rVert}_{k,\alpha}$ with larger $\alpha
$ is to apply operators $D_x - \gamma$ to $v$, where $\gamma\in\{k +
\beta\ell: \, (k, \ell) \in\mathbb{N}_0^2 \setminus\{(0,1)\}\}$, since
$x^\gamma$ spans the kernel of $D_x - \gamma$. Indeed, it is immediate
that for instance $D_x v = O(x)$ and $\tilde q(D_x) D_x v = O\left(x^2\right)$ as $x
\searrow0$ given expansion~\reftext{\eqref{expansion_v}}. Then it can be proved
that $ {\left\lVert v^{(0)} \right\rVert}_\mathrm{init}$
controls the Lipschitz constant
of $v^{(0)}$ and ${\skliaustask v\skliaustasd}_\mathrm{Sol}$ the
supremum in time of the
Lipschitz constant of $v$ in $(x,y)$ (cf.~\reftext{Lemmata~\ref{lem:equivalence_init},
\ref{lem:equivalence_sol}, and \ref{lem:bc0_bounds}}, \S \ref{sec:norms_prop}).

\subsection{Main result}

For our main result we need to make the following assumptions on the
number of derivatives $k$, $\tilde k$, $\check k$, and $\breve k$:
\begin{assumptions}\label{ass:parameters}
The numbers $k$, $\tilde k$, $\check k$, and $\breve k$
entering the
definitions of the norms ${\skliaustask\cdot\skliaustasd}_\mathrm
{Sol}$ (cf.~\reftext{\eqref{norm_sol_simple}}), $ {\left\lVert\cdot
\right\rVert
}_\mathrm{init}$ (cf.~\reftext{\eqref{norm_init}}), and ${\skliaustask\cdot
\skliaustasd}_\mathrm{rhs}$
(cf.~\reftext{\eqref{norm_rhs}})
meet the conditions
\begin{equation}\label{parameters}
\max\left\{13, \left\lfloor\frac{k+1}2\right\rfloor+3,\breve k+8\right\}\le \min\left\{\tilde k, \check k\right\}, \quad \max\left\{\tilde k, \check k+2\right\}+4\le k, \quad 4 \le \breve k.
\end{equation}
\end{assumptions}

We remark that we can satisfy conditions~\reftext{\eqref{parameters}} for the
explicit choice $k=19$, $\tilde k=\check k =13$, and $\breve k = 4$.

\medskip

Our main result reads as follows:
\begin{theorem}[Well-posedness and stability]\label{main_th}
Suppose that $\delta \in \left(0,\frac12\left(\beta-\frac12\right
)\right]$,
the numbers $k, \tilde k, \check k,\breve k \in\mathbb{N}$ satisfy
conditions~\reftext{\eqref{parameters}} of \reftext{Assumptions~\ref{ass:parameters}}, and
$I = [0,\infty)$. Then there exists $\varepsilon > 0$ such that if
$v^{(0)} =
v^{(0)}(x,y): \, (0,\infty) \times\mathbb{R}\to\mathbb{R}$ is
locally integrable and
satisfies $ {\left\lVert v^{(0)} \right\rVert}_\mathrm
{init} \le \varepsilon$, then problem~\reftext{\eqref{nonlin_cauchy}}
has a unique locally integrable solution
$v = v(t,x,y)\colon (0,\infty)^2 \times\mathbb{R}\to\mathbb{R}$
such that $\skliaustask v \skliaustasd_\mathrm{Sol} < \infty$.
This solution satisfies the a-priori estimate
\begin{equation}\label{main_estimate}
{\skliaustask v\skliaustasd}_\mathrm{Sol} \le C  {\left
\lVert v^{(0)} \right\rVert}_\mathrm{init},
\end{equation}
where $C$ only
depends on $k$, $\tilde k$, $\check k$, $\breve k$, and $\delta$.
Furthermore, we have
\[
\lVert v(t,\break\cdot,\cdot)
\rVert_\mathrm{init} \to 0 \quad \mbox{as}
\quad t \to\infty,
\]
i.e., the traveling-wave $v_\mathrm{TW}= 0$ is
asymptotically stable.
\end{theorem}

We recall that \reftext{\eqref{eq_transformed}} is obtained from \reftext{\eqref{tfe_higher}} via the coordinate changes \reftext{\eqref{hodograph}} and \reftext{\eqref{def_v}}. Inverting these transformations, the above result states that
we can ensure global existence and uniqueness under the condition that
in a suitable norm the initial data $h_{|t = 0}$ are close to the
traveling wave $h_\mathrm{TW}(t,z) = H\left(z - V t\right)$
(cf.~\reftext{\eqref{tw_profile}}). Moreover, as we will find in \S \ref{sec:transforiginal},
in the original coordinates the solution $h$ approaches the traveling wave $h_\mathrm{TW}$ as time $t \to\infty$ and in this sense,
$h_\mathrm{TW}$ is asymptotically stable.

We note that the principal improvement in the above statement with
respect to previous works for Navier slip and complete-wetting boundary
conditions in \cite{ggko.2014,g.2016}, is that we work in physical
space dimensions. Due to this, we meet the difficulty of designing
appropriate norms that account for the additional tangential coordinate
$y$. This leads to the additional contributions $D_y^2 r\left
(D_x\right
) v$ and $D_y^4 v$ in the linearized evolution \reftext{\eqref{eq_lin}} which
need to be absorbed in the linear estimates (cf.~\ref{sec:lin_heur}),
leading to a restriction of the range of weights $\alpha$ for which the
spatial part of the linear operator is coercive. As a consequence, we
obtain a hierarchy of estimates that need to be combined appropriately
in order to obtain sufficient control on the solution $v$. As explained
already in \S \ref{sec:lin_evol}, the transformations in
\cite{ggko.2014,g.2016} are not directly applicable in our case. In
particular, the boundary value $v_0 = v_0(t,y)$, determining the
position of the contact line $Z_0(t,y)$ (cf.~\reftext{\eqref{hodograph}}\&\reftext{\eqref{def_v}}), cannot be eliminated from our problem. Furthermore, the
nonlinearity ${\mathcal N}(v)$ cannot be expressed as a sum of
multilinear forms
as in \cite{ggko.2014,g.2016}, but is rather a rational
function of $
\begin{pmatrix} D^\ell v: \, \ell\in\mathbb{N}_0^2, \, 1 \le
 {\left\lvert\ell\right\rvert} \le4
\end{pmatrix}
$. This complicates the symmetry considerations and the algebraic
structure of the nonlinear terms discussed in \S \ref{sec:nonlinear_struct} and makes the proof of appropriate estimates for
${\mathcal N}(v)$ (cf.~\S \ref{sec:nonlinear_est}) more involved.

A rather natural question in the context of parabolic problems is how
the regularity of the free boundary propagates. This question has been
studied by Kienzler \cite{k.2016}, Koch \cite{k.1999}, and Seis \cite{s.2015} for the porous-medium equation \reftext{\eqref{pme}} and by John \cite{j.2015} and Seis \cite{s.2017} in the case of the thin-film
equation with linear mobility (i.e., for the case of \reftext{\eqref{tfe_general}} with $n=1$), but it has not been addressed in physical
dimensions for the more realistic case $n=2$ with complete-wetting
boundary conditions treated here. We do not prove a regularizing effect
in the tangential variables and we only prove partial regularity in the
normal variables. Note that through the a-priori estimate \reftext{\eqref{main_estimate}} and finiteness of
\begin{equation*}
\int_I  \left({\left\lVert v \right\rVert
}_{k + 2, - \delta- \frac1 2}^2
+ {\left\lVert(D_x-3) (D_x-2) \tilde q(D_x) D_x v
\right\rVert}_{\check k + 2,
\delta+ \frac3 2}^2\right) \mathrm{d} t
\end{equation*}
(cf.~\reftext{\eqref{norm_sol_simple}}) in conjunction with \reftext{Lemma~\ref{lem:mainhardy}} and the fact that $1$, $x$, $x^{1+\beta}$ and $x^2$ are
in the kernel of $(D_x-3) (D_x-2) \tilde q(D_x) D_x$, we infer that the
solution obeys the asymptotic expansion
\begin{equation}\label{as_sol}
D^\ell v(t,x,y) = D^\ell\left(v_0(t,y) + v_1(t,y) x + v_{1+\beta}(t,y)
x^{1+\beta} +v_2(t,y)x^2\right) + o\left(x^{2 + \delta}\right)
\end{equation}
as $x \searrow0$ almost everywhere, where we have $\ell\in\mathbb{N}_0^2$
satisfying the bounds $ {\left\lvert\ell\right\rvert
}\le\check k + 9$ and $\ell_y\le
\check k + 2$. Here, the coefficients $v_0$, $v_1,v_{1+\beta}$, $v_2$
fulfill a-priori estimates in terms of $ {\left\lVert
v^{(0)} \right\rVert}_\mathrm{init}$
in the following normed spaces (cf.~\reftext{Lemma~\ref{lem:bc0_bounds}}):
\begin{subequations}\label{bc0_bounds_summ}
\begin{align}
v_0 &\in BC^0\left([0,\infty);BC^1(\mathbb
{R})\cap H^2(\mathbb R)\right) \cap L^2\left([0,\infty);H^2(\mathbb{R})\right), \\
v_1 &\in BC^0\left([0,\infty);BC^0(\mathbb{R})\cap H^1(\mathbb R)\right)\cap L^2\left([0,\infty);H^1(\mathbb{R})\right), \\
v_{1+\beta} &\in L^2\left((0,\infty);H^1(\mathbb{R})\right
), \\
v_2&\in L^2\left((0,\infty)\times \mathbb{R}\right).
\end{align}
\end{subequations}
An analogous partial regularity result has been found in \cite[Equation~(3.1)]{ggko.2014}
and upgraded to \reftext{\eqref{exp_1d}} in \cite[Equation~(2.4)]{g.2016}.
Our expectation is that the full regularity study
of \cite{g.2016} can be adapted for proving the power
expansion of the solution near the free boundary to arbitrary order in
the higher-dimensional case considered here. More precisely, our
expectation is that the unique solution $v$ to the nonlinear Cauchy
problem \reftext{\eqref{nonlin_cauchy}} fulfills
\begin{equation}\label{exp_2d}
v(t,x,y) =v_0(t,y) + \sum_{\substack{(k,\ell) \in\mathbb{N} \times \mathbb{N}_0 \\
k + \beta\ell< N}} v_{k+\beta\ell}(t,y) x^{k+\beta\ell} + O\left
(x^N\right) \quad\mbox{as} \quad x \searrow0 \quad\mbox{classically},
\end{equation}
where $N \in\mathbb{N}$ is arbitrary and the functions $v_{k+\beta
\ell} =
v_{k + \beta\ell}(t,y)$ are (at least) continuous. On the other hand,
the main new difficulties due to the introduction of the extra
dimensions appear already for the setting present here, and the
extension to \reftext{\eqref{exp_2d}} following \cite{g.2016} would add a second
layer of technical detail. Therefore, we believe that investigating the
regularizing effect in the tangential as well as normal variables is
better suited for a separate future work.

We note that we expect our strategy to be applicable also for
(nonphysical) higher space dimensions, at the cost of some extra
technical difficulties. One approach could be by replacing the present
framework based on $L^2$-norms by a framework using $L^p$-norms with
$p>2$. This is required in order to obtain the spatial regularity in
the extra dimensions, for which in the current framework we use the
Sobolev embedding of $W^{k,2}$-spaces into $BC^\ell$-spaces in our
nonlinear estimates. For the case of $1+2$ dimensions, we may already
directly extend our bounds to $L^p$-bounds via the openness of the
range of exponents for which $L^p$-estimates are valid, as available
for instance due to Kalton and Mitrea \cite[Theorem~2.5]{km.1998}. However, for more
general dimension $d+1$ where $d > 2$, we would require to use more
general Sobolev embeddings, and the full strength of Calder\'on--Zygmund
estimates for $q(D_x)$. Note that while the basic embeddings are
available also in $L^p$, it is not yet clear how to prove the required
$L^p$ maximal-regularity for the linearized operator for general $p \in
(2,\infty)$ (see the study of Pr{\"u}ss and Simonett \cite{ps.2008}, where this question has
been addressed but not completely solved).

Another possible approach would be to apply more $\partial
_y$-derivatives to the linearization \reftext{\eqref{lin_cauchy}} of \reftext{\eqref{tfe_free}}, which then retains the same form. Hence, the linear
estimates that we are going to develop,
remain valid upon taking an arbitrary
number of $\partial_y$-derivatives. This implies better control in the
tangential directions $y$ and by Sobolev embedding, we obtain the
desired $BC^\ell$-bounds as well. However, in this case, treating the
nonlinear terms would produce further case distinctions in the products
that appear when treating the nonlinearity.

In either approach, while being mathematically more challenging, the
extension of our results to general dimensions would add a further
layer of technical difficulties to the current work, while not being
directly motivated by a physical model. Therefore, we leave this
endeavor to future work.

\subsection{Transformation into the original set of variables}\label{sec:transforiginal}

We reformulate the statement of \reftext{Theorem~\ref{main_th}} in terms of the
quantities appearing in the original problem \reftext{\eqref{tfe_free}}. First,
due to \reftext{\eqref{def_v}} we have that
\begin{equation}
Z(t,x,y) = Z_\mathrm{TW}(t,x,y) + v(t,x,y)= x-\frac3 8 t +v(t,x,y).
\end{equation}
Via \reftext{\eqref{hodograph}} we find
\begin{equation}\label{h_via_v}
h\left(t,y,x-\tfrac3 8 t + v(t,x,y)\right)= x^{\frac3 2} \quad
\mbox
{for } \, t, x > 0 \quad\mbox{and} \quad y \in\mathbb{R},
\end{equation}
and we have seen that through \reftext{\eqref{transf_deriv}}, equation~\reftext{\eqref{tfe_higher}} for $h(t,y,z)$ is equivalent to \reftext{\eqref{tfe_transformed}},
which in terms of $v$ is re-expressed as \reftext{\eqref{eq_transformed}}.

The assumed smallness condition $ {\left\lVert v^{(0)}
\right\rVert}_\mathrm{init} < \varepsilon$
at the initial time $t = 0$ means that $Z(0,x,y) = x + v^{(0)}(x,y)$ is
a small perturbation of the linear profile $x$. By \reftext{\eqref{h_via_v}} this
in turn means that $h^{(0)}(y,z):=\lim_{t\searrow0}h(t,y,z)=h(0,y,z)$
is a small perturbation of the traveling-wave profile $h_\mathrm{TW}$
(cf.~\reftext{\eqref{tw_profile}}) at time $t=0$. Indeed, using estimate~\reftext{\eqref{bc0_grad_init}} of \reftext{Lemma~\ref{lem:bc0_bounds}} we find
\begin{equation}\label{est_bc0_intro_2}
{\left\lVert D^\ell\nabla v^{(0)} \right\rVert
}_{BC^0\left((0,\infty) \times\mathbb{R}\right)} + \left\lvert v_0^{(0)} \right\rvert_{BC^0(\R)} \le C  {\left\lVert v^{(0)} \right\rVert}_\mathrm{init}
\end{equation}
for $0 \le {\left\lvert\ell\right\rvert} \le\min
\left\{\tilde k - 2, \check k - 2\right
\}$, where $C$ is a constant
depending only on $\tilde k$, $\check k$, and $\delta$.

The transform \reftext{\eqref{hodograph}} is well defined due to the
point-wise estimate
%
%
\begin{equation}\label{Z_incr_cond}
 {\left\lvert\nabla v(t,x,y) \right\rvert} < 1 \quad
\mbox{for} \quad t, x > 0 \quad\mbox
{and} \quad y \in\mathbb{R}.
\end{equation}
Property \reftext{\eqref{Z_incr_cond}} is a consequence of the bound
$ {\left\lVert v^{(0)} \right\rVert}_\mathrm{init} <
\varepsilon$, of the a-priori estimate \reftext{\eqref{main_estimate}}, and of
%
%
\begin{equation}\label{est_bc0_intro}
{\skliaustask D^\ell \nabla v \skliaustasd
}_{BC^0\left(I \times (0,\infty) \times\mathbb{R}
\right)} \le C {\skliaustask v\skliaustasd}_\mathrm{sol}
\end{equation}
for $0 \le {\left\lvert\ell\right\rvert} \le\min
\left\{\tilde k - 2, \check k - 2\right\}
$ coming from estimate~\reftext{\eqref{bc0_grad_sol}} of \reftext{Lemma~\ref{lem:bc0_bounds}}, and where $C$
is a constant depending only on $\tilde k$, $\check k$, and $\delta$.\vspace*{1pt}

Furthermore, because of \reftext{\eqref{est_bc0_intro_2}}, $ {\left
\lVert v(t,\cdot,\cdot) \right\rVert}_\mathrm{init} \le
{\skliaustask v\skliaustasd}_\mathrm{sol}$ (cf.~\reftext{\eqref{norm_init}} and \reftext{\eqref{norm_sol_simple}}), and the a-priori estimate \reftext{\eqref{main_estimate}}, we have that $v(t,x,y)$
stays close to a translate of the
linear profile $x$ for all $t > 0$ and
\reftext{Theorem~\ref{main_th}} further
implies that
\begin{equation*}
\left\lVert D^\ell \nabla v(t,\cdot,\cdot)
\right\rVert_{BC^0\left((0,\infty) \times\mathbb{R}\right)}
+ \left\lvert v_0(t,\cdot) \right\rvert_{BC^0(\R)}  \to0 \quad \mbox{as} \quad t \to\infty
\end{equation*}
for $0 \le {\left\lvert\ell\right\rvert} \le\min\left\{\tilde
k - 2, \check k - 2\right\}$ giving stability
of the traveling-wave profile because of \reftext{\eqref{h_via_v}}.\vspace*{1pt}

The explicit computations pertaining to the remainder of this
subsection concentrate on finding expansions of $h$ and the velocity $V
= h \nabla\Delta h$ close to the free boundary and are contained in
Appendix~\ref{app:advection}. Our results can be re-expressed in terms
of the original formulation as follows. If we parametrize
\begin{eqnarray*}
\partial\{(y,z)\in\mathbb{R}^2:\, h(t,y,z) > 0\} &=&
\{
(y,Z_0(t,y)): \, y\in\mathbb{R}\},\\
\partial\{(y,z)\in\mathbb{R}^2:\, h(0,y,z) > 0\} &=&
\left\{\left
(y,Z_0^{(0)}(y)\right):\ y\in\mathbb{R}\right\},
\end{eqnarray*}
then we have almost everywhere
%
%
\begin{align}\label{hodo3}
h(t,y,z) = \tilde z^{\frac32}\left(\frac{1}{(1+v_1)^{\frac
32}}-\frac
32\frac{v_{1+\beta}}{(1+v_1)^{\frac52+\beta}} \tilde z^\beta-
\frac32
\frac{v_2}{(1+v_1)^{\frac72}}\tilde z + o\left(\tilde z^{1+\delta
}\right
)\right)\text{ as } \tilde z \searrow0,
\end{align}
where the coordinate $\tilde z$ is given in terms of the distance to
the free boundary as $\tilde z := z-Z_0(t,y)$.

In order to compute $Z_0(t,y)$, we may express $\partial\{h > 0\}$ as the
solution of the following system of ODEs
%
%
\begin{subequations}\label{advection}
%
%
\begin{align}
\partial_t(Y,Z)(t,y)&=V_0(t,Y(t,y)) \quad\mbox{for} \quad(t,y)\in
(0,\infty)\times\mathbb{R},
\eqncr
(Y,Z)(0,y)&=(y, Z_0^{(0)}(y)) \quad\mbox{for} \quad y\in\mathbb{R},
\end{align}
\end{subequations}
where $V_0(t,y)$, appearing also in \reftext{\eqref{contact2}}, can be
characterized as the first term of the asymptotic expansion of the
advection velocity $V$ from \reftext{\eqref{def_V}}
%
%
\begin{equation}\label{asy_V}
V(t,y,Z(t,x,y)) = V_0(t,y) + V_1(t,y)x + o\left(x^{1+\delta}\right)
\quad\mbox{as} \quad x \searrow0
\end{equation}
and can be expressed in terms of the function $v$ from \reftext{Theorem~\ref{main_th}} as
%
%
\begin{equation}\label{def_V0}
V_0 =
\begin{pmatrix}V_0^{(y)}\\V_0^{(z)}
\end{pmatrix}
= -\frac38 \frac{1+(v_0)_y^2}{(1+v_1)^3}
\begin{pmatrix} -(v_0)_y \\ 1
\end{pmatrix}
.
\end{equation}
We remark that singular terms in expansion~\reftext{\eqref{asy_V}} enter in
higher-order terms given that $v$ meets an expansion as in \reftext{\eqref{exp_2d}}
(for details in the $1+1$-dimensional case see
\cite[\S2.3]{g.2016}). Since~-- as pointed out in \S \ref{sec:tfe_gen} -- the
free-boundary problem \reftext{\eqref{tfe_free}} is the lubrication approximation
of the Navier--Stokes equations in a moving-cusp domain, there is the
possibility that such singular expansions also occur in the underlying
Navier--Stokes problem itself, as the velocity $V$ in \reftext{\eqref{asy_V}} is
nothing but the vertically-averaged horizontal component of the fluid
velocity. This question is in fact not addressed in the literature so
far, but it may have connections to the elliptic regularity theory in
non-smooth domains, where such expansions are known to occur
(see \cite{g.2011,kmr.1997}).

\subsection{Outline}\label{sec:outline}

We start our study in \S \ref{sec:linear_th} by studying the
linearization \reftext{\eqref{eq_lin}} of \reftext{\eqref{tfe_transformed}}, in order to
determine suitable norms which allow for a maximal-regularity estimate
of the linear problem \reftext{\eqref{lin_cauchy}} that reads
%
%
\begin{equation}\label{mr_lin}
{\skliaustask v\skliaustasd}_\mathrm{Sol} \le C \left(
{\left\lVert v^{(0)} \right\rVert}_\mathrm{init} +
{\skliaustask f\skliaustasd}_\mathrm{rhs}\right),
\end{equation}
where $C$ is a
constant depending only on $k$, $\tilde k$, $\check k$, $\breve k$, and
$\delta$, and the norm ${\skliaustask \cdot \skliaustasd}_\mathrm{rhs}$ is defined through
%
%
\begin{eqnarray}\label{norm_rhs}
{\skliaustask f\skliaustasd}_\mathrm{rhs}^2 &:=& \int_I
\left({\left\lVert f \right\rVert}_{k-2,-\delta-\frac12}^2 + {\left\lVert(D_x-1)f \right\rVert}_{\tilde
k-2,-\delta+\frac12}^2 \right)\mathrm{d} t \nonumber\\
&&+\int_I\left( {\left\lVert\tilde q(D_x-1)(D_x-1)f
\right\rVert}_{\tilde k-2,\delta
+\frac12}^2 +  {\left\lVert\tilde q(D_x-1)(D_x-1)f \right
\rVert}_{\check k - 2,-\delta
+\frac32}^2\right)\mathrm{d} t\nonumber\\
&&+ \int_I \left({\left\lVert(D_x-4)(D_x-3)\tilde
q(D_x-1)(D_x-1)f \right\rVert}_{\check
k-2,\delta+\frac32}^2 + {\left\lVert D_y \tilde q(D_x-1) (D_x-1) f \right\rVert
}_{\breve k-2, -\delta+\frac
52}^2\right) \mathrm{d} t \nonumber\\
\end{eqnarray}
We remark that ${\skliaustask \cdot \skliaustasd}_\mathrm{rhs}$ is a norm on the space of all
locally integrable $f: \, I \times (0,\infty) \times \mathbb R \to \mathbb R$ with ${\skliaustask f \skliaustasd}_\mathrm{rhs} < \infty$.
Note that for estimate \reftext{\eqref{mr_lin}} to yield maximal regularity, the
solution norm ${\skliaustask\cdot\skliaustasd}_\mathrm{Sol}$ has to
control $4$ spatial
and $1$ temporal derivative more than the norm ${\skliaustask\cdot
\skliaustasd}_\mathrm
{rhs}$ for the right-hand side $f$. It is known from the theory of
linear (higher-order) parabolic equations that $L^2$ maximal regularity
follows from \emph{coercivity} of the spatial part of the linear
operator (cf.~Mielke \cite{m.1987}). For equation~\reftext{\eqref{eq_lin}} we cannot
obtain such a coercivity bound for unweighted Sobolev norms, but only
for \emph{weighted} ones. Essentially, our elliptic estimates are based
on the study of the one-dimensional operator $q(D_x)$ in \reftext{\eqref{eq_lin}}
for which a quantitative result is given by \reftext{Lemma~\ref{lem:coercivity}}
in \S \ref{sec:coerc} giving coercivity with respect to $
(\cdot,\cdot)_\alpha$ for $\alpha\in\left(- \frac{\sqrt
{13}-3}{4},0\right)$, where $\left(v,w\right)_\alpha := \int_0^\infty x^{-2\alpha} v w \frac{\mathrm{d} x}{x}$, that is, $\left(v,q(D_x) v\right)_\alpha \ge C (v,v)_\alpha$ for $v \in C^\infty_\mathrm{c}((0,\infty))$, with $C > 0$ depending only on $\alpha$. Since
we furthermore need to absorb the terms coming from the operators
$D_y^2 r(D_x)$ and $D_y^4$, the coercivity range is further restricted
to $\left(-\frac{1}{10},0\right)$ (cf.~\S \ref{sec:basic_weak}) and we
arrive at a maximal-regularity estimate of the form
\begin{align*}
& \sup_{t \in I}  {\left\lVert D_y^j v \right\rVert
}_{k,-\delta- 1 + j}^2 + \int_I \left
( {\left\lVert\partial_t D_y^j v \right\rVert
}_{k-2,-\delta- \frac3 2 + j}^2 +  {\left\lVert D_y^j v
\right\rVert}_{k+2,-\delta- \frac1 2 + j}^2\right) \mathrm{d} t \\
& \quad\le C \left( {\left\lVert D_y^j v_{|t = 0} \right
\rVert}_{k,-\delta- 1 + j}^2 +
\int_I  {\left\lVert D_y^j f \right\rVert}_{k-2,-\delta
- \frac1 2 + j}^2 \mathrm{d} t\right)
\quad\mbox{for} \quad\delta\in\left(0,\tfrac{1}{10}\right),\
j\ge
0,\ k\ge2,
\end{align*}
and where $C$ depends only on $j$, $k$, and
$\delta$, as stated in \reftext{\eqref{estimate1_int}} of \S \ref{sec:basicstrongest}. Although we can take $k \ge2$ arbitrarily large,
this only leads to better regularity of the solution $v$ in the bulk $\{
x > 0\}$ but not at the free boundary $\{x = 0\}$. In particular, due
to the negative weight $-\delta\in\left(-\frac{1}{10},0\right)$,
control of both the Lipschitz constant of $v$ and even the position of
the free boundary determined by $v_{|x = 0} = v_0 = v_0(t,y)$ fails due
to the scaling in $x$ of the norms that appear above. For control of
$v_0$ we would have to allow for weight exponents $\alpha> 0$ while
for control of the Lipschitz constant we would require $\alpha> 1$.
Yet, these weights are excluded as a nonzero boundary value $v_0$ would
lead to blow-up of terms appearing on the left-hand side of
estimate~\reftext{\eqref{estimate1_int}}. As a first step, we apply the operator
$D_x$ to equation~\reftext{\eqref{eq_lin}}, canceling $v_0$ in expansion~\reftext{\eqref{expansion_v}} and leading to $D^\ell v = O(x)$ as $x \searrow0$, in
which $\ell= (\ell_x,\ell_y) \in\mathbb{N}_0^2$, $
{\left\lvert\ell\right\rvert}\le L+4$, $\ell_x \ge 1$, and
$\ell_y \le L_y+4$, where $L$ and $L_y$ are the total number of $D$-
and of $D_y$-derivatives controlled for $f$, respectively. The
resulting equation is \reftext{\eqref{lin_cauchy2}}, i.e.,
\begin{equation*}
\left(x \partial_t + \tilde q(D_x)\right) D_x v + D_y^2 \tilde
r(D_x) v
+ D_y^4 (D_x+3) v = (D_x-1)f,
\end{equation*}
where $\tilde q(\zeta)$ and $\tilde r(\zeta)$ are fourth- and third-order real polynomials in $\zeta$, respectively,
and $\tilde q(D_x)$ has coercivity range $(0,1)$. Unlike in the
$1+1$-dimensional case, the equation does \emph{not} have a closed form
in $D_x v$ and $(D_x-1) f$, so that additional remnant terms appear
that need to be absorbed with \reftext{\eqref{estimate1_int}}. We consider
weights $\delta$ and $1-\delta$ in this range. For the weight $1-\delta$
the resulting estimate is \reftext{\eqref{higher_est_1}} presented at the end of
\S \ref{sec:higher_1}. In order to allow for control of the norm
${\skliaustask v_x\skliaustasd}_{BC^0\left(I \times(0,\infty)
\times\mathbb{R}\right)}$ by
Sobolev embedding as in \reftext{Lemma~\ref{lem:bc0_bounds}}, we require to use
weights $2\pm\delta$. Therefore, firstly in \S \ref{sec:higher_2} we
apply the operator $\tilde q(D_x-1)$ to \reftext{\eqref{lin_cauchy2}} leading to
equation~\reftext{\eqref{lin_eq_qd}}, i.e.,
\begin{equation*}
\left( x\partial_t + \tilde q(D_x - 1)\right) \tilde q(D_x) D_x v -
\eta^2 x^2 \tilde q(D_x+1) \tilde r(D_x) v + \eta^4 x^4 \tilde
q(D_x+3)(D_x+3)v = \tilde q(D_x-1) (D_x-1) f.
\end{equation*}
Now the coercivity range of $\tilde q(D_x-1)$ contains the interval
$(1,2)$, which allows for the weight $2-\delta$, and the operator
$\tilde q(D_x) D_x = \left(D_x + \frac1 2\right) \left(D_x + \beta-
\frac1 2\right) \left(D_x - 1\right) \left(D_x - \beta- 1\right) D_x$
cancels expansion~\reftext{\eqref{expansion_v}} up to order $O\left(x^2\right)$. Since
equation~\reftext{\eqref{lin_eq_qd}} does not have a closed form in $\tilde
q(D_x) D_x v$ and $\tilde q(D_x-1) (D_x-1) f$ as in the
$1+1$-dimensional case, we again need to restrict the range of
admissible weights and absorb the remnant contributions coming from the
additional terms in \reftext{\eqref{lin_eq_qd}} by those coming from 
\reftext{\eqref{higher_est_1}}. This leads to estimate \reftext{\eqref{higher_est_2}} 
from \S \ref{sec:higher_2}.

A third step is needed in order to reach a coercivity range including
weights $2+\delta$, which -- together with $2-\delta$ -- are required
for controlling the norm ${\skliaustask v_x\skliaustasd}_{BC^0\left(I
\times(0,\infty)
\times\mathbb{R}\right)}$ as in \reftext{Lemma~\ref{lem:bc0_bounds}}. For
this reason,
we apply $(D_x-4)(D_x-3)$ to \reftext{\eqref{lin_eq_qd}}, and obtain equation 
\reftext{\eqref{lin_eq_qdqd}}, i.e.
\begin{align*}
& \left( x\partial_t + \breve q(D_x)\right)(D_x-3) (D_x-2) \tilde
q(D_x) D_x v - \eta^2 x^2 \breve r_1(D_x) (D_x-1) D_x v + \eta^4 x^4
\breve r_2(D_x) D_x v \\
& \quad= (D_x-4) (D_x-3) \tilde q(D_x-1) (D_x-1) f,
\end{align*}
where $\breve q(\zeta)$, $\breve r_1(\zeta)$, and $\breve r_2(\zeta)$ are real polynomials of order four, seven, and six, respectively, and $\breve q(D_x)$ has coercivity range containing the interval $\left[2,2+\frac{1}{10}\right)$.
In order to reach the final estimate with weight $2+\delta$ via this
equation, we also make use of versions of \reftext{\eqref{higher_est_1}} and \reftext{\eqref{higher_est_2}} with weights $\delta$ and $1+\delta$, respectively.
These bounds will appear in \reftext{\eqref{higher_est_1_alt}}, \reftext{\eqref{higher_est_2_alt}}, respectively. The resulting bound is \reftext{\eqref{higher_est_3}} from \S \ref{sec:higher_3}. Finally, by suitably
combining the previous estimates, we obtain a maximal-regularity
estimate of the form \reftext{\eqref{mr_lin}} with slightly more complicated
norms. These results are summarized in \S \ref{sec:max_reg}. Then in
\S \ref{sec:equivnorms} we reabsorb some of the terms in the norms, in
order to simplify their form. This allows to reach \reftext{\eqref{mr_lin}}
itself with norms $ {\left\lVert\cdot\right\rVert
}_\mathrm{init}$ as introduced in \reftext{\eqref{norm_init}}, ${\skliaustask\cdot
\skliaustasd}_\mathrm{Sol}$ as
given by \reftext{\eqref{norm_sol_simple}}, and ${\skliaustask\cdot\skliaustasd
}_\mathrm
{rhs}$ as defined in \reftext{\eqref{norm_rhs}}.

In \S \ref{sec:elliptic} several properties of the norms and in
particular embeddings and control of the coefficients $v_0$, $v_1$, and
$v_{1+\beta}$ are discussed. The reasoning mainly relies on elementary
estimates, combined with elliptic-regularity estimates based on Hardy's
inequality. The embeddings are necessary to rigorously define
appropriate function spaces, but also for the treatment of the
nonlinear Cauchy problem \reftext{\eqref{nonlin_cauchy}}. In order to lighten the
presentation, many of the proofs of this section have been outsourced
to Appendix~\ref{app:proofs}. In \S \ref{sec:lin_rig} the treatment of
the linear theory is concluded by discussing all arguments to make
estimate~\reftext{\eqref{mr_lin}} rigorous and to prove existence and uniqueness
for the linear problem \reftext{\eqref{lin_cauchy}} (cf.~\reftext{Proposition~\ref{prop:maxregevolution}}). This is achieved through a time-discretization
procedure, relying on a thorough understanding of the resolvent equation
\begin{equation*}
x v + q(D_x) v - \eta^2 x^2 r\left(D_x\right) v + \eta^4 x^4 v = f
\quad\mbox{for} \quad x > 0 \quad \mbox{and} \quad \eta \in \R,
\end{equation*}
which is \reftext{\eqref{resolvent}}, discussed in \S \ref{sec:resolvent}. This
is done
through a matching argument of solution manifolds with convenient asymptotic
properties as $x \searrow0$ and $x \to\infty$
(cf.~\reftext{Proposition~\ref{prop:resolvent}}). In particular the construction
of the
solution manifold with convenient asymptotics as $x \to\infty$ is quite
different from the arguments of
\cite[\S6]{ggko.2014}, because of the
additional terms coming from the tangential direction $\partial_y$.
Note that
the resolvent equation \reftext{\eqref{resolvent}} is in essence nothing but the
time-discretized linear equation \reftext{\eqref{eq_lin}}. A solution for the linear
problem \reftext{\eqref{lin_cauchy}} is obtained by compactness in the limit in which
the time step tends to zero. The bounds for the discrete case are the
same as
in the equations coming from the computations of, and leading to,
\S \ref{sec:max_reg}, with the only exception that the continuous time
derivative has to be replaced by a discrete difference quotient. This allows
to prove the bounds for \reftext{\eqref{lin_cauchy}} with the usual time
derivative in
the limit. Note that our approach does not essentially rely on the
time-discretization argument and a semi-group approach appears to be
applicable, too (see monographs by Lunardi \cite{l.1995} and Pazy \cite{p.1983} and the work of
Mielke \cite{m.1987} for this approach). Yet, firstly the mathematical ingredients
of both approaches are essentially the same, namely a solid
understanding of
the resolvent equation \reftext{\eqref{resolvent}} has to be obtained, and secondly
also the resulting estimates will be identical.

In order to apply the maximal-regularity estimate \reftext{\eqref{mr_lin}} valid
for the linearized evolution \reftext{\eqref{lin_cauchy}} to the actual nonlinear
problem \reftext{\eqref{nonlin_cauchy}} to prove well-posedness and the a-priori
estimate \reftext{\eqref{main_estimate}}, we dedicate \S \ref{sec:nonlinear} to
studying the nonlinearity ${\mathcal N}(v)$ defined in \reftext{\eqref{def_nonlinearity}}.
We note that ${\mathcal N}(v)$ is in fact a local rational function in
$x$ and
$\mathcal D:=
\begin{pmatrix} D^\ell v: \, \ell\in\mathbb{N}_0^2, \, 1 \le
 {\left\lvert\ell\right\rvert} \le
4
\end{pmatrix}
$ and super-linear in $\mathcal D$, too. In \S \ref{sec:nonlinear_struct} we detail the algebraic structure of ${\mathcal N}(v)$
and derive suitable decompositions that in particular make the
nontrivial expansion \reftext{\eqref{expansion_f}} apparent also on the level of
the nonlinear problem \reftext{\eqref{nonlin_cauchy}}. In \S \ref{sec:nonlinear_est} we then prove our main estimate for the
nonlinearity (cf.~\reftext{\eqref{non_est_main}} of \reftext{Proposition~\ref{prop:non_est}}), i.e.,
\begin{equation*}
{\skliaustask{\mathcal N}\left(v^{(1)}\right) - {\mathcal N}\left
(v^{(2)}\right)\skliaustasd}_\mathrm
{rhs} \le C \left({\skliaustask v^{(1)}\skliaustasd}_\mathrm{sol} +
{\skliaustask v^{(2)}\skliaustasd}_\mathrm{sol}\right) {\skliaustask
v^{(1)} - v^{(2)}\skliaustasd}_\mathrm{sol},
\end{equation*}
provided that ${\skliaustask v^{(j)}\skliaustasd}_\mathrm{sol} \le
c$, for constants $C$ and $c > 0$ depending only on $k$, 
$\tilde k$, $\check k$, $\breve k$, and $\delta$. The above estimate allows to
establish Lipschitz continuity of ${\mathcal N}(v)$ in the norm
${\skliaustask\cdot\skliaustasd}_\mathrm{rhs}$ (appearing on the
right-hand side of the
maximal-regularity estimate \reftext{\eqref{mr_lin}}) with a small Lipschitz
constant, if $v$ belongs to a small ${\skliaustask\cdot\skliaustasd
}_\mathrm{sol}$-ball
centered at $0$. This allows to use \reftext{\eqref{mr_lin}} for the linearized
evolution \reftext{\eqref{lin_cauchy}} in order to solve \reftext{\eqref{tfe_transformed}}
and to prove the a-priori estimate \reftext{\eqref{main_estimate}} of
\reftext{Theorem~\ref{main_th}} using the contraction-mapping theorem with the underlying
norm ${\skliaustask\cdot\skliaustasd}_\mathrm{sol}$ (cf.~\S \ref{sec:well}). The
uniqueness will follow by a classical argument based on the continuity
of $t\mapsto {\left\lVert v(t,\cdot,\cdot) \right\rVert
}_\mathrm{init}$, as shown in
\reftext{Corollary~\ref{coroll:contnorm}} from \S \ref{sec:approx}.

\subsection{Notation}
Throughout the paper, all constants have a finite value.

If $f:\mathbb R\to\mathbb R, g:\mathbb R\to[0,\infty)$ we write
\begin{enumerate}[(i)]
\item[(i)]\emph{$f(x)=O(g(x))$ as $x\to a\in\mathbb R\cup\{\pm\infty\}
$} if
there exist constants $C$ and a neighborhood $U$ of $a$ such that
$ {\left\lvert f(x) \right\rvert}\le C g(x)$ for all $x\in U$.
\item[(ii)]\emph{$f(x)=o(g(x))$ as $x\to a\in\mathbb R\cup\{\pm\infty\}
$} if
$\lim_{x\to a} {\left\lvert f(x)/g(x) \right\rvert}=0$.
\end{enumerate}

If $A,B\in\mathbb R$ and $\mathcal P$ is a set of parameters, then
\begin{enumerate}[(i)]
\item[(i)] we will write $A\lesssim_\mathcal{P} B$ if there exists a
constant $C$, with $C$ only depending on $\mathcal P$, such that
$A\le C B$. We write $S \sim_\mathcal{P} B$ if $A \lesssim_\mathcal{P}
B$ and $B \lesssim_\mathcal{P} A$.
\item[(ii)] we say that \emph{property $(X)$ holds if $A\gg_\mathcal{P} B$}
in case there exists a constant $C$ only depending on
$\mathcal P$ such that $(X)$ holds whenever $A\ge C B$.
\end{enumerate}

We write $\mathbb{N}:=\{1,2,3,\ldots\}$ and we denote by $\mathbb
{N}_0:=\{0,1,2,3,\ldots
\}$ the set of natural numbers including $0$. If $a\in\mathbb{R}$
then we
denote by $\left\lfloor a\right\rfloor$ the integer part of $a$,
i.e., the largest
integer smaller than $a$.

For a multi-index $\alpha\in\mathbb{N}_0^k$ of the form $\alpha
=(\alpha
_1,\ldots,\alpha_k)$ we write $|\alpha|:=\alpha_1+\cdots+\alpha
_k$. We
will indicate by $\ell\in\mathbb{N}_0^2$ the multi-index of
derivatives $\ell
=(\ell_x,\ell_y)$ and in that case we write $D^\ell:=D_y^{\ell
_y}D_x^{\ell_y}$, as already indicated in \reftext{\eqref{der_scale}}.

For a complex number $z \in\mathbb{C}$ we write $z^*$ for its complex
conjugate.

If $E_1, E_2, \ldots, E_N$ are a finite number of expressions of the
form $E_i=\prod_{j=1}^{N_i} D^{\ell_{j,i}} f_{j,i}$ for $i=1,\ldots,N$,
then we write $E_1 \times E_2 \times\ldots\times E_N$ to indicate
that operators $D$ within the expression $E_i$ act on \emph{everything
to their right within $E_i$}. Otherwise, we write $f_t$, $f_x$, or
$f_y$ if the derivative shall act on the function $f$ only.

For derivatives we usually use the notation $\partial_t, \partial_x,
\partial_y$ whenever several variables play a role in our computations,
while total derivatives $\frac{\mathrm{d}}{\mathrm{d} t}, \frac
{\mathrm{d}}{\mathrm{d} x}$ are used in
order to emphasize that we deal with (in-)equalities depending of only
one variable, i.e., ordinary differential equations (ODE) theory is used.

Our function space norms will be denoted by $ {\left\lvert
\cdot\right\rvert},  {\left\lVert\cdot\right\rVert
}, {\skliaustask\cdot\skliaustasd}$ if they are norms on functions depending
respectively on $1$, $2$, or $3$ of the variables $t,x,y$ appearing in
our original problem.

We will denote by $I=[0,T)\subseteq[0,\infty)$ with $T \in (0,\infty]$ a time interval
starting at zero, which will appear in our norms.

If $A,B\subseteq\mathbb R^d$ we write $A\Subset B$ in case there
exists a compact set $K$ and an open set $U$ such that $A\subseteq
K\subseteq U\subseteq B$.

\section{Linear theory}\label{sec:linear_th}

\subsection{Coercivity}\label{sec:coerc}

Our aim is to derive suitable maximal-regularity estimates for
problem~\reftext{\eqref{lin_cauchy}} that are strong enough to control the
nonlinearity. We consider a general fourth-order operator $Q(D_x)$,
where $Q(\zeta)$ is a polynomial in $\zeta$ of degree $4$ with real
roots $\gamma_1 \le\gamma_2 \le\gamma_3 \le\gamma_4$. A key concept
is the coercivity of $Q(D_x)$ with respect to the inner product
\begin{equation*}
(f,g)_\alpha:=\int_0^\infty x^{-2\alpha} f(x) g^*(x) \, \frac
{\mathrm{d} x}{x}.
\end{equation*}
More precisely, our aim is to find $\alpha\in\mathbb{R}$ such that
%
%
\begin{equation}\label{coercivity_q}
\left(v,Q(D_x) v\right)_\alpha\gtrsim_\alpha {\left
\lvert v \right\rvert}_{2,\alpha}^2
\quad\mbox{for} \quad v: \, (0,\infty) \to\mathbb{R}\mbox{ smooth
with }
 {\left\lvert v \right\rvert}_{2,\alpha} < \infty,
\end{equation}
where we have introduced the norm $\lvert\cdot\rvert_{k,\alpha}$ with
%
%
\begin{equation}\label{norm_1d}
 {\left\lvert v \right\rvert}_{k,\alpha}^2 := \sum_{j =
0}^k \int_0^\infty x^{-2 \alpha}
 {\left\lvert D_x^j v \right\rvert}^2 \, \frac{\mathrm
{d} x}{x}, \quad \mbox{where} \quad k \in\mathbb{N}
_0, \; \alpha\in\mathbb{R}.
\end{equation}
In what follows, we will also write $ {\left\lvert v \right
\rvert}_\alpha:=  {\left\lvert v \right\rvert
}_{0,\alpha}$. If \reftext{\eqref{coercivity_q}} holds, then we say that
$Q(D_x)$ is \textit{coercive with respect to $(\cdot,\cdot)_\alpha
$}. A
sufficient criterion for this property to hold is given in \cite[Proposition~5.3]{ggko.2014}.
We need a quantitative version of this result,
stated as follows:
%
%
\begin{lemma}\label{lem:coercivity}
Suppose that $\alpha\in\mathbb{R}$ and $v: \, (0,\infty) \to
\mathbb{R}$ is smooth
with $ {\left\lvert v \right\rvert}_{2,\alpha} < \infty
$, then
%
%
\begin{subequations}\label{coer_quant}
%
%
\begin{equation}
\left(Q(D_x)v,v\right)_\alpha=  {\left\lvert(D_x -
\alpha)^2v \right\rvert}^2_\alpha+
\omega(\alpha)  {\left\lvert(D_x -\alpha)v \right
\rvert}^2_\alpha+ Q(\alpha)  {\left\lvert v \right
\rvert}_\alpha^2,
\end{equation}
where
%
%
\begin{equation}\label{def_om}
\omega(\alpha) := - \sum_{1 \le j < k \le4} (\gamma_j - \alpha)
(\gamma
_k - \alpha) = 2 \left(\sigma^2 - 3 (\alpha-m)^2\right)
\end{equation}
and
%
%
\begin{equation}\label{mean_var}
m = \frac1 4 \sum_{j = 1}^4 \gamma_j, \quad\sigma= \sqrt{\frac1 4
\sum_{j = 1}^4 (\gamma_j - m)^2}.
\end{equation}
\end{subequations}
In particular, the operator $Q(D_x)$ is coercive with respect to
$(\cdot
,\cdot)_\alpha$ if the weight $\alpha$ meets the necessary condition
%
%
\begin{subequations}\label{coer_cond}
%
%
\begin{equation}\label{coer_cond_1}
\alpha\in(-\infty,\gamma_1) \cup(\gamma_2,\gamma_3) \cup(\gamma
_4,\infty)
\end{equation}
and additionally
%
%
\begin{equation}\label{coer_cond_2}
\alpha\in\left[m - \frac{\sigma}{\sqrt3},m + \frac{\sigma
}{\sqrt
3}\right].
\end{equation}
\end{subequations}
\end{lemma}

Since the quantitative version has not been stated in \cite[Proposition~5.3]{ggko.2014}, we briefly outline the reasoning once more:
\begin{proof}[Proof of \reftext{Lemma~\ref{lem:coercivity}}]
Passing to the variable $s := \ln x$ and employing the Fourier
transform in~$s$ (which as in our usual convention is not highlighted
in the notation), we obtain
\begin{align*}
\left(Q(D_x)v,v\right)_\alpha &= \int_0^\infty x^{-2 \alpha} v(x)
Q(D_x) v^*(x) \, \frac{\mathrm{d} x}{x} = \int_0^\infty\underbrace
{x^{- \alpha
} v(x)}_{=: w(\log x)} Q(D_x+\alpha) x^{- \alpha} v^*(x) \, \frac
{\mathrm{d}
x}{x} \\
&= \int_{-\infty}^\infty w(s) Q(\partial_s+\alpha) w^*(s) \,
\mathrm{d} s =
\int_{-\infty}^\infty Q(i \xi+ \alpha)  {\left\lvert
w(\xi) \right\rvert}^2 \, \mathrm{d}\xi\\
&= \int_{-\infty}^\infty\left(\xi^4 - \sum_{1 \le j < k \le4}
(\gamma
_j - \alpha) (\gamma_k - \alpha) \xi^2 + \prod_{j = 1}^4 (\gamma
_j -
\alpha) \right)  {\left\lvert w(\xi) \right\rvert}^2
\, \mathrm{d}\xi\\
&= \int_{-\infty}^\infty\left( {\left\lvert\partial
_s^2 w(s) \right\rvert}^2 - \sum_{1
\le j < k \le4} (\gamma_j - \alpha) (\gamma_k - \alpha)
 {\left\lvert\partial_s w(s) \right\rvert}^2 + \prod
_{j = 1}^4 (\gamma_j - \alpha)  {\left\lvert w(s) \right
\rvert}^2\right)
\, \mathrm{d} s \\
&=  {\left\lvert(D_x - \alpha)^2v \right\rvert
}^2_\alpha\underbrace{- \sum_{1 \le j < k
\le4} (\gamma_j - \alpha) (\gamma_k - \alpha)}_{= \omega(\alpha)}
 {\left\lvert(D_x -\alpha)v \right\rvert}^2_\alpha+
\underbrace{\prod_{j=1}^4(\gamma
_j-\alpha)}_{=Q(\alpha)}  {\left\lvert v \right\rvert
}_\alpha^2.
\end{align*}
Then condition~\reftext{\eqref{coer_cond_1}} ensures positivity of
$Q(\alpha)$ and condition~\reftext{\eqref{coer_cond_2}} ensures non-negativity of
$\omega(\alpha)$, given that \reftext{\eqref{def_om}} holds true. To prove \reftext{\eqref{def_om}}, observe that
\begin{align*}
\omega(\alpha) &= - \sum_{1 \le j < k \le4} \gamma_j \gamma_k +
12 m
\alpha- 6 \alpha^2 = - \sum_{1 \le j < k \le4} \gamma_j \gamma_k
+ 6
m^2 - 6 (\alpha-m)^2 \\*
&= \frac1 2 \sum_{j = 1}^4 \gamma_j^2 - \frac1 2 \sum_{j, k = 1}^4
\gamma_j \gamma_k + 6 m^2 - 6 (\alpha-m)^2 = 2 \left(\frac1 4 \sum_{j
= 1}^4 \gamma_j^2 - m^2 - 3 (\alpha-m)^2\right) \\
&= 2 \left(\sigma^2 - 3 (\alpha-m)^2\right).
\end{align*}
\end{proof}

In the specific situation of the operator $q(D_x)$ defined through \reftext{\eqref{poly_q}} (the left-hand side of \reftext{\eqref{lin_pde}} neglecting the
$D_y$-terms), the roots of $q(\zeta)$ are in strictly increasing order
$\gamma_1 := - \frac1 2$, $\gamma_2 := \frac1 2 - \beta$, $\gamma_3
:= 0$, $\gamma_4: = 1 + \beta$, their mean is
$m=\frac1 4$ and their variance is $\sigma^2 = \frac{11}{16}$. In our
case, we have from \reftext{\eqref{coer_cond}}
\begin{equation*}
m-\tfrac{\sigma}{\sqrt3} \approx-0.2287 < \gamma_2 \stackrel
{\text{\reftext{\eqref{poly_q}}}}{=} \frac1 2 - \beta= \frac{3-\sqrt{13}}{4} \approx-0.1514.
\end{equation*}
On the other hand
\begin{equation*}
m + \frac{\sigma}{\sqrt3} \approx0.7287 > \gamma_3 = 0.
\end{equation*}
Therefore, we have
\begin{equation*}
\gamma_1 < m - \frac{\sigma}{\sqrt3} < \gamma_2 < \gamma_3 < m +
\frac
{\sigma}{\sqrt3} < \gamma_4.
\end{equation*}
This implies:
%
%
\begin{lemma}\label{lem:coer_q}
The operator $q(D_x)$ given through \reftext{\eqref{poly_q}} is coercive with
respect to $\left(\cdot,\cdot\right)_\alpha$ if
%
%
\begin{equation}\label{coer_range_x}
\alpha\in\left(\frac1 2 - \beta,0\right) = \left(- \frac{\sqrt
{13} -
3}{4},0\right) \supset\left(-\frac{1}{10}, 0\right).
\end{equation}
\end{lemma}

Observe that in view of \reftext{\eqref{expansion_v}}, the boundary value $v_0$
of the solution $v$ to \reftext{\eqref{nonlin_cauchy}} is in general
non-vanishing, which makes the condition $\alpha< 0$ in \reftext{\eqref{coer_range_x}} intuitive.

\subsection{Heuristic treatment of the linear equation}\label{sec:lin_heur}

In this part, we derive estimates for the linear equation that are
sufficient in order to treat the nonlinear problem \reftext{\eqref{nonlin_cauchy}}. The following arguments are based on the assumption
that sufficiently regular solutions already exist on a time interval $I
= [0,T) \subseteq[0,\infty)$, where $T \in (0,\infty]$. A rigorous justification is based on
treating the resolvent problem associated to \reftext{\eqref{nonlin_cauchy}} in
conjunction with a time-discretization argument. We postpone these
arguments to \S \ref{sec:lin_rig}.

\subsubsection{A basic weak estimate}\label{sec:basic_weak}

We assume that $-\delta$ belongs to the range \reftext{\eqref{coer_range_x}} and
test \reftext{\eqref{lin_pde}} with $v$ in the inner product $(\cdot,\cdot
)_{-\delta}$
%
%
\begin{equation}\label{test_basic}
(x\partial_tv,v)_{-\delta} + \left(q(D_x)v,v\right)_{-\delta} -
\eta
^2\left(x^2 r(D_x)v,v\right)_{-\delta} + \eta^4(x^4v,v)_{-\delta} =
(f,v)_{-\delta}.
\end{equation}
We treat each term in \reftext{\eqref{test_basic}} separately: For the first term
we obtain
%
%
\begin{equation}\label{1term}
(x\partial_tv,v)_{-\delta}=\frac12\frac{\mathrm{d}}{\mathrm{d} t}
 {\left\lvert v \right\rvert}^2_{-\delta
-\frac12}.
\end{equation}
For the second term in \reftext{\eqref{test_basic}} we use the coercivity of
$q(D_x)$ (cf.~\reftext{\eqref{coercivity_q}}) in the quantitative version \reftext{\eqref{coer_quant}} of \reftext{Lemma~\ref{lem:coercivity}} under the assumptions \reftext{\eqref{coer_range_x}} of \reftext{Lemma~\ref{lem:coer_q}} as
%
%
\begin{equation}\label{2term}
\left(q(D_x)v,v\right)_{-\delta} =  {\left\lvert
(D_x+\delta)^2 v \right\rvert}_{-\delta}^2
+ \omega(-\delta)  {\left\lvert(D_x+\delta) v \right
\rvert}_{-\delta}^2 + q(-\delta)  {\left\lvert v \right
\rvert}_{-\delta}^2,
\end{equation}
where $\omega(-\delta)$ is defined in \reftext{\eqref{def_om}} and the polynomial
$q(-\delta)$ was introduced in \reftext{\eqref{poly_q}}. For the term in \reftext{\eqref{test_basic}} proportional to $\eta^2$, we have
%
%
\begin{eqnarray}
- \eta^2 \left(x^2 r(D_x)v,v\right)_{-\delta}&
\stackrel{\text{\reftext{\eqref{poly_r}}}}{=}& - 2 \eta^2 \left(\left(D_x + 1\right) \left(D_x +
\tfrac
1 2\right) v, x^{2 + 2 \delta} v\right)_0 \nonumber\\
&=& 2 \eta^2 \left(\left(D_x + \tfrac1 2\right) v, \left(D_x -
1\right
) x^{2 + 2 \delta} v\right)_0 \nonumber\\
&=& 2 \eta^2 \left((D_x +\tfrac12)v,(D_x +1+2\delta)v\right
)_{-\delta
-1} \nonumber\\
&=& \eta^2 \left(2  {\left\lvert(D_x+1+\delta) v \right
\rvert}_{-\delta-1}^2 - \delta(2
\delta+ 1) {\left\lvert v \right\rvert}_{-\delta
-1}^2\right), \label{3term}
\end{eqnarray}
where we have used the skew-symmetry of $D_x$ with respect to $(\cdot
,\cdot)_0$ and $(v,D_x v)_{-\delta-1} = -(\delta+1)
{\left\lvert v \right\rvert}_{-\delta
-1}^2$ (which both immediately follow through integration by parts).
Finally, for the term in $\eta^4$ we have
%
%
\begin{equation}\label{4term}
\eta^4 (x^4v,v)_{-\delta} = \eta^4  {\left\lvert v
\right\rvert}_{-\delta-2}^2.
\end{equation}
For the right-hand side of \reftext{\eqref{test_basic}}, we use Young's
inequality and obtain
%
%
\begin{equation}\label{5term}
(f,v)_{-\delta} \le\frac{1}{2\varepsilon}  {\left\lvert
f \right\rvert}_{-\delta}^2 + \frac{\varepsilon
}{2}  {\left\lvert v \right\rvert}_{-\delta}^2 \quad
\mbox{for any} \quad\varepsilon> 0.
\end{equation}
The contribution from the last term in \reftext{\eqref{3term}} will be absorbed
by the last terms of \reftext{\eqref{2term}} and \reftext{\eqref{4term}} as follows:
%
%
\begin{eqnarray}
\eta^2 \delta(2 \delta+ 1)  {\left\lvert v \right\rvert
}_{-\delta-1}^2 &\le& 2 \left(\eta
^2  {\left\lvert v \right\rvert}_{-\delta-2}\right)
\left(\delta(\delta+ \tfrac12)  {\left\lvert v \right
\rvert}_{-\delta}\right)\nonumber\\
&\le& K(\delta) \left(q(-\delta)  {\left\lvert v \right
\rvert}_{-\delta}^2 + \eta^4  {\left\lvert v \right
\rvert}_{-\delta-2}^2\right) \quad\mbox{for} \quad\delta\in
\left
(0,\tfrac{1}{10}\right), \label{est3term}
\end{eqnarray}
where $K(\delta) < 1$. Estimate~\reftext{\eqref{est3term}} follows by
Young's inequality, that is,
\begin{equation*}
2 \left(\eta^2  {\left\lvert v \right\rvert}_{-\delta
-2}\right) \left(\delta(\delta+\tfrac
12)  {\left\lvert v \right\rvert}_{-\delta}\right) \le
C \delta^2 (\delta+\tfrac1 2)^2
 {\left\lvert v \right\rvert}_{-\delta}^2 + C^{-1} \eta
^4  {\left\lvert v \right\rvert}_{-\delta-2}^2 \quad
\mbox{with} \quad C > 1.
\end{equation*}
Then we have
\begin{equation*}
q(-\delta) - C \delta^2 \left(\delta+\tfrac12\right)^2 \stackrel{\text{\reftext{\eqref{poly_q}}}}{=} \delta\left((1-C)\delta^3 + (1-C) \delta^2 - \left
(\tfrac
C 4 + 1\right) \delta+ \tfrac1 8\right)
\end{equation*}
and in the limit $C \searrow1$ the expression is positive for $\delta
\in\left(0,\tfrac{1}{10}\right)$, which is more restrictive than \reftext{\eqref{coer_range_x}}.

We combine \reftext{\eqref{1term}}, \reftext{\eqref{2term}}, \reftext{\eqref{3term}}, \reftext{\eqref{4term}},
\reftext{\eqref{est3term}}, and then absorb the second sum in \reftext{\eqref{5term}} with
$\varepsilon$ small enough depending on $\delta$. This gives
%
%
\begin{equation}\label{lin_est1}
\frac{\mathrm{d}}{\mathrm{d} t} {\left\lvert v \right
\rvert}_{-\delta-\frac12}^2 + \sum_{\ell=0}^2 \eta
^{2\ell}  {\left\lvert v \right\rvert}_{\ell,-\delta
-\ell}^2 \lesssim_\delta {\left\lvert f \right\rvert
}_{-\delta}^2 \quad\mbox{for} \quad\delta\in\left(0,\tfrac
{1}{10}\right).
\end{equation}
This is a weak estimate, as the regularity gain in space is only $2$
(i.e., up to two derivatives $D \stackrel{\text{\reftext{\eqref{der_scale}}}}{=}
(D_x,D_y)$ more compared to those acting on the right-hand side $f$),
while equation~\reftext{\eqref{lin_pde}} is of order $4$. Upgrading it to a
strong estimate will be addressed in what follows.

\subsubsection{A basic strong estimate}\label{sec:basicstrongest}

For upgrading \reftext{\eqref{lin_est1}} to a strong estimate, we test \reftext{\eqref{lin_pde}} against
%
%
\begin{equation}\label{basicstrongtest}
\eta^{2 \ell} \left(\cdot,(-D_x - 1 - 2 \ell- 2 \delta)^{k-\ell}
D_x^{k-\ell} v\right)_{-\delta-\ell} =: \left(\cdot,S\right
)_{-\delta
-\ell},
\end{equation}
where $k \ge2$ and $0 \le\ell\le k$. We obtain the following terms:
\begin{subequations}\label{terms_estimate_k_deriv}
\begin{eqnarray}
(x\partial_tv,S)_{-\delta-\ell} &=& \frac1 2 \eta
^{2 \ell} \frac{\mathrm
{d}}{\mathrm{d} t} {\left\lvert D_x^{k-\ell} v \right\rvert
}_{-\delta- \ell- \frac1 2}^2,
\label{terms_estimate_k_deriv_1}
\eqncr
(q(D_x)v,S)_{-\delta-\ell} &=& \eta^{2 \ell} \left
(q(D_x)(D_x-1)^{k-\ell
-2}v,(D_x+1+2\ell+ 2\delta)^2 D_x^{k-\ell} v\right)_{-\delta-\ell
}\nonumber\\
&\stackrel{\text{\reftext{\eqref{poly_q}}}}{\ge}& \eta^{2 \ell} \left(\frac1 2
 {\left\lvert v \right\rvert}^2_{k-\ell+2,-\delta-\ell
} - C  {\left\lvert v \right\rvert}_{-\delta-\ell
}^2\right),
\label{terms_estimate_k_deriv_2}
\eqncr
- \eta^2 \left(x^2 r(D_x)v,S\right)_{-\delta-\ell} &\stackrel{\text{\reftext{\eqref{poly_r}}}}{\ge}& \eta^{2 (\ell+1)} \left
( {\left\lvert v \right\rvert}_{k+1-\ell,-\delta-\ell
-1}^2 - C  {\left\lvert v \right\rvert}_{-\delta-\ell
-1}^2\right
),\!\!\!\!\label{terms_estimate_k_deriv_3}
\eqncr
\eta^4(x^4v,S)_{-\delta-\ell} &=& \eta^{2(\ell+2)} \left(
(D_x+3)^{k-\ell} v,D_x^{k-\ell} v\right)_{-\delta-\ell
-2}^2\nonumber\\
&\ge& \eta^{2 (\ell+2)} \left(\frac1 2  {\left\lvert v
\right\rvert}_{k-\ell,-\delta-\ell
-2}^2 - C {\left\lvert v \right\rvert}_{-\delta-\ell
-2}^2\right),\label{terms_estimate_k_deriv_4}
\end{eqnarray}
\end{subequations}
where $C \gg_{k,\delta} 1$ comes from interpolation estimates for
intermediate terms. We do not have the coefficient $\frac12$ in \reftext{\eqref{terms_estimate_k_deriv_3}} because $r(\zeta)$ has leading coefficient
$2$ (cf.~\reftext{\eqref{poly_r}}). For the right-hand side we obtain through
integration by parts and Young's inequality
\begin{align}\label{terms_estimate_k_deriv_5}
\eta^{2 \ell} (f,S)_{-\delta} \lesssim_{k,\delta}
\begin{cases} \varepsilon^{-1}  {\left\lvert f \right
\rvert}_{k-2,-\delta}^2 + \varepsilon {\left\lvert v
\right\rvert}_{k+2,-\delta}^2 & \mbox{for } \, \ell= 0, \\\noalign{\vspace{2pt}}
\varepsilon^{-1}  {\left\lvert f \right\rvert
}_{k-2,-\delta}^2 + \varepsilon\, \eta^4  {\left\lvert v
\right\rvert}_{k,-\delta-2}^2 & \mbox
{for } \, \ell= 1, \\\noalign{\vspace{2pt}} \varepsilon^{-1} \eta^{2 (\ell-2)}
 {\left\lvert f \right\rvert}_{k-\ell
,-\delta+2-\ell}^2 + \varepsilon\, \eta^{2 (\ell+2)}
{\left\lvert v \right\rvert}_{k-\ell,-\delta
-\ell-2}^2 & \mbox{for } \, 2 \le\ell\le k,
\end{cases}
\end{align}
where $\varepsilon> 0$ is arbitrary. By absorbing the $v$-terms of \reftext{\eqref{terms_estimate_k_deriv_5}} on the left-hand side, we obtain
\begin{align*}
& \eta^{2 \ell} \frac{\mathrm{d}}{\mathrm{d} t} 
{\left\lvert D_x^{k-\ell} v \right\rvert}_{-\delta
-\ell-\frac12}^2 + \sum_{\ell^\prime=0}^2 \eta^{2 (\ell
+\ell
^\prime)} \left( {\left\lvert v \right\rvert}_{k+2-\ell
-\ell^\prime,-\delta-\ell-\ell^\prime
}^2 - C  {\left\lvert v \right\rvert}_{-\delta-\ell
-\ell^\prime}^2\right) \\
& \quad\lesssim_\delta
\begin{cases}  {\left\lvert f \right\rvert}_{k-2,-\delta
}^2 & \mbox{for } \, \ell\in\{0,
1\}, \\ \eta^{2 (\ell-2)}  {\left\lvert f \right\rvert
}_{k-\ell,-\delta+2-\ell}^2 & \mbox
{for } \, 2 \le\ell\le k,
\end{cases}
\end{align*}
where $C \gg_{k,\delta} 1$ and $k \ge2$. An induction argument,
starting with the base \reftext{\eqref{lin_est1}}, leads to
\begin{equation}\label{estimate1}
\begin{aligned}
& \frac{\mathrm{d}}{\mathrm{d} t} \left(\sum_{\ell= 0}^k C_\ell
\, \eta^{2 \ell} {\left\lvert D_x^{k-\ell} v \right
\rvert}_{-\delta-\frac1 2-\ell}^2 + \left\lvert v \right\rvert_{-\delta-\frac12}^2\right)
+ \sum_{\ell=0}^{k+2} \eta^{2 \ell}
 {\left\lvert v \right\rvert}_{k+2-\ell,-\delta-\ell
}^2 \\
& \quad \lesssim_{k,\delta} \sum_{\ell= 0}^{k-2} \eta^{2 \ell}
 {\left\lvert f \right\rvert}_{k-2-\ell,-\delta-\ell
}^2, \quad \mbox{with constants} \quad C_\ell > 0 \quad \mbox{and} \quad \delta\in\left(0,\tfrac{1}{10}\right). 
\end{aligned}
\end{equation}
Next we return to physical space-time $(t,x,y)$ and use the norms
$ {\left\lVert\cdot\right\rVert}_{k,\alpha}$, defined through
\begin{eqnarray}
 {\left\lVert v \right\rVert}_{k,\alpha}^2 &\stackrel{\eqref{norm_2d_1}}{=}& \sum_{0
\le j + j^\prime\le k} \int_{-\infty
}^\infty\eta^{2j} {\left\lvert D_x^{j^\prime} v \right
\rvert}_{\alpha+\frac12-j}^2 \mathrm{d}\eta
= \sum_{0 \le j + j^\prime\le k} \int_{-\infty}^\infty
{\left\lvert D_y^j D_x^{j^\prime} v \right\rvert}_{\alpha+ \frac12}^2\mathrm{d} y \nonumber\\
&=& \sum_{0 \le j + j^\prime\le k} \int_{-\infty}^\infty\int
_0^\infty
x^{-2 \alpha} \!\left(D_y^j D_x^{j^\prime} v\right)^2 x^{-2} \,
\mathrm{d} x \,
\mathrm{d} y = \sum_{0 \le {\left\lvert\ell\right
\rvert} \le k} \int_{(0,\infty)_x \times\mathbb{R}_y}
x^{-2 \alpha} \!\left(D^\ell v\right)^2 x^{-2} \, \mathrm{d} x \, \mathrm{d} y, \nonumber\\ \label{norm_2d}
\end{eqnarray}
where $\ell= (\ell_x,\ell_y) \in\mathbb{N}_0^2$ with $
{\left\lvert\ell\right\rvert} := \ell_x
+ \ell_y$ and $D^\ell:= D_y^{\ell_y} D_x^{\ell_x}$ (cf.~\reftext{\eqref{der_scale}}). Again, we use the short-hand
notation $ {\left\lVert v \right\rVert}_\alpha:=
 {\left\lVert v \right\rVert}_{0,\alpha}$. Integration
in time of \reftext{\eqref{estimate1}} multiplied with $\eta^{2 j}$, where $j \in
\mathbb{N}_0$ is
arbitrary, Fourier back transform, and a standard interpolation
estimate for the norm \reftext{\eqref{norm_2d}} give
\begin{equation*}
\sup_{t \in I}  {\left\lVert D_y^j v \right\rVert
}_{k,-\delta-1 + j}^2 + \int_I  {\left\lVert D_y^j v
\right\rVert}_{k+2,-\delta- \frac1 2 + j}^2 \mathrm{d} t \lesssim
_{k,\delta}
 {\left\lVert D_y^j v_{|t = 0} \right\rVert}_{k,-\delta
-1 + j}^2 + \int_I  {\left\lVert D_y^j f \right\rVert
}_{k-2,-\delta- \frac1 2 + j}^2 \mathrm{d} t.
\end{equation*}
With help of the linear equation \reftext{\eqref{eq_lin}}, we can also obtain
control on the time derivative $\partial_t v$ according to
\begin{eqnarray*}
\int_I  {\left\lVert\partial_t D_y^j v \right\rVert
}_{k-2,-\delta- \frac3 2 + j}^2 \mathrm{d} t &\lesssim_k& \int_I \left( {\left\lVert D_y^j f \right\rVert
}_{k-2,-\delta- \frac1 2 + j}^2 +
 {\left\lVert q\left(D_x-j\right) D_y^j v \right\rVert
}_{k-2,-\delta- \frac1 2 +
j}^2\right) \mathrm{d} t \\
&& + \int_I \left( {\left\lVert D_y^2 r(D_x-j) D_y^j v
\right\rVert}_{k-2,-\delta- \frac
1 2 + j}^2 +  {\left\lVert D_y^4 D_y^j v \right\rVert
}_{k-2,-\delta- \frac1 2 +
j}^2\right) \mathrm{d} t \\
&\lesssim_{j,k}& \int_I \left( {\left\lVert D_y^j v \right
\rVert}_{k+2,-\delta- \frac1 2+ j}^2
+  {\left\lVert D_y^j f \right\rVert}_{k-2,-\delta- \frac1 2
+ j}^2\right) \mathrm{d} t
\end{eqnarray*}
and thus, with the notation 
\begin{equation}\label{not_3barnorm}
\skliaustask w\skliaustasd_{\kappa,\alpha,I}^2:=\sup_{t\in I}\left\lVert w\right\rVert_{\kappa,\alpha}^2+\int_I\left(\left\lVert \partial_tw\right\rVert_{\kappa-2,\alpha-\frac12}^2+\left\lVert w\right\rVert_{\kappa+2,\alpha+\frac12}^2\right)\mathrm{d} t,
\end{equation}
we obtain
\begin{equation}\label{estimate1_int}
\skliaustask D_y^jv\skliaustasd_{k,-\delta-1+j,I}^2 \lesssim_{j,k,\delta}  {\left\lVert D_y^j v_{|t =
0} \right\rVert}_{k,-\delta- 1
+ j}^2 + \int_I  {\left\lVert D_y^j f \right\rVert
}_{k-2,-\delta- \frac1 2 + j}^2 \mathrm{d} t
\quad\mbox{for} \quad\delta\in\left(0,\tfrac{1}{10}\right),
\end{equation}
where $j \in\mathbb{N}_0$. While estimate~\reftext{\eqref{estimate1_int}} is
maximal in the sense that one time derivative or four spatial derivatives
more are controlled for the solution $v$ than for the right-hand side $f$
(the maximal gain in regularity possible in view of
equation~\reftext{\eqref{eq_lin}}), the estimate is still too weak in order
to treat the nonlinearity ${\mathcal N}(v)$
(cf.~\reftext{\eqref{def_nonlinearity}}). The reason for this is a loss of
regularity of the employed norms on approaching the boundary $\{x = 0\}$. Increasing the number $k$ of derivatives in the norms does not improve the control on regularity because of $D_x
\stackrel{\text{\reftext{\eqref{log_der}}}}{=} x \partial_x$ and $D_y
\stackrel{\text{\reftext{\eqref{log_der}}}}{=} x \partial_y$,
while increasing $j$ only improves the scaling in the tangential
direction, and no matter how large we take $k$, the
norms $\sup_{t
\in I}   {\left\lVert D_y^j v \right\rVert}_{k,-\delta- 1 + j}$
cannot bound the norm ${\skliaustask\nabla v\skliaustasd}_{BC^0\left(I
\times(0,\infty) \times\mathbb{R}\right)}$ because $-\delta-1 < 1$. On the other
hand, control of the norm ${\skliaustask\nabla v\skliaustasd}_{BC^0\left(I
\times(0,\infty) \times\mathbb {R}\right)}$ appears to be necessary in order
to estimate the nonlinearity, because from \reftext{\eqref{def_nonlinearity}}
we recognize that factors $F \stackrel{\text{\reftext{\eqref{def_fg}}, \reftext{\eqref{def_v}}}}{=} (1+v_x)^{-1}$ need to be controlled.
This problem can be overcome by
combining parabolic estimates of the form \reftext{\eqref{estimate1_int}}
with elliptic regularity estimates, which will be dealt with in \S
\ref{sec:higher_1}, \S \ref{sec:higher_2}, and \S \ref{sec:higher_3}.

\subsubsection{Higher regularity I}\label{sec:higher_1}

In view of expansion~\reftext{\eqref{expansion_v}} we can only hope for estimates
with negative weights for $v$, since $v_0 = v_{|x = 0}$ can be nonzero (it is zero only if the free boundary is a straight line) and
norms $ {\left\lvert v \right\rvert}_\alpha$ can be
infinite if this happens. Applying $D_x-1$ to \reftext{\eqref{lin_cauchy}}, we get an equation in terms of $D_x v$ and $v$, which reads
Fourier-transformed in the $y$-variable
\begin{equation}\label{lin_cauchy2}
\left(x \partial_t + \tilde q(D_x)\right) D_x v - \eta^2 x^2\tilde
r(D_x) v + \eta^4 x^4(D_x+3) v = (D_x-1)f,
\end{equation}
where $\tilde q(\zeta)$ and $\tilde r(\zeta)$ are monic polynomials of
degrees $4$ and $3$, respectively:
%
%
\begin{subequations}\label{poly_qr2}
%
%
\begin{align}
\tilde q(\zeta) &:= \left(\zeta+ \tfrac1 2\right) \left(\zeta+
\beta
- \tfrac1 2\right) \left(\zeta-1\right) \left(\zeta- \beta-
1\right
),\label{poly_q2}
\eqncr
\tilde r(\zeta) &:= 2\left(\zeta+ 1\right)^2 \left(\zeta+ \tfrac12\right).\label{poly_r2}
\end{align}
\end{subequations}
Now we have $D_x v \stackrel{\text{\reftext{\eqref{expansion_v}}}}{=} O(x)$ as $x \searrow0$ and thus we expect to obtain
maximal-regularity estimates with new increased weight exponents. The
subsequent reasoning is similar to the one leading from \reftext{\eqref{lin_pde}}
to \reftext{\eqref{lin_est1}}, only that this time we test against $D_x v$
instead of $v$ in the inner product $(\cdot,\cdot)_{\tilde\alpha}$:
%
%
\begin{align}\label{test_cauchy2}
\begin{aligned}
& \left(x \partial_t D_x v, D_x v\right)_{\tilde\alpha} + \left
(\tilde
q(D_x) D_x v, D_x v\right)_{\tilde\alpha} - \eta^2 \left(x^2\tilde
r(D_x) v,D_x v\right)_{\tilde\alpha} + \eta^4 \left(x^4(D_x+3) v,D_x
v\right)_{\tilde\alpha} \\
& \quad= \left((D_x-1)f,D_x v\right)_{\tilde\alpha}.
\end{aligned}
\end{align}
This gives
%
%
\begin{equation}\label{1term1}
\left(x \partial_t D_x v, D_x v\right)_{\tilde\alpha} = \frac1 2
\frac
{\mathrm{d}}{\mathrm{d} t}  {\left\lvert D_x v \right
\rvert}_{\tilde\alpha-\frac1 2}^2
\end{equation}
as in \reftext{\eqref{1term}} for the first term of \reftext{\eqref{test_cauchy2}}. The new
mean $\tilde m$ and variance $\tilde\sigma$ of the zeros of $\tilde
q(\zeta)$, characterizing the coercivity range of $\tilde q(\zeta)$ as
in \reftext{Lemma~\ref{lem:coercivity}}, are given by
\begin{equation*}
\tilde m = \frac1 2 \quad\mbox{and} \quad\tilde\sigma^2 \stackrel
{\text{\reftext{\eqref{poly_q}}, \reftext{\eqref{def_beta}}}}{=} \frac1 2 \left(\beta^2 +
\frac12 \beta+ \frac3 4\right) = \frac3 4,
\end{equation*}
so that with criterion~\reftext{\eqref{coer_cond}} we infer that $\tilde q(\zeta
)$ is coercive with respect to $\tilde\alpha$ if
%
%
\begin{equation}\label{meanvar2}
\tilde\alpha\in\left(\frac1 2 - \beta, 1\right) \cap\left
[\tilde
m-\frac{\tilde\sigma}{\sqrt3}, \tilde m+\frac{\tilde\sigma
}{\sqrt
3}\right] = [0,1).
\end{equation}
Then indeed for $\tilde\alpha$ meeting \reftext{\eqref{meanvar2}} we have for the
second term in \reftext{\eqref{test_cauchy2}}
%
%
\begin{equation}\label{2term1}
\left(\tilde q(D_x) D_x v, D_x v\right)_{\tilde\alpha} \stackrel
{\text{\reftext{\eqref{coercivity_q}}}}{\gtrsim_{\tilde\alpha}}  {\left\lvert
D_x v \right\rvert}_{2,\tilde\alpha}^2.
\end{equation}
For the third term in \reftext{\eqref{test_cauchy2}} we obtain through
integration by parts and interpolation
\begin{equation}\label{3term1}
- \left(\tilde r(D_x) v, D_x v\right)_{{\tilde\alpha}-1} \stackrel
{\text{\reftext{\eqref{poly_r2}}}}{\ge}  {\left\lvert D_x v \right
\rvert}^2_{1,\tilde\alpha-1} - C  {\left\lvert v \right
\rvert}^2_{{\tilde\alpha}-1},
\end{equation}
where $C = C(\tilde\alpha)$ is sufficiently large. By interpolation
furthermore
\begin{equation}\label{4term1}
\left(x^4(D_x+3) v ,D_x v\right)_{\tilde\alpha} \ge\frac1 2
 {\left\lvert D_x v \right\rvert}_{\tilde\alpha-2}^2 -
C  {\left\lvert v \right\rvert}_{{\tilde\alpha}-2}^2,
\end{equation}
with $C = C(\tilde\alpha)$ sufficiently large. For the right-hand side
of \reftext{\eqref{test_cauchy2}} we may use the elementary
%
%
\begin{equation}\label{5term1}
\left((D_x-1)f, D_xv\right)_{\tilde\alpha} \le\frac
{1}{2\varepsilon}  {\left\lvert(D_x - 1) f \right\rvert
}_{\tilde\alpha}^2 + \frac{\varepsilon}{2}  {\left
\lvert D_x v \right\rvert}_{\tilde
\alpha}^2
\end{equation}
for $\varepsilon> 0$. Now we combine \reftext{\eqref{1term1}}, \reftext{\eqref{2term1}},
\reftext{\eqref{3term1}}, \reftext{\eqref{4term1}}, and \reftext{\eqref{5term1}} in \reftext{\eqref{test_cauchy2}} and get by absorption of the terms with prefactor $\varepsilon$ in \reftext{\eqref{5term1}} and
interpolation
\begin{equation}\label{lin_est2}
\frac{\mathrm{d}}{\mathrm{d} t}  {\left\lvert D_x v
\right\rvert}_{{\tilde\alpha}-\frac12}^2 +  {\left
\lvert D_xv \right\rvert}_{2,\tilde\alpha}^2 + \eta^2
{\left\lvert D_x v \right\rvert}_{1,\tilde\alpha-1}^2 +
\eta^4  {\left\lvert D_x v \right\rvert}_{\tilde\alpha
- 2}^2 \lesssim_{\tilde\alpha} {\left\lvert(D_x - 1) f
\right\rvert}_{\tilde\alpha}^2 + \eta^2  {\left\lvert
v \right\rvert}_{\tilde\alpha-1}^2 +
\eta^4  {\left\lvert v \right\rvert}_{\tilde\alpha-2}^2.
\end{equation}
We let ${\tilde k}\ge2$, $0 \le\ell\le\tilde k$, and test
equation~\reftext{\eqref{lin_cauchy2}} against
\begin{equation*}
T:= \eta^{2 \ell} (-D_x - 1 - 2 \ell+ 2 \tilde\alpha)^{\tilde k -
\ell
} D_x^{\tilde k + 1 - \ell} v
\end{equation*}
in the $(\cdot,\cdot)_{\tilde\alpha-\ell}$ inner product, giving
%
%
\begin{eqnarray}\label{eqn2test}
\begin{aligned}
&\eta^{2 \ell} \left(\partial_t D_x v, T\right)_{\tilde\alpha - \ell
-\frac12}
+ \eta^{2 \ell} \left(\tilde q(D_x) D_x v, T\right)_{\tilde
\alpha - \ell}
- \eta^{2 (\ell+1)} \left(\tilde r(D_x)v, T\right)_{\tilde\alpha - \ell
-1} \\
&+ \eta^{2 (\ell+2)} \left((D_x+3) v, T\right)_{\tilde\alpha - \ell-2} 
= \eta^{2 \ell} \left((D_x-1)f, T\right)_{\tilde\alpha - \ell}.
\end{aligned}
\end{eqnarray}
Then the terms appearing in \reftext{\eqref{eqn2test}} can be simplified
according to
%
%
\begin{subequations}\label{simplify_2t}
\begin{eqnarray}
\eta^{2 \ell} \left(\partial_t D_x v, T\right)_{{\tilde\alpha
} - \ell-\frac12} &=& \frac1 2 \eta^{2\ell} \frac{\mathrm{d}}{\mathrm{d} t} 
{\left\lvert D_x^{\tilde k + 1} v \right\rvert
}_{\tilde\alpha - \ell- \frac1 2}^2,
\eqncr
\eta^{2 \ell} \left(\tilde q(D_x)D_xv, T\right)_{\tilde\alpha - \ell}
&=& \eta
^{2 \ell} \left(\tilde q(D_x)(D_x-1)^{{\tilde k}-2-\ell}D_xv,
(D_x+1+2\ell-2{\tilde
\alpha})^2D_x^{\tilde k+1-\ell}v\right)_{\tilde\alpha - \ell}\nonumber\\
&\ge& \eta^{2 \ell} \left(\frac1 2  {\left\lvert D_x v
\right\rvert}_{\tilde k + 2 -\ell, \tilde
\alpha - \ell}^2 - \tilde C  {\left\lvert D_x v \right\rvert
}_{\tilde\alpha - \ell}^2\right),\label{2term2}
\eqncr
-\eta^{2 (\ell+1)} \left(\tilde r(D_x)v, T\right)_{{\tilde\alpha} - \ell-1}
&\stackrel{\text{\reftext{\eqref{poly_r2}}}}{\ge}& \eta^{2 (\ell+1)} \left
( {\left\lvert D_x v \right\rvert}_{\tilde k + 1-\ell, \tilde
\alpha - \ell-1}^2 - \tilde C  {\left\lvert D_x v \right\rvert
}_{{\tilde
\alpha} - \ell-1}^2\right),
\eqncr
\eta^{2 (\ell+2)} \left((D_x+3)v, T\right)_{{\tilde\alpha} - \ell-2}
&\ge& \eta
^{2 (\ell+2)} \left(\frac1 2  {\left\lvert D_x v \right
\rvert}_{\tilde k-\ell, \tilde\alpha - \ell
-2}^2 - \tilde C  {\left\lvert D_x v \right\rvert
}_{{\tilde\alpha} - \ell-2}^2\right),
\end{eqnarray}
\end{subequations}
where in the last two estimates we used the elliptic regularity result of \reftext{Lemma \ref{lem:mainhardy}}, contained in \reftext{\S \ref{sec:elliptic}}. Here the constant $\tilde C$ only depends on $\tilde\alpha$ and
$\tilde k$. The right-hand side of \reftext{\eqref{eqn2test}} can be estimated
as follows:
%
%
\begin{eqnarray}\label{5term2}
\begin{aligned}
& \eta^{2 \ell} ((D_x-1) f,T)_{\tilde\alpha - \ell} \\
& \lesssim_{\tilde k,\tilde\alpha}
\begin{cases} \varepsilon^{-1}  {\left\lvert(D_x-1) f
\right\rvert}_{\tilde k-2,\tilde\alpha}^2 +
\varepsilon {\left\lvert D_x v \right\rvert}_{\tilde
k+2,\tilde\alpha}^2 & \mbox{for } \, \ell=
0, \\ \varepsilon^{-1}  {\left\lvert(D_x-1) f \right
\rvert}_{\tilde k-2,\tilde\alpha}^2 + \varepsilon\,
\eta^4  {\left\lvert D_x v \right\rvert}_{\tilde
k,\tilde\alpha-2}^2 & \mbox{for } \, \ell
= 1, \\ \varepsilon^{-1} \eta^{2 (\ell-2)}  {\left\lvert
(D_x-1) f \right\rvert}_{\tilde k-\ell
,\tilde\alpha+2-\ell}^2 + \varepsilon\, \eta^{2 (\ell+2)}
 {\left\lvert D_x v \right\rvert}_{\tilde k-\ell,\tilde
\alpha-\ell-2}^2 & \mbox{for } \, 2 \le\ell
\le k,
\end{cases}
\end{aligned}\nonumber\\
\end{eqnarray}
for any $\varepsilon> 0$. Now we insert \reftext{\eqref{simplify_2t}} and \reftext{\eqref{5term2}}, with $\varepsilon> 0$ sufficiently small, into \reftext{\eqref{eqn2test}} and
arrive at
%
%
\begin{equation}\label{lin_est1.5}
\begin{aligned}
& \eta^{2\ell} \frac{\mathrm{d}}{\mathrm{d} t}
{\left\lvert D_x^{\tilde k + 1 - \ell} v \right\rvert}_{{\tilde
\alpha - \ell} - \frac1 2}^2 + \sum_{\ell^\prime=0}^2 \eta
^{2 \left(\ell+\ell^\prime\right)} \left( {\left
\lvert D_x v \right\rvert}_{\tilde k + 2 -
\ell- \ell^\prime, \tilde\alpha- \ell-\ell^\prime}^2 - \tilde C
 {\left\lvert D_x v \right\rvert}_{\tilde\alpha- \ell-
\ell^\prime}^2 \right) \\
& \quad\lesssim_\alpha
\begin{cases}  {\left\lvert f \right\rvert}_{k-2,\alpha
}^2 & \mbox{for } \, \ell\in\{0,
1\}. \\ \eta^{2 (\ell-2)}  {\left\lvert f \right\rvert
}_{k-\ell,\alpha+2-\ell}^2 & \mbox
{for } \, 2 \le\ell\le k.
\end{cases}
\end{aligned}
\end{equation}
Then we infer inductively from \reftext{\eqref{lin_est2}} and \reftext{\eqref{lin_est1.5}} that
%
%
\begin{equation}\label{estimate2-}
\begin{aligned}
& \frac{\mathrm{d}}{\mathrm{d} t} \left(\sum_{\ell= 0}^{\tilde k}
\tilde C_\ell\, \eta
^{2 \ell} {\left\lvert D_x^{\tilde k + 1 - \ell} v
\right\rvert}_{\tilde\alpha-\frac12-\ell}^2
+ \left\lvert D_x v\right\rvert_{\tilde\alpha-\frac12}^2\right) + \sum
_{\ell=0}^{\tilde k+2} \eta^{2 \ell}  {\left\lvert D_x v
\right\rvert}_{\tilde k + 2 - \ell
,\tilde\alpha-\ell}^2 \\
& \quad\lesssim_{\tilde k,\tilde\alpha} \sum_{\ell= 0}^{\tilde k-2}
\eta^{2 \ell}  {\left\lvert(D_x - 1) f \right\rvert
}_{\tilde k-2-\ell,\tilde\alpha-\ell
}^2 + \eta^2  {\left\lvert v \right\rvert}_{\tilde
\alpha-1}^2 + \eta^4  {\left\lvert v \right\rvert
}_{\tilde
\alpha-2}^2,
\end{aligned}
\end{equation}
with constants $\tilde C_\ell> 0$ and $\tilde\alpha\in[0,1)$.
Passing to physical space-time $(t,x,y)$ by Fourier back transforming in
$\eta$, we get with the same steps as those leading from \reftext{\eqref{estimate1}} to \reftext{\eqref{estimate1_int}}
%
%
\begin{equation}\label{estimate2-_int}
{\skliaustask D_y^{\tilde j} D_x v \skliaustasd}_{\tilde k,\tilde
\alpha- 1 + \tilde j, I}^2 \lesssim_{\tilde j, \tilde k,\tilde\alpha}  {\left\lVert
D_y^{\tilde j} D_x v_{|t = 0} \right\rVert}_{\tilde k,\tilde\alpha-
1 + \tilde j}^2 + \int_I  {\left\lVert D_y^{\tilde j}
(D_x-1) f \right\rVert}_{\tilde k-2,\tilde\alpha- \frac1 2 +
\tilde j}^2 \mathrm{d} t + \sum_{\ell= 1}^2 \int_I
{\left\lVert D_y^{\tilde j + \ell} v \right\rVert}_{\tilde\alpha
- \frac1 2 + \tilde j}^2 \mathrm{d} t,
\end{equation}
where $\tilde j \in\mathbb{N}_0$ is arbitrary. Observe that we can
absorb the
remnant terms, i.e.,
%
%
\begin{equation}\label{remnant_terms}
\sum_{\ell= 1}^2 \int_I  {\left\lVert D_y^{\tilde j +
\ell} v \right\rVert}_{\tilde\alpha
- \frac1 2 + \tilde j}^2 \mathrm{d} t
\end{equation}
in the second line of \reftext{\eqref{estimate2-_int}}, by those appearing on the
left-hand side of \reftext{\eqref{estimate1_int}} if $\tilde\alpha= 1-\delta$,
if $-\delta$ is in the intersection of $\left(-\frac{1}{10},0\right)$
(cf.~\reftext{\eqref{estimate1_int}}) and the range \reftext{\eqref{meanvar2}} (which
simply leads to $\delta\in\left(0,\frac{1}{10}\right)$ again), and if
$j = \tilde j + 1$. Then adding a large multiple of \reftext{\eqref{estimate1_int}} to \reftext{\eqref{estimate2-_int}} to eliminate \reftext{\eqref{remnant_terms}} (where for convenience we make the choice $k := \tilde
k-1$), we obtain a higher-order maximal-regularity estimate of the form
%
%
\begin{align}\nonumber
& \skliaustask D_y^j v\skliaustasd_{\tilde k - 1, -\delta-1 + j,I}^2 + \skliaustask D_y^{j-1} D_x v \skliaustasd_{\tilde k,-\delta- 1 + j,I}^2 \\
& \quad\lesssim_{j,\tilde k, \delta} {\left\lVert D_y^j v_{|t =
0} \right\rVert}_{\tilde k -
1,-\delta- 1 + j}^2 + {\left\lVert D_y^{j-1} D_x v_{|t =
0} \right\rVert}_{\tilde
k,-\delta- 1 + j}^2 \nonumber\\
& \quad \phantom{\lesssim_{j,\tilde k, \delta}} + \int_I \left({\left\lVert D_y^j f \right\rVert
}_{\tilde k-3,-\delta- \frac1 2 +
j}^2 + {\left\lVert D_y^{j-1} (D_x-1) f \right\rVert
}_{\tilde k-2,- \delta- \frac1 2 +
j}^2\right) \mathrm{d} t \label{higher_est_1}
\end{align}
for $\delta\in\left(0,\frac{1}{10}\right)$, where $\tilde k \ge3$
and $j \ge1$. Estimate~\reftext{\eqref{higher_est_1}} gives stronger control of the solution
$v$ in terms of the initial data $v_{|t = 0}$ and the right-hand side
$f$, as the weight $\alpha$ has increased by $+1$. Yet,
estimate~\reftext{\eqref{higher_est_1}} is still insufficient to treat the
nonlinear problem \reftext{\eqref{nonlin_cauchy}}. The reason is that the scaling of neither $\sup
_{t \in I}  {\left\lVert D_y^j v \right\rVert}_{\tilde
k,-\delta-1+j}$ nor $\sup_{t \in I}
 {\left\lVert D_y^{j-1} D_x v \right\rVert}_{\tilde
k,-\delta- 1 + j}$ is sufficient in
order to get control of ${\skliaustask\nabla v\skliaustasd
}_{BC^0\left(I \times
(0,\infty) \times\mathbb{R}\right)}$ (for details we refer to the discussion
of such scalings at the end of \S \ref{sec:norms}). This requires to
apply another polynomial in $D_x$ to \reftext{\eqref{lin_cauchy2}}, which will be
the subject of \S \ref{sec:higher_2}.

Estimate~\reftext{\eqref{higher_est_1}} is $\delta$-sub-critical with respect to
the addend $v_1 x$ in expansion~\reftext{\eqref{expansion_v}} for $v$ while
estimate~\reftext{\eqref{estimate1_int}} is $\delta$-sub-critical with respect to
the addend $v_0$ of $v$ in \reftext{\eqref{expansion_v}}. For convenience for the
subsequent reasoning, we will also use an estimate which is $\delta
$-super-critical with respect to $v_0$. Therefore, we additionally use the values $\tilde\alpha= \delta$,
$\tilde k = k$, and $\tilde j = j-1$ where $j \ge1$ in \reftext{\eqref{estimate2-_int}}, so that in view of \reftext{\eqref{remnant_terms}} the remnant
contributions are $\int_I  {\left\lVert D_y^{j+\ell-1} v \right
\rVert}_{\delta- \frac 3 2 + j}^2 \mathrm{d}
t$ with $\ell= 1,2$. By the definition of the norm \reftext{\eqref{norm_2d}} we
have the interpolation estimate
\begin{equation*}
 {\left\lVert D_y^{j+\ell-1} v \right\rVert}_{\delta-
\frac3 2 + j} \lesssim {\left\lVert D_y^{j+\ell-1} v
\right\rVert}_{- \delta- \frac3 2 + j} +  {\left\lVert
D_y^{j+\ell-1} v \right\rVert}_{- \delta- \frac1 2 + j}.
\end{equation*}
Then we notice that $\int_I {\left\lVert D_y^{j+\ell-1} v
\right\rVert}_{- \delta- \frac32 + j}^2 \mathrm{d} t$ can be absorbed by $\int_I  {\left
\lVert D_y^{j-1} v \right\rVert}_{k+2,-\delta
-\frac3 2 + j}^2 \mathrm{d} t$ appearing on the left-hand side of \reftext{\eqref{estimate1_int}} with $j$ replaced by $j-1$ while $\int_I
{\left\lVert D_y^{j+\ell-1} v \right\rVert}_{- \delta - \frac 1 2 + j}^2
\mathrm{d} t$ can be absorbed by $\int_I
 {\left\lVert D_y^j v \right\rVert}_{\tilde k + 2, -
\delta- \frac1 2 + j}^2 \mathrm{d} t$,
appearing in \reftext{\eqref{estimate1_int}} as well where $j$ is the same and
$k$ is replaced by $\tilde k - 1$. By combining these estimates, we
find an inequality of the same form as \reftext{\eqref{higher_est_1}} for the
exponent $\tilde\alpha=\delta$:
%
%
\begin{align}\nonumber
& \skliaustask D_y^{j-1}v\skliaustasd_{k,-\delta-2+j,I}^2 +
\skliaustask D_y^{j-1}D_x v \skliaustasd_{k,\delta
-2+j,I}^2 + \skliaustask D_y^jv \skliaustasd_{\tilde k -
1,-\delta-1 + j,I}^2 \\
& \quad \lesssim_{j, k, \tilde k, \delta} {\left\lVert D_y^{j-1}
v_{|t = 0} \right\rVert}_{k,-\delta-2+j}^2 + {\left
\lVert D_y^{j-1} D_x v_{|t = 0} \right\rVert}_{k,\delta
-2+j}^2 + {\left\lVert D_y^j v_{|t = 0} \right\rVert
}_{\tilde k - 1,-\delta-1+j}^2
\nonumber\\
& \quad \phantom{\lesssim_{j, k, \tilde k, \delta}} + \int_I \left({\left\lVert D_y^{j-1} f \right\rVert
}_{k-2,-\delta- \frac3 2 + j}^2
+ {\left\lVert D_y^{j-1} (D_x-1) f \right\rVert
}_{k-2,\delta- \frac3 2 + j}^2 + {\left\lVert D_y^j f
\right\rVert}_{\tilde k-3,-\delta- \frac1 2 + j}^2\right) \mathrm
{d} t, \label{higher_est_1_alt}
\end{align}
where $\delta\in\left(0,\frac{1}{10}\right)$, $k \ge2$, $\tilde k
\ge3$, and $j \ge1$.

\subsubsection{Higher regularity II}\label{sec:higher_2}

Next, we apply $\tilde q(D_x - 1)$ to \reftext{\eqref{lin_cauchy2}} and arrive at
%
%
\begin{align}\label{lin_eq_qd}
\left( x\partial_t + \tilde q(D_x - 1)\right) \tilde q(D_x) D_x v -
\eta
^2 x^2 \check r_1(D_x) D_x v + \eta^4 x^4 \check r_2(D_x) v = \tilde
q(D_x-1) (D_x-1) f,
\end{align}
where in view of \reftext{\eqref{poly_qr2}}, $\check
r_1(\zeta)$ and $\check r_2(\zeta)$ are polynomials of degrees six and five, respectively, defined by
\begin{subequations}\label{def_check_r12}
\begin{eqnarray}
\check r_1(\zeta) &:=& \zeta^{-1}\tilde q(\zeta+ 1)\tilde r(\zeta) = 2
\left(\zeta+ \tfrac3 2\right) \left(\zeta+1\right)^2 \left
(\zeta+
\tfrac1 2\right) \left(\zeta+ \beta+ \tfrac1 2\right) \left
(\zeta-\beta\right), \label{def_check_r1}\\
\check r_2(\zeta) &:=& \tilde q(\zeta+3)(\zeta+3) \stackrel{\text{\reftext{\eqref{poly_q2}}}}{=} \left(\zeta+ \tfrac72\right) \left(\zeta+ \beta+
\tfrac{5}{2}\right) \left(\zeta+ 3\right) \left(\zeta+ 2\right)
\left
(\zeta- \beta+ 2\right).\label{def_check_r2}
\end{eqnarray}
\end{subequations}
The motivation for this step is that in view of expansion~\reftext{\eqref{expansion_v}} we expect $\tilde q(D_x) D_x v(x) = O\left(x^2\right)$ as $x
\searrow0$, as the powers $x$ and $x^{1+\beta}$ are in the kernel of
$\tilde q(D_x)$ (cf.~\reftext{\eqref{poly_q2}}). This enables us to obtain
estimates in norms with larger weight exponents and therefore stronger control
at the free boundary $\{x = 0\}$.

We test \reftext{\eqref{lin_eq_qd}} with $\tilde q(D_x) D_x v$ in the inner
product $\left(\cdot,\cdot\right)_{\check\alpha}$, where $\check
\alpha$
is in the coercivity range of $\tilde q(D_x-1)$. The latter is the
coercivity range of $\tilde q(D_x)$ shifted by $+1$, that is, it suffices to have $\check
\alpha \in [1,2)$. This gives, like before in \reftext{\eqref{lin_est2}},
\begin{equation}\label{lin_est3}
\begin{aligned}
& \frac{\mathrm{d}}{\mathrm{d} t}  {\left\lvert\tilde
q(D_x) D_x v \right\rvert}_{\check\alpha- \frac
12}^2 +  {\left\lvert\tilde q(D_x) D_x v \right\rvert
}_{2,\check\alpha}^2 + \eta^2  {\left\lvert D_x v \right
\rvert}_{5,\check\alpha-1}^2+ \eta^4  {\left\lvert v
\right\rvert}_{5,\check\alpha-2}^2 \\
& \quad\lesssim_{\check\alpha}  {\left\lvert\tilde
q(D_x-1) (D_x-1) f \right\rvert}_{\check\alpha}^2 + \eta^2
 {\left\lvert D_x v \right\rvert}_{\check\alpha-1}^2 +
\eta^4
 {\left\lvert v \right\rvert}_{\check\alpha-2}^2.
\end{aligned}
\end{equation}
The only difference in the argumentation leading to \reftext{\eqref{lin_est3}}
compared to the reasoning before \reftext{\eqref{lin_est2}} is the second but
last term $ {\left\lvert D_x v \right\rvert}_{\check
\alpha-1}^2$ in the second line, which
can be obtained as in \reftext{\eqref{3term1}} noting that $\tilde q(\zeta+1)$
has a root in $\zeta= 0$. With the same reasoning as before, we can
upgrade \reftext{\eqref{lin_est3}} to
\begin{align}\label{estimate3-}
\begin{aligned}
& \frac{\mathrm{d}}{\mathrm{d} t} \left(\sum_{\ell= 0}^{\check k}
\check C_\ell\, \eta
^{2 \ell}  {\left\lvert D_x^{\check k - \ell}\tilde
q(D_x)D_x v \right\rvert}_{\check\alpha
-\frac1 2-\ell}^2 +  {\left\lvert\tilde q(D_x)D_x v
\right\rvert}_{\check\alpha- \frac
1 2}^2\right) + \sum_{\ell=0}^{\check k+2} \eta^{2 \ell}
 {\left\lvert\tilde q(D_x) D_x v \right\rvert}_{\check
k + 2 - \ell,\check\alpha-\ell}^2 \\
& \quad\lesssim_{\check k,\check\alpha} \sum_{\ell= 0}^{\check k-2}
\eta^{2 \ell}  {\left\lvert\tilde q(D_x-1)(D_x - 1) f
\right\rvert}_{\check k-2-\ell
,\check\alpha-\ell}^2 + \eta^2  {\left\lvert D_x v
\right\rvert}_{\check\alpha-1}^2 + \eta
^4  {\left\lvert v \right\rvert}_{\check\alpha-2}^2,
\end{aligned}
\end{align}
with constants $\check C_\ell > 0$ and $\check\alpha \in [1,2)$. Then
from the above we obtain a strong estimate of the form:
\begin{eqnarray}\nonumber
\skliaustask D_y^{\check j}
\tilde q(D_x) D_x v \skliaustasd_{\check k,\check\alpha- 1 +
\check j,I}^2 &\lesssim_{\check j, \check k,\check\alpha}& {\left\lVert
D_y^{\check j} \tilde q(D_x) D_x v_{|t = 0} \right\rVert}_{\check
k,\check\alpha- 1 + \check j}^2 + \int_I
{\left\lVert D_y^{\check j} \tilde q(D_x-1) (D_x-1) f
\right\rVert}_{\check k-2,\check
\alpha- \frac1 2 + \check j}^2 \mathrm{d} t \nonumber\\
&& + \sum_{\ell= 1}^2 \int_I {\left\lVert D_y^{\check j
+ \ell} D_x v \right\rVert}_{\check\alpha- \frac1 2 + \check j}^2
\mathrm{d} t. \label{estimate3-_int}
\end{eqnarray}
The last line in \reftext{\eqref{estimate3-_int}} contains the remnant terms,
which can be absorbed by adding a large multiple of \reftext{\eqref{higher_est_1}}, where we choose $j = \check j + 2$ and $\tilde k = \check k - 1$ provided $\check\alpha= -\delta
+ 2$. Thus,
estimate~\reftext{\eqref{estimate3-_int}} upgrades to
\begin{align}\nonumber
& \skliaustask D_y^j v \skliaustasd_{\check k - 2,-\delta-
1 + j,I}^2 +  {\skliaustask D_y^{j-1} D_x v \skliaustasd}_{\check k - 1,-\delta- 1 + j, I}^2 +
 {\skliaustask D_y^{j-2} \tilde q(D_x) D_x v \skliaustasd}_{\check k,-\delta- 1 +
j,I}^2 \nonumber\\
& \quad\lesssim_{j, \check k, \delta} {\left\lVert D_y^j v_{|t =
0} \right\rVert}_{\check k -
2,-\delta- 1 + j}^2 \! +\!  {\left\lVert D_y^{j-1} D_x v_{|t =
0} \right\rVert}_{\check k -
1,-\delta- 1 + j}^2 \! +\!  {\left\lVert D_y^{j-2} \tilde
q(D_x) D_x v_{|t = 0} \right\rVert}_{\check k,-\delta- 1 + j}^2
\nonumber\\
& \quad\phantom{\lesssim_{j, \check k, \delta}} + \int_I \left( {\left\lVert D_y^j f \right\rVert
}_{\check k - 4,-\delta- \frac1 2 +
j}^2 +  {\left\lVert D_y^{j-1} (D_x-1) f \right\rVert
}_{\check k - 3,-\delta- \frac1 2 +
j}^2\right) \mathrm{d} t \nonumber\\
& \quad\phantom{\lesssim_{j, \check k, \delta}} + \int_I  {\left\lVert D_y^{j-2} \tilde q(D_x-1)
(D_x-1) f \right\rVert}_{\check
k-2,-\delta- \frac1 2 + j}^2 \mathrm{d} t
\label{higher_est_2}
\end{align}
for $\delta\in\left(0,\frac{1}{10}\right)$, where $\check k \ge4$
and $j \ge2$.

Estimate~\reftext{\eqref{higher_est_2}} is $\delta$-sub-critical with respect to
the addend $v_2 x^2$ in expansion~\reftext{\eqref{expansion_v}}. For later
purposes we also choose the weight $\check\alpha= \delta+ 1$ in \reftext{\eqref{estimate3-_int}} being $\delta$-super-critical with respect to
the addend $v_1 x$ in \reftext{\eqref{expansion_v}} (i.e., respective norms blow up unless this addend is annihilated). Choosing $\check k :=
\tilde
k$ and $\check j := j-2$, where $j \ge2$, it then suffices to bound
the remnant terms $\int_I  {\left\lVert D_y^{j+\ell-2}
D_x v \right\rVert}_{\delta
-\frac3 2+j}^2 \mathrm{d} t$, where $\ell= 1, 2$. We now use that
$\int_I
 {\left\lVert D_y^{j+\ell-2} D_x v \right\rVert}_{\delta
-\frac3 2 + j}^2 \mathrm{d} t$ can be
absorbed by $\int_I  {\left\lVert D_y^{j-1} v \right
\rVert}_{\tilde k+1,\delta- \frac3 2
+ j}^2 \mathrm{d} t$ appearing on the left-hand side of \reftext{\eqref{higher_est_1_alt}}, where $k$ is replaced by $\tilde k - 1$ and $\tilde
k$ is replaced by $\check k - 1$. By combining these estimates, we find
the following bounds for the case $\check \alpha = \delta+1$
%
%
\begin{align}\nonumber
& {\skliaustask D_y^{j-1} v
\skliaustasd}_{\tilde k - 1,-\delta
-2+j,I}^2 +  {\skliaustask D_y^{j-1} D_x v \skliaustasd}_{\tilde k - 1,\delta-2+j,I}^2 +  {\skliaustask D_y^{j-2}
\tilde q(D_x) D_x v \skliaustasd}_{\tilde k,\delta- 2 + j,I}^2 + {\skliaustask D_y^j v \skliaustasd}_{\check k - 2,-\delta-1 + j, I}^2 \nonumber\\
& \quad\lesssim_{j, \tilde k, \check k, \delta} {\left\lVert
D_y^{j-1} v_{|t = 0} \right\rVert}_{\tilde k - 1,-\delta-2+j}^2 \! +\!
 {\left\lVert D_y^{j-1} D_x v_{|t = 0} \right\rVert
}_{\tilde k - 1,\delta-2+j}^2 \! +\!  {\left\lVert D_y^{j-2}
\tilde q(D_x) D_x v_{|t = 0} \right\rVert}_{\tilde k,\delta- 2 +
j}^2 \qquad\qquad\nonumber\\
& \quad\phantom{\lesssim_{j, \tilde k, \check k, \delta}} +  {\left\lVert D_y^j v_{|t = 0} \right\rVert}_{\check
k - 2,-\delta-1+j}^2 \nonumber\\
& \quad\phantom{\lesssim_{j, \tilde k, \check k, \delta}} + \int_I \left( {\left\lVert D_y^{j-1} f \right\rVert
}_{\tilde k - 3,-\delta- \frac32 + j}^2 +  {\left\lVert D_y^{j-1} (D_x-1) f \right\rVert
}_{\tilde k - 3,\delta- \frac32 + j}^2\right) \mathrm{d} t \nonumber\\
& \quad\phantom{\lesssim_{j, \tilde k, \check k, \delta}} + \int_I \left( {\left\lVert D_y^{j-2} \tilde q(D_x-1)
(D_x-1) f \right\rVert}_{\tilde
k-2,\delta- \frac3 2 + j}^2 +  {\left\lVert D_y^j f
\right\rVert}_{\check k - 4,-\delta
- \frac1 2 + j}^2\right) \mathrm{d} t \label{higher_est_2_alt}
\end{align}
for $\delta\in\left(0,\frac{1}{10}\right)$, $\tilde k \ge3$,
$\check
k \ge4$, and $j \ge2$.

\subsubsection{Higher regularity III}\label{sec:higher_3}

In order to additionally control the term $v_2 x^2$ in \eqref{expansion_v}, we apply $(D_x-4) (D_x-3)$ to equation~\reftext{\eqref{lin_eq_qd}} leading to
\begin{align}\label{lin_eq_qdqd}
\begin{aligned}
& \left( x\partial_t + \breve q(D_x)\right)(D_x-3) (D_x-2) \tilde
q(D_x) D_x v - \eta^2 x^2 \breve r_1(D_x) (D_x-1) D_x v + \eta^4 x^4
\breve r_2(D_x) D_x v \\
& \quad= (D_x-4) (D_x-3) \tilde q(D_x-1) (D_x-1) f,
\end{aligned}
\end{align}
where in view of \reftext{\eqref{poly_r2}}
%
%
\begin{subequations}\label{poly_breve}
%
%
\begin{eqnarray}
\breve r_1(\zeta) &=& (\zeta-2) \check r_1(\zeta) \stackrel{\text{\reftext{\eqref{def_check_r1}}}}{=} 2 \left(\zeta+ \tfrac3 2\right) \left(\zeta
+1\right
)^2 \left(\zeta+ \tfrac1 2\right) \left(\zeta+ \beta+ \tfrac12\right) \left(\zeta- \beta\right) (\zeta-2),\qquad \label{poly_breve_r} \\
\breve r_2(\zeta) &=& (\zeta+1) \check r_2(\zeta) \stackrel{\text{\reftext{\eqref{def_check_r2}}}}{=} (\zeta+1) \tilde q(\zeta+3)
(\zeta
+3) \nonumber \\
&=& \left(\zeta+ \tfrac72\right) \left(\zeta+ \beta+ \tfrac
{5}{2}\right) \left(\zeta+ 3\right) \left(\zeta+ 2\right) \left
(\zeta
- \beta+ 2\right) (\zeta+1)
\end{eqnarray}
are polynomials of order $7$ and $6$, respectively, and where we have
introduced the operator $\breve q(D_x)$ fulfilling $(D_x-4) (D_x-3)
\tilde q(D_x-1) = \breve q(D_x) (D_x-3) (D_x-2)$, i.e., (cf.~\reftext{\eqref{poly_q2}}) $\breve q(\zeta)$ is given by
%
%
\begin{equation}\label{def_breve_q}
\breve q(\zeta) = \left(\zeta- \tfrac1 2\right) \left(\zeta+
\beta-
\tfrac3 2\right) \left(\zeta- \beta- 2\right) \left(\zeta-
4\right).
\end{equation}
\end{subequations}
Our motivation for applying the operator $(D_x-4) (D_x-3)$ is that we
want the resulting operator $\breve q(D_x)$ to be coercive with respect
to a weight which is $\delta$-super-critical with respect to the term
$v_2x^2$ from expansion~\reftext{\eqref{expansion_v}}. We are led to prefer the
weight $\delta+2$ since remnant contributions can be easily absorbed
using estimates from the previous steps (see below). Note that applying operators
$D_x-\gamma$, where $\gamma-1$ is a root
of $\check q(\zeta)$, preserves the structure of the $\eta^0$-terms on
the left-hand side of \reftext{\eqref{lin_eq_qd}}. Since the root $2$ is the
third-largest root of $\tilde q(\zeta-1)$ and by criterion \reftext{\eqref{coer_cond_1}} of \reftext{Lemma~\ref{lem:coercivity}} limits the coercivity
range, we are lead to apply $D_x-3$ to \reftext{\eqref{lin_eq_qd}}, resulting in
the linear operator
\begin{equation*}
\left(D_x - \tfrac1 2\right) \left(D_x + \beta- \tfrac3 2\right)
\left(D_x - \beta- 2\right) \left(D_x - 3\right)
\end{equation*}
acting on $v$ in the $\eta^0$-term. This operator has coercivity range
containing a neighborhood around~$2$. It also produces a factor $D_x-1$
acting on $D_x v$ in the $\eta^2$-term, so that the fist two terms
$v_0$ and $v_1 x$ in expansion~\reftext{\eqref{expansion_v}} are canceled and in
a weak estimate with weight $\delta+2$, where $\delta> 0$, the
resulting term remains finite. However, in the resulting equation the
$\eta^4$-term has no $D_x$ acting on $v$, which because of the boundary
value $v_0$ (cf.~\reftext{\eqref{expansion_v}}) would lead to a blow-up in a weak
estimate with weight $\delta+2$ where $\delta> 0$. For this reason, we
need the operator to be applied to~\reftext{\eqref{lin_eq_qd}} to have a factor
$(D_x-4)$, so that after commuting with $x^4$, the operator $D_x$ acts
on $v$ in the $\eta^4$-term.

Further note that the operator $(D_x-3) (D_x-2) \tilde q(D_x) D_x$
acting on $v$ (cf.~\reftext{\eqref{poly_q2}}) does not cancel the addend $O\left
(x^{1+2\beta}\right)$ from \reftext{\eqref{expansion_v}}, which, as it turns out
later, generically appears via the nonlinear equation (cf.~\S \ref{sec:nonlinear_struct}).
However, the weight $\delta+2$,
which we introduce for $\delta\in\left(0,\frac{1}{10}\right)$ as before,
meets $\delta+2 < 1+2\beta$ in view of \reftext{\eqref{def_beta}} and thus also
avoids the aforementioned $O(x^{1+2\beta})$-terms.\footnote{In the
framework of the more general thin-film equation \reftext{\eqref{tfe_general}},
where $\beta$ is a function of the mobility exponent $n$, this
restricts our analysis to those $n$ for which $\beta> \frac1 2$.}

It remains to verify coercivity of $\breve q(D_x)$ employing
criterion~\reftext{\eqref{coer_cond}} of \reftext{Lemma~\ref{lem:coercivity}}. We find that
the mean of the zeros of $\breve q(\zeta)$ is given by $\breve m = 2$
and the root of the variance fulfills $\breve\sigma> \frac5 4 >
\frac
{\sqrt3}{10}$, so that indeed weights $2+\delta$ with $\delta\in
\left
(0,\frac{1}{10}\right)$ are admissible.

Testing \reftext{\eqref{lin_eq_qdqd}} with $(D_x-3) (D_x-2) \tilde q(D_x) D_x v$
in the inner product $\left(\cdot,\cdot\right)_{\delta+2}$, we obtain
analogously to the computations in \S \ref{sec:basic_weak} and \S \ref{sec:higher_1}
\begin{eqnarray}\label{test_higher_3}
\begin{aligned}
& \frac{\mathrm{d}}{\mathrm{d} t}  {\left\lvert(D_x-2)
\tilde q(D_x) D_x v \right\rvert}_{1,\delta+ \frac
3 2}^2 +  {\left\lvert(D_x-2) \tilde q(D_x) D_x v \right
\rvert}_{3,\delta+2}^2 + \eta^2
 {\left\lvert(D_x-1) D_x v \right\rvert}_{6,\delta+1}^2
+ \eta^4  {\left\lvert D_x v \right\rvert}_{6,\delta
}^2 \\
& \quad\lesssim_\delta {\left\lvert(D_x-4) (D_x-3)
\tilde q(D_x-1) (D_x-1) f \right\rvert}_{\delta+2}^2 + \eta^2
 {\left\lvert (D_x-1) D_x v \right\rvert}_{\delta
+1}^2 + \eta
^4  {\left\lvert D_x v \right\rvert}_\delta^2.
\end{aligned}
\end{eqnarray}
Upgrading the weak estimate~\reftext{\eqref{test_higher_3}} follows the lines of
the respective discussion when passing from the weak estimates \reftext{\eqref{lin_est1}}, \reftext{\eqref{lin_est2}}, and \reftext{\eqref{lin_est3}} to the strong
estimates \reftext{\eqref{estimate1}}, \reftext{\eqref{estimate2-}}, and \reftext{\eqref{estimate3-}}, respectively. Therefore, we skip all details and arrive at
\begin{align}
& \frac{\mathrm{d}}{\mathrm{d} t} \left(\sum_{\ell= 0}^{\check k}
\breve C_\ell\eta
^{2\ell}  {\left\lvert D_x^{\check k - \ell} (D_x-3)
(D_x-2) \tilde q(D_x) D_x v \right\rvert}_{\delta+ \frac3 2 - \ell
}^2 + {\left\lvert(D_x-3) (D_x-2) \tilde q(D_x) D_x v
\right\rvert}_{\delta+ \frac3 2}^2\vphantom{\sum_{\ell= 0}^{\check k}}\right) \nonumber\\
& + \sum_{\ell= 0}^{\check k + 2} \eta^{2\ell}  {\left
\lvert(D_x-3) (D_x-2) \tilde q(D_x) D_x v \right\rvert}_{\check k +
2 - \ell,\delta+2 - \ell}^2 \label{strong_higher_3}\\
& \quad\lesssim_{\check k,\delta} \sum_{\ell= 0}^{\check k - 2}
\eta
^{2\ell}  {\left\lvert(D_x-4) (D_x-3) \tilde q(D_x-1)
(D_x-1) f \right\rvert}_{\check k -
2 - \ell,\delta+2}^2 + \eta^2  {\left\lvert (D_x-1) D_x v \right\rvert}_{\delta+1}^2
+ \eta^4  {\left\lvert D_x v \right\rvert}_\delta
^2,\nonumber
\end{align}
with constants $\check k \ge2$, $\breve C_\ell> 0$, and $\delta\in
\left(0,\frac{1}{10}\right)$. In integrated form and after employing
equation~\reftext{\eqref{lin_eq_qdqd}} to obtain control on the time derivative
$\partial_t v$, we find
\begin{align}\nonumber
& {\skliaustask D_y^{\breve j} (D_x-3) (D_x-2) \tilde q(D_x) D_x v \skliaustasd}_{\check k,\delta+1 + \breve j, I}^2 \\
&\quad\lesssim_{\breve j, \check k,\delta} {\left\lVert D_y^{\breve j}
(D_x-3) (D_x-2) \tilde q(D_x) D_x v_{|t = 0} \right\rVert}_{\check
k,\delta+ 1 + \breve j}^2
\nonumber\\
&\quad\phantom{\lesssim_{\breve j, \check k,\delta}} + \int_I  {\left\lVert D_y^{\breve j} (D_x-4) (D_x-3)
\tilde q(D_x-1) (D_x-1) f \right\rVert}_{\check k-2,\delta+ \frac32 + \breve j}^2 \mathrm{d} t \nonumber
\\
&\quad\phantom{\lesssim_{\breve j, \check k,\delta}} + \int_I  {\left\lVert D_y^{\breve j + 1} (D_x-1)
D_x v \right\rVert}_{\delta+
\frac3 2 + \breve j}^2 \mathrm{d} t + \int_I  {\left
\lVert D_y^{\breve j + 2} D_x v \right\rVert}_{\delta+ \frac3 2 +
\breve j}^2 \mathrm{d} t
\label{strong_higher_int}
\end{align}
for $\breve j \ge0$.

In order to absorb the remnant terms
\begin{equation*}
\int_I \left( {\left\lVert D_y^{\breve j + 1} (D_x-1) D_x v \right\rVert}_{\delta+
\frac3 2 + \breve j}^2 +  {\left\lVert D_y^{\breve j + 2}
D_x v \right\rVert}_{\delta+
\frac3 2 + \breve j}^2\right) \mathrm{d} t,
\end{equation*}
we choose $\breve j := j - 3$ and assume $j \ge3$,
so that the remnant terms can be bounded by
\begin{equation*}
\int_I \left( {\left\lVert D_y^{j-1} D_x v \right\rVert
}_{\check k,\delta-\frac3 2 + j}^2
+  {\left\lVert D_y^{j-2} \tilde q(D_x) D_x v \right\rVert
}_{\check k + 1,\delta- \frac32 + j}^2\right) \mathrm{d} t,
\end{equation*}
appearing on the left-hand side of \reftext{\eqref{higher_est_2_alt}} with
$\tilde k$ replaced by $\check k - 1$ and $\check k$ replaced by
$\breve k$, where $\breve k \ge4$. Note that this requires
using \reftext{Lemma~\ref{lem:mainhardy}} of \S\ref{sec:elliptic}.
This results in the estimate
\begin{align}\nonumber
& {\skliaustask D_y^{j-1} v \skliaustasd}_{\check k - 2,-\delta -2+j, I}^2 
+  {\skliaustask D_y^{j-1} D_x v \skliaustasd}_{\check k - 2,\delta-2+j,I}^2 +  {\skliaustask D_y^{j-2}
\tilde q(D_x) D_x v \skliaustasd}_{\check k - 1,\delta- 2 +j,I}^2
\\
& + {\skliaustask D_y^{j-3} (D_x-3) (D_x-2) \tilde q(D_x) D_x v \skliaustasd}_{\check k,\delta-2 + j,I}^2 +  {\skliaustask D_y^j v \skliaustasd}_{\breve k
- 2,-\delta-1 + j, I}^2 \nonumber\\
&\quad\lesssim_{j,\check k, \breve k, \delta} {\left\lVert
D_y^{j-1} v_{|t = 0} \right\rVert}_{\check k - 2,-\delta-2+j}^2 \!+\!
 {\left\lVert D_y^{j-1} D_x v_{|t = 0} \right\rVert
}_{\check k - 2,\delta-2+j}^2 \!+\!  {\left\lVert D_y^{j-2}
\tilde q(D_x) D_x v_{|t = 0} \right\rVert}_{\check k - 1,\delta- 2 +
j}^2 \nonumber \\
&\quad\phantom{\lesssim_{j,\check k, \breve k, \delta}} +  {\left\lVert D_y^{j-3} (D_x-3) (D_x-2)
\tilde q(D_x) D_x v_{|t = 0} \right\rVert}_{\check k,\delta- 2 +
j}^2 +  {\left\lVert D_y^j v_{|t = 0} \right\rVert
}_{\breve k
- 2,-\delta-1+j}^2 \nonumber \\
&\quad\phantom{\lesssim_{j,\check k, \breve k, \delta}} + \int_I \left( {\left\lVert D_y^{j-1} f
\right\rVert}_{\check k - 4,-\delta-
\frac3 2 + j}^2 +  {\left\lVert D_y^{j-1} (D_x-1) f \right
\rVert}_{\check k -4,\delta-
\frac3 2 + j}^2\right) \mathrm{d} t \nonumber\\
&\quad\phantom{\lesssim_{j,\check k, \breve k, \delta}} + \int_I  {\left\lVert D_y^{j-2} \tilde
q(D_x-1) (D_x-1) f \right\rVert}_{\check k - 3,\delta- \frac3 2 +
j}^2 \mathrm{d} t \nonumber\\
&\quad\phantom{\lesssim_{j,\check k, \breve k, \delta}} + \int_I\left(  {\left\lVert D_y^{j - 3}
(D_x-4) (D_x-3) \tilde q(D_x-1) (D_x-1) f \right\rVert}_{\check
k-2,\delta- \frac3 2 + j}^2 +  {\left\lVert D_y^j f
\right\rVert}_{\breve k - 4,-\delta- \frac1 2 + j}^2 \right
)\mathrm{d} t \nonumber\\
\label{higher_est_3}
\end{align}
for $\delta\in\left(0,\frac{1}{10}\right)$, where $\check k \ge4$,
$\breve k \ge4$, and $j \ge3$.\goodbreak

\subsubsection{Maximal regularity for the linear degenerate-parabolic
equation}\label{sec:max_reg}

In this section, we combine
\begin{itemize}
\item estimate~\reftext{\eqref{estimate1_int}} with $j=0$,
\item estimate~\reftext{\eqref{higher_est_1}} with $j =1$,
\item estimate~\reftext{\eqref{higher_est_2}} with $j = 2$,
\item estimate~\reftext{\eqref{higher_est_2}} with $j=3$ and $\tilde k$ replaced by $\breve k$,
\item estimate~\reftext{\eqref{higher_est_2_alt}} with $j = 2$, and
\item estimate~\reftext{\eqref{higher_est_3}} with $j = 3$
\end{itemize}
to obtain maximal regularity for the
linear degenerate-parabolic problem \reftext{\eqref{lin_cauchy}}
in form of
\begin{equation}\label{mr}
{\skliaustask v\skliaustasd}_\mathrm{sol} \lesssim_{k,\tilde k,
\check k, \breve k, \delta}  {\left\lVert v^{(0)} \right\rVert
}_\mathrm{init}^\prime+ {\skliaustask f\skliaustasd}_\mathrm
{rhs}^\prime,
\end{equation}
where we have introduced the following norms:

\medskip

The norm ${\skliaustask\cdot\skliaustasd}_{\mathrm{sol}}$ for the
solution $v$ is defined through
\begin{align}\nonumber
{\skliaustask v\skliaustasd}_\mathrm{sol}^2
&:=
{\skliaustask v \skliaustasd}_{k, - \delta-1 ,I}^2
+ {\skliaustask D_y v \skliaustasd}_{\tilde k - 1, - \delta,I}^2
+ {\skliaustask D_x v \skliaustasd}_{\tilde k, - \delta,I}^2
+  {\skliaustask D_y D_x v \skliaustasd}_{\tilde k - 1, \delta, I}^2
+ {\skliaustask\tilde q(D_x) D_x v \skliaustasd}_{\tilde k, \delta, I}^2
\nonumber\\
&\phantom{:=} + {\skliaustask D_y^2 v \skliaustasd}_{\check k - 2, - \delta+1,I}^2
+ {\skliaustask D_y D_x v \skliaustasd}_{\check k -1, - \delta+ 1, I}^2
+ {\skliaustask\tilde q(D_x) D_x v\skliaustasd}_{\check k, - \delta+ 1, I}^2
+ {\skliaustask D_y^2 D_x v \skliaustasd}_{\check k - 2, \delta+ 1, I}^2
\nonumber\\
&\phantom{:=} + {\skliaustask D_y \tilde q(D_x)D_x v \skliaustasd}_{\check k - 1, \delta+ 1, I}^2 
+ {\skliaustask(D_x-3) (D_x-2) \tilde q(D_x) D_x v \skliaustasd}_{\check k, \delta+ 1, I}^2 
\nonumber\\
&\phantom{:=} +  {\skliaustask D_y^3 v \skliaustasd}_{\breve k - 2, - \delta+ 2,I}^2
+ {\skliaustask D_y^2D_x v \skliaustasd}_{\breve k - 1, - \delta+ 2,I}^2
+ {\skliaustask D_y \tilde q(D_x) D_x v \skliaustasd}_{\breve k, - \delta+ 2,I}^2.
\label{norm_sol}
\end{align}
Its equivalence to the norm ${\skliaustask\cdot\skliaustasd}_\mathrm
{Sol}$ defined in \reftext{\eqref{norm_sol_simple}}
for $I = [0,\infty)$ is given through \reftext{Lemma~\ref{lem:equivalence_sol}} (cf.~\S\ref{sec:equivnorms}).\goodbreak

The norm $ {\left\lVert\cdot\right\rVert}_\mathrm
{init}^\prime$ for the initial data
$v^{(0)}$ is given through
%
%
\begin{align}\nonumber
\left( {\left\lVert v^{(0)} \right\rVert}_\mathrm{init}^\prime\right)^2
&:=
{\left\lVert v^{(0)} \right\rVert}_{k, -1- \delta}^2
+ {\left\lVert D_y v^{(0)} \right\rVert}_{\tilde k-1, - \delta}^2
+ {\left\lVert D_x v^{(0)} \right\rVert}_{\tilde k, - \delta}^2
+ {\left\lVert D_y D_x v^{(0)}\right\rVert}_{\tilde k - 1, \delta}^2
+ {\left\lVert\tilde q(D_x) D_x v^{(0)}\right\rVert}_{\tilde k, \delta}^2
\nonumber\\
&\phantom{:=} + {\left\lVert D_y^2 v^{(0)} \right\rVert}_{\check k - 2, - \delta+ 1}^2
+ {\left\lVert D_y D_x v^{(0)} \right\rVert}_{\check k - 1, - \delta+ 1}^2
+ {\left\lVert \tilde q(D_x) D_x v^{(0)} \right\rVert}_{\check k, - \delta+ 1}^2 + {\left\lVert D_y^2 D_x v^{(0)} \right\rVert}_{\check k - 2, \delta+ 1}^2 \nonumber\\
&\phantom{:=}+ {\left\lVert D_y \tilde q(D_x) D_x v^{(0)} \right\rVert}_{\check k - 1, \delta+
1}^2 + {\left\lVert(D_x-3) (D_x-2) \tilde q(D_x) D_x v^{(0)} \right\rVert}_{\check k,
\delta+ 1}^2 \nonumber\\
&\phantom{:=}+  {\left\lVert D_y^3 v^{(0)} \right\rVert}_{\breve k - 2, - \delta+2}^2
+ {\left\lVert D_y^2D_x v^{(0)} \right\rVert}_{\breve k - 1, - \delta+2}^2
+ {\left\lVert D_y \tilde q(D_x) D_x v^{(0)} \right\rVert}_{\breve k, - \delta+2}^2
\label{norm_init_p}
\end{align}

The norm ${\skliaustask\cdot\skliaustasd}_\mathrm{rhs}^\prime$ for
the right-hand side $f$ is determined by
\begin{align}\nonumber
\left({\skliaustask f\skliaustasd}_\mathrm{rhs}^\prime\right)^2
&:= \int_I \left({\left\lVert f \right\rVert}_{k - 2, - \delta- \frac1 2}^2
+ {\left\lVert D_y f \right\rVert}_{\tilde k - 3, - \delta+ \frac1 2}^2
+ {\left\lVert(D_x-1) f \right\rVert}_{\tilde k - 2, - \delta+ \frac1 2}^2
+ {\left\lVert D_y (D_x-1) f \right\rVert}_{\tilde k - 3, \delta+ \frac1 2}^2 \right)
\mathrm{d} t \nonumber\\
&\phantom{:=}
+ \int_I\left( {\left\lVert\tilde q(D_x-1) (D_x-1) f\right\rVert}_{\tilde k - 2,\delta+ \frac1 2}^2
+ {\left\lVert D_y^2 f \right\rVert}_{\check k - 4, - \delta+\frac3 2}^2
+ {\left\lVert D_y (D_x-1) f \right\rVert}_{\check k - 3, - \delta+ \frac32}^2
\right)\mathrm{d} t \nonumber\\
&\phantom{:=}
+ \int_I \left( {\left\lVert\tilde q(D_x-1) (D_x-1) f \right\rVert}_{\check k - 2, -\delta+ \frac3 2}^2
+ {\left\lVert D_y^2 (D_x-1) f\right\rVert}_{\check k - 4, \delta+ \frac3 2}^2
\right) \mathrm{d} t \nonumber\\
&\phantom{:=}
+ \int_I \left(
{\left\lVert D_y \tilde q(D_x-1)(D_x-1) f \right\rVert}_{\check k - 3, \delta+ \frac3 2}^2
+ {\left\lVert(D_x-4) (D_x-3)\tilde q(D_x-1) (D_x-1) f \right\rVert}_{\check k - 2, \delta+ \frac
3 2}^2\right)
\mathrm{d} t \nonumber\\
&\phantom{:=} + \int_I \left(
{\left\lVert D_y^3 f \right\rVert}_{\breve k - 4, - \delta+ \frac5 2}^2
+ {\left\lVert D_y^2(D_x-1) f \right\rVert}_{\breve k - 3, - \delta+ \frac5 2}^2
+ {\left\lVert D_y \tilde q(D_x-1) (D_x-1) f \right\rVert}_{\breve k - 2, - \delta+ \frac5 2}^2
\right) \mathrm{d}t. \nonumber\\
\label{norm_rhs_p}
\end{align}
Here, $k\ge2, \tilde k \ge3$, $\check k \ge4$, $\breve k \ge4$, and $\delta
\in\left(0,\frac{1}{10}\right)$. Further conditions on the number of
derivatives $k$, $\tilde k$, $\check k$, and $\breve k$ will be imposed for the approximation results of \S \ref{sec:approx} and when
treating the nonlinearity ${\mathcal N}(v)$ (cf.~\reftext{\eqref{def_nonlinearity}}) in \S \ref{sec:nonlinear_est}. A simplification of the norms $
{\left\lVert\cdot\right\rVert}_\mathrm{init}$ and ${\skliaustask
\cdot\skliaustasd}_\mathrm{rhs}^\prime$ is provided
in \reftext{Lemmata~\ref{lem:equivalence_init} and~\ref{lem:equivalence_rhs}} of
\S \ref{sec:elliptic}. A rigorous justification of the
maximal-regularity estimate \reftext{\eqref{mr}} will be the subject of \S \ref{sec:lin_rig}.

\subsection{Norms and their properties}\label{sec:norms_prop}

\subsubsection{Simplifications of our norms}\label{sec:equivnorms}

The following weighted trace estimate gives control on the time trace,
which follows directly by using the fundamental theorem of calculus and
Young's inequality (see for instance \cite[Lemma~B.2]{ggko.2014}). The precise version included below is convenient
because it can be used to estimate the supremum terms from the
definition \reftext{\eqref{norm_sol}} of ${\skliaustask v\skliaustasd}_\mathrm
{sol}$ by the
integral terms appearing in the same norm.
%
%
\begin{lemma}[Time trace]\label{lem:trace_estimate}
Assume that $T \in (0,\infty]$, $I=[0,T) \subseteq[0,\infty)$, and let $w \colon I \times(0,\infty) \times\mathbb{R} \to \R$ be locally integrable. Then we have
%
%
\begin{equation}\label{traceest0}
\sup_{t\in I}  {\left\lvert w \right\rvert}_{\ell
,\gamma}^2\lesssim_{\ell,\gamma}  {\left\lvert w|_{t=0}
\right\rvert}_{\ell,\gamma}^2+\int_I {\left\lvert
\partial_tw \right\rvert}_{\ell-2,\gamma
-\frac12}^2\mathrm{d} t + \int_I {\left\lvert w \right
\rvert}_{\ell+2,\gamma+\frac12}^2 \mathrm{d} t,
\end{equation}
and for the case $I=[0,\infty)$ we have
%
%
\begin{equation}\label{traceest}
\sup_{t\in [0,\infty)}  {\left\lvert w \right\rvert}_{\ell
,\gamma}^2\lesssim_{\ell,\gamma} \int_0^\infty {\left\lvert\partial_tw \right\rvert
}_{\ell-2,\gamma-\frac12}^2\mathrm{d} t + \int
_0^\infty {\left\lvert w \right\rvert}_{\ell+2,\gamma+\frac
12}^2 \mathrm{d} t
\end{equation}
for $\gamma\in\mathbb{R}$ and $\ell\ge2$.
\end{lemma}

Next, we give the following simplification of the norms $
{\left\lVert\cdot\right\rVert}_\mathrm{init}^\prime$,
${\skliaustask\cdot\skliaustasd}_\mathrm{rhs}^\prime$, and
${\skliaustask\cdot\skliaustasd}_\mathrm{sol}$, whose proofs can be
found in \S \ref{app:equivnorm}.
%
%
\begin{lemma}\label{lem:equivalence_init}
Suppose $\delta\in\left(0,\frac1{10}\right)$, $k \ge \max\{\tilde k+4, \check k + 6, \breve k+5\}$, $\tilde k \ge1$,
$\check k \ge2$, $\breve k\ge2$, and define the norm $
{\left\lVert v^{(0)} \right\rVert}_\mathrm{init}$ for a locally
integrable function $v^{(0)}: \,
(0,\infty) \times\mathbb{R}\to\mathbb{R}$ through \reftext{\eqref{norm_init}}, i.e.,
\begin{eqnarray*}
 {\left\lVert v^{(0)} \right\rVert}_\mathrm{init}^2 &=&
 {\left\lVert v^{(0)} \right\rVert}_{k, -\delta-1}^2 + {\left\lVert D_x v^{(0)} \right\rVert}_{\tilde k,
-\delta}^2 + {\left\lVert\tilde q(D_x) D_x v^{(0)}
\right\rVert}_{\tilde k, \delta}^2 +
 {\left\lVert\tilde q(D_x)D_xv^{(0)} \right\rVert
}_{\check k, -\delta+1}^2\\
&&+ {\left\lVert(D_x-3)(D_x-2)\tilde q(D_x)D_xv^{(0)}
\right\rVert}_{\check k,\delta
+1}^2+ {\left\lVert D_y \tilde q(D_x) D_xv^{(0)} \right\rVert}_{\breve
k,-\delta+2}^2.
\end{eqnarray*}
Then the norms $ {\left\lVert\cdot\right\rVert
}_\mathrm{init}^\prime$ (cf.~\reftext{\eqref{norm_init_p}}) and $ {\left
\lVert\cdot\right\rVert
}_\mathrm{init}$ are equivalent on the
space of all locally integrable $v^{(0)}: \, (0,\infty) \times\mathbb
{R}\to\mathbb{R}
$ with $ {\left\lVert v^{(0)} \right\rVert}_\mathrm
{init} < \infty$.
\end{lemma}
%

%
%
\begin{lemma}\label{lem:equivalence_rhs}
Suppose $\delta\in\left(0,\frac1{10}\right)$,
$k \ge \max\{\tilde k+4, \check k +6, \breve k+5\}$,
$\tilde k \ge3$, $\check k \ge4$, $\breve k \ge4$,
and define the norm ${\skliaustask
f\skliaustasd}_\mathrm{rhs}$ for a locally integrable function $f: \,
I \times
(0,\infty) \times\mathbb{R}\to\mathbb{R}$, where $I = [0,T)
\subseteq[0,\infty)$ with $T \in (0,\infty]$,
through \reftext{\eqref{norm_rhs}}, i.e.,
\begin{eqnarray*}
{\skliaustask f\skliaustasd}_\mathrm{rhs}^2 &=& \int_I
\left({\left\lVert f \right\rVert}_{k-2,-\delta-\frac12}^2 
+ {\left\lVert(D_x-1)f \right\rVert}_{\tilde
k-2,-\delta+\frac12}^2 \right)\mathrm{d} t  \\
&&+\int_I\left( {\left\lVert\tilde q(D_x-1)(D_x-1)f
\right\rVert}_{\tilde k-2,\delta
+\frac12}^2 +  {\left\lVert\tilde q(D_x-1)(D_x-1)f \right
\rVert}_{\check k - 2,-\delta
+\frac32}^2\right)\mathrm{d} t\\
&&+ \int_I\left( {\left\lVert(D_x-4)(D_x-3)\tilde
q(D_x-1)(D_x-1)f \right\rVert}_{\check
k-2,\delta+\frac32}^2+ {\left\lVert D_y \tilde q(D_x-1) (D_x-1)f \right\rVert
}_{\breve k-2, -\delta+\frac
52}^2\right)\mathrm{d} t.
\end{eqnarray*}
Then the norms ${\skliaustask\cdot\skliaustasd}_\mathrm{rhs}^\prime
$ (cf.~\reftext{\eqref{norm_rhs_p}}) and ${\skliaustask\cdot\skliaustasd}_\mathrm{rhs}$
are equivalent on the
space of all locally integrable $f: \, (0,\infty)^2 \times\mathbb
{R}\to\mathbb{R}$
with ${\skliaustask f\skliaustasd}_\mathrm{rhs} < \infty$.
\end{lemma}

Thanks to \reftext{Lemma~\ref{lem:trace_estimate}} we can also prove the
following analogue of \reftext{Lemmata~\ref{lem:equivalence_init} and \ref{lem:equivalence_rhs}}.
%
%
\begin{lemma}\label{lem:equivalence_sol}
Suppose that $\delta\in\left(0,\frac1{10}\right)$,
$k \ge \max\{\tilde k+4, \check k +6, \breve k+5\}$,
$\tilde k \ge3$, $\check k \ge4$, $\breve k\ge4$,
and consider the norm ${\skliaustask v\skliaustasd}_\mathrm{Sol}$
for a locally integrable function
$v \colon [0,\infty) \times (0,\infty) \times\mathbb{R}\to\mathbb{R}$, given by
\begin{eqnarray*}
{\skliaustask v\skliaustasd}_\mathrm{Sol}^2 &\stackrel{\text{\reftext{\eqref{norm_sol_simple}}}}{=}&  \int_0^\infty \left
({\left\lVert\partial_t v \right\rVert}_{k- 2, - \delta- \frac3 2}^2 + {\left\lVert\partial_t D_x v \right\rVert}_{\tilde k
- 2, - \delta- \frac1 2}^2 \right) \mathrm{d}
t\nonumber\\
&& + \int_0^\infty \left( {\left\lVert\partial_t
\tilde q(D_x) D_x v \right\rVert}_{\tilde k - 2, \delta- \frac12}^2 + {\left\lVert\partial_t \tilde q(D_x) D_x v \right
\rVert}_{\check k - 2, - \delta+ \frac1 2}^2 \right) \mathrm{d} t
\nonumber\\
&& + \int_0^\infty \left( {\left\lVert\partial_t (D_x-3)
(D_x-2) \tilde q(D_x) D_x v \right\rVert}_{\check k - 2, \delta+
\frac1 2}^2 + {\left\lVert\partial_t D_y \tilde q(D_x) D_x v \right
\rVert}_{\breve k - 2, - \delta+ \frac3 2}^2 \right) \mathrm{d} t \nonumber\\
&& + \int_0^\infty \left( {\left\lVert v \right\rVert
}_{k+2, - \delta- \frac1 2}^2
+ {\left\lVert D_x v \right\rVert
}_{\tilde k + 2, - \delta+ \frac1 2}^2
+ {\left\lVert\tilde q(D_x) D_x v \right\rVert}_{\tilde
k + 2, \delta+ \frac1 2}^2 \right) \mathrm{d} t \\
&& + \int_0^\infty \left( {\left\lVert\tilde q(D_x) D_x v \right\rVert}_{\check k
+ 2, - \delta+ \frac32}^2 + {\left\lVert(D_x-3) (D_x-2) \tilde
q(D_x) D_x v \right\rVert}_{\check
k + 2, \delta+ \frac3 2}^2 \right) \mathrm{d} t \\
&& + \int_0^\infty {\left\lVert D_y \tilde q(D_x) D_x v \right
\rVert}_{\breve k+2, - \delta+
\frac5 2}^2 \mathrm{d} t.
\end{eqnarray*}
Then the norms ${\skliaustask\cdot\skliaustasd}_\mathrm{sol}$ for the case $I=[0,\infty)$
(cf.~\reftext{\eqref{norm_sol}})
and ${\skliaustask\cdot\skliaustasd}_\mathrm{Sol}$ are equivalent
on the space of all
locally integrable $v: [0,\infty) \times(0,\infty) \times\mathbb{R}\to
\mathbb{R}$ with
${\skliaustask v\skliaustasd}_\mathrm{Sol} < \infty$.
\end{lemma}
%
%

\subsubsection{Elliptic regularity and embeddings}\label{sec:elliptic}

Note that the norms ${\skliaustask\cdot\skliaustasd}_\mathrm{sol}$,
$ {\left\lVert\cdot\right\rVert}_\mathrm{init}$, and
${\skliaustask\cdot\skliaustasd}_\mathrm{rhs}$ (cf.~\reftext{\eqref{norm_sol}},
\reftext{\eqref{norm_init}}, and \reftext{\eqref{norm_rhs}}) contain terms in
$(D_x-3-\kappa) (D_x-2-\kappa) \tilde q(D_x-\kappa) (D_x-\kappa) w$, $\tilde q(D_x-\kappa) (D_x-\kappa) w$, or $(D_x-\kappa) w$, where $(w,\kappa) \in \left\{(v,0),\left(v^{(0)},0\right),(f,1)\right\}$, that are not
explicit enough in terms of $w$ in the sense that they partially do not change on
changing boundary values such as $w_0$, $w_1$, $w_{1+\beta}$, or $w_2$. On the
other hand, we are precisely interested in control on these terms,
which requires to study the elliptic regularity of the operators
$D_x-\gamma$ with $\gamma \in \mathbb R$. Therefore, we consider the following direct
consequence of Hardy's inequality (cf.~\cite[Lemma~A.1]{gko.2008},
\cite[Lemma~7.4]{ggko.2014},
\cite[Proposition~3.1 and Lemma~3.1]{g.2016}).
\begin{lemma}\label{lem:mainhardy}
Suppose $\underline f, \overline f \in H^1_{\mathrm{loc}}\left((0,\infty)\right)$ and $\rho, \underline\gamma, \overline\gamma
\in\mathbb{R}
$ with $\underline\gamma< \rho< \overline\gamma$ such that
\begin{equation*}
\lim_{n\to\infty} \underline x_n^{-\underline \gamma} \underline f\left(x_{1,n}\right)
= \lim_{n\to\infty} \overline x_n^{-\overline \gamma} \overline f\left(x_{2,n}\right)
= 0 \quad\mbox{for sequences} \quad \underline x_n \searrow 0
\quad \mbox{and} \quad \overline x_n \to \infty.
\end{equation*}
Then
\begin{equation}\label{est_hardy}
 {\left\lvert\underline f \right\rvert}_{1,\rho
} \lesssim_{\underline\gamma
,\rho} {\left\lvert(D_x-\underline\gamma) \underline f \right
\rvert}_\rho  \qquad\mbox{and} \qquad {\left\lvert\overline f \right\rvert}_{1,\rho} 
\lesssim_{\overline\gamma
,\rho}
 {\left\lvert(D_x-\overline
\gamma)\overline f \right\rvert}_\rho.
\end{equation}
\end{lemma}
A proof is contained in \S \ref{app:embedding}. Note that for each of the inequalities in \reftext{\eqref{est_hardy}}, if the left-hand side is finite, then the asymptotic behavior in the respective case is automatically fulfilled and the respective estimate holds true.

\medskip

The following lemma yields $BC^0$-bounds on the solution $v$ to \reftext{\eqref{nonlin_cauchy}}, and its gradient. Its proof is contained in \S \ref{app:embedding}.
\begin{lemma}\label{lem:bc0_bounds}
Suppose that $k\ge \max\left\{\tilde k+4,\check k + 6, \breve k +5\right\}$, $\tilde k\ge3$, $\check k\ge4$, $\breve k \ge4$, and $\delta\in
\left(0,\frac1{10}\right)$. Then for functions $v^{(0)}:
(0,\infty) \times\mathbb{R}\to\mathbb{R}$, $v: (0,\infty)^2
\times\mathbb{R}\to\mathbb{R}$, and
$f: (0,\infty)^2 \times\mathbb{R}\to\mathbb{R}$ that are locally
integrable with
$ {\left\lVert v^{(0)} \right\rVert}_\mathrm{init} <
\infty$, ${\skliaustask v\skliaustasd}_\mathrm{sol} <
\infty$, and ${\skliaustask f\skliaustasd}_\mathrm{rhs} < \infty$,
the expansions
\begin{equation}\label{expansion_v0}
D^\ell v^{(0)}(x,y) = D^\ell\left(v_0^{(0)}(y) + v_1^{(1)}(y) x\right) + o\left(x^{\frac 32+\delta}\right) \quad \mbox{as} \quad x \searrow 0
\end{equation}
(with $\ell = (\ell_x,\ell_y) \in \mathbb N_0^2$
such that $\lvert\ell\rvert \le \check k + 7$ and $\ell_y \le \check k$),
\reftext{\eqref{expansion_v}} (with $L := \check k + 5$ and $L_y := \check k - 2$),
and \reftext{\eqref{expansion_f}} (with $L := \check k + 5$ and $L_y := \check k - 2$)
are valid almost everywhere, where the coefficients $v_j^{(0)}$ ($j \in \{0,1\}$), $v_j$ ($j \in \{0,1,1+\beta,2\}$), and $f_j$ ($j \in \{1,2\}$) satisfy the
following estimates, with constants only depending on $\tilde k$,
$\check k$, and $\delta$: 
\begin{subequations}\label{bc0_bounds}
\begin{equation}\label{bc0_grad_init}
{\left\lVert D^\ell v^{(0)}_{x} \right\rVert}_{BC^0\left((0,\infty) \times\mathbb{R}\right)}
+ {\left\lVert D^\ell v^{(0)}_{y} \right\rVert}_{BC^0\left
((0,\infty) \times\mathbb{R}\right)} 
+ {\left\lvert v_0^{(0)} \right\rvert}_{BC^1\left(\mathbb{R}\right
)} 
+  {\left\lvert v^{(0)}_1 \right\rvert}_{BC^0\left(\mathbb{R}\right)} \lesssim {\left\lVert v^{(0)} \right\rVert}_\mathrm{init}
\end{equation}
and
\begin{equation}\label{bc0_grad_sol}
{\skliaustask D^\ell v_x \skliaustasd}_{BC^0\left(I\times(0,\infty)
\times\mathbb{R}\right)} + {\skliaustask D^\ell v_y \skliaustasd}_{BC^0\left(I\times(0,\infty) \times\mathbb{R}\right)} +  {\left\lVert v_0 \right\rVert}_{BC^0\left
(I;BC^1\left(\mathbb{R}\right)\right)} +  {\left\lVert v_1 \right\rVert}_{BC^0\left(I\times\mathbb{R}
\right)} \lesssim {\skliaustask v\skliaustasd}_\mathrm{sol}
\end{equation}
for $0 \le {\left\lvert\ell\right\rvert} \le\min
\left\{\tilde k - 2, \check k - 2\right
\}$,
\begin{align} \label{bc0_v1beta_v2}
& {\left\lVert v_1 \right\rVert}_{L^2\left(I; H^1\left(\mathbb{R}\right)\right)}
+ {\left\lVert \left(v_0\right)_y \right\rVert}_{L^2\left(I; H^1\left(\mathbb{R}\right)\right)}
+{\left\lVert v_1\right\rVert}_{BC^0\left(I; H^1\left(\mathbb R\right)\right)}
+ {\left\lVert \left(v_0\right)_y\right\rVert}_{BC^0\left(I; H^1\left(\mathbb R\right)\right)} \nonumber \\
& \quad + {\left\lVert v_{1+\beta} \right\rVert}_{L^2\left(I;BC^0(\mathbb{R})\right)}
+ {\left\lVert v_{1+\beta}\right\rVert}_{L^2\left(I ;H^1(\mathbb{R})\right)}
+ {\left\lVert v_2 \right \rVert}_{L^2(I\times\mathbb{R})}
\lesssim {\skliaustask v\skliaustasd}_\mathrm{sol},
\end{align}
and
\begin{equation}\label{bcint_f1_f2}
{\left\lVert f_1 \right
\rVert}^2_{L^2(I\times\mathbb{R})} +  {\left\lVert f_2 \right\rVert}^2_{L^2(I\times\mathbb{R})}
\lesssim
{\skliaustask f\skliaustasd}_\mathrm{rhs}.
\end{equation}
\end{subequations}
\end{lemma}
%

\subsubsection{Approximation results}\label{sec:approx}

Our aim is to approximate the solution $v$, the initial data $v^{(0)}$,
and the right-hand side $f$ in the norms ${\skliaustask\cdot
\skliaustasd}_\mathrm
{sol}$, ${\skliaustask\cdot\skliaustasd}_\mathrm{init}$, and
${\skliaustask\cdot\skliaustasd}_\mathrm
{rhs}$, where the approximating sequences have well-controlled behavior
as $x \searrow0$ and $x \to\infty$. The following two definitions
specify the necessary asymptotic behavior precisely:
%
%
\begin{definition}\label{def:g0eta}
A function $f : \, (0,\infty) \times \mathbb{R}_\eta\to\mathbb{C}$ such
that $f(\cdot,\eta
) \in C^\infty((0,\infty))$ for every \mbox{$\eta\in\mathbb{R}$},
satisfies $(G_0)$
if for every compact interval
${\mathcal J}\Subset\mathbb{R}\setminus\{0\}$ or ${\mathcal J} = \{0\}$, there
exists $\varepsilon> 0$
and a function $\bar f = \bar f\left(x_1,x_2, \eta\right) \colon [0,\varepsilon]
\times \left[0,\varepsilon^\beta\right]\times{\mathcal J}\to \mathbb{C}$ such that
\begin{enumerate}[(i)]
\item[(i)]$f(x, \eta) = \bar f\left(x, x^\beta, \eta\right)$ for
$(x,\eta) \in\left(0,\varepsilon^\beta\right]\times{\mathcal J}$ and
for each $\eta\in{\mathcal J}$
the function $\bar f(\cdot,\cdot,\eta)$ is analytic with $\partial_{x_2}^j \bar f (0,0,\eta) = 0$ for $j \ge 1$;
\item[(ii)] the map ${\mathcal J}\owns\eta\mapsto
\bar f \in
\Omega$, where
\begin{equation*}
\Omega:= \left\{\bar g: \,  {\left\lvert\bar g \right
\rvert}_\omega:= \sum_{(k,\ell) \in
\mathbb{N}_0^2} \frac{\varepsilon^{k + \beta\ell}}{k! \ell!}
 {\left\lvert\partial_{x_1}^k \partial_{x_2}^\ell\bar
g(0,0) \right\rvert} < \infty\right\},
\end{equation*}
is well-defined and continuous.
\end{enumerate}
\end{definition}

Note that the power $\beta$ is in line with our expectation
stated in \eqref{exp_2d} an also appears
in \cite{ggko.2014} (compare to \S \ref{sec:lin_evol} and in
particular the discussion leading to \reftext{\eqref{expansion_u}}). As
already outlined in \S \ref{sec:lin_evol} (see the discussion after \reftext{\eqref{lin_cauchy_f}}) the terms $- \eta^2 x^2 r\left(D_x\right) v$ and
$\eta^4 x^4 v$ in \reftext{\eqref{eq_lin}} are, for fixed $\eta\in\mathbb
{R}$, merely higher-order corrections.

We introduce an analogous definition for the asymptotic behavior as $x \to \infty$.
\begin{definition}\label{def:ginftyeta}
A function $f : \,
(0,\infty)
\times\mathbb{R}\to\mathbb{C}$ such that $f(\cdot,\eta) \in
C^\infty((0,\infty))$
for every $\eta\in\mathbb{R}$, meets $(G_\infty)$ if
\begin{enumerate}[(i)]
\item[(i)] for $\eta= 0$, for all $j\in\mathbb
N_0$, and all $\nu \in \left[0,2^{\frac32}\right)$ we have
\begin{equation*}
\limsup_{x\to\infty}\left|e^{\nu \sqrt[4]{x}} \tfrac{\mathrm{d}^j
f}{\mathrm{d} x^j}(x)\right
|<\infty;
\end{equation*}
\item[(ii)] for $\eta\ne0$, for all $\nu\in
[0,|\eta|)$, and all $j\in\mathbb N_0$ we have
\begin{equation*}
\limsup_{x\to\infty}\left|e^{\nu x} \tfrac{\mathrm{d}^j
f}{\mathrm{d} x^j}(x)\right
|<\infty;
\end{equation*}
\item[(iii)] for all $x_0 > 0$, $k \ge0$, and $\nu
\in
\left[0, {\left\lvert\eta\right\rvert}\right)$ the
mapping $\left(\mathbb{R}\setminus\{0\}\right
) \owns\eta\mapsto f(\cdot,\eta) \in{\mathcal V}$ is well-defined and
continuous, where
\begin{equation*}
{\mathcal V}:= \left\{g: \,  {\left\lvert g \right\rvert
}_{k,x_0,\nu} := \max_{j = 0,\ldots,k} \sup
_{x \ge x_0} e^{\nu x}  {\left\lvert\tfrac{\mathrm{d}^j
g}{\mathrm{d} x^j}(x) \right\rvert}\right\}.
\end{equation*}
\end{enumerate}
\end{definition}

Note that item~(i) coincides with the choice
in \cite[Definition~6.2, $(G_\infty)$]{ggko.2014}.

For the proof of the following approximation results, see \S \ref{app:approx}.
\begin{lemma}\label{lem:approx_init}
Suppose that $k\ge \max\left\{\tilde k+4,\check k + 6, \breve k +5\right\}$, $\tilde k\ge3$, $\check k\ge4$, $\breve k \ge4$, and $\delta\in
\left(0,\frac1{10}\right)$. Then for each locally integrable $v^{(0)}:
(0,\infty
)\times\mathbb R_y\to\mathbb R$ with $ {\left\lVert v^{(0)}
\right\rVert}_\mathrm
{init}<\infty$ there exists a sequence $\left(v^{(0,n)}\right)_n$ in $C^\infty
((0,\infty)\times\mathbb R_y)$ such that $ {\left\lVert
v^{(0,n)}-v^{(0)} \right\rVert}_\mathrm{init}\to0$ as $n\to\infty
$ and
\begin{enumerate}[\textit{(iii)}]
\item[\textit{(i)}]$v^{(0,n)}$ is smooth in $(0,\infty) \times \mathbb R_\eta$ and $(0,\infty) \times \mathbb R_y$,
\item[\textit{(ii)}]$v^{(0,n)}(x,\eta)= v_0^{(0,n)}(\eta) + v_1^{(0,n)}(\eta)x $ for
$x\ll_n 1$, where $ v_0^{(0,n)}, v_1^{(0,n)}$ are smooth, so that in
particular $v^{(0,n)}$ satisfies $(G_0)$ from \reftext{Definition~\ref{def:g0eta}}.
\item[\textit{(iii)}]$v^{(0,n)}(x,y) =0$ for $x \gg_n 1$ or $\lvert y\rvert \gg_n 1$, so that in particular
$v^{(0,n)}$ satisfies $(G_{\infty})$ from \reftext{Definition~\ref{def:ginftyeta}}.
\end{enumerate}
\end{lemma}
\begin{lemma}\label{lem:approx_rhs}
Suppose that $k\ge \max\left\{\tilde k+4,\check k + 6, \breve k +5\right\}$, $\tilde k\ge3$, $\check k \ge4$, $\breve k \ge4$, $\delta\in
\left
(0,\frac1{10}\right)$,
$T \in (0,\infty]$, and $I := [0,T)$.
Then for each locally integrable $f: I \times(0,\infty) \times\mathbb R_y
\to\mathbb R$ with ${\skliaustask f\skliaustasd}_\mathrm{rhs}<\infty
$ there exists a
sequence $\left(f^{(n)}\right)_n$ in $C^\infty(I \times(0,\infty) \times
\mathbb
R_y)$ such that ${\skliaustask f^{(n)}-f\skliaustasd}_\mathrm{rhs}\to
0$ as $n\to\infty$ and
\begin{enumerate}[\textit{(iii)}]
\item[\textit{(i)}]$f^{(n)}$ is smooth in $I \times (0,\infty) \times \mathbb R_\eta$ and $I \times (0,\infty) \times \mathbb R_y$,
\item[\textit{(ii)}] for every $t \in I$ we have $f^{(n)}(t,x,\eta)= f_1^{(n)}(t,\eta)x + f_2^{(n)}(t,\eta) x^2$
for $x\ll_n 1$, where $f_1^{(n)}$ and $f_2^{(n)}$
are smooth on $I \times\mathbb{R}$, so that in
particular $f^{(n)}(t,\cdot,\cdot)$ satisfies $(G_0)$ from
\reftext{Definition~\ref{def:g0eta}}.
\item[\textit{(iii)}]$f^{(n)}(t,x,\eta) =0$ for $x \gg_n 1$ or $\lvert y\rvert \gg_n 1$, so that in particular
$f^{(n)}(t,\cdot,\cdot)$ satisfies $(G_{\infty})$ from
\reftext{Definition~\ref{def:ginftyeta}} for all $t \in I$.
\end{enumerate}
\end{lemma}
\begin{lemma}\label{lem:approx_sol}
Suppose $k\ge \max\left\{\tilde k+4,\check k + 6, \breve k +5\right\}$, $\tilde k \ge3$, $\check k \ge4$, $\breve k\ge4$, $\delta\in
\left
(0,\frac1{10}\right)$, $T \in (0,\infty]$, and $I := [0,T)$.
Then for each locally integrable $v: I \times(0,\infty) \times
\mathbb
R_y \to\mathbb R$ with ${\skliaustask v\skliaustasd}_\mathrm
{sol}<\infty$ there exists a
sequence $\left(v^{(n)}\right)_n$ in $C^\infty(I \times(0,\infty) \times
\mathbb
R_y)$ such that ${\skliaustask v^{(n)}-v\skliaustasd}_\mathrm{sol}\to
0$ as $n\to\infty$ and
\begin{enumerate}[\textit{(iii)}]
\item[\textit{(i)}] if $I=(0, \infty)$ then $v^{(n)}(t,x,\eta)=0$ whenever $t>2n$,
\item[\textit{(ii)}]$v^{(n)}$ is smooth in $I \times (0,\infty) \times \mathbb R_\eta$ and $I \times (0,\infty) \times \mathbb R_y$,
\item[\textit{(iii)}] for every $t \in I$ we have
\begin{equation*}
v^{(n)}(t,x,\eta) = v_0^{(n)}(t,\eta)+v_1^{(n)}(t,\eta)x +
v_{1+\beta
}^{(n)}(t,\eta)x^{1+\beta} + v_2^{(n)}(t,\eta) x^2 \quad\mbox{for}
\quad x\ll_n 1,
\end{equation*}
where $v_0^{(n)}, v_1^{(n)}, v_{1+\beta}^{(n)}, v_2^{(n)}: I \times
\mathbb{R}
\to\mathbb{R}$ are smooth, so that in particular $v^{(n)}(t,\cdot
,\cdot)$
satisfies $(G_0)$ from \reftext{Definition~\ref{def:g0eta}} for all $t\in I$,
\item[\textit{(iv)}]$v^{(n)}(t,x,\eta) =0$ for $x \gg_n 1$ or $\lvert y\rvert \gg_n 1$, so that in particular
$v^{(n)}(t,\cdot,\cdot)$ satisfies $(G_{\infty})$ from
\reftext{Definition~\ref{def:ginftyeta}} for all $t \in I$.
\end{enumerate}
\end{lemma}

We obtain the following corollary of \reftext{Lemma~\ref{lem:approx_sol}}:

\begin{corollary}\label{coroll:contnorm}
Suppose that $k\ge \max\left\{\tilde k+4,\check k + 6, \breve k +5\right\}$, $\tilde k\ge3$,
$\check k \ge4$, $\breve k\ge4$, $\delta\in \left(0,\frac1{10}\right)$,
$T \in (0,\infty]$, and $I := [0,T)$.
Then for each locally integrable $v: I\times(0,\infty) \times\mathbb
R\to\mathbb R$ with ${\skliaustask v\skliaustasd}_\mathrm
{sol}<\infty$, using the
notation of \reftext{Lemma~\ref{lem:approx_sol}}, the following hold:
\begin{enumerate}[\textit{(iii)}]
\item[\textit{(i)}] If $I=(0,\infty)$ then
$\lim_{t\to\infty}(v_0)_y(t,y) = 0$
for all $y\in\mathbb R$ as well as $\lim_{t\to\infty}
\lVert v(t,\allowbreak\cdot,\cdot) \rVert_\mathrm{init}=0$.
\item[\textit{(ii)}] The function $(v_0)_y(t,y) =\lim_{x\searrow0} v_y(t,x,y)$ is
continuous and bounded.
\item[\textit{(iii)}] For $\tau\in(0,1)$ we denote
by ${\skliaustask\cdot\skliaustasd}_{\mathrm{sol},\tau}$ the norm
${\skliaustask v\skliaustasd}_\mathrm
{sol}$ as defined in \reftext{\eqref{norm_sol}} with the choice $I:=[0,\tau)$. Then
the functions $(0,\infty)\ni\tau\mapsto{\skliaustask v\skliaustasd
}_{\mathrm{sol},\tau}$ and
$(0,\infty)\ni t\mapsto {\left\lVert v(t,\cdot,\cdot)
\right\rVert}_\mathrm{init}$ are
continuous and bounded and ${\skliaustask v\skliaustasd}_{\mathrm
{sol},\tau}\to {\left\lVert v_{|t=0} \right\rVert
}_\mathrm{init}$ as $\tau\searrow0$.
\end{enumerate}
\end{corollary}
%

\subsection{Rigorous treatment of the linear equation}\label{sec:lin_rig}

The construction of solutions to the linear problem \reftext{\eqref{lin_cauchy}}
with appropriate estimates is based on a time-discretization argument
(cf.~\S \ref{sec:time_discr}) which in turn is based on a solid
understanding of the corresponding resolvent equation (cf.~\S \ref{sec:resolvent}). Note that the reasoning of \S \ref{sec:time_discr}
partially follows \cite[\S7]{ggko.2014} while the treatment of the
resolvent problem has similarities with \cite[\S6]{ggko.2014}. We will thus
emphasize differences and introduce some simplifications to the
approach in \cite[\S6--7]{ggko.2014}.

\subsubsection{The resolvent equation}\label{sec:resolvent}

Suppose we are given $v^{(t)} = v^{(t)}(x)$ for a given time $t > 0$.
Then an approximate solution of \reftext{\eqref{lin_pde}} at time $t + \delta t$,
where $\delta t > 0$ is small, can be found by solving
\begin{equation}\label{discrete_main}
x \frac{v^{(t+\delta t)} - v^{(t)}}{\delta t} + q\left(D_x\right)
v^{(t+\delta t)} - \eta^2 x^2 r\left(D_x\right) v^{(t+\delta t)} +
\eta
^4 x^4 v^{(t+\delta t)} = \frac{1}{\delta t} \int_t^{t+\delta t}
f\left
(t^\prime\right) \mathrm{d} t^\prime
\end{equation}
for $x > 0$. Replacing $\frac{1}{\delta t} x v^{(t)} + \frac
{1}{\delta
t} \int_t^{t+\delta t} f\left(t^\prime\right) \mathrm{d} t^\prime
$ by $f =
f(x)$, setting $\lambda:= \frac{1}{\delta t}$, and writing $v$ instead
of $v^{(t+\delta t)}$, we end up with the resolvent equation
\begin{equation*}
\lambda x v + q(D_x) v - \eta^2 x^2 r\left(D_x\right) v + \eta^4
x^4 v
= f \quad\mbox{for} \quad x > 0.
\end{equation*}
By scaling $x$ and $y$, we can assume without loss of generality
$\lambda= 1$, so that the resolvent equation simplifies to
\begin{equation}\label{resolvent}
x v + q(D_x) v - \eta^2 x^2 r\left(D_x\right) v + \eta^4 x^4 v = f
\quad\mbox{for} \quad x > 0.
\end{equation}
Equation~\reftext{\eqref{resolvent}} is for every fixed $\eta\in\mathbb{R}$
an ODE in
$x$, i.e., we consider $\eta$ merely as a parameter in the problem. In
what follows by a slight abuse of notation, we will be indicating
$v(\cdot)=v(\cdot,\eta)$, etc., whenever the fixed value of $\eta$ is clear
from the context. Furthermore, since $\eta$ is merely a parameter, we
will write total derivatives $\frac{\mathrm{d}}{\mathrm{d} x}$
instead of partial ones.
The notation with partial derivatives is used when it comes to
controlling the singular expansion at $x = 0$.

\paragraph*{Finding well-controlled solutions}
It is obvious that standard ODE theory yields a four-parameter family
of global solutions to \reftext{\eqref{resolvent}} given sufficient regularity of
$f$. However, our aim is not to construct just some solution to \reftext{\eqref{resolvent}} but solutions that are well-controlled as $x \searrow0$
and $x \to\infty$. This is detailed in conditions $(G_0)$ and
$(G_\infty)$ of \reftext{Definitions~\ref{def:g0eta} and \ref{def:ginftyeta}}.

The main result of \S \ref{sec:resolvent} reads as follows.
\begin{proposition}\label{prop:resolvent}
Suppose that $f : \, (0,\infty) \times \mathbb{R}_\eta \to\mathbb{C}$ such that $f(\cdot
,\eta) \in C^\infty((0,\infty))$ for every $\eta\in\mathbb{R}$, satisfies
$(G_0)$ with $(\bar f, \partial_{x_1}\partial_{x_2}\bar f)(0,0,\eta) = (0,0)$,
$(G_\infty)$, and is such that all derivatives
$\frac{\mathrm{d}^j f}{\mathrm{d} x^j}$
for $j \ge0$ are continuous in $\left\{(x,\eta): \, x > 0, \, \eta \ne 0\right\}$.
Then for every $\eta\in\mathbb{R}$ there exists exactly one solution
$v \colon (0,\infty)\times\mathbb{R}_\eta\to\mathbb{C}$ to the resolvent equation \reftext{\eqref{resolvent}}
such that $v(\cdot,\eta) \in C^\infty((0,\infty))$ for $\eta \in \mathbb R $
and for $k=0,1,2$ we have $ {\left\lvert v \right
\rvert}_{k,-\delta- k} < \infty
$ for some $\delta\in\left(0,\frac1{10}\right)$. Moreover, $v$
satisfies conditions $(G_0)$ and $(G_\infty)$, and $\frac{\mathrm
{d}^j v}{\mathrm{d}
x^j}$ for $j \ge0$ are continuous in $\left\{(x,\eta): \, x > 0, \,
\eta\ne0\right\}$.
\end{proposition}

Note that for later purpose the continuous dependence on $\eta$ will
not be necessary but the weaker measurability is sufficient.

The strategy for constructing a solution to \reftext{\eqref{resolvent}} proceeds
in three steps:
\begin{enumerate}[(i)]
\item[(i)] We construct a $2$-dimensional manifold of solutions to \reftext{\eqref{resolvent}} for $x\ll1$ meeting $(G_0)$ (cf.~\reftext{Definition~\ref{def:g0eta}}).
\item[(ii)] We construct a $2$-dimensional manifold of solutions to \reftext{\eqref{resolvent}} for $x\gg1$ fulfilling $(G_\infty)$ (cf.~\reftext{Definition~\ref{def:ginftyeta}}).
\item[(iii)] We find exactly one solution to \reftext{\eqref{resolvent}} satisfying
contemporarily the above two conditions, by intersecting the
above two
$2$-dimensional solution manifolds in $4$-dimensional phase space
spanned by coordinates $\left(v,\frac{\mathrm{d} v}{\mathrm{d} x},
\frac{\mathrm{d}^2v}{\mathrm{d}
x^2}, \frac{\mathrm{d}^3 v}{\mathrm{d} x^3}\right)$.
\end{enumerate}
The above last step is done as in
\cite[Proposition~6.3]{ggko.2014} using
the coercivity of $q(D_x)$, though we additionally prove continuity in
$\eta$ (see also Angenent \cite{a.1988} for a similar argument). This is based
on the analogue of the uniqueness result
\cite[Proposition~5.5]{ggko.2014},
which now can be formulated as follows:
%
%
\begin{lemma}\label{lem:res_unique}
Suppose that $\eta\in\mathbb{R}$ and $\delta \in \left(0,\frac{1}{10}\right)$ are fixed. If $v \in C^{\infty
}((0,\infty))$
solves \reftext{\eqref{resolvent}} with $f=0$ in $(0,\infty)$ and for
$k=0,1,2$ we have $ {\left\lvert v \right\rvert
}_{k,-\delta-k} < \infty$, then
$v=0$ in $(0,\infty)$.
\end{lemma}

For the convenience of the reader, we will give a proof of
\reftext{Proposition~\ref{prop:resolvent}} and \reftext{Lemma~\ref{lem:res_unique}} at the
end of \S \ref{sec:resolvent}. In order to complete the presentation of 
the above sketch of proof, we first construct solutions to \reftext{\eqref{resolvent}} as described above,
separately near $x = 0$ and near $x = \infty$.

\paragraph*{Construction of solutions for $x\ll1$}
The treatment for $x \ll0$ is analogous to the reasoning in
\cite[\S6.2]{ggko.2014} since -- as mentioned twice before -- the extra terms
coming from the additional dimension are of order $O\left(x^2\right)$
and $O\left(x^4\right)$, respectively, i.e., they are perturbative
terms in the fixed-point problem for $x \ll1$. We first
concentrate on constructing solutions to \reftext{\eqref{resolvent}} for $x \ll
1$ and $\eta\in\mathbb{R}$ fixed,
so that we initially suppress the
dependence on $\eta$ in the notation.
The continuous dependence on
$\eta$ will be discussed at the end of the paragraph.

Here the leading-order operator in \reftext{\eqref{resolvent}} is $q(D_x)$ and
the corresponding homogeneous equation has two linearly independent
bounded solutions $x^0,x^{1+\beta}$ corresponding to the positive roots
$0, 1+\beta$ of the polynomial $q(\zeta)$. As a consequence, we expect
that a generic bounded solution of $q(D_x)v=f$ behaves like
$v(x)=a_1+a_2x^{1+\beta} + o\left(x^{1+\beta}\right)$ as $x \searrow0$, where
$a_1, a_2 \in\mathbb R$ provided $f(x)=o\left(x^{1+\beta}\right)$ as 
$x \searrow0$. Then we unfold the singular behavior for
$0 \le x \ll1$ by using
two variables $x_1,x_2$ rather than the single variable $x$,
i.e., we employ the following substitutions:
\begin{equation*}
v(x) \mapsto\bar v\left(x_1,x_2\right), \quad f(x) \mapsto\bar
f(x_1,x_2), \quad D = x \tfrac{\mathrm{d}}{\mathrm{d} x} \mapsto
\bar D := x_1\partial
_{x_1} + \beta x_2\partial_{x_2},
\end{equation*}
where on the characteristic $(x_1,x_2) = \left(x,x^\beta\right)$ we
can identify
\begin{equation*}
D^j v(x) = \bar D^j \bar v\left(x,x^\beta\right) \quad\mbox{and}
\quad D^j f(x) = \bar D^j \bar f\left(x,x^\beta\right) \quad
\mbox
{for} \quad j \in\mathbb{N}_0.
\end{equation*}
For fixed $\eta\in\mathbb{R}\setminus\{0\}$, we then consider the unfolded problem
\begin{subequations}\label{solbigbar}
\begin{equation}\label{eq_bigbar}
x_1 \bar v + q\left(\bar D\right)\bar v - \eta^2 x_1^2 r\left(\bar
D\right) \bar v + \eta^4 x_1^4 \bar v = \bar f \quad\mbox{in} \quad Q
:= [0,\varepsilon] \times[0,L],
\end{equation}
subject to
\begin{equation}
\left(\bar v, \partial_{x_1}\partial_{x_2}\bar v\right)(0,0)=(a_1,a_2),
\end{equation}
\end{subequations}
where $a_1, a_2 \in\mathbb{R}$ are parameters and by compatibility necessarily
$\left(\bar f, \partial_{x_2} \bar f\right)(0,0) = (0,0)$ holds true.
Problem~\reftext{\eqref{solbigbar}} can be solved with a fixed-point argument,
for which we define the norms (compare to item~(ii)
in \reftext{Definition~\ref{def:g0eta}})
\begin{equation}\label{norm_omega}
 {\left\lvert\bar f \right\rvert}_\omega:= \sum
_{(k,\ell) \in\mathbb{N}_0^2} \frac{\varepsilon^k L^\ell
}{k ! \ell!}  {\left\lvert\partial_{x_1}^k \partial
_{x_2}^\ell\bar f(0,0) \right\rvert}
\quad\mbox{and} \quad {\left\lvert\bar v \right\rvert
}_{4,\omega} := \sum_{m = 0}^4
 {\left\lvert\bar D^m \bar v \right\rvert}_\omega,
\end{equation}
with parameters $\varepsilon, L > 0$. Note that finiteness of
$ {\left\lvert\bar f \right\rvert}_\omega$ is
equivalent to the property that the series $\sum_{(k,\ell
) \in\mathbb{N}_0^2} \frac{\partial_{x_1}^k \partial_{x_2}^\ell
\bar f(0,0)}{k
! \ell!} x_1^k x_2^\ell$ converges absolutely in $Q$. We obtain the
following result being the analogue of \cite[Lemma~6.5]{ggko.2014},
though the norms \reftext{\eqref{norm_omega}} are chosen slightly differently and
are more in line with the reasoning in \cite[\S4.4]{bgk.2016}:
%
%
\begin{lemma}\label{lem:fixedpointbar}
For all compact intervals ${\mathcal J}\Subset\mathbb{R}\setminus\{0\}$,
there
exists $\varepsilon> 0$
such that for all $\eta\in{\mathcal J}$, any $L > 0$, and any
$a_1,a_2\in
\mathbb R$, for any $\bar f = \bar f(x_1,x_2)$ analytic in $Q = \left
[0,\varepsilon\right] \times\left[0,L\right]$ with
$\left(\bar f,\partial_{x_1}\partial_{x_2}\bar f\right)(0,0) = (0,0)$
and $ {\left\lvert\bar f \right\rvert}_\omega< \infty$,
problem~\reftext{\eqref{solbigbar}} has an analytic solution $\bar v$, which
additionally satisfies
\begin{equation}\label{bound_bar}
 {\left\lvert\bar v \right\rvert}_{4,\omega} \lesssim_{\varepsilon,\mathcal J}
 {\left\lvert a_1 \right\rvert} +  {\left
\lvert a_2 \right\rvert} +  {\left\lvert\bar f \right
\rvert}_\omega.
\end{equation}
Now define $v^{(j)}(x):=\bar v(x,x^\beta)$ solving \reftext{\eqref{resolvent}} for $x\le\varepsilon$ with
\begin{enumerate}[\textit{(2)}]
\item[\textit{(0)}]$(a_1,a_2) = (0,0)$ and given
$f = \bar f\left(x,x^\beta\right)$ for $j = 0$,
\item[\textit{(1)}]$(a_1,a_2) = (1,0)$ and $f = 0$ in $(0,\infty)$ for $j = 1$,
\item[\textit{(2)}]$(a_1,a_2) = (0,1)$ and $f = 0$ in $(0,\infty)$ for $j = 2$.
\end{enumerate}
Then
\begin{subequations}\label{closetozeroresolvent}
\begin{align}
v^{(0)}(x) &= - \frac8 9 \tfrac{\mathrm{d} f}{\mathrm{d} x}(0) x +
O\left(x^2\right)
\quad\mbox{as} \quad x \searrow0,
\eqncr
v^{(1)}(x) &= 1 + O\left(x^2\right) \quad\mbox{as} \quad x \searrow0,
\eqncr
v^{(2)}(x) &= x^{1+\beta} + O\left(x^2\right) \quad\mbox{as} \quad x
\searrow0,
\end{align}
\end{subequations}
and $v := v^{(0)} + a_1 v^{(1)} + a_2 v^{(2)}$ is a solution to \reftext{\eqref{resolvent}} on the interval $(0,\varepsilon]$ that can be extended to a
solution for all $x > 0$. Furthermore, if
${\mathcal J}
\owns\eta\mapsto\bar f(\cdot,\cdot,\eta)$ is continuous with values in $\{
\bar
f: \,  {\left\lvert\bar f \right\rvert}_\omega< \infty
\}$, then the solution map ${\mathcal J}
\owns\eta\mapsto\bar v$ is continuous with values in $\{\bar w: \,
{\left\lvert\bar w \right\rvert}_{4,\omega} < \infty\}$.
\end{lemma}
\begin{proof}[Proof of \reftext{Lemma~\ref{lem:fixedpointbar}}]
\textit{Step~1: Linear solution operator.}
We first observe that for $\bar f$ analytic on $Q$ satisfying
\begin{equation*}
(\bar f,  \partial_{x_1}\partial_{x_2}\bar f)(0,0)=(0,0) \quad\mbox{and} \quad
 {\left\lvert\bar f \right\rvert}_\omega< \infty,
\end{equation*}
there exists an analytic solution $T\bar f$ with $\sum_{m = 0}^4
 {\left\lvert\bar D^m T \bar f \right\rvert}_\omega< \infty$ of
\begin{equation}\label{solhombar}
q(\bar D) T \bar f = \bar f \quad\mbox{in} \quad Q, \quad\mbox{with}
\quad\left(T \bar f, \partial_{x_1}\partial_{x_2}T \bar f\right)(0,0) = (0,0),
\end{equation}
where
\begin{equation}\label{t_prod}
T = T_{-\frac12} T_{-\beta+\frac12} T_0T_{1+\beta}
\end{equation}
and $T_\gamma$ is determined through
\begin{equation}\label{1order_bar}
(\bar D -\gamma) T_\gamma\bar g = \bar g \quad\mbox{in} \quad Q,
\quad\mbox{subject to} \quad
\left(T_\gamma\bar g, \partial_{x_1}\partial_{x_2}
T_\gamma\bar g\right)(0,0)=(0,0),
\end{equation}
and where by compatibility we need to assume
$(\bar g, \partial_{x_1}\partial_{x_2}\bar g)(0,0)=(0,0)$.
Further assuming that $\bar g$ is analytic
with $ {\left\lvert\bar g \right\rvert}_\omega< \infty
$, we may define $T_\gamma\bar g$ through
\begin{equation*}
\partial_{x_1}^k \partial_{x_2}^\ell T_\gamma\bar g(0,0) := \frac
{\partial_{x_1}^k \partial_{x_2}^\ell\bar g(0,0)}{k + \beta
\ell - \gamma}, \quad\mbox{where} \quad(k,\ell) \in\mathbb{N}_0^2
\setminus\{(0,0), (1,1)\},
\end{equation*}
and
\begin{equation*}
\left(T \bar g, \partial_{x_1} \partial_{x_2} T \bar g\right)(0,0) := (0,0),
\end{equation*}
so that $T_\gamma\bar g(x_1,x_2) := \sum_{(k,\ell) \in\mathbb
{N}_0^2} \frac
{1}{k ! \ell!} \partial_{x_1}^k \partial_{x_2}^\ell T_\gamma\bar
g(0,0) x_1^k x_2^\ell$ is indeed an analytic solution to \reftext{\eqref{1order_bar}}. We find moreover that
\begin{equation*}
\sum_{m = 0}^1  {\left\lvert\bar D^m T_\gamma\bar g
\right\rvert}_\omega\lesssim {\left\lvert\bar g \right
\rvert}_\omega,
\end{equation*}
which, by the product decomposition \reftext{\eqref{t_prod}} and due to the
commutation relation $\bar D T_\gamma= T_\gamma\bar D$, immediately
upgrades to the maximal-regularity estimate
%
%
\begin{equation}\label{mr_bar}
 {\left\lvert T \bar f \right\rvert}_{4,\omega} \lesssim
 {\left\lvert\bar f \right\rvert}_\omega.
\end{equation}
\textit{Step~2: Fixed-point argument.}
We may reformulate problem~\reftext{\eqref{solbigbar}} in terms of the
fixed-point problem
%
%
\begin{equation}\label{fixedptnearzero}
\bar v = {\mathcal T}[\bar v] \quad\mbox{for} \quad{\mathcal T}[\bar
v]:= a_1 + a_2 x_2
+ T[\bar f]+ T\left[-x_1\bar v + \eta^2 x_1^2r(\bar D)\bar v - \eta^4
x_1^4 \bar v\right],
\end{equation}
where the operator ${\mathcal T}$ acts on the space of analytic
functions $\bar
v$ with $ {\left\lvert\bar v \right\rvert}_{4,\omega} <
\infty$. Now notice that with help
of \reftext{\eqref{norm_omega}} and \reftext{\eqref{mr_bar}}, by the sub-multiplicativity
of $ {\left\lvert\cdot\right\rvert}_\omega$
(cf.~\cite[Lemma~3(b)]{ggo.2013} for an
analogous case) we have
%
%
\begin{equation}\label{bar_self}
 {\left\lvert{\mathcal T}\left[\bar v\right] \right
\rvert}_{4,\omega} \le C_1 \left( {\left\lvert a_1
\right\rvert} +
 {\left\lvert a_2 \right\rvert} L +  {\left
\lvert\bar f \right\rvert}_\omega\right) + C_2 \left(\varepsilon
+ \eta^2
\varepsilon^2 + \eta^4 \varepsilon^4\right)  {\left
\lvert\bar v \right\rvert}_{4,\omega},
\end{equation}
with constants $C_1$, $C_2$, and for $\bar v$ and $\bar w$ such that $
{\left\lvert\bar v \right\rvert}_{4,\omega} < \infty$ and
$ {\left\lvert\bar w \right\rvert}_{4,\omega} < \infty
$ we have
%
%
\begin{equation}\label{bar_contract}
 {\left\lvert{\mathcal T}\left[\bar v\right] - {\mathcal
T}\left[\bar w\right] \right\rvert}_{4,\omega} \le
C_2 \left(\varepsilon+ \eta^2 \varepsilon^2 + \eta^4 \varepsilon
^4\right)  {\left\lvert\bar v - \bar w \right\rvert
}_{4,\omega}.
\end{equation}
Estimates~\reftext{\eqref{bar_self}} and \reftext{\eqref{bar_contract}} imply that
${\mathcal T}$ is
a contraction, provided
\begin{equation}\label{contraction_condition_xll1}
C_2 \left(\varepsilon+ \eta^2 \varepsilon^2
+ \eta^4 \varepsilon^4\right) < 1 \quad \mbox{for} \quad \eta \in \mathcal J,
\end{equation}
which can be uniformly if $0 < \varepsilon
\ll1$. The contraction-mapping theorem yields
existence of
a unique solution to \reftext{\eqref{fixedptnearzero}} and therefore also to \reftext{\eqref{solbigbar}}. The bound \reftext{\eqref{bound_bar}} follows immediately from \reftext{\eqref{bar_self}} for $0 < \varepsilon\ll\min\left\{1,
{\left\lvert\eta\right\rvert}^{-1}\right
\}$. Inserting the power series $\bar v(x_1,x_2) = \sum_{(k,\ell) \in
\mathbb{N}
_0^2} \frac{\partial_{x_1}^k \partial_{x_2}^\ell\bar v(0,0)}{k !
\ell
!} x_1^k x_2^\ell$ and $\bar f(x_1,x_2) = \sum_{(k,\ell) \in\mathbb{N}_0^2}
\frac{\partial_{x_1}^k \partial_{x_2}^\ell\bar f(0,0)}{k ! \ell!}
x_1^k x_2^\ell$ in \reftext{\eqref{eq_bigbar}}, we find
\begin{equation*}
\partial_{x_1} \bar v(0,0) = \frac{\partial_{x_1} \bar f(0,0) -
a_1}{q(1)} \stackrel{\text{\reftext{\eqref{poly_q}}}}{=} \frac8 9 \left(a_1 -
\partial
_{x_1} \bar f(0,0)\right),
\end{equation*}
from which we conclude that \reftext{\eqref{closetozeroresolvent}} holds true. The extension of $v$ to $(0,\infty)$ follows by standard theory.

\medskip
\noindent\textit{Step~3: Continuous dependence on $\eta$.}
From the fixed-point equation \reftext{\eqref{fixedptnearzero}} it follows that
for a given fixed point $\bar v$ for given right-hand side $\bar f$
with $ {\left\lvert\bar f \right\rvert}_\omega< \infty
$, we have
\begin{eqnarray}\nonumber
 {\left\lvert\bar v(\cdot,\eta_1) - \bar v(\cdot,\eta
_2) \right\rvert}_{4,\omega} &\lesssim
& C_1  {\left\lvert\bar f(\cdot,\eta_1) - \bar g(\cdot
,\eta_2) \right\rvert}_\omega\nonumber\\
&&+ C_2
\left(\varepsilon+ \eta_1^2 \varepsilon^2 + \eta_1^4 + \varepsilon
^4\right)  {\left\lvert\bar v(\cdot,\eta_1) - \bar
v(\cdot,\eta_2) \right\rvert}_{4,\omega} \nonumber\\
&& + C_2\left( {\left\lvert\eta_1^2 - \eta_2^2 \right
\rvert} \varepsilon^2 +  {\left\lvert\eta_1^4 - \eta
_2^4 \right\rvert}\right)  {\left\lvert\bar v(\cdot
,\eta_2) \right\rvert}_{4,\omega}, \label{bar_continuity}
\end{eqnarray}
where $C_1$ and $C_2$ are as in \reftext{\eqref{bar_self}} and \reftext{\eqref{bar_contract}}. Under the same smallness assumption \eqref{contraction_condition_xll1} on $\varepsilon$, we get from \reftext{\eqref{bar_self}} that $ {\left\lvert\bar v(\cdot
,\eta_2) \right\rvert}_{4,\omega
}$ is uniformly bounded, and further using
\reftext{\eqref{bar_continuity}}, we find that
\begin{equation*}
 {\left\lvert\bar v(\cdot,\eta_1) - \bar v(\cdot,\eta
_2) \right\rvert}_{4,\omega} \to0
\quad\mbox{as} \quad \eta_2 \to \eta_1 \quad\mbox{in} \quad
{\mathcal J}.
\end{equation*}
\end{proof}

\paragraph*{Construction of solutions for $x\gg1$}

Our aim is to prove the following result which is a generalization of
\cite[Lemma~6.6]{ggko.2014}:
%
%
\begin{lemma}[Resolvent equation for $x\gg1$]\label{lem:xgg1}
Assume that $f \in C\left((0,\infty) \times \left(\mathbb{R}_\eta\setminus\{0\}\right)\right)$, with
$f(\cdot
,\eta) \in C^\infty\left((0,\infty)\right)$ for every $\eta\in
\mathbb{R}\setminus\{0\}$,
satisfies $(G_\infty)$. Then there exists
a two-parameter family of
smooth solutions of \reftext{\eqref{resolvent}} of the form
\begin{equation*}
v(x) = v^{(\infty)}(x) + a_3 v^{(3)}(x) + a_4 v^{(4)}(x) \quad\mbox
{for} \quad x > 0, \quad where \quad a_1,a_2\in\mathbb R,
\end{equation*}
$v^{(\infty)} = v^{(\infty)}(x)$ is a smooth solution to \reftext{\eqref{resolvent}}, and $v^{(3)} = v^{(3)}(x)$, $v^{(4)} = v^{(4)}(x)$ are two
linearly independent solutions to \reftext{\eqref{resolvent}} with homogenous right-hand side,
such that $v^{(\infty)}$, $v^{(3)}$, and $v^{(4)}$ satisfy
$(G_\infty)$.
\end{lemma}
\begin{proof}[Proof of \reftext{Lemma~\ref{lem:xgg1}}]
We first study the homogeneous version of the resolvent equation, i.e.,
equation~\reftext{\eqref{resolvent}} with $f = 0$ in $(0,\infty)$. Notice
that in the case $\eta=0$, equation~\reftext{\eqref{resolvent}} takes the form
\begin{equation*}
xv +q(D_x)v = f \quad\mbox{for} \quad x > 0,
\end{equation*}
which is the same form as
\cite[Equation~(6.1)]{ggko.2014} except for the
fact that $q(\zeta)$ is a different fourth-order polynomial. However,
the leading coefficient of $q(\zeta)$ here is the same as that of the
polynomial $p(\zeta)$ in
\cite[Equation~(6.1)]{ggko.2014}, which allows to
apply the same strategy as in the proof of
\cite[Lemma~6.6]{ggko.2014}, leading to the statement of \reftext{Lemma~\ref{lem:xgg1}} for
$\eta= 0$. Hence, in what follows we concentrate on the case $\eta\ne0$.

\medskip
\noindent\textit{Step~1: Finding two linearly independent solutions
to the
homogeneous resolvent equation.}

Collecting only the leading-order terms in the homogeneous version of
the resolvent equation \reftext{\eqref{resolvent}}, we find that the principal
part in $\frac{\mathrm{d}}{\mathrm{d} x}$ and $\eta$ is given by
\begin{equation*}
x^4 \left(\tfrac{\mathrm{d}^4}{\mathrm{d} x^4} - 2 \eta^2 \tfrac
{\mathrm{d}^2}{\mathrm{d} x^2} + \eta
^4\right) = x^4 \left(\tfrac{\mathrm{d}^2}{\mathrm{d} x^2} - \eta
^2\right)^2 = x^4 \left
(\tfrac{\mathrm{d}}{\mathrm{d} x} -  {\left\lvert\eta
\right\rvert}\right)^2 \left(\tfrac{\mathrm{d}}{\mathrm{d} x} +
 {\left\lvert\eta\right\rvert}\right)^2.
\end{equation*}
The kernel of this operator is spanned by $\left\{e^{\pm
{\left\lvert\eta\right\rvert}
x}, x e^{\pm {\left\lvert\eta\right\rvert} x}\right\}
$. In what follows, we will
construct two linearly independent solutions that are stable as $x \to
\infty$, so that -- neglecting algebraic factors for the time being --
we are lead to factor out $e^{- {\left\lvert\eta\right
\rvert} x}$ from $v$, i.e., we define
\begin{equation}\label{def_phi}
v =: e^{- {\left\lvert\eta\right\rvert} x} \phi
\end{equation}
and hence the homogeneous version of the resolvent equation \reftext{\eqref{resolvent}} turns into
\begin{equation*}
x\phi+ q\left(D_x\right) \phi+  {\left\lvert\eta
\right\rvert} x \left(- 4 D_x^3 - 3
D_x^2 + D_x + \tfrac9 8\right) \phi+ \eta^2 x^2 \left(4 D_x^2 + 6 D_x
+ 2\right) \phi= 0 \quad \mbox{in} \quad (0,\infty).
\end{equation*}
Next, it turns out to be convenient to apply the scaling $
{\left\lvert\eta\right\rvert}
x \mapsto x$, leading to the simplification
\begin{equation}\label{eq_xgg1_phi}
 {\left\lvert\eta\right\rvert}^{-1} x \phi+ q\left
(D_x\right) \phi+ x \left(- 4 D_x^3 -
3 D_x^2 + D_x + \tfrac9 8\right) \phi+ x^2 \left(4 D_x^2 + 6 D_x +
2\right) \phi= 0 \quad \mbox{in} \quad (0,\infty).
\end{equation}
We will now consider the leading-order operator $x^2{\mathcal A}$, where
\begin{subequations}\label{xgg1_phi}
\begin{equation}\label{xgg1_lead}
{\mathcal A}:= \left(\tfrac{\mathrm{d}^2}{\mathrm{d} x^2} - 4
\tfrac{\mathrm{d}}{\mathrm{d} x} + 4\right) \left
(D_x^2 + \tfrac3 2 D_x + \tfrac1 2\right) = \left(\tfrac{\mathrm
{d}}{\mathrm{d} x} -
2\right)^2 \left(D_x + 1\right) \left(D_x + \tfrac1 2\right).
\end{equation}
The criteria for selecting ${\mathcal A}$ as above are that $x^2
{\mathcal A}$ contains
all terms from \reftext{\eqref{eq_xgg1_phi}} having the maximal number of $4$
derivatives while at the same time collecting the terms $x^2 \left(4
D_x^2 + 6 D_x + 2\right)$ with the largest pre-factor $x^2$ (which
scaling-wise give the leading order in $x$ for \reftext{\eqref{eq_xgg1_phi}} as
\mbox{$x\to\infty$}). Furthermore, it has the useful factorization \reftext{\eqref{xgg1_lead}}, which makes it easier to find an inverse with appropriate
estimates. Utilizing \reftext{\eqref{xgg1_lead}} in \reftext{\eqref{eq_xgg1_phi}}, we
infer that \reftext{\eqref{eq_xgg1_phi}} can be reformulated as
\begin{equation}\label{eq_xgg1_phi_alt}
{\mathcal A}\phi= -  {\left\lvert\eta\right\rvert
}^{-1} x^{-1} \phi+ \frac3 2 x^{-1} \tfrac{\mathrm{d}
}{\mathrm{d} x} \left(D_x^2 - \tfrac1 4\right) \phi- 3 x^{-1}
\left(D_x^2 +
D_x + \tfrac3 8\right) \phi \quad \mbox{in} \quad (0,\infty).
\end{equation}
\end{subequations}
Now we notice that the powers $x^{-\frac1 2}$ and $x^{-1}$ are in the
kernel of the operator ${\mathcal A}$, so that we may set
\begin{equation}\label{def_psi}
\phi=: x^\alpha(1+\psi), \quad\mbox{where} \quad\alpha\in\left
\{
-\tfrac1 2, -1\right\}.
\end{equation}
Our goal is to construct solutions $\psi$ with $\psi\to0$ as $x \to
\infty$. Indeed, undoing the transformations \reftext{\eqref{def_phi}}, \reftext{\eqref{def_psi}}, and the scaling of $x$, this will result in two linearly
independent solutions $v^{(1)}$ and $v^{(2)}$ to the homogeneous
version of the resolvent equation \reftext{\eqref{resolvent}} with the
asymptotic behavior
\begin{equation}\label{as_vj}
v^{(j)} = x^{\alpha_j} e^{- {\left\lvert\eta\right
\rvert} x} (1+o(1)) \quad\mbox{as}
\quad x \to\infty, \quad\mbox{where} \quad\alpha_1 := - \frac1 2,
\quad\alpha_2 := -1.
\end{equation}
The decay as given in \reftext{Definition~\ref{def:ginftyeta}(ii)} will also
immediately follow from the fixed-point argument.

In the two situations formulated in \reftext{\eqref{def_psi}} we consider
\begin{subequations}\label{xgg1_psi}
\begin{equation}\label{xgg1_lead_psi}
{\mathcal A}_{-\frac1 2} := \left(\tfrac{\mathrm{d}}{\mathrm{d}
x}-2\right)^2 \left(D_x +
\tfrac1 2\right) D_x \quad\mbox{and} \quad{\mathcal A}_{-1} :=
\left(\tfrac
{\mathrm{d}}{\mathrm{d} x}-2\right)^2 \left(D_x - \tfrac1 2\right
) D_x,
\end{equation}
and with this notation \reftext{\eqref{eq_xgg1_phi_alt}} becomes
\begin{eqnarray}\nonumber
{\mathcal A}_\alpha\psi&=& -  {\left\lvert\eta\right
\rvert}^{-1} x^{-1} + x^{-1} P_\alpha(0) +
x^{-2} R_\alpha(0) \\
&& -  {\left\lvert\eta\right\rvert}^{-1} x^{-1} \psi+
x^{-1} \tfrac{\mathrm{d}}{\mathrm{d} x} Q_\alpha
(D_x) \psi+ x^{-1} P_\alpha(D_x) \psi+ x^{-2} R_\alpha(D_x) \psi \quad \mbox{in} \quad (0,\infty),
\label{eq_xgg1_psi}
\end{eqnarray}
\end{subequations}
where $Q_\alpha(\zeta)$, $P_\alpha(\zeta)$, and $R_\alpha(\zeta)$ are
real-valued polynomials of degree less or equal to $2$.

Next we invert the operator ${\mathcal A}_\alpha$ and prove estimates in
suitable norms. We first note that the product structure of ${\mathcal
A}_\alpha$
induces a product structure of the solution operator $T = T_{-2}^2
S_{\pm\frac1 2} S_0$, where $T_{-2} \varrho$ is a solution of $\left
(\tfrac{\mathrm{d}}{\mathrm{d} x} - 2\right) T_{-2} \varrho=
\varrho$ and $S_\mu
\varrho$ solves $(D_x-\mu) S_\mu\varrho= \varrho$. We use the
explicit definitions (which have the role of
fixing the behavior of $T_{-2}\varrho$ and $S_\mu\varrho$ as $x\to
\infty$)
\begin{equation}\label{def_t-2}
T_{-2} \varrho(x) := - \int_x^\infty e^{2 \left(x-x^\prime\right)}
\varrho\left(x^\prime\right) \mathrm{d} x^\prime= - \int
_0^\infty e^{- 2
x^\prime} \varrho\left(x+x^\prime\right) \mathrm{d} x^\prime
\end{equation}
and
%
%
\begin{equation}\label{def_s_gamma}
S_\mu\varrho(x) := - x^\mu\int_x^\infty\left(x^\prime\right
)^{-\mu}
\varrho\left(x^\prime\right) \frac{\mathrm{d} x^\prime}{x^\prime
} = - \int
_1^\infty r^{-\mu} \varrho(x r) \frac{\mathrm{d} r}{r}.
\end{equation}

For proving estimates for $T_{-2}$, observe that for $x \ge x_0 > 0$ we have
\begin{equation*}
x  {\left\lvert T_{-2} \varrho(x) \right\rvert}
\stackrel{\text{\reftext{\eqref{def_t-2}}}}{\le} x \left(\int
_0^\infty e^{-2 x^\prime} \left(x+x^\prime\right)^{-1} \mathrm{d}
x^\prime\right
) \sup_{x \ge x_0} x  {\left\lvert\varrho(x) \right
\rvert} \le\frac1 2 \sup_{x \ge x_0}
x  {\left\lvert\varrho(x) \right\rvert},
\end{equation*}
that is,
\begin{equation*}
\sup_{x \ge x_0} x  {\left\lvert T_{-2} \varrho(x) \right
\rvert} \le\frac1 2 \sup_{x \ge
x_0} x  {\left\lvert\varrho(x) \right\rvert},
\end{equation*}
and because of $\tfrac{\mathrm{d}}{\mathrm{d} x} T_{-2} \varrho= 2
T_{-2} \varrho+
\varrho$ and therefore also $\tfrac{\mathrm{d}^{j+1}}{\mathrm{d}
x^{j+1}} T_{-2}
\varrho= 2 \tfrac{\mathrm{d}^j}{\mathrm{d} x^j} T_{-2} \varrho+
\tfrac{\mathrm{d}^j}{\mathrm{d} x^j}
\varrho$, we have
%
%
\begin{equation}\label{est_t-2}
\max_{j = 0,\ldots,J+1} \sup_{x \ge x_0} x  {\left\lvert
\frac{\mathrm{d}^j}{\mathrm{d} x^j} T_{-2} \varrho(x) \right\rvert
} \lesssim_J \max_{j = 0,\ldots,J} \sup_{x \ge x_0} x
 {\left\lvert\frac{\mathrm{d}^j \varrho}{\mathrm{d}
x^j}(x) \right\rvert} \quad\mbox{for every} \quad J
\in\mathbb{N}_0.
\end{equation}

Considering $S_\mu$, where $\mu\in\left\{-\frac1 2, 0, \frac12\right\}$, we estimate
\begin{equation*}
x  {\left\lvert S_\mu\varrho(x) \right\rvert} \le x
\left(\int_1^\infty r^{-\mu} x^{-1}
r^{-1} \frac{\mathrm{d} r}{r}\right) \sup_{x \ge x_0} x
{\left\lvert\varrho(x) \right\rvert} =
\frac{1}{1+\mu} \sup_{x \ge x_0} x  {\left\lvert\varrho
(x) \right\rvert},
\end{equation*}
that is,
\begin{equation*}
\sup_{x \ge x_0} x  {\left\lvert S_\mu\varrho(x) \right
\rvert} \le\frac{1}{1+\mu} \sup_{x
\ge x_0} x  {\left\lvert\varrho(x) \right\rvert}.
\end{equation*}
The relation $D_x^{k+1} S_\mu\varrho= \mu D_x^k S_\mu\varrho+
D_x^k \varrho$ leads to the upgraded version
%
%
\begin{equation}\label{est_s_mu}
\max_{0\le k\le K+1} \sup_{x \ge x_0} x  {\left\lvert
D_x^k S_\mu\varrho(x) \right\rvert} \lesssim_K \max_{0\le k\le K} \sup_{x \ge x_0} x  {\left\lvert D_x^k
\varrho(x) \right\rvert} \quad\mbox{for every} \quad K \in\mathbb{N}_0.
\end{equation}

Now we use the factorization $T = T_{-2}^2 S_{\pm\frac1 2} S_0$
together with the commutation relation $\frac{\mathrm{d}}{\mathrm{d}
x} D_x = \left
(D_x+1\right) \frac{\mathrm{d}}{\mathrm{d} x}$, so that \reftext{\eqref{est_t-2}}
and \reftext{\eqref{est_s_mu}} result in the maximal-regularity estimate
%
%
\begin{subequations}\label{mr_norms_xgg1}
%
%
\begin{equation}\label{mr_xgg1}
 {\left\lvert T \varrho\right\rvert}_{J+2,K+2,x_0}
\lesssim_{J,K}  {\left\lvert\varrho\right\rvert
}_{J,K,x_0} \quad\mbox{for all} \quad J, K \in\mathbb{N}_0,
\end{equation}
where
%
%
\begin{equation}\label{def_norms_xgg1}
 {\left\lvert\varrho\right\rvert}_{J,K,x_0} := \max
_{\substack{0\le j\le J\\ 0\le k\le K}} \sup_{x \ge x_0} x  {\left\lvert\tfrac
{\mathrm{d}^j}{\mathrm{d} x^j} D_x^k \varrho(x) \right\rvert}.
\end{equation}
\end{subequations}
Notice that estimate~\reftext{\eqref{mr_xgg1}} reflects a regularity gain of two
$\frac{\mathrm{d}}{\mathrm{d} x}$-derivatives and two
$D_x$-deriva\-tives, which is the
maximal gain in regularity to be expected from \reftext{\eqref{xgg1_lead_psi}}.

Next, we discuss the fixed-point argument for the full equation \reftext{\eqref{eq_xgg1_psi}} on the space
\begin{equation*}
\Psi_{J,K,x_0}:= \left\{\psi \colon {\left\lvert \psi \right\rvert}_{J+2,K+2,x_0} < \infty\right\}.
\end{equation*}
First apply the solution operator $T$ to both sides of \reftext{\eqref{eq_xgg1_psi}} and obtain $\psi={\mathcal T}[\psi]$ for
%
%
\begin{eqnarray}\nonumber
{\mathcal T}[\psi] &:=& T\left[-  {\left\lvert\eta
\right\rvert}^{-1} x^{-1} + x^{-1} P_\alpha(0) +
x^{-2} R_\alpha(0)\right] \\
&& + T\left[-  {\left\lvert\eta\right\rvert}^{-1}
x^{-1} \psi+ x^{-1} \tfrac{\mathrm{d}}{\mathrm{d} x}
Q_\alpha(D_x) \psi+ x^{-1} P_\alpha(D_x) \psi+ x^{-2} R_\alpha(D_x)
\psi\right]. \qquad\label{fixed_xgg1}
\end{eqnarray}
For establishing the self-map and contraction property of ${\mathcal
T}$ in $\Psi
_{J,K,x_0}$, observe that
%
%
\begin{subequations}\label{aux_xgg1}
%
%
\begin{align}\label{aux_xgg1_1}
 {\left\lvert T\left[-  {\left\lvert\eta
\right\rvert}^{-1} x^{-1} + x^{-1} P_\alpha(0) + x^{-2} R_\alpha
(0)\right] \right\rvert}_{J+2,K+2,x_0} \stackrel{\text{\reftext{\eqref{mr_norms_xgg1}}}}{\le}
C_1 \left( {\left\lvert\eta\right\rvert}^{-1} +
 {\left\lvert P_\alpha(0) \right\rvert} +
{\left\lvert R_\alpha(0) \right\rvert}
x_0^{-1}\right)
\end{align}
for $C_1 = C_1(J,K)$ and secondly
%
%
\begin{equation}\label{aux_xgg1_2}
\begin{aligned}
&  {\left\lvert T\left[-  {\left\lvert\eta
\right\rvert}^{-1} x^{-1} \psi+ x^{-1} \tfrac{\mathrm{d}}{\mathrm
{d} x} Q_\alpha(D_x) \psi+ x^{-1} P_\alpha(D_x) \psi+ x^{-2}
R_\alpha(D_x) \psi\right] \right\rvert}_{J+2,K+2,x_0} \\
& \quad\stackrel{\text{\reftext{\eqref{mr_norms_xgg1}}}}{\le} C_2 \left(
{\left\lvert\eta\right\rvert}^{-1} x_0^{-1} + x_0^{-1} +
x_0^{-2}\right)  {\left\lvert\psi\right\rvert}_{J+2,K+2,x_0}
\end{aligned}
\end{equation}
\end{subequations}
for $C_2 = C_2(J,K)$. From \reftext{\eqref{aux_xgg1}} it is immediate that
${\mathcal T}$ is a self-map in $\Psi_{J,K,x_0}$ and estimate~\reftext{\eqref{aux_xgg1_2}}
also yields the contraction property of ${\mathcal T}$ provided
%
%
\begin{equation}\label{smallness_xgg1}
C_2 \left( {\left\lvert\eta\right\rvert}^{-1}
x_0^{-1} + x_0^{-1} + x_0^{-2}\right) < 1,
\end{equation}
which is true for $x_0 \gg_{J,K, {\left\lvert\eta\right
\rvert}} 1$ uniformly in $\eta\in
{\mathcal J}\Subset\mathbb{R}\setminus\{0\}$. Hence, for all $J, K
\in\mathbb{N}_0$ there
exists an $x_0 = x_0(J,K) > 0$ such that \reftext{\eqref{aux_xgg1}} and therefore
also \reftext{\eqref{eq_xgg1_psi}} has a solution $\psi= \psi_{J,K}$ with finite
norm $ {\left\lvert\psi\right\rvert}_{J+2,K+2,x_0}$.
Next, we verify the continuity in
$\eta$.

For a fixed point $\psi$ to \reftext{\eqref{fixed_xgg1}} we have with the same
constants $C_1$ and $C_2$ as in \reftext{\eqref{aux_xgg1}}
\begin{eqnarray}\nonumber
 {\left\lvert\psi(\cdot,\eta_1) - \psi(\cdot
,\eta_2) \right\rvert}_{J+2,K+2,x_0}
&\le& C_1 \frac{ {\left\lvert\eta_1 - \eta_2 \right
\rvert}}{ {\left\lvert\eta_1 \right\rvert}
{\left\lvert\eta_2 \right\rvert}}
+ C_2 \left( {\left\lvert\eta_1 \right\rvert}^{-1}
x_0^{-1} + x_0^{-1} + x_0^{-2}\right)
 {\left\lvert\psi(\cdot,\eta_1) - \psi(\cdot,\eta_2)
\right\rvert}_{J+2,K+2,x_0} \\
&& + C_2 \frac{ {\left\lvert\eta_1 - \eta_2 \right
\rvert}}{x_0  {\left\lvert\eta_1 \right\rvert}
 {\left\lvert\eta_2 \right\rvert}}  {\left
\lvert\psi(\cdot,\eta_2) \right\rvert}_{J+2,K+2,x_0}. \label{continuity_xgg1}
\end{eqnarray}
Under the same smallness assumption \reftext{\eqref{smallness_xgg1}} as for the
contraction property, we obtain from \reftext{\eqref{fixed_xgg1}} and \reftext{\eqref{aux_xgg1}} uniform boundedness of $ {\left\lvert\psi
(\cdot,\eta_2) \right\rvert}_{J+2,K+2,x_0}$ for $\eta_2 \in
{\mathcal J}\Subset\mathbb{R}\setminus\{0\}$ and
from \reftext{\eqref{continuity_xgg1}} also
\begin{equation*}
 {\left\lvert\psi(\cdot,\eta_1) - \psi(\cdot,\eta_2)
\right\rvert}_{J+2,K+2,x_0} \to0
\quad\mbox{as} \quad \eta_2 \to \eta_1 \quad\mbox{in} \quad
{\mathcal J}.
\end{equation*}

By standard theory, the solution $\psi= \psi_{J,K}$ can be uniquely
extended to a solution of \reftext{\eqref{eq_xgg1_psi}} for all $x > 0$ which is
also continuous for $\eta\in{\mathcal J}\Subset\mathbb{R}\setminus
\{0\}$. Now for
pairs of indices $(J,K), \left(J^\prime,K^\prime\right) \in\mathbb{N}_0^2$
such that $J \le J^\prime$ and $K \le K^\prime$ we may choose $x_0 > 0$
large enough such that the fixed points $\psi_{J,K}$ and $\psi
_{J^\prime
,K^\prime}$ exist. Because $\Psi_{J^\prime,K^\prime,x_0} \subseteq
\Psi
_{J,K,x_0}$, the uniqueness part of the contraction-mapping theorem
gives $\psi_{J,K} = \psi_{J^\prime,K^\prime}$ on
$[x_0,\infty
)$ and the unique extension property yields the same identity on the
whole interval $(0,\infty)$. Undoing the transformations \reftext{\eqref{def_phi}}, \reftext{\eqref{def_psi}}, and the scaling of $x$, the two values
$\alpha= - \frac1 2$ and $\alpha= -1$ yield two solutions $v^{(3)}$
and $v^{(4)}$ with asymptotics as in \reftext{\eqref{as_vj}} that are therefore
linearly independent. Note that $v^{(3)}$ and $v^{(4)}$ also meet the
decay and continuity properties of condition $(G_\infty)$ of
\reftext{Definition~\ref{def:ginftyeta}}, because $ {\left\lvert
\psi\right\rvert}_{J+2,K+2,x_0} <
\infty$, where $J \in\mathbb{N}_0$ (or $K \in\mathbb{N}_0$) can be
chosen arbitrarily
large if $x_0 > 0$ is sufficiently large.

\medskip
\noindent\textit{Step~2: Finding a particular solution for the full equation.}
Although the reasoning below for $\eta\ne0$
is similar to the construction of a particular solution in
\cite[Lemma~6.6(a)]{ggko.2014}, we nevertheless
provide all details in the situation at hand for the sake of clarity.
We apply again the scaling $x  {\left\lvert\eta\right
\rvert} \mapsto x$ to equation~\reftext{\eqref{resolvent}} and obtain
\begin{equation}\label{resolventscaled}
 {\left\lvert\eta\right\rvert}^{-1} x v + q(D_x) v -
x^2 r(D_x) v + x^4v = f \quad\mbox
{for} \quad x > 0.
\end{equation}
By expressing equation~\reftext{\eqref{resolventscaled}} in terms of $\tfrac
{\mathrm{d}
}{\mathrm{d} x}$-derivatives rather than $D_x$-derivatives, after
which we
divide by $x^4$, we can rewrite it in the following form
\begin{equation}\label{resolvent2}
\left(\tfrac{\mathrm{d}^2}{\mathrm{d} x^2} - 1\right)^2 v = x^{-4}
f- x^{-1} P\left
(x^{-1},\tfrac{\mathrm{d}}{\mathrm{d} x},\eta\right) v,
\end{equation}
where $P(a,\xi,\eta)$ is a third-order polynomial in $(a,\xi)$ given by
\begin{equation}\label{defpetaab}
P(a,\xi,\eta) = 5 \xi^3 + 3 a \xi^2 - \frac9 8 a^2 \xi- 5 \xi+
 {\left\lvert\eta\right\rvert}^{-1} a^2 - a.
\end{equation}
We now use the truncation function $\chi_n(x) = \chi_1\left(\tfrac
{x}{n}\right)$, where $\chi_1$ is a smooth positive function $\chi_1 = 1$ on $[2,\infty)$ and $\chi_1 = 0$ on $\left(-\infty
,1\right]$. Then we consider the equation
\begin{equation}\label{resolvent3}
\left(\tfrac{\mathrm{d}^2}{\mathrm{d} x^2} - 1\right)^2 v = \chi
_n \left(x^{-4} f -
x^{-1} P\left(x^{-1},\tfrac{\mathrm{d}}{\mathrm{d} x},\eta\right)
v \right).
\end{equation}
We will find a solution for truncated data with $n\gg1$ by a
contraction principle, and then we use unique continuation to extend it
to a solution of \reftext{\eqref{resolvent}} on the whole interval $(0,\infty)$.
Note that the operator on the right-hand side of \reftext{\eqref{resolvent3}} is
translation-invariant in $x$, and thus in order to find a particular
solution to \reftext{\eqref{resolvent3}}, we will apply a fixed-point argument,
where we invert the operator $\left(\tfrac{\mathrm{d}^2}{\mathrm{d}
x^2} - 1\right)^2$
by convolving with the associated fundamental solution $g$, defined as
the unique solution to the equation
\begin{equation*}\label{fundamsol}
\left(\tfrac{\mathrm{d}^2}{\mathrm{d} x^2} - 1\right)^2 g = \delta_0 \quad
\mbox{subject to} \quad
\lim_{x\to\pm\infty} g(x)=0.
\end{equation*}
By classical ODE methods, (e.g.~by Fourier transforming the above equation) we find
\begin{equation*}\label{fundamsol1}
g(x) =
\begin{cases} \frac{1}{4}e^{-x} + \frac1{4}xe^{-x}&\text{ for }x > 0.\\
\frac1{4}e^{x} - \frac1{4}xe^{x}&\text{ for }x<0.
\end{cases}
\end{equation*}
Using the convolution with $g$, we desire to solve the fixed-point equation
\begin{equation}\label{fixed_xgg1_par}
{\mathcal T}_n[v] = v \quad\mbox{with} \quad{\mathcal T}_n[v]:= g *
\left(\chi_n
x^{-4}f\right) - g * \left(\chi_n x^{-1} P\left(x^{-1},\tfrac
{\mathrm{d}}{\mathrm{d}
x},\eta\right) v \right),
\end{equation}
which is equivalent to \reftext{\eqref{resolvent3}}. In order to obtain a contraction, note that for $\nu\in(0,1)$ and for
$j=1,\ldots,4$, $e^{\nu {\left\lvert x \right\rvert
}}\tfrac{\mathrm{d}^j}{\mathrm{d} x^j}g(x)$ is
exponentially decaying as $x\to\pm\infty$ and absolutely integrable.
Based on this, for $\nu\in(0,1)$ and for $N\in\mathbb N$ we define the
norms $ {\left\lvert\cdot\right\rvert}_{N,\nu,\infty
}$ and the space $X_\nu$ as follows:
\begin{equation*}
 {\left\lvert v \right\rvert}_{N,\nu,\infty} := \sup
_{x\in\mathbb{R}} \max_{0\le j\le N}  {\left\lvert
e^{\nu x} \tfrac{\mathrm{d}^j v}{\mathrm{d} x^j}(x) \right\rvert}
\quad\mbox{and} \quad X_\nu:=
\left\{v\in C^4(\mathbb{R}): \,  {\left\lvert v \right
\rvert}_{4,\nu,\infty} < \infty\right\}.
\end{equation*}
Then using the decay of $e^{\nu|x|}\partial_x^jg(x)$ and the property
that $e^{\nu\cdot}(f*g) = (e^{\nu\cdot} f)*(e^{\nu\cdot}g)$, we find 
\begin{align*}
&\sup_{x\in\mathbb R}\left\lvert e^{\nu x}
\tfrac{\mathrm{d}^j}{\mathrm{d} x^j} \left(g * \left(\chi_n
x^{-1} P\left(x^{-1},\tfrac{\mathrm{d}}{\mathrm{d} x},\eta\right
)v\right)\right)\right
\rvert \\
&\quad = \sup_{x\in\mathbb{R}} \left\lvert\left(e^{\nu
x}\left(\tfrac{\mathrm{d}}{\mathrm{d} x}\right)^{\min\{4,j\}}
g\right) * \left(e^{\nu x} \left(\tfrac{\mathrm{d}}{\mathrm{d}
x}\right)^{\max\{0,j-4\}} \left(\chi_n x^{-1} P\left
(x^{-1},\tfrac{\mathrm{d}}{\mathrm{d} x},\eta\right)v\right
)\right) \right\rvert \\
&\quad \lesssim \sup_{x\in\mathbb{R}}\left\lvert e^{\nu x}
\left(\tfrac{\mathrm{d}}{\mathrm{d} x}\right)^{\max\{0,j-4\}}
\left(\chi_n x^{-1} P\left(x^{-1},\tfrac{\mathrm{d}}{\mathrm{d}
x},\eta\right)v\right)\right\rvert,
\end{align*}
which with help of \eqref{defpetaab} gives the first bound
\begin{align}\label{ineqbig}
\sup_{x\in\mathbb R} {\left\lvert e^{\nu x} \tfrac
{\mathrm{d}^j}{\mathrm{d} x^j} \left(g * \left(\chi_n x^{-1}
P\left(x^{-1},\tfrac{\mathrm{d}}{\mathrm{d} x},\eta\right
)v\right)\right) \right\rvert} \le C_2 \frac{1 +
{\left\lvert\eta\right\rvert}^{-1}}{n}\sup_{x\ge n} \max
_{1\le k\le\max\{j-1,3\}}  {\left\lvert e^{\nu x} \tfrac
{\mathrm{d}^k v}{\mathrm{d} x^k} \right\rvert}
\end{align}
for a constant $C_2 = C_2(j)$.
Along the same lines, we also get the
following estimate for the first term in \reftext{\eqref{resolvent3}}:
\begin{equation}\label{ineqsmall}
\sup_{x\in\mathbb{R}}  {\left\lvert e^{\nu x} \tfrac
{\mathrm{d}^j}{\mathrm{d} x^j} \left(g * \left(\chi_n x^{-4}
f\right)\right) \right\rvert} \le C_1 \sup_{x\in\mathbb{R}}
 {\left\lvert e^{\nu x} \left(\tfrac{\mathrm
{d}}{\mathrm{d} x}\right)^{\max\{0,j-4\}} \left(\chi_nf\right)
\right\rvert}
\end{equation}
for a constant $C_1 = C_1(j)$. The above inequalities \reftext{\eqref{ineqbig}} and \reftext{\eqref{ineqsmall}} imply
\begin{subequations}\label{normfp}
\begin{equation}\label{normfp_self}
 {\left\lvert{\mathcal T}_n[v] \right\rvert}_{j,\nu
,\infty} \le C_1  {\left\lvert\chi_n f \right\rvert
}_{j-4,\nu,\infty
} + C_2 \frac{1 +  {\left\lvert\eta\right\rvert
}^{-1}}{n}  {\left\lvert v \right\rvert}_{j,\nu,\infty}
\quad
\mbox{for} \quad j \ge4
\end{equation}
and
\begin{equation}
 {\left\lvert{\mathcal T}_n[v] - {\mathcal T}_n[w] \right
\rvert}_{4,\nu,\infty} \le C_2 \frac{1 +  {\left\lvert
\eta\right\rvert}^{-1}}{n}  {\left\lvert v-w \right
\rvert}_{4,\nu,\infty}.
\end{equation}
\end{subequations}
The bounds \reftext{\eqref{normfp}} imply that for $C_2 \tfrac{1 +
 {\left\lvert\eta\right\rvert}^{-1}}{n} < 1$, the mapping ${\mathcal
T}_n$ is a contraction
on the space $X_\nu$. The condition $C_2 \tfrac{1 +
 {\left\lvert\eta\right\rvert}^{-1}}{n} < 1$ holds for large enough $n$ provided $\eta\in
{\mathcal J}\Subset\mathbb{R}\setminus\{0\}$, and in this case we have existence of a unique fixed point to \reftext{\eqref{resolvent3}} for any $\nu\in(0,1)$. By undoing the scaling
$x \mapsto {\left\lvert\eta\right\rvert} x$ in
particular we find the decay bounds in (ii) of \reftext{Definition~\ref{def:ginftyeta}} for up to
four derivatives. In order to
obtain bounds on more than four
derivatives, we note that
\begin{eqnarray*}
 {\left\lvert v \right\rvert}_{N,\nu,\infty} &\le&
 {\left\lvert g*(\chi_n f) \right\rvert}_{N,\nu,\infty
} +
 {\left\lvert g*\left(\chi_n x^{-1}P_\eta\left
(x^{-1},\tfrac{\mathrm{d}}{\mathrm{d} x}\right)v\right) \right
\rvert}_{N,\nu,\infty} \\
&\lesssim_N&  {\left\lvert\chi_n f \right\rvert
}_{N-4,\nu,\infty} + \frac1 n  {\left\lvert v \right
\rvert}_{N-1,\nu,\infty},
\end{eqnarray*}
which can be iterated and yields estimates for any order of
derivatives. As the above fixed-point argument also yields uniqueness
of solutions, we find that our so-described solution to \reftext{\eqref{resolvent3}} is the same for all $n$, can be proved to exist provided
that $f$ satisfies $(G_\infty)$, can be extended to all $x > 0$, and
satisfies property (ii) of \reftext{Definition~\ref{def:ginftyeta}}. Hence, it remains
to prove continuity in $\eta$.

For the fixed point $v$ we have from \reftext{\eqref{fixed_xgg1_par}}
\begin{eqnarray}\nonumber
 {\left\lvert{\mathcal T}_n[v(\cdot,\eta_1)] - {\mathcal
T}_n[v(\cdot,\eta_2)] \right\rvert}_{j,\nu}&\le&
C_1  {\left\lvert f(\cdot,\eta_1) - f(\cdot,\eta_2)
\right\rvert}_{j-4,\nu}\\
&&+ C_2 \frac{1 +
 {\left\lvert\eta_1 \right\rvert}^{-1}}{n}
{\left\lvert v(\cdot,\eta_1) - v(\cdot,\eta_2) \right\rvert
}_{j,\nu
}\nonumber\\
&& + C_2 \frac{ {\left\lvert\eta_1 - \eta_2 \right
\rvert}}{n  {\left\lvert\eta_1 \right\rvert}
 {\left\lvert\eta_2 \right\rvert}}  {\left
\lvert v(\cdot,\eta_2) \right\rvert}_{j,\nu},\label{continuity_xgg1_par}
\end{eqnarray}
where $j \ge4$. Therefore, under the assumption $C_2 \frac{1 +
 {\left\lvert\eta\right\rvert}^{-1}}{n} < 1$ for $\eta
\in{\mathcal J}\Subset\mathbb{R}\setminus\{0\}$, we
have uniform boundedness of $ {\left\lvert v(\cdot,\eta
_2) \right\rvert}_{4,\nu}$ from \reftext{\eqref{normfp_self}} and thus continuity in ${\mathcal J}$ because of \reftext{\eqref{continuity_xgg1_par}}. The continuity statement of \reftext{Definition~\ref{def:ginftyeta}}(iii) follows by standard ODE theory,
since the obtained continuity in $\eta$ propagates from $x \ge2 n$ to
all $x > 0$.
\end{proof}

\paragraph*{Proof of existence and uniqueness of solutions to the
resolvent equation}
As already noted, the uniqueness result (\reftext{Lemma~\ref{lem:res_unique}}) is
not difficult to obtain, so we only sketch the arguments.
\begin{proof}[Proof sketch of \reftext{Lemma~\ref{lem:res_unique}} (uniqueness)]
It suffices to
test the homogeneous version of the resolvent equation \reftext{\eqref{resolvent}} with a cut-off $\chi_n^2 v$ in the inner product $(\cdot
,\cdot)_{-\delta}$, where $\chi_n(x) := \chi\left(\frac{\log
x}{n}\right
)$ with $\chi\in C^\infty((-\infty,\infty))$,
$\chi(s) = 1$ for $s \in (-1,1)$,
and $\chi(s) = 0$ for $s \in (-\infty,-2) \cup (2,\infty)$.
This is analogous to the treatment of the time evolution equation \reftext{\eqref{test_basic}} in \S \ref{sec:basic_weak} except for the fact that
now the
cut-off $\chi_n$ has to be commuted with the differential operators
$q(D_x)$ and $r(D_x)$. Note that the commutators $\left[q(D_x),\chi
_n\right]$ and $\left[r(D_x),\chi_n\right]$ exclusively consist of
addends in which at least one $D_x = \partial_s$ (where $s := \log x$)
acts on $\chi_n$, giving a pre-factor of order $O\left(\frac1 n\right)$.
Otherwise following the computations leading from \reftext{\eqref{test_basic}} to \reftext{\eqref{lin_est1}}, we obtain
\begin{equation*}
 {\left\lvert\chi_n v \right\rvert}_{-\delta-\frac
12}^2 + \sum_{\ell=0}^2 \eta^{2\ell}
 {\left\lvert\chi_n v \right\rvert}_{\ell,-\delta
-\ell}^2 \le O\left(\tfrac1 n\right).
\end{equation*}
Now taking $n \to\infty$, we conclude that $v = 0$ in $(0,\infty)$ as desired.
\end{proof}

The technical \reftext{Lemmata~\ref{lem:fixedpointbar} and~\ref{lem:xgg1}}
provide all ingredients for proving existence to the resolvent equation \reftext{\eqref{resolvent}} as stated in \reftext{Proposition~\ref{prop:resolvent}}. This
is essentially the same as the arguments leading to
\cite[Proposition~6.3]{ggko.2014}, except for the
proof of continuity in $\eta$.
\begin{proof}[Proof of \reftext{Proposition~\ref{prop:resolvent}}]
We use the notation of \reftext{Lemmata~\ref{lem:fixedpointbar} and~\ref{lem:xgg1}}, i.e., $v^{(0)} + a_1 v^{(1)} + a_2 v^{(2)}$ with $a_1, a_2
\in\mathbb{R}$ and $v^{(\infty)} + a_3 v^{(3)} + a_4 v^{(4)}$ with
$a_3, a_4
\in\mathbb{R}$ define two-parameter solution families to \reftext{\eqref{resolvent}}
satisfying the decay conditions $(G_0)$ (cf.~\reftext{Definition~\ref{def:g0eta}}) and $(G_\infty)$ (cf.~\reftext{Definition~\ref{def:ginftyeta}}),
respectively. We first show that the family $\left
(v^{(1)},v^{(2)},v^{(3)},v^{(4)}\right)$ is linearly independent.
Afterwards we show that we can construct a $v \in C^\infty\left
((0,\infty)\right)$ solving the resolvent equation \reftext{\eqref{resolvent}}
and meeting $(G_0)$ and $(G_\infty)$. Note that $v$ is unique, since
uniqueness already holds for a larger class of functions
(cf.~\reftext{Lemma~\ref{lem:res_unique}}). Continuity in $\eta$ will be proved at the end.

\medskip
\noindent\textit{Step~1: linear independence.}
Suppose that we have $(a_1,a_2,a_3,a_4) \in\mathbb{R}^4$ with
\begin{equation*}
a_1 v^{(1)} + a_2 v^{(2)} - a_3 v^{(3)} - a_4 v^{(4)}
= 0 \quad \mbox{in} \quad (0,\infty).
\end{equation*}
Then we recognize that $v := a_1 v^{(1)} + a_2 v^{(2)} = a_3 v^{(3)} + a_4
v^{(4)}$ satisfies $(G_0)$ and $(G_{\infty})$ and is a solution to the
resolvent equation \reftext{\eqref{resolvent}} with $f = 0$ in $(0,\infty)$. \reftext{Lemma~\ref{lem:res_unique}} gives $v = 0$, i.e.,
$a_1 v^{(1)} + a_2 v^{(2)} = 0$ and $a_3 v^{(3)} + a_4 v^{(4)} = 0$ in $(0,\infty)$ and
linear independence of the tuples $\left(v^{(1)}, v^{(2)}\right)$ and
$\left(v^{(3)}, v^{(4)}\right)$ implies $a_1 = a_2 = a_3 = a_4 = 0$.

\medskip
\noindent\textit{Step~2: existence.}
For given right-hand side $f \in C^\infty((0,\infty))$ satisfying
$\left
(G_0\right)$ with
$\left(\bar f, \partial_{x_2} \bar f\right)(0,\allowbreak0) = (0,0)$ and $\left
(G_\infty\right)$, the function $v^{(\infty)} - v^{(0)}$ defines a
smooth solution to the homogeneous version of the resolvent equation \reftext{\eqref{resolvent}}. Since $\left(v^{(1)},v^{(2)},v^{(3)},v^{(4)}\right)$
is a fundamental system to the left-hand side of \reftext{\eqref{resolvent}},
there exist $\left(a_1,a_2,a_3,a_4\right) \in\mathbb{R}^4$ such that
\begin{equation}\label{sys_vj}
v^{(\infty)} - v^{(0)} = a_1 v^{(1)} + a_2 v^{(2)} - a_3 v^{(3)} - a_4 v^{(4)} \quad \mbox{in} \quad (0,\infty).
\end{equation}
Now define $v := v^{(0)} + a_1 v^{(1)} + a_2 v^{(2)} = v^{(\infty)} +
a_3 v^{(3)} + a_4 v^{(4)}$. Then by construction $v$ satisfies $\left
(G_0\right)$ and $\left(G_\infty\right)$, and is a solution to
the inhomogeneous resolvent equation \reftext{\eqref{resolvent}}.

\medskip
\noindent\textit{Step~3: Continuity in $\eta$.}
We know from \reftext{Lemma~\ref{lem:fixedpointbar}} that the analytic extensions
of $v^{(0)}$, $v^{(1)}$, and $v^{(2)}$ depend continuously on $\eta\in
{\mathcal J}\Subset\mathbb{R}\setminus\{0\}$ in a small ${\mathcal J}$-dependent
neighborhood of the origin.
In particular $\frac{\mathrm{d}^k v^ {(j)}}{\mathrm{d} x^k}$ for $j
= 0, 1, 2$ and $k
\ge0$ depend continuously on $\eta$ at a given sufficiently small $x >
0$. By standard ODE theory this implies that the extensions of $\frac
{\mathrm{d}
^k v^{(j)}}{\mathrm{d} x^k}$ for $j = 0, 1, 2$ and $k \ge0$ to all $x
> 0$
depend continuously on $\eta$. \reftext{Lemma~\ref{lem:xgg1}} further yields that
$\frac{\mathrm{d}^k v^{(\infty)}}{\mathrm{d} x^k}$, $\frac{\mathrm
{d}^k v^{(3)}}{\mathrm{d} x^k}$, and
$\frac{\mathrm{d}^k v^{(4)}}{\mathrm{d} x^k}$ for $k \ge0$ depend
continuously on
$\eta\in\mathbb{R}\setminus\{0\}$. Hence, it suffices to show that the
coefficients $a_1$, $a_2$, $a_3$, and $a_4$ of \reftext{\eqref{sys_vj}} are
continuous functions of $\eta\in\mathbb{R}\setminus\{0\}$.

Note that differentiation of \reftext{\eqref{sys_vj}} yields
\begin{equation*}
\tfrac{\mathrm{d}^k v^{(\infty)}}{\mathrm{d} x^k} - \tfrac{\mathrm
{d}^k v^{(0)}}{\mathrm{d} x^k} = a_1
\tfrac{\mathrm{d}^k v^{(1)}}{\mathrm{d} x^k} + a_2 \tfrac{\mathrm
{d}^k v^{(2)}}{\mathrm{d} x^k} - a_3
\tfrac{\mathrm{d}^k v^{(3)}}{\mathrm{d} x^k} - a_4 \tfrac{\mathrm
{d}^k v^{(4)}}{\mathrm{d} x^k} \quad
\mbox{for} \quad k = 0, 1, 2, 3.
\end{equation*}
Then linear independence of $\left
(v^{(1)},v^{(2)},v^{(3)},v^{(4)}\right
)$ implies by Liouville's formula that the Wronskian
\begin{equation*}
\det
\begin{pmatrix} v^{(1)} & v^{(2)} & v^{(3)} & v^{(4)} \\ \tfrac
{\mathrm{d}
v^{(1)}}{\mathrm{d} x} & \tfrac{\mathrm{d} v^{(2)}}{\mathrm{d} x} &
\tfrac{\mathrm{d} v^{(3)}}{\mathrm{d} x} &
\tfrac{\mathrm{d} v^{(4)}}{\mathrm{d} x} \\\noalign{\vspace{3pt}} \tfrac{\mathrm{d}^2
v^{(1)}}{\mathrm{d} x^2} & \tfrac{\mathrm{d}^2
v^{(2)}}{\mathrm{d} x^2} & \tfrac{\mathrm{d}^2 v^{(3)}}{\mathrm{d}
x^2} & \tfrac{\mathrm{d}^2
v^{(4)}}{\mathrm{d} x^2} \\\noalign{\vspace{3pt}} \tfrac{\mathrm{d}^3 v^{(1)}}{\mathrm{d}
x^3} & \tfrac{\mathrm{d}^3
v^{(2)}}{\mathrm{d} x^3} & \tfrac{\mathrm{d}^3 v^{(3)}}{\mathrm{d}
x^3} & \tfrac{\mathrm{d}^3
v^{(4)}}{\mathrm{d} x^3}
\end{pmatrix}
\end{equation*}
is non-vanishing for all $x > 0$, so that the continuous dependence of
$a_1$, $a_2$, $a_3$, and $a_4$ on $\eta\in\mathbb{R}\setminus\{0\}$ follows
by Cramer's rule.
\end{proof}
%

\subsubsection{Time discretization for the linear equation}\label{sec:time_discr}

In this section, we prove existence, uniqueness, and maximal regularity
for the linear equation \reftext{\eqref{lin_cauchy}}
(cf.~\cite[Proposition~7.6]{ggko.2014} for the equivalent result for the $1+1$-dimensional
counterpart of the free boundary problem \reftext{\eqref{tfe_free}}):
\begin{proposition}\label{prop:maxregevolution}
Suppose $T \in (0,\infty]$, $I := [0,T) \subseteq[0,\infty)$, $k, \tilde k, \check k, \breve k
\in\mathbb{N}$ with $k\ge \max\left\{\tilde k+4,\check k + 6, \breve k +5\right\}$, $\tilde k\ge3$, $\check k \ge4$, $\breve k \ge
4$, and
$\delta\in\left(0,\tfrac1{10}\right)$. Then for every locally
integrable $f:I\times(0,\infty)\times\mathbb R\to\mathbb R$ such that
${\skliaustask f\skliaustasd}_{\mathrm{rhs}}<\infty$ and for every locally integrable $v^{(0)}:(0,\infty)\times\mathbb R\to\mathbb R$ such that ${\left\lVert v^{(0)} \right\rVert}_{\mathrm{init}}<\infty$, there exists
exactly one solution
$v:I\times(0,\infty)\times\mathbb R\to\mathbb R$ of the linear
degenerate-parabolic problem \reftext{\eqref{eq_lin}} that is locally integrable
with ${\skliaustask v\skliaustasd}_{\mathrm{sol}}<\infty$ and
$v_{|t=0}=v^{(0)}$. This solution obeys the maximal-regularity estimate
\begin{equation}\label{mr_interval}
{\skliaustask v\skliaustasd}_{\mathrm{sol}} \lesssim_{k,\tilde k,
\check k, \breve k, \delta}
 {\left\lVert v^{(0)} \right\rVert}_{\mathrm{init}} +
{\skliaustask f\skliaustasd}_{\mathrm{rhs}}.
\end{equation}
Moreover, $v$ is continuous, the function $q(D_x)v + D_y^2r(D_x)v +
D_y^4v$ is continuous and defined classically, and the initial
condition $v_{|t=0}=v^{(0)}$ holds in the classical sense.
\end{proposition}

The last statement of the proposition follows from ${\skliaustask
v\skliaustasd}_{\mathrm
{sol}}<\infty$ via \reftext{Lemmata~\ref{lem:trace_estimate} and \ref{lem:bc0_bounds}}, and the approximation results from \reftext{Corollary~\ref{coroll:contnorm}}, and is independent from the rest of the proof.

The proof of the main statement of \reftext{Proposition~\ref{prop:maxregevolution}} will proceed as follows:
\begin{enumerate}[(i)]
\item[(i)] On each discretization interval
$[(m-1)\delta t, m \delta t)$, where $\delta t := \frac{T}{M}$ with $M
\in\mathbb{N}$ and $m = 1, \ldots, M$, we solve the resolvent
problem \reftext{\eqref{resolvent}} via \reftext{Proposition~\ref{prop:resolvent}}.
\item[(ii)] We derive estimates dealing with one
discrete time step $t \mapsto t + \delta t$ for $\delta t = \frac T M$,
where $M \in\mathbb{N}$ is arbitrary, which are discretized versions
of the
estimates that have been heuristically derived in \S \ref{sec:lin_heur}
(assuming that sufficiently regular solutions to \reftext{\eqref{lin_cauchy}}
exist, see \reftext{Lemma~\ref{lem:discrete_step}}). We also use the simple trace
estimate \reftext{\eqref{traceest}} of \reftext{Lemma~\ref{lem:trace_estimate}} to show
that the operation of fixing $v$ at $t=0$ is robust with respect to our
approximation.
\item[(iii)] We perform a piecewise constant in time
approximation of the data $f^{(m\delta t)}$ and a piecewise affine in
time interpolation of the solutions $v^{(m\delta t)}$ in order to have
objects defined in continuous time, while retaining the estimates
(ii).
\item[(iv)] We treat the discretized problem on a
non-vanishing interval $I = [0,T)$ by summing the estimates obtained in
(iii).
\item[(v)] The solution in the continuum is
obtained from (iv) in the limit $\delta t\to0$,
using (ii) coupled with a compactness argument. By
(i) and (ii) this limit also
fulfills the linear equation \reftext{\eqref{eq_lin}} and it recovers the desired
initial condition $v|_{t=0}$. Via (ii) and
(iii) we finally obtain the bounds stated in
\reftext{Proposition~\ref{prop:maxregevolution}}.
\end{enumerate}
Consider the discretization \reftext{\eqref{discrete_main}} of the linear
equation \reftext{\eqref{eq_lin}} for a time step $0 \mapsto\delta t$ from \S \ref{sec:resolvent}, i.e.,
%
%
\begin{equation}\label{resolvent_disc}
x\frac{v^{(\delta t)} - v^{(0)}}{\delta t} + q(D_x) v^{(\delta t)} -
 {\left\lvert\eta\right\rvert}^2 x^2r(D_x) v^{(\delta
t)} +  {\left\lvert\eta\right\rvert}^4 x^4 v^{(\delta
t)} = f^{(\delta t)},
\end{equation}
where $f^{(\delta t)} := \frac{1}{\delta t} \int_0^{\delta t} f\left
(t^\prime\right) \mathrm{d} t^\prime$. Up to approximating
$v^{(0)}$ and $f^{(\delta
t)}$ as in \reftext{Lemmata~\ref{lem:approx_init} and~\ref{lem:approx_rhs}}, respectively, we
can assume that $v^{(0)}$ and $f^{(\delta t)}$
satisfy the following assumptions:
\begin{enumerate}[$(\mathcal D1)$]
\item[$(\mathcal D1)$] For all $\eta\in\mathbb{R}\setminus\{0\}$ we have
$v^{(0)}(\cdot
,\eta)\in C^\infty((0,\infty))$ such that for all $j \ge0$ the
function $\partial_x^j v^{(0)}$ is continuous in $\left\{(x,\eta):
\, x
> 0, \, \eta\ne0\right\}$, and $v^{(0)}$ satisfies $(G_0)$ and
$(G_{\infty})$ (see \reftext{Lemma~\ref{lem:approx_init}}).
\item[$(\mathcal D2)$] For all $\eta\in\mathbb{R}\setminus\{0\}$ we have
$f^{(\delta
t)}(\cdot,\eta) \in C^\infty((0,\infty))$ such that for all $j \ge0$
the function $\partial_x^j f^{(\delta t)}$ is continuous in $\left\{
(x,\eta): \, x > 0, \, \eta\ne0\right\}$, and $f^{(\delta t)}$
satisfies $(G_0)$ with $\left(\bar f^{(\delta t)}, \partial_{x_2}
\bar
f^{(\delta t)}\right)(0,0)=(0,0)$ and $(G_{\infty})$ (see \reftext{Lemma~\ref{lem:approx_rhs}}).
\end{enumerate}
The following estimates are the discrete analogues of estimates~\reftext{\eqref{estimate1}}, \reftext{\eqref{estimate2-}}, \reftext{\eqref{estimate3-}}, and \reftext{\eqref{strong_higher_3}}:
%
\begin{lemma}\label{lem:discrete_step}
Suppose we are given $v^{(0)}$ and $f^{(\delta t)}$ satisfying
$(\mathcal D1)$ and $(\mathcal D2)$, respectively. Then there exists a solution
$v^{(\delta t)}$ to \reftext{\eqref{resolvent_disc}} satisfying $(G_0)$ and
$(G_\infty)$, and for $k \ge 2$, $\tilde k \ge 2$, $\check k \ge 2$,
and $\delta\in\left(0,\frac{1}{10}\right)$ fulfilling the estimates
%
%
\begin{subequations}\label{estimate123_discr}
%
%
\begin{align}\label{estimate1_discr}
\begin{aligned}
& \frac{1}{\delta t} \left(\sum_{\ell= 0}^k C_\ell\, \eta^{2 \ell}
 {\left\lvert D_x^{k-\ell} v^{(\delta t)} \right\rvert
}_{-\delta-\frac1 2-\ell}^2 +  {\left\lvert v^{(\delta
t)} \right\rvert}_{-\delta- \frac1 2}^2 - \sum_{\ell= 0}^k C_\ell
\,
\eta^{2 \ell}  {\left\lvert D_x^{k-\ell} v^{(0)} \right
\rvert}_{-\delta-\frac1 2-\ell}^2 -
 {\left\lvert v^{(0)} \right\rvert}_{-\delta- \frac12}^2\right) \\
& \quad+ \sum_{\ell=0}^{k+2} \eta^{2 \ell}  {\left
\lvert v^{(\delta t)} \right\rvert}_{k+2-\ell,-\delta-\ell}^2
\lesssim_{k,\delta} \sum_{\ell=
0}^{k-2} \eta^{2 \ell}  {\left\lvert f^{(\delta t)}
\right\rvert}_{k-2-\ell,-\delta-\ell}^2,
\end{aligned}
\end{align}
with constants $C_\ell> 0$ and $\delta\in\left(0,\tfrac
{1}{10}\right)$;
%
%
\begin{equation}\label{estimate2_discr}
\begin{aligned}
& \frac{1}{\delta t} \left(\sum_{\ell= 0}^{\tilde k} \tilde C_\ell
\,
\eta^{2 \ell}  {\left\lvert D_x^{\tilde k + 1 - \ell}
v^{(\delta t)} \right\rvert}_{\tilde
\alpha-\frac1 2-\ell}^2 +  {\left\lvert D_x v^{(\delta
t)} \right\rvert}_{\tilde\alpha-
\frac1 2}^2 - \sum_{\ell= 0}^{\tilde k} \tilde C_\ell\, \eta^{2
\ell
}  {\left\lvert D_x^{\tilde k + 1 - \ell} v^{(0)} \right
\rvert}_{\tilde\alpha-\frac12-\ell}^2 -  {\left\lvert D_x v^{(0)} \right\rvert
}_{\tilde\alpha- \frac1 2}^2\right) \\
& + \sum_{\ell=0}^{\tilde k+2} \eta^{2 \ell}  {\left
\lvert D_x v^{(\delta t)} \right\rvert}_{\tilde k + 2 - \ell,\tilde
\alpha-\ell}^2 \lesssim_{\tilde
k,\tilde\alpha} \sum_{\ell= 0}^{\tilde k-2} \eta^{2 \ell}
 {\left\lvert(D_x - 1) f^{(\delta t)} \right\rvert
}_{\tilde k-2-\ell,\tilde\alpha-\ell}^2 + \eta^2
 {\left\lvert v^{(\delta t)} \right\rvert}_{\tilde\alpha
-1}^2 + \eta^4  {\left\lvert v^{(\delta t)} \right\rvert
}_{\tilde\alpha-2}^2
\end{aligned}
\end{equation}
with constants $\tilde C_\ell> 0$ and $\tilde\alpha\in[0,1)$;
%
%
\begin{align}\label{estimate3_discr}
& \frac{1}{\delta t} \left(\sum_{\ell= 0}^{\check k} \check C_\ell
\,
\eta^{2 \ell}  {\left\lvert D_x^{\check k - \ell}\tilde
q(D_x)D_x v^{(\delta t)} \right\rvert}_{\check\alpha-\frac12-\ell}^2 +  {\left\lvert\tilde q(D_x)D_x v^{(\delta t)}
\right\rvert}_{\check\alpha- \frac1 2}^2\right.\nonumber\\
&\left.-\sum_{\ell= 0}^{\check k} \check C_\ell\, \eta^{2 \ell}
 {\left\lvert D_x^{\check k - \ell}\tilde q(D_x)D_x
v^{(0)} \right\rvert}_{\check\alpha-\frac12-\ell}^2 -  {\left\lvert\tilde q(D_x)D_x v^{(0)} \right
\rvert}_{\check\alpha- \frac12}^2\right) \nonumber\\
&+ \sum_{\ell=0}^{\check k+2} \eta^{2 \ell}
{\left\lvert\tilde q(D_x) D_x v^{(\delta t)} \right\rvert}_{\check
k + 2 - \ell,\check\alpha-\ell}^2 \nonumber\\
& \quad\lesssim_{\check k,\check\alpha} \sum_{\ell= 0}^{\check k-2}
\eta^{2 \ell}  {\left\lvert\tilde q(D_x-1)(D_x - 1)
f^{(\delta t)} \right\rvert}_{\check
k-2-\ell,\check\alpha-\ell}^2 + \eta^2  {\left\lvert
D_x v^{(\delta t)} \right\rvert}_{\check\alpha-1}^2 + \eta^4
 {\left\lvert v^{(\delta t)} \right\rvert}_{\check\alpha-2}^2
\end{align}
with constants $\check C_\ell> 0$ and $\check\alpha\in(1,2)$;
%
%
\begin{equation}\label{estimate4_discr}
\begin{aligned}
& \frac{1}{\delta t} \left(
\sum_{\ell= 0}^{\check k} \breve C_\ell\eta^{2\ell}
{\left\lvert D_x^{\check k - \ell} (D_x-3) (D_x-2) \tilde q(D_x) D_x
v^{(\delta t)} \right\rvert}_{\delta+ \frac3 2 - \ell}^2 +
 {\left\lvert(D_x-3) (D_x-2) \tilde q(D_x) D_x v^{(\delta
t)} \right\rvert}_{\delta+ \frac3 2}^2\right.\\
&\left.-\sum_{\ell= 0}^{\check k} \breve C_\ell\eta^{2\ell}
 {\left\lvert D_x^{\check k - \ell} (D_x-3) (D_x-2) \tilde
q(D_x) D_x v^{(0)} \right\rvert}_{\delta+ \frac3 2 - \ell}^2 -
 {\left\lvert(D_x-3) (D_x-2) \tilde q(D_x) D_x v^{(0)}
\right\rvert}_{\delta+ \frac3 2}^2\vphantom{\sum_{\ell= 0}^{\check k}}
\right) \\
& + \sum_{\ell= 0}^{\check k + 2} \eta^{2\ell}  {\left
\lvert(D_x-3) (D_x-2) \tilde q(D_x) D_x v^{(\delta t)} \right\rvert
}_{\check k + 2 - \ell,\delta+2 - \ell
}^2 \\
& \quad\lesssim_{\check k,\delta} \sum_{\ell= 0}^{\check k - 2}
\eta
^{2\ell}  {\left\lvert(D_x-4) (D_x-3) \tilde q(D_x-1)
(D_x-1) f^{(\delta t)} \right\rvert}_{\check k - 2 - \ell,\delta
+2}^2 + \eta^2  {\left\lvert\tilde q(D_x) D_x
v^{(\delta t)} \right\rvert}_{\delta
+1}^2 + \eta^4  {\left\lvert D_x v^{(\delta t)} \right
\rvert}_\delta^2
\end{aligned}
\end{equation}
with constants $\breve C_\ell > 0$ and $\delta\in\left(0,\frac
1{10}\right)$.
\end{subequations}
\end{lemma}
\begin{proof}[Proof of \reftext{Lemma~\ref{lem:discrete_step}}]
By scaling according to $x\mapsto\frac{x}{\delta t}$ and $\eta
\mapsto
\delta t \, \eta$, we need to solve the resolvent equation (cf.~\reftext{\eqref{resolvent}})
\begin{align}\label{resolvg}
x v^{(\delta t)} + q(D_x) v^{(\delta t)} - \eta^2 x^2 r(D_x)
v^{(\delta t)} +  \eta^4 x^4
v^{(\delta t)} = g, \quad\mbox
{where} \quad g := f^{(\delta t)} + x v^{(0)}.
\end{align}
Then $g$ also satisfies $(G_0)$ with $\left(\bar g, \partial
_{x_2}\bar
g\right)(0,0) = (0,0)$ and $(G_\infty)$. Therefore, we are in the
setting of \reftext{Proposition~\ref{prop:resolvent}}, i.e., there exists a
solution $v^{(\delta t)} \in C^\infty((0,\infty))$ to \reftext{\eqref{resolvg}}
satisfying $(G_0)$ and $(G_\infty)$ such that all derivatives
$\partial
_x^j v^{(\delta t)}$ for $j \ge0$ are continuous in $\left\{(x,\eta):
\, x > 0, \, \eta\ne0\right\}$.

The reasonings for the four cases in \reftext{\eqref{estimate123_discr}} are
alike and follow the derivations of \reftext{\eqref{estimate1}},
\reftext{\eqref{estimate2-}}, \reftext{\eqref{estimate3-}}, and
\reftext{\eqref{strong_higher_3}}, respectively. Therefore, we describe the
procedure only for the first case \reftext{\eqref{estimate1_discr}} and leave
the details in the other cases \reftext{\eqref{estimate2_discr}},
\reftext{\eqref{estimate3_discr}}, and \reftext{\eqref{estimate4_discr}}
to the reader.

Consider the un-rescaled
equation \reftext{\eqref{resolvent_disc}} and test it with $v^{(\delta t)}$
with respect to the inner product $(\cdot,\cdot )_{-\delta}$ for
$\delta\in\left(0,\tfrac1{10}\right)$:
%
%
\begin{equation}\label{test_basic-discr}
\begin{aligned}
& \frac{1}{\delta t}\left(x\left(v^{(\delta t)}-v^{(0)}\right),
v^{(\delta t)}\right)_{-\delta} + \left(q(D_x)v^{(\delta
t)},v^{(\delta
t)}\right)_{-\delta} - \eta^2\left(x^2 r(D_x)v^{(\delta
t)},v^{(\delta
t)}\right)_{-\delta} +\eta^4\left(x^4v^{(\delta t)},v^{(\delta t)}\right
)_{-\delta} \\
& \quad = \left(f^{(\delta t)}, v^{(\delta t)}\right)_{-\delta}.
\end{aligned}
\end{equation}
This equation is the discrete version of \reftext{\eqref{test_basic}}, and is
treated in the same way, except for the first term, for which rather
than \reftext{\eqref{1term}} we obtain
%
%
\begin{align}\label{1term-discr}
\left(x\left(v^{(\delta t)}-v^{(0)}\right), v^{(\delta t)}\right
)_{-\delta} =  {\left\lvert v^{(\delta t)} \right\rvert
}_{-\delta-\frac12}^2 - \left
(v^{(0)}, v^{(\delta t)}\right)_{-\delta-\frac12} \ge\frac12
 {\left\lvert v^{(\delta t)} \right\rvert}_{-\delta
-\frac12}^2 -\frac12  {\left\lvert v^{(0)} \right\rvert
}_{-\delta
-\frac12}^2.
\end{align}
The bound \reftext{\eqref{1term-discr}}, together with estimates~\reftext{\eqref{2term}}--\reftext{\eqref{est3term}}, allows to obtain the following analogue of \reftext{\eqref{lin_est1}}:
%
%
\begin{equation}\label{lin_est1-discr}
\frac{1}{\delta t}\left( {\left\lvert v^{(\delta t)}
\right\rvert}_{-\delta- \frac12}^2 -
 {\left\lvert v^{(0)} \right\rvert}_{-\delta-\frac
12}^2\right) +\sum_{\ell=0}^2\eta^{2\ell
} {\left\lvert v^{(\delta t)} \right\rvert}_{\ell,
-\delta-\ell}^2 \lesssim_\delta {\left\lvert f^{(\delta
t)} \right\rvert}_{-\delta}^2 \quad\mbox{for} \quad \delta\in\left
(0,\tfrac
1{10}\right).
\end{equation}
This can be upgraded to a strong estimate like in \S \ref{sec:basicstrongest}, by testing \reftext{\eqref{resolvent_disc}} against \reftext{\eqref{basicstrongtest}}, i.e., against $\eta^{2 \ell} \left(\cdot,(-D_x -
1 -
2 \ell+ 2 - \delta)^{k-\ell} D_x^{k-\ell} v^{(\delta t)}\right
)_{-\delta-\ell} =: \left(\cdot,S\right)_{-\delta-\ell}$. We
proceed as
for \reftext{\eqref{terms_estimate_k_deriv}} and \reftext{\eqref{terms_estimate_k_deriv_5}}, with the only difference appearing for the
treatment of the time derivative term \reftext{\eqref{terms_estimate_k_deriv_1}},
which is discretized here. Similar to the discussion leading to \reftext{\eqref{1term-discr}}, we find the time-discrete analogue of \reftext{\eqref{terms_estimate_k_deriv_1}}:
%
%
\begin{equation}\label{terms_estimate_k_deriv_1-discr}
\begin{aligned}
\frac{1}{\delta t}\left(x\left(v-v^{(0)}\right), S\right
)_{-\delta
-\ell} &=\frac{1}{\delta t}\left((xv, S)_{-\delta-\ell} - \left(x
v^{(0)},S\right)_{-\delta-\ell}\right) \\
&= \frac{1}{\delta t}\left(\eta^{2\ell} {\left
\lvert D_x^{k-\ell} v \right\rvert}_{-\delta-\ell-\frac12} -\eta
^{2\ell}\left(D_x^{k-\ell}v,D_x^{k-\ell
}v^{(0)}\right)_{-\delta-\ell-\frac12}\right) \\
&\ge\frac{1}{2}\frac{1}{\delta t}\left(\eta^{2\ell
} {\left\lvert D_x^{k-\ell}v \right\rvert}_{-\delta
-\ell-\frac12}^2-\eta^{2\ell} {\left\lvert D_x^{k-\ell
}v^{(0)} \right\rvert}_{-\delta-\ell-\frac12}^2\right).
\end{aligned}
\end{equation}
Adding the bound \reftext{\eqref{terms_estimate_k_deriv_1-discr}} to the bounds \reftext{\eqref{terms_estimate_k_deriv_2}}--\reftext{\eqref{terms_estimate_k_deriv_5}},
afterwards summing over $\ell$ and reabsorbing by using \reftext{\eqref{lin_est1-discr}} (precisely as in the steps in \S \ref{sec:basicstrongest} leading up to \reftext{\eqref{estimate1}}), we obtain the
time-discrete estimate \reftext{\eqref{estimate1_discr}}.
\end{proof}

We are now ready to prove our main result for the linear equation:
\begin{proof}[Proof of \reftext{Proposition~\ref{prop:maxregevolution}}]
Throughout the proof, estimates may depend on $k$, $\tilde k$, $\check k$,
$\breve k$, and $\delta$. By writing $[0,\infty) = \bigcup_{T \in
\mathbb{N}}
[0,T)$ and using the uniqueness part of \reftext{Proposition~\ref{prop:maxregevolution}} for $T < \infty$, we may without loss of
generality assume that $T$ is finite. As already mentioned in the
discussion preceding the proof, up to approximating as in \reftext{Lemmata~\ref{lem:approx_init} and~\ref{lem:approx_rhs}}, we may assume that
$(\mathcal D1)$ and $(\mathcal D2)$ hold for $v^{(0)}$ and $f^{(\delta t)}$,
respectively. We subdivide the proof into several parts.

\medskip
\noindent\textit{Step 1: Bounds for a single time interval.}
We first deduce the integral version of \reftext{\eqref{estimate1_discr}} which
is the discrete analogue of estimate~\reftext{\eqref{estimate1_int}}. To this
aim, for equal time steps $\delta t$ we iteratively solve for
$m=1,2,\ldots,M:=\frac{T}{\delta t}$
%
%
\begin{subequations}
%
%
\begin{equation}\label{discretestep}
x\frac{v^{(m\delta t)} - v^{((m-1)\delta t)}}{\delta t} + q(D_x) v^{(m
\delta t)} - \eta^2 x^2 r(D_x) v^{(m \delta t)} + \eta^4 x^4
v^{(m\delta t)} = f^{(m\delta t)}
\end{equation}
on the time interval $[(m-1)\delta t, m \delta t)$, where
%
%
\begin{equation}\label{deffnudeltat}
f^{(m \delta t)} := \frac{1}{\delta t}\int_{(m-1)\delta t}^{m\delta t}
f\left(t^\prime\right) \mathrm{d} t^\prime.
\end{equation}
\end{subequations}
We then apply estimate~\reftext{\eqref{estimate1_discr}} with the substitutions
%
%
\begin{equation}\label{substitutions_discr}
v^{(\delta t)} \mapsto v^{(m\delta t)}, \quad v^{(0)} \mapsto v^{((m -
1)\delta t)}, \quad f^{(0)} \mapsto f^{(m\delta t)},
\end{equation}
and we sum the resulting estimates over $m=1,\ldots, M$. The sum of the
discretized time derivatives, coming from the first term in \reftext{\eqref{estimate1_discr}}, form a telescoping sum, so that after multiplying by
$\delta t$ we find
%
%
\begin{equation}\label{estimate1_discr_int}
\begin{aligned}
&\sum_{\ell=0}^k C_\ell\eta^{2\ell}  {\left\lvert
D_x^{k-\ell}v^{(T)} \right\rvert}_{-\delta- \frac{1}{2} -\ell}^2
+  {\left\lvert v^{(T)} \right\rvert}_{-\delta
-\frac12}^2 + \delta t \sum_{m=1}^{M} \sum_{\ell= 0}^{k+2} \eta
^{2\ell
}  {\left\lvert v^{(m \delta t )} \right\rvert}_{k+2-\ell
, -\delta-\ell}^2\\
& \quad\lesssim\sum_{\ell=0}^kC_\ell\eta^{2\ell}
{\left\lvert D_x^{k-\ell}v^{(0)} \right\rvert}_{-\delta-\frac
12-\ell}^2 +  {\left\lvert v^{(0)} \right\rvert
}_{-\delta-\frac
12}^2 +\delta t \sum_{m=1}^{M}\sum_{\ell=0}^{k-2}\eta^{2\ell
} {\left\lvert f^{(m\delta t)} \right\rvert}_{k-2-\ell,
-\delta- \ell}^2.
\end{aligned}
\end{equation}
The constant in estimate~\reftext{\eqref{estimate1_discr_int}} does not depend on
$\delta t$ and $M$, and upon increasing its value, we can replace the
first two terms in \reftext{\eqref{estimate1_discr_int}} by
\begin{equation*}
\max_{0\le m\le M}\left(\sum_{\ell= 0}^kC_\ell\eta^{2\ell
} {\left\lvert D_x^{k-\ell}v^{(m\delta t)} \right\rvert
}_{-\delta-\frac12 - \ell}^2 + {\left\lvert v^{(m\delta
t)} \right\rvert}_{-\delta-\frac12}^2\right).
\end{equation*}
Multiplying with $\eta^{2 j}$ and applying the inverse Fourier
transform in $\eta$, we obtain from \reftext{\eqref{estimate1_discr_int}}
%
%
\begin{equation}\label{estimate1_discr_2}
\begin{aligned}
&\max_{0\le m\le M} {\left\lVert D_y^jv^{(m\delta t)}
\right\rVert}_{k,-\delta-1+j}^2 +
\delta t\sum_{m=1}^{M} {\left\lVert D_y^jv^{(m\delta t)}
\right\rVert}_{k+2,-\delta-\frac
12 +j}^2\\
&\quad\lesssim {\left\lVert D_y^jv^{(0)} \right\rVert
}_{k,-\delta-1 +j}^2 +\delta t\sum
_{m=1}^{M}  {\left\lVert D_y^j f^{(m \delta t)} \right
\rVert}_{k-2,-\delta-\frac12 +j}^2.
\end{aligned}
\end{equation}
In order to justify convergence of the integrals underlying inverse Fourier
transform, we note that we are working under the hypothesis that the terms on
the right-hand side of \reftext{\eqref{estimate1_discr_2}} are bounded. Due
to the local continuity in $\eta$
(cf.~\reftext{Proposition~\ref{prop:resolvent}}), which implies (global)
measurability in~$\eta$, we find due to \reftext{\eqref{estimate1_discr_int}}
also the integrals required for defining the left-hand side of
\reftext{\eqref{estimate1_discr_2}} to converge.

\medskip
\noindent\textit{Step 2: Interpolation in time.}
We define the continuum version of the left-hand side $f$ by a
piecewise constant extension
%
%
\begin{equation}\label{defpiecewiseconst}
\Phi_{M} (t,x):= \sum_{m=1}^{M}f^{(m\delta t)}(x) \mathbh
{1}_{[(m-1)\delta
t,m\delta t)}(t),
\end{equation}
and for approximating the solution $v$ itself for time steps $T =
0,\delta t, \ldots, M \delta t$, define the piecewise affine interpolant
%
%
\begin{equation}\label{affineinterpol}
\Psi_{M}(t,x):=\sum_{m=1}^{M}\left(\frac{t-(m-1)\delta t}{\delta t}
v^{(m\delta t)}(x) + \frac{m\delta t - t}{\delta t} v^{((m-1)\delta
t)}(x)\right) \mathbh{1}_{[(m-1)\delta t, m\delta t)}(t).
\end{equation}
We note that for $1\le m\le M$ and for $t \in((m-1)\delta t, m \delta
t)$, based on \reftext{\eqref{discretestep}}, \reftext{\eqref{defpiecewiseconst}}, and \reftext{\eqref{affineinterpol}}, we have 
\begin{equation}\label{derivpsi}
x\partial_t\Psi_{M}+q(D_x)v^{(m\delta t)} + D_y^2 r(D_x)v^{(m\delta
t)} + D_y^4 v^{(m\delta t)} = \Phi_M.
\end{equation}
Then we use the same reasoning as the one preceding \reftext{\eqref{estimate1_int}} in order to bound $\partial_t D_y^j \Psi_{M}$ via
equation~\reftext{\eqref{derivpsi}}. On the time interval $[(m-1)\delta t,
m\delta t)$, we have:
\begin{equation}\label{timedernorms}
 {\left\lVert\partial_tD_y^j\Psi_{M} \right\rVert
}_{k-2,-\delta-\frac32+j}^2 \lesssim
 {\left\lVert D_y^jv^{(m\delta t)} \right\rVert
}_{k+2,-\delta-\frac12+j}^2 +  {\left\lVert D_y^j \Phi_M
\right\rVert}_{k-2,-\delta-\frac12+j}^2.
\end{equation}
On each time interval $[(m-1)\delta t, m\delta t)$ we may use the
triangle inequality for \reftext{\eqref{affineinterpol}} in order to bound the
norm of $D_y^j\Psi_{M}$ by the sum of the norms of
$D_y^jv^{((m-1)\delta t)}$ and $D_y^jv^{(m\delta t)}$ and finally
integrate \reftext{\eqref{timedernorms}} on $t\in[(m-1)\delta t, m\delta t)$.
Summing these estimates over $m$ allows to pass from \reftext{\eqref{estimate1_discr_2}} to the following bound for $\Psi_{M}$:
\begin{equation}\label{intbound_psi}
{\skliaustask D_y^j\Psi_{M} \skliaustasd}_{k,-\delta-1 +j,I}^2\lesssim {\left\lVert v^{(0)} \right\rVert
}_{k,-\delta-1 +j} + \int_I {\left\lVert D_y^j\Phi_{M}
\right\rVert}_{k-2,-\delta-\frac12+j}^2 \mathrm{d} t.
\end{equation}
By the same reasoning applied to \reftext{\eqref{estimate2_discr}}, \reftext{\eqref{estimate3_discr}}, and \reftext{\eqref{estimate4_discr}}, we obtain discrete
integral analogues thereof, which are also discrete versions of \reftext{\eqref{higher_est_1}}, \reftext{\eqref{higher_est_2}}, \reftext{\eqref{higher_est_2_alt}}, and \reftext{\eqref{higher_est_3}}, respectively. These are again obtained directly
from \reftext{\eqref{higher_est_1}}, \reftext{\eqref{higher_est_2}}, \reftext{\eqref{higher_est_2_alt}}, and \reftext{\eqref{higher_est_3}} by replacing $\left
(v,v_{|t=0},f\right)$ by $\left(\Psi_M,v^{(0)}, \Phi_M\right)$. For
example the bound corresponding to \reftext{\eqref{higher_est_1}} reads
\begin{align}\nonumber
& {\skliaustask D_y^j \Psi_{M} \skliaustasd}_{\tilde k -1,-\delta-1 + j, I}^2 
+  {\skliaustask D_y^{j-1} D_x \Psi_{M}\skliaustasd}_{\tilde k,-\delta-1 + j, I}^2 \\
&\quad\lesssim {\left\lVert D_y^j v^{(0)} \right\rVert
}_{\tilde k - 1,-\delta-1 + j}^2 +
 {\left\lVert D_y^{j-1} D_x v^{(0)} \right\rVert}_{\tilde
k,-\delta-1 + j}^2 \nonumber\\
&\quad\phantom{\lesssim} + \int_I \left( {\left\lVert D_y^j \Phi_{M} \right
\rVert}_{\tilde k-3,-\delta-\frac12 +
j}^2 +  {\left\lVert D_y^{j-1} (D_x-1) \Phi_{M} \right
\rVert}_{\tilde k-2,-\delta-\frac12
+ j}^2\right) \mathrm{d} t \label{higher_est_1_discr}
\end{align}
and in the same way we can derive the analogues to \reftext{\eqref{higher_est_2}}, \reftext{\eqref{higher_est_2_alt}} and \reftext{\eqref{higher_est_3}},
which we omit for the sake of keeping the presentation concise.

\medskip
\noindent\textit{Step 3: Passage to the continuum limit and existence
of a solution.}
Let $T$ be fixed and consider the limit $M \to\infty$ with $\delta t =\frac T M$. In order to obtain that the right-hand sides in \reftext{\eqref{higher_est_1_discr}} and of the higher-order discretized versions of
estimates \reftext{\eqref{higher_est_2}}, \reftext{\eqref{higher_est_2_alt}}, and \reftext{\eqref{higher_est_3}} are bounded, we will use the bounds on $v^{(0)}$ which
follow from the assumption $ {\left\lVert v^{(0)} \right
\rVert}_{\mathrm{init}}<\infty$,
and for the remaining terms we use bounds following from ${\skliaustask
f\skliaustasd}_{\mathrm{rhs}}<\infty$ and the fact that for any polynomial operator of the form $Q(D_y)P(D_x)$ and any choice of the
$\left\lVert\cdot\right\rVert_{k,\rho}$-norm as
occurring in the definition of ${\skliaustask f \skliaustasd}_\mathrm{rhs}$, we have
\begin{equation}\label{boundedness}
\int_I {\left\lVert Q(D_y)P(D_x)\Phi_M \right
\rVert}_{k,\rho}^2 \mathrm{d} t \le\int
_I {\left\lVert Q(D_y)P(D_x) f \right\rVert
}_{k,\rho}^2 \mathrm{d} t.
\end{equation}
In order to verify \reftext{\eqref{boundedness}}, we may recall the
definition \reftext{\eqref{defpiecewiseconst}} of $\Phi_{M}$ and the definition \reftext{\eqref{deffnudeltat}} of $f^{(m\delta t)}$. Indeed, using Jensen's
inequality for interchanging the $\delta t$-interval average in \reftext{\eqref{deffnudeltat}} with the norm squared, we find
\begin{align*}
& \int_I {\left\lVert Q(D_y)P(D_x) f \right\rVert
}_{k,\rho}^2\mathrm{d} t \\
& \quad = \int_I {\left\lVert
Q(D_y)P(D_x)\left(\sum_{m=1}^M \mathbh{1}_{[(m-1)\delta
t,m\delta t)}f(t,\cdot,\cdot)\right) \right\rVert}_{k,\rho
}^2\mathrm{d} t = \sum_{m=1}^M\int_{(m-1)\delta t}^{m\delta t} {\left
\lVert Q(D_y)P(D_x) f(t,\cdot,\cdot) \right\rVert
}_{k,\rho}^2\mathrm{d} t \\
& \quad \ge \sum_{m=1}^M\int_{(m-1)\delta t}^{m\delta t} {\left
\lVert Q(D_y)P(D_x) f^{(m\delta t)} \right\rVert}_{k,\rho
}^2\mathrm{d} t = \int_I {\left\lVert Q(D_y)P(D_x) \Phi_M \right
\rVert}_{k,\rho}^2 \mathrm{d} t.
\end{align*}
As a consequence of \reftext{\eqref{boundedness}}, we have
\begin{equation}\label{bound_phi_rhs}
{\skliaustask\Phi_M\skliaustasd}_\mathrm{rhs}\le{\skliaustask
f\skliaustasd}_\mathrm{rhs}
\end{equation}
and thus the inequalities \reftext{\eqref{higher_est_1_discr}}, and similarly the
discretized analogues of \reftext{\eqref{higher_est_2}}, \reftext{\eqref{higher_est_2_alt}}
and \reftext{\eqref{higher_est_3}} continue to hold if we use $f$ rather than
$\Phi_M$ on the right-hand side.

As a result, by the same mechanism that led from \reftext{\eqref{higher_est_1}},
\reftext{\eqref{higher_est_2}}, \reftext{\eqref{higher_est_2_alt}}, and \reftext{\eqref{higher_est_3}} to the definition of the norms ${\skliaustask\cdot
\skliaustasd}_{\mathrm
{sol}},  {\left\lVert\cdot\right\rVert}_{\mathrm
{init}}^{\prime}$, and ${\skliaustask\cdot\skliaustasd}_{\mathrm
{rhs}}^{\prime}$, we can prove for $I=[0,T)$
\begin{equation*}
{\skliaustask\Psi_{M}\skliaustasd}_{\mathrm{sol}} \lesssim
 {\left\lVert v^{(0)} \right\rVert}_{\mathrm
{init}}^\prime+ {\skliaustask f\skliaustasd}_{\mathrm{rhs}}^\prime,
\end{equation*}
where the constant is independent of $M$. By using \reftext{Lemmata~\ref{lem:equivalence_init} and~\ref{lem:equivalence_rhs}}, the norms
$ {\left\lVert\cdot\right\rVert}_\mathrm{init}^\prime
$ and ${\skliaustask\cdot\skliaustasd}_\mathrm{rhs}^\prime$
can be replaced by the simpler $ {\left\lVert\cdot\right
\rVert}_\mathrm{init}$ and
${\skliaustask\cdot\skliaustasd}_\mathrm{rhs}$
(cf.~\reftext{\eqref{norm_init}}\&\reftext{\eqref{norm_rhs}}), respectively, so that we end up with
%
%
\begin{equation}\label{finalbound_discr}
{\skliaustask\Psi_{M}\skliaustasd}_{\mathrm{sol}} \lesssim
 {\left\lVert v^{(0)} \right\rVert}_{\mathrm
{init}} + {\skliaustask f\skliaustasd}_{\mathrm{rhs}}.
\end{equation}

Due to the definition \reftext{\eqref{derivpsi}} and to the bounds \reftext{\eqref{bound_phi_rhs}} and \reftext{\eqref{finalbound_discr}} we
find that a subsequence of $\left(\Phi_M,\Psi_M\right)_M$ (which we
denote by $\left(\Phi_M,\Psi_M\right)$ again) weak-$*$-converges to
$(f,v)$. By weak lower-semicontinuity of the norms, we find
that the maximal-regularity estimate \reftext{\eqref{mr_interval}} holds in the
limit as well.

In order to ensure the validity of the boundary condition
$v_{|t=0}=v^{(0)}$ in the limit, we note that the function $v^{(0)}$ is
the initial value for \emph{every} $\Psi_M$ regardless of the choice of
$M$. Then we use the time-trace bounds~\reftext{\eqref{traceest}} of \reftext{Lemma~\ref{lem:trace_estimate}} for the norms $ {\left\lVert\cdot
\right\rVert}_{\ell,\rho}$ with
indices as appearing in the definition~\reftext{\eqref{norm_sol}} of
${\skliaustask\cdot\skliaustasd}_\mathrm{sol}$ in order to be able
to define the evaluation at a
specific time $t$ in the limit.

In order to pass to the limit $M \to\infty$ in equation~\reftext{\eqref{derivpsi}}, we first rewrite it in terms of $\Psi_M$ only:
\begin{subequations}\label{derivpsi_rewr}
\begin{equation}
x\partial_t \Psi_M + q(D_x)\Psi_M + D_y^2 r(D_x) \Psi_M + D_y^4 \Psi_M = \Phi_M + R_M,
\end{equation}
with
\begin{align}
R_M &:= q(D_x)\tilde R_M + D_y^2 r(D_x) \tilde R_M + D_y^4 \tilde R_M,
\eqncr
\tilde R_M &:= \Psi_M - \sum_{m=1}^Mv^{(m\delta t)} \mathbh
{1}_{[(m-1)\delta
t, m\delta t)} \stackrel{\text{\reftext{\eqref{affineinterpol}}}}{=} \sum_{m=1}^M\frac{m\delta t -
t}{\delta t}\left(v^{((m-1)\delta t)} -v^{(m\delta t)}\right) \mathbh{1}
_{[(m-1)\delta t, m\delta t)}.
\end{align}
\end{subequations}
In order to prove distributional convergence to zero of the remainder
$R_M$, it suffices to prove distributional convergence to zero of
$\tilde R_M$. For this we first use \reftext{\eqref{discretestep}} leading to
\begin{equation}\label{reexprdiffv}
v^{((m-1)\delta t)} - v^{(m\delta t)} = \frac{\delta t }{x}\left
(-\Phi
_M + q(D_x)v^{(m\delta t)} + D_y^2 r(D_x) v^{(m\delta t)} + D_y^4 v^{(m\delta t)}\right).
\end{equation}
Note that on one side, the $\Phi_M$-terms in \reftext{\eqref{reexprdiffv}}
already weak-*-converge to $f$, so that when multiplied by $\delta t=
\frac T M$ it converges to zero. For the remaining terms, we note that
the linear operator $q(D_x)+D_y^2 r(D_x) + D_y^4$ acting on
$v^{(m\delta t)}$ is independent of $t$ and contains no
$t$-derivatives. Therefore, we may consider directly the distributional
convergence of the term
\begin{align}\nonumber
& \frac{\delta t}{x} \sum_{m=1}^M\left(q(D_x)v^{(m\delta t)} +D_y^2 r(D_x) v^{(m\delta t)} + D_y^4 v^{(m\delta t)}\right
)\mathbh{1}
_{[(m-1)\delta t, m\delta t)} \\
&\quad = \frac{\delta t}{x} \left(q(D_x)+D_y^2 r(D_x) + D_y^4 \right) \left(\sum_{m=1}^Mv^{(m\delta t)}\mathbh{1}_{[(m-1)\delta t,
m\delta
t)}\right),\label{term_to_converge}
\end{align}
and so we consider the function
\begin{equation}\label{whatneedstoconverge}
\sum_{m=1}^Mv^{(m\delta t)}\mathbh{1}_{[(m-1)\delta t, m\delta t)}.
\end{equation}
In view of \reftext{\eqref{affineinterpol}} and by using \reftext{\eqref{finalbound_discr}}
together with \reftext{\eqref{bc0_grad_sol}} of \reftext{Lemma~\ref{lem:bc0_bounds}}, we
find the bound
\begin{align}
& \max_{1\le m\le M} \left( {\left\lVert
v^{(m\delta t)}_x \right\rVert}_{BC^0((0,\infty)_x\times\mathbb
{R}_y)}+ {\left\lVert v^{(m\delta t)}_y \right\rVert
}_{BC^0((0,\infty)_x\times\mathbb{R}_y)}\right)\nonumber\\
& \quad\lesssim\sup_{t\in I}\left( {\left\lVert\partial
_x\Psi_M(t,\cdot,\cdot) \right\rVert}_{BC^0((0,\infty)_x\times
\mathbb{R}_y)}+ {\left\lVert\partial_y\Psi_M(t,\cdot
,\cdot) \right\rVert}_{BC^0((0,\infty)_x\times\mathbb
{R}_y)}\right)\lesssim{\skliaustask\Psi_M\skliaustasd}_\mathrm
{sol}, \label{psi_bd_v}
\end{align}
which shows that the gradient of \reftext{\eqref{whatneedstoconverge}} is bounded
in the $ {\left\lVert\cdot\right\rVert
}_{BC^0((0,\infty)_x\times\mathbb{R}_y)}$-norm. Now we
note that the operator that in \reftext{\eqref{term_to_converge}} is applied to \reftext{\eqref{whatneedstoconverge}} is in divergence form and tends to zero distributionally as $M\to
\infty
$, due to the factor $\delta t$. Therefore, the boundedness of \reftext{\eqref{psi_bd_v}} implies that distributionally in the limit $M\to\infty$ the
term \reftext{\eqref{term_to_converge}} tends to zero, and therefore
equation~\reftext{\eqref{eq_lin}} is satisfied. Furthermore, as noted before, $v$ fulfills the initial condition $v_{|t=0} = v^{(0)}$. This completes the proof of existence of a solution $v$ meeting \reftext{\eqref{mr_interval}}.

\medskip
\noindent\textit{Step 4: Uniqueness of solutions.}
For proving uniqueness of the solution $v$, we use the following
elementary arguments, which follow the lines of the proof of \reftext{Lemma~\ref{lem:res_unique}} on uniqueness of
solutions to the resolvent equation: Without
loss of generality, we may assume $f = 0$ and $v^{(0)}=0$ by
linearity. Since ${\skliaustask v\skliaustasd}_{\mathrm{sol}}<\infty
$, we have in particular
\begin{equation}\label{basicfiniteness}
\sup_{t\in I} {\left\lVert v \right\rVert}_{\tilde k,-\delta-1
}^2 + \int_I {\left\lVert v \right\rVert}_{\tilde k+2,-\delta
- \frac12}^2 \mathrm{d} t < \infty.
\end{equation}
Testing\vspace*{1pt} \reftext{\eqref{eq_lin}} with $\chi_n^2v$ where $\chi_n$ is a cut-off
$\chi_n(x)=\tilde\chi\left(\frac{\log x}{n}\right)$ such that $\tilde\chi$
is a smooth non-negative function $\tilde\chi(s) = 1$ for $s \in [-1,1]$
and $\tilde\chi(s) = 0$ for $s \in \mathbb{R}\setminus[-2,2]$,
we can proceed along the steps
leading from \reftext{\eqref{test_basic}} to \reftext{\eqref{lin_est1}} and integrate in
$t\in I$. Note, however, that remnant terms appear when commuting the
multiplication by $\chi_n$ with the operators $q(D_x)$ and $r(D_x)$.
Due to the scaling of $\chi_n$ in $n$, we find that $R_n = O\left
(\frac
1n\right)$ as $n\to\infty$. This leads to the bound
\begin{equation}\label{testunique}
\sup_{t\in I}\frac12 {\left\lVert\chi_n v(t) \right
\rVert}_{-\delta-1}^2 + \int_I  {\left\lVert\chi_n v
\right\rVert}_{2,-\delta-\frac12}^2 \mathrm{d} t \le \frac12
 {\left\lVert\chi_n v_{|t=0} \right\rVert}_{-\delta-1}^2 + \int_I R_n\ \mathrm{d} t.
\end{equation}
Since $v_{|t=0}=v^{(0)}=0$, the first term on the right-hand side of \reftext{\eqref{testunique}} vanishes, and as $n\to\infty$ so does the second
term. By dominated convergence we conclude $v = 0$.
\end{proof}
%

\section{Nonlinear theory}\label{sec:nonlinear}

\subsection{The structure of the nonlinearity}\label{sec:nonlinear_struct}

We begin by making some observations on the structure of the
nonlinearity ${\mathcal N}(v)$ given through \reftext{\eqref{def_nonlinearity}}, i.e.,
\begin{align*}
{\mathcal N}(v) := \, & - x F^{-1} \Big( D_y^2 - D_y G \left(D_x -
\tfrac12\right) - G D_y \left(D_x + \tfrac3 2\right) + G \left(D_x +
\tfrac32\right) G \left(D_x - \tfrac1 2\right) \\
& + F \left(D_x + \tfrac3 2\right) F \left(D_x - \tfrac1 2\right
)\Big
) \Big(D_y G - G \left(D_x + \tfrac1 2\right) G - F \left(D_x +
\tfrac
1 2\right) F\Big) \\
& + \frac3 8 x + q(D_x) v + D_y^2 r(D_x) v + D_y^4 v.
\end{align*}
Recall that due to \reftext{\eqref{def_fg}} and \reftext{\eqref{def_v}} we have
%
%
\begin{equation}\label{recall_fg}
F^{-1}=Z_x=v_x+1 \quad\mbox{and} \quad G=Z_x^{-1}Z_y=\frac{v_y}{v_x+1}.
\end{equation}
Our main objective is to re-write ${\mathcal N}(v)$ in a form that reflects
expansion~\reftext{\eqref{expansion_f}}, i.e., almost everywhere
\begin{equation*}
D^\ell{\mathcal N}(v) = D^\ell\left(\left({\mathcal N}(v)\right)_1
(t,y) x + \left({\mathcal N}
(v)\right)_2 (t,y) x^2\right) + o\left(x^{2+\delta}\right) \quad
\mbox
{as} \quad x \searrow0,
\end{equation*}
where $\ell\in\mathbb{N}_0^2$ with
${\left\lvert\ell\right\rvert} \le\check k + 5$ and $\ell_y \le\check k - 2$,
given that $v$ meets expansion~\reftext{\eqref{as_sol}},
i.e., almost everywhere
\begin{equation}\label{recall_exp_v}
D^\ell v(t,x,y) = D^\ell\left(v_0(t,y) + v_1(t,y) x + v_{1+\beta}(t,y)
x^{1+\beta}+v_2(t,y) x^2\right) + o\left(x^{2+\delta}\right)
\end{equation}
as $x \searrow 0$, where $\ell\in\mathbb{N}_0^2$ with $ {\left\lvert\ell
\right\rvert} \le\check k + 9$ and $\ell
_y \le\check k + 2$. Because of $q(0) \stackrel{\text{\reftext{\eqref{poly_q}}}}{=} 0$,
we have almost everywhere $D^\ell q(D_x) v = O(x)$ as $x \searrow0$,
where ${\left\lvert\ell\right\rvert} \le \check k + 5$,
and the asymptotics $D^\ell D_y^2 r(D_x) v = O(x^2)$ and
$D^\ell D_y^4 v = O(x^4)$ as $x \searrow0$ for
$ {\left\lvert\ell\right\rvert} \le \check k + 5$
are valid almost everywhere, so that indeed
$D^\ell{\mathcal N}(v) = O(x)$ as $x \searrow0$
for $ {\left\lvert\ell\right\rvert} \le \check k + 5$ holds true almost everywhere.
In order to see that the
contribution $O\left(x^{1+\beta}\right)$ is canceled due to the
structure of the nonlinearity ${\mathcal N}(v)$, we separate the terms
appearing in~\reftext{\eqref{def_nonlinearity}} by defining the operators
%
%
\begin{subequations}
%
%
\begin{eqnarray}
A_1&:=&D_y^2 - D_yG\left(D_x -\tfrac12\right) - G D_y\left
(D_x+\tfrac
32\right),
\eqncr
A_2&:=&G\left(D_x+\tfrac32\right)G\left(D_x-\tfrac12\right) +
F\left
(D_x+\tfrac32\right)F\left(D_x-\tfrac12\right),
\end{eqnarray}
and expressions
%
%
\begin{equation}
B_1:=D_yG,\quad\quad B_2:=G\left(D_y+\tfrac12\right)G + F\left
(D_x+\tfrac12\right)F.\label{subdivAB}
\end{equation}
\end{subequations}
Then we write
%
%
\begin{subequations}\label{decomp_first_all}
%
%
\begin{equation}\label{decomp_first}
{\mathcal N}(v) = {\mathcal N}^{(1)}(v) + {\mathcal N}^{(2)}(v),
\end{equation}
where
%
%
\begin{eqnarray}
{\mathcal N}^{(1)}(v)&:=&-xF^{-1}(A_1B_1 + A_2B_1 -
A_1B_2)+D_y^2r(D_x)v+D_y^4v,\label{defn1}
\eqncr
{\mathcal N}^{(2)}(v)&:=&xF^{-1}A_2B_2 + \frac38x+q(D_x)v\label{defn2}.
\end{eqnarray}
\end{subequations}
Thus ${\mathcal N}^{(1)}(v)$ is the combination of terms from \reftext{\eqref{def_nonlinearity}} which contain factors of the form $D_y^2, D_yG$ or
$G D_y$, and ${\mathcal N}^{(2)}(v)$ contains only products of $x$, $D_x$,
$F$, and $G$.

\subsubsection{The structure of ${\mathcal N}^{(1)}(v)$}
By series expansion of the factors $F=(1+v_x)^{-1}$ for $
{\left\lvert v_x \right\rvert} <
1$, we can write ${\mathcal N}^{(1)}(v)$ as a convergent series of
terms of the form
\begin{subequations}\label{simple_n1v}
\begin{equation}\label{simple_n1v_terms}
T\left(a^{(1)},\ldots,a^{(m)},v\right) := c\left(a^{(1)},\ldots
,a^{(m)}\right) \, x^{1-m} \bigtimes_{j = 1}^m D^{a^{(j)}} v,
\end{equation}
where coefficients $c(a^{(1)},\ldots,a^{(m)})$ and multi-indices
$a^{(j)} = \left( a^{(j)}_x,a^{(j)}_y \right)$ satisfy (with notations and conventions
like in \reftext{\eqref{der_scale}})
\begin{equation}\label{simple_n1v_cond}
\begin{aligned}
& m \ge2, \quad a^{(1)}, \ldots, a^{(m)} \in\mathbb{N}_0^2
\setminus\{(0,0)\}
,\quad {\left\lvert a^{(1)} \right\rvert} \ge2,\quad
a^{(1)}_y \ge1, \\
& m \le\sum_{j = 1}^m  {\left\lvert a^{(j)} \right\rvert
} \le m + 3, \quad c\left
(a^{(1)},\ldots,a^{(m)}\right) \in\mathbb{R}.
\end{aligned}
\end{equation}
\end{subequations}
Since almost everywhere $D^\ell D^{a^{(j)}} v = O(x)$ for
$j \in\{2,\ldots,m\}$ and
$D^\ell D^{a^{(1)}} v = O\left(x^2\right)$ as $x \searrow0$, each term
in \reftext{\eqref{simple_n1v_terms}} behaves almost everywhere
as $D^\ell T\left(a^{(1)},\ldots
,a^{(m)},v\right) = O\left(x^2\right)$ as $x \searrow0$. As a result,
we have ${\mathcal N}^{(1)}(v) = O\left(x^2\right)$ as well.

\subsubsection{The structure of ${\mathcal N}^{(2)}(v)$}

Introducing the $4$-linear form
\begin{equation}\label{def_4lin}
{\mathcal M}\left(H^{(1)},H^{(2)},H^{(3)},H^{(4)}\right) := x H^{(1)}
\left(D_x
+ \tfrac3 2\right) H^{(2)} \left(D_x - \tfrac1 2\right) H^{(3)}
\left
(D_x + \tfrac1 2\right) H^{(4)},
\end{equation}
we may decompose the term ${\mathcal N}^{(2)}(v)$ from \reftext{\eqref{decomp_first}} as follows:
\begin{subequations}\label{def_n25v}
\begin{equation}\label{nonlin_decomp_n2}
{\mathcal N}^{(2)}(v) = {\mathcal N}^{(2,1)}(v) + {\mathcal
N}^{(2,2)}(v) + {\mathcal N}^{(2,3)}(v) + {\mathcal N}
^{(2,4)}(v)
\end{equation}
with
\begin{eqnarray}
{\mathcal N}^{(2,1)}(v) &:=& {\mathcal M}(1,F,F,F) + q(D_x) v + \frac38 x, \label{def_n21v}
\eqncr
{\mathcal N}^{(2,2)}(v) &:=& {\mathcal M}\left(F^{-1} G, G, G, G\right
), \label{def_n22v}
\eqncr
{\mathcal N}^{(2,3)}(v) &:=& {\mathcal M}\left(1, F, G, G\right),
\label{def_n23v}
\eqncr
{\mathcal N}^{(2,4)}(v) &:=& {\mathcal M}\left(F^{-1} G, G, F, F\right
). \label{def_n24v}
\end{eqnarray}
\end{subequations}
For treating the terms appearing in \reftext{\eqref{def_n25v}}, we keep
expansion~\reftext{\eqref{recall_exp_v}} in mind in order to keep track of the
$x$-power series contributions appearing in this regime. We introduce
the notation
\begin{equation}\label{notation_nonlin}
\phi:= \left(1 + v_1\right)^{-1} \left(v - v_0 - v_1 x\right)
\quad
\mbox{and} \quad\psi:= v - v_0
\end{equation}
leading to
\begin{subequations}\label{simple_nonlin}
\begin{eqnarray}
F &\stackrel{\text{\reftext{\eqref{def_fg}}}}{=}& (1 + v_1)^{-1} \left(1 + \phi
_x\right)^{-1},
\eqncr
F^{-1} G &\stackrel{\text{\reftext{\eqref{def_fg}}}}{=}& (v_0)_y + \psi_y,
\eqncr
G &\stackrel{\text{\reftext{\eqref{def_fg}}}}{=}& \left(1 + v_1\right)^{-1} \left(1 +
\phi_x\right)^{-1} \left((v_0)_y + \psi_y\right).
\end{eqnarray}
\end{subequations}
Employing \reftext{\eqref{notation_nonlin}} and \reftext{\eqref{simple_nonlin}} for \reftext{\eqref{def_n21v}} gives through power series expansion in $\phi_x$ (using
$q(0) \stackrel{\text{\reftext{\eqref{poly_q}}}}{=} 0$)
\begin{eqnarray*}
{\mathcal N}^{(2,1)}(v) &=& \left(- \tfrac3 8 (1+v_1)^{-3} + q(1) v_1
+ \tfrac
3 8\right) x \\
&& - (1 + v_1)^{-3} x \left(D_x + \tfrac3 2\right) \left(- \tfrac1 4
\phi_x + \tfrac1 2 \left(D_x - \tfrac1 2\right) \phi_x + \left
(D_x -
\tfrac1 2\right) \left(D_x + \tfrac1 2\right) \phi_x\right) \\
&& + (1 + v_1) q(D_x) \phi\\
&& + (1 + v_1)^{-3} \sum_{\tau_2+\tau_3+\tau_4 \ge2}
(-1)^{ {\left\lvert\tau\right\rvert}} {\mathcal
M}\left(1,\phi_x^{\tau_2},\phi_x^{\tau_3},\phi_x^{\tau_4}\right).
\end{eqnarray*}
Hence, with help of \reftext{\eqref{poly_q}}
\begin{eqnarray}\nonumber
{\mathcal N}^{(2,1)}(v) &=& - \frac3 8 (1 + v_1)^{-3} \left(6 v_1^2 +
8 v_1^3 +
3 v_1^4\right) x \\
&& + (1+v_1)^{-3} \left(4 v_1 + 6 v_1^2 + 4 v_1^3 + v_1^4\right) q(D_x)
\phi\nonumber\\
&& + (1 + v_1)^{-3} \sum_{\tau_2+\tau_3+\tau_4 \ge2}
(-1)^{ {\left\lvert\tau\right\rvert}} {\mathcal
M}\left(1,\phi_x^{\tau_2},\phi_x^{\tau_3},\phi_x^{\tau_4}\right).
\label{simple_n21}
\end{eqnarray}

With an analogous reasoning, using \reftext{\eqref{poly_q}}, \reftext{\eqref{def_beta}},
and \reftext{\eqref{simple_nonlin}} in \reftext{\eqref{def_n22v}}, we also obtain
\begin{eqnarray}
{\mathcal N}^{(2,2)}(v) &=& - \frac3 8 (1+v_1)^{-3} (v_0)_y^4 x
\nonumber\\
&& - (1+v_1)^{-3} (v_0)_y^4 q(D_x) \phi\nonumber\\
&& + (1+v_1)^{-3} (v_0)_y^3 \left(D_x + \tfrac1 2\right) \left
(D_x^2 -
\tfrac3 2 D_x - \tfrac5 8\right) D_y \psi\nonumber\\
&& + (1+v_1)^{-3} \sum_{\substack{ {\left\lvert\mu
\right\rvert} +  {\left\lvert\tau\right\rvert} \ge2
\\
\mu_j + \nu_j = 1 \\ \tau_1 = 0}} (-1)^{ {\left\lvert
\tau\right\rvert}} {\mathcal M}\left(\left
(\psi_y^{\mu_j} (v_0)_y^{\nu_j}\right)_{j=1}^4\right). \label{simple_n22}
\end{eqnarray}

For dealing with ${\mathcal N}^{(2,3)}(v)$ (cf.~\reftext{\eqref{def_n23v}}), we
may use \reftext{\eqref{poly_q}} and \reftext{\eqref{simple_nonlin}} once more and arrive at
\begin{eqnarray}\nonumber
{\mathcal N}^{(2,3)}(v) &=& - \frac3 8 (1+v_1)^{-2} (v_0)_y^2 x \\
&& - (1+v_1)^{-3} (v_0)_y^2 q\left(D_x\right) \phi\nonumber\\
&& + (1+v_1)^{-3} (v_0)_y \left(D_x+\tfrac1 2\right) \left(D_x -
\tfrac3 2\right) D_x D_y \psi\nonumber\\
&& + (1+v_1)^{-3} \sum_{\substack{ {\left\lvert\mu
\right\rvert} +  {\left\lvert\tau\right\rvert} \ge2
\\
\tau_1 = 0\\\mu_j = \nu_j = 0\; \mathrm{for}\; j\in\{1,2\}\\ \mu
_j + \nu
_j = 1 \; \mathrm{for} \; j \in\{3, 4\}}} (-1)^{ {\left
\lvert\tau\right\rvert}} {\mathcal M}\left
(\left(\psi_y^{\mu_j} (v_0)_y^{\nu_j} \phi_x^{\tau_j}\right
)_{j=1}^4\right). \label{simple_n23}
\end{eqnarray}

For ${\mathcal N}^{(2,4)}(v)$ (cf.~\reftext{\eqref{def_n24v}}), employing \reftext{\eqref{poly_q}}
and \reftext{\eqref{simple_nonlin}} gives
\begin{eqnarray}\nonumber
{\mathcal N}^{(2,4)}(v) &=& - \frac3 8 (1+v_1)^{-3} (v_0)_y^2 x \\
&& - (1+v_1)^{-3} (v_0)_y^2 q(D_x) \phi\nonumber\\
&& - \frac3 8 (1+v_1)^{-3} (v_0)_y \left(\tfrac2 3 D_x + \tfrac43\right) D_y \psi\nonumber\\
&& + (1+v_1)^{-3} \sum_{\substack{ {\left\lvert\mu
\right\rvert} +  {\left\lvert\tau\right\rvert} \ge2
\\
\tau_1 = 0\\ \mu_j + \nu_j = 1\; \mathrm{for}\; j\in\{1,2\}\\ \mu
_j =
\nu_j = 0 \; \mathrm{for} \; j \in\{3, 4\}}} (-1)^{
{\left\lvert\tau\right\rvert}} {\mathcal M}
\left(\left(\psi_y^{\mu_j} (v_0)_y^{\nu_j} \phi_x^{\tau_j}\right
)_{j=1}^4\right). \label{simple_n24}
\end{eqnarray}

By combining \reftext{\eqref{nonlin_decomp_n2}}, \reftext{\eqref{simple_n21}}, \reftext{\eqref{simple_n22}}, \reftext{\eqref{simple_n23}}, and \reftext{\eqref{simple_n24}}, we find
\begin{subequations}\label{simple_n2}
\begin{eqnarray}\nonumber
{\mathcal N}^{(2)}(v) &=& - \frac3 8 (1+v_1)^{-3} \left(6
v_1^2 + 8 v_1^3
+ 3 v_1^4 + 2 (v_0)_y^2 + (v_0)_y^4\right) x \nonumber\\
&& + (1+v_1)^{-3} \left(4 v_1 + 6 v_1^2 + 4 v_1^3 + v_1^4 - 2 (v_0)_y^2
- (v_0)_y^4\right) q(D_x) \phi\nonumber\\
&& + \left(1+v_1\right)^{-3} \left((v_0)_y^{3}
\left(D_x + \tfrac 1 2\right) \left(D_x^2 - \tfrac3 2 D_x - \tfrac5 8\right)
+ (v_0)_y \left(D_x^2 - \tfrac5 4 D_x - \tfrac5 4\right)\right) D_y \psi\nonumber
\\
&& + (1+v_1)^{-3} \sum_{(\mu,\nu,\tau) \in{\mathcal I}} {\mathcal
M}\left(\left(\psi
_y^{\mu_j} (v_0)_y^{\nu_j} \phi_x^{\tau_j}\right)_{j=1}^4\right),
\label{simple_n2_terms}
\end{eqnarray}
where $(\mu,\nu,\tau) \in\left(\mathbb{N}_0^4\right)^3$ and
\begin{eqnarray}\nonumber
{\mathcal I}&:=& \left\{ {\left\lvert\mu\right\rvert}
+  {\left\lvert\tau\right\rvert} \ge2, \; \tau_1 =
0, \; \mu
= \nu= 0, \right\} \\
&& \cup\left\{ {\left\lvert\mu\right\rvert} +
 {\left\lvert\tau\right\rvert} \ge2, \; \tau_1 = 0,
\; \mu
_j + \nu_j = 1 \; \mathrm{for} \; j \in\{1,2,3,4\}\right\}
\nonumber\\
&& \cup\left\{ {\left\lvert\mu\right\rvert} +
 {\left\lvert\tau\right\rvert} \ge2, \; \tau_1 = 0,
\; \mu
_j = \nu_j = 0\; \mathrm{for}\; j\in\{1,2\}, \; \mu_j + \nu_j = 1
\;
\mathrm{for} \; j \in\{3,4\}\right\} \nonumber\\
&& \cup\left\{ {\left\lvert\mu\right\rvert} +
 {\left\lvert\tau\right\rvert} \ge2, \; \tau_1 = 0,
\; \mu
_j + \nu_j = 1\; \mathrm{for}\; j\in\{1,2\}, \; \mu_j = \nu_j = 0
\;
\mathrm{for} \; j \in\{3,4\}\right\} \nonumber\\
&& \label{simple_n2_ind}
\end{eqnarray}
\end{subequations}
Note that the precise definition \reftext{\eqref{simple_n2_ind}} of ${\mathcal
I}$ is not
essential for the subsequent arguments and that we will simply use
${\mathcal I}
\subseteq\left\{ {\left\lvert\mu\right\rvert} +
 {\left\lvert\tau\right\rvert} \ge2, {\left\lvert\mu\right\rvert}+{\left\lvert\nu\right\rvert}\le 4\right\}$. Regarding
expression~\reftext{\eqref{simple_n2_terms}}, we remark that due to \reftext{\eqref{recall_exp_v}} and \reftext{\eqref{notation_nonlin}} we have almost everywhere
\begin{align}\label{expansion_phi}
D^\ell\phi(t,x,y) = D^\ell\left(\left(1+v_1(t,y)\right)^{-1}
\left
(v_{1+\beta}(t,y) \, x^{1+\beta} +v_2(t,y)x^2\right)\right) +
o\left
(x^{2+\delta}\right) \quad\mbox{as} \quad x \searrow0,
\end{align}
where $\ell\in\mathbb{N}_0^2$ with $ {\left\lvert\ell
\right\rvert} \le\check k + 9$ and $\ell
_y \le\check k + 2$.

The main considerations to be kept in mind concerning \reftext{\eqref{simple_n2}} are:
\begin{enumerate}[(i)]
\item[(i)] The second line of \reftext{\eqref{simple_n2_terms}} is $(\phi,\psi
)$-independent but super-linear in $v$ and forms the
$O(x)$-contribution of ${\mathcal N}^{(2)}(v)$.
\item[(ii)] As a consequence of the fact that $q(1+\beta) \stackrel{\text{\reftext{\eqref{poly_q}}}}{=} 0$ and using \reftext{\eqref{expansion_phi}}, the third line of \reftext{\eqref{simple_n2_terms}} is of order $O\left(x^2\right)$ as $x
\searrow
0$ and is again super-linear in $v$ (but linear in~$\phi$).
\item[(iii)] Line four of \reftext{\eqref{simple_n2_terms}} contains one
$D_y$-derivative acting on $\psi$, therefore this line is of order
$O\left(x^2\right)$ as $x \searrow0$ and super-linear in $v$ but
linear in $\psi$.
\item[(iv)] Finally, the last line of \reftext{\eqref{simple_n2_terms}} is the
remainder, being super-linear in $(\phi,\psi)$ (cf.~\reftext{\eqref{simple_n2_ind}} where $ {\left\lvert\mu\right\rvert} +
 {\left\lvert\tau\right\rvert} \ge2$), and thus
also in $v$. Concerning the behavior near $x = 0$, we have $D^\ell\phi_x
\stackrel{\text{\reftext{\eqref{expansion_phi}}}}{=} O\left(x^\beta\right)$ as $x
\searrow0$ (with $\beta\in(\frac1 2,1)$ by \reftext{\eqref{def_beta}}) and
$D^\ell\psi_y = O(x)$ as $x \searrow0$. Thus, recalling that
${\mathcal M}$
also features an extra factor of $x$, we find that this term is of
order $O\left(x^{1+2\beta}\right)$ as $x \searrow0$.
\item[(v)] The above considerations also imply that, using the notation \reftext{\eqref{expansion_v}} and \reftext{\eqref{recall_exp_v}} for the power series
coefficient expansion near $x=0$, we have almost everywhere
\begin{equation*}
D^\ell{\mathcal N}^{(2)}(v) = D^\ell\left(\left({\mathcal
N}^{(2)}(v)\right)_1 x + \left
({\mathcal N}^{(2)}(v)\right)_2 x^2\right) + o\left(x^{2+\delta
}\right) \quad
\mbox{as} \quad x \searrow0
\end{equation*}
for $\ell\in\mathbb{N}_0^2$ with $ {\left\lvert\ell
\right\rvert} \le\check k + 5$ and $\ell_y
\le\check k - 2$.
\end{enumerate}
%

\subsection{Nonlinear estimates}\label{sec:nonlinear_est}

In this section, we derive our main estimate for the nonlinearity
${\mathcal N}
(v)$ (cf.~\reftext{\eqref{def_nonlinearity}}).
%
%
\begin{proposition}\label{prop:non_est}
Suppose that $T \in (0,\infty]$, $I := [0,T) \subseteq [0,\infty)$ is an interval, $\delta
\in \left(0,\frac12\left(\beta-\frac12\right)\right]$, and $k$, $\tilde k$,
$\check k$, and $\breve k$ satisfy the bounds \reftext{\eqref{parameters}} of
\reftext{Assumptions~\ref{ass:parameters}}. Then\goodbreak
%
%
\begin{equation}\label{non_est_main}
{\skliaustask{\mathcal N}\left(v^{(1)}\right) - {\mathcal N}\left
(v^{(2)}\right)\skliaustasd}_\mathrm
{rhs} \lesssim_{k, \tilde k, \check k, \breve k, \delta} \left
({\skliaustask v^{(1)}\skliaustasd}_\mathrm{sol} + {\skliaustask
v^{(2)}\skliaustasd}_\mathrm{sol}\right) {\skliaustask v^{(1)} -
v^{(2)}\skliaustasd}_\mathrm{sol},
\end{equation}
where
\begin{equation*}
v^{(j)} = v^{(j)}(t,x,y) : \, (0,\infty)^2 \times\mathbb{R}\to
\mathbb{R}
\end{equation*}
are locally integrable with ${\skliaustask v^{(j)}\skliaustasd
}_\mathrm{sol} \ll_{k, \tilde
k,\check k,\breve k,\delta} 1$ for $j = 1,2$ (cf.~\reftext{\eqref{norm_rhs}} and \reftext{\eqref{norm_sol}} for the definition of the norms ${\skliaustask\cdot
\skliaustasd}_\mathrm{rhs}$ and ${\skliaustask\cdot\skliaustasd
}_\mathrm{sol}$).
\end{proposition}

The detailed proof of \reftext{Theorem~\ref{main_th}} employing the
maximal-regularity estimate \reftext{\eqref{mr_interval}} of \reftext{Proposition~\ref{prop:maxregevolution}} and the nonlinear estimate \reftext{\eqref{non_est_main}}
of \reftext{Proposition~\ref{prop:non_est}} is the subject of \S \ref{sec:well}.

In order to prove \reftext{Proposition~\ref{prop:non_est}}, we show the following
auxiliary interpolation result:
\begin{lemma}\label{lem:interpolation_dy3}
Assume that $w^{(i)}\in C^\infty((0,\infty
)_x\times\mathbb{R}_y)\cap C^0_\mathrm{c}([0,\infty)_x\times
\mathbb{R}_y)$ for $i=1,2$. Then we
have the following bound, in which the implicit constants are
independent of $\delta$:
\begin{equation}
\begin{aligned}
& {\left\lVert D_yw^{(1)} \times D_yw^{(2)} \right\rVert}_{-\delta+ \frac
72} \\
&\quad\lesssim{\left\lVert D_y^2w^{(1)}
\right\rVert}_{-\delta+\frac
52}^{\frac12} {\left\lVert x^{-1} w^{(1)} \right\rVert
}_{BC^0((0,\infty)_x\times\mathbb{R}
_y)}^{\frac12}{\left\lVert D_y^2w^{(2)}
\right\rVert}_{-\delta+\frac
52}^{\frac12} {\left\lVert x^{-1} w^{(2)} \right\rVert
}_{BC^0((0,\infty)_x\times\mathbb{R}
_y)}^{\frac12}. \label{dy_dy}
\end{aligned}
\end{equation}
\end{lemma}
\begin{proof}[Proof of \reftext{Lemma~\ref{lem:interpolation_dy3}}]
First suppose that $w^{(1)} = w^{(2)} =: w$. Using the definition of the norm and integrating by parts we can write
\begin{eqnarray}\label{dw_abs}
{\left\lVert D_yw \times D_yw\right\rVert}_{-\delta+\frac72}^2
&=&\int_0^\infty \int_{-\infty}^\infty 
x^{2\delta - 5}
\left(w_y\right)^4 \mathrm{d}y \, \mathrm{d}x 
= -3\int_0^\infty\int_{-\infty}^\infty x^{2\delta - 5}\left(w_y\right)^2w_{yy}w \, \mathrm{d}y \, \mathrm{d}x \nonumber\\
&\le&3{\left\lVert x^{-1}w\right\rVert}_{BC^0((0,\infty)\times \R)} \left(\int_0^\infty\int_{-\infty}^\infty x^{2\delta-3}\left(w_{yy}\right)^2 \mathrm{d}y \, \mathrm{d}x\right)^{\frac12}\nonumber\\
&&\quad\times\left(\int_0^\infty\int_{-\infty}^\infty x^{2\delta-5}\left(w_y\right)^4 \, \mathrm{d}y \, \mathrm{d}x\right)^{\frac12} \nonumber\\
&=&3{\left\lVert x^{-1}w\right\rVert}_{BC^0((0,\infty)\times \R)}{\left\lVert D_y^2w\right\rVert}_{-\delta+\frac52}{\left\lVert D_y w \times D_y w \right\rVert}_{-\delta+\frac72}.
\end{eqnarray}
By dividing left- and right-hand sides of the inequality \reftext{\eqref{dw_abs}} by the common factor, we find
\begin{equation}\label{dw_abs2}
{\left\lVert D_yw \times D_yw\right\rVert}_{-\delta+\frac72}^2\le 3{\left\lVert x^{-1}w\right\rVert}_{BC^0((0,\infty)\times \R)}{\left\lVert D_y^2w\right\rVert}_{-\delta+\frac52}.
\end{equation}
Now by Cauchy-Schwarz inequality, we have in the general case that
\begin{equation*}
{\left\lVert D_yw^{(1)} \times D_yw^{(2)} \right\rVert}_{-\delta+ \frac
72}
\le
{\left\lVert D_y w^{(1)} \times D_y w^{(1)} \right\rVert}_{-\delta+ \frac
72}{\left\lVert D_y w^{(2)} \times D_y w^{(2)} \right\rVert}_{-\delta+ \frac72},
\end{equation*}
which together with \reftext{\eqref{dw_abs2}} directly gives \reftext{\eqref{dy_dy}}.
\end{proof}
\begin{proof}[Proof of \reftext{Proposition~\ref{prop:non_est}}]
Throughout the proof, estimates and constants, the latter usually denoted by $C$, will depend on $k$, $\tilde k$, $\check k$,
$\breve k$, and $\delta$, and for lightness of notation we will not
indicate this dependence explicitly. We first prove estimate~\reftext{\eqref{non_est_main}} for $v^{(1)} = v$ and $v^{(2)} = 0$. The general case
will be dealt with at the end of the proof. We use the decomposition \reftext{\eqref{decomp_first_all}} and treat the norms ${\skliaustask{\mathcal
N} ^{(1)}(v)\skliaustasd}_\mathrm{rhs}$ and ${\skliaustask{\mathcal
N}^{(2)}(v)\skliaustasd}_\mathrm{rhs}$ separately.\looseness=1

\medskip
\noindent\textit{Estimate of ${\skliaustask{\mathcal
N}^{(1)}(v)\skliaustasd}_\mathrm{rhs}$.}
As we have already noted in \reftext{\eqref{simple_n1v}}, we can write
${\mathcal N}
^{(1)}(v)$ as a convergent power series of terms of the form \reftext{\eqref{simple_n1v_terms}}, i.e.,
\begin{equation*}
T\left(a^{(1)},\ldots,a^{(m)},v\right) := c\left(a^{(1)},\ldots
,a^{(m)}\right) \, x^{1-m} \bigtimes_{j = 1}^m D^{a^{(j)}} v,
\end{equation*}
where
\begin{align*}
& m \ge2, \quad a^{(1)}, \ldots, a^{(m)} \in\mathbb{N}_0^2
\setminus\{(0,0)\}
,\quad {\left\lvert a^{(1)} \right\rvert} \ge2,\quad
a^{(1)}_y \ge1, \\
& m \le\sum_{j = 1}^m  {\left\lvert a^{(j)} \right\rvert
} \le m + 3, \quad c\left
(a^{(1)},\ldots,a^{(m)}\right) \in\mathbb{R}.
\end{align*}
Our goal is to show
\begin{equation}\label{est_n1v_terms}
{\skliaustask x^{1-m} \bigtimes_{j = 1}^m D^{a^{(j)}} v\skliaustasd
}_\mathrm{rhs} \le
\left(C {\skliaustask v\skliaustasd}_\mathrm{sol}\right)^m,
\end{equation}
where $C$ is
independent of $m$. Adding up all terms of the form \eqref{est_n1v_terms} will produce a sum of terms which can be bounded by a geometric series for $\skliaustask v\skliaustasd_\mathrm{sol} \le C^{-1}$, up to increasing $C$. A discussion of this estimate, as well as a bound on the number of terms which we obtain by expansion and during the proof by distributing derivatives, will be provided in a unified treatment at the end of the proof, contemporarily for $\skliaustask \mathcal N^{(1)}(v)\skliaustasd_\mathrm{sol}$ and $\skliaustask \mathcal N^{(2)}(v)\skliaustasd_\mathrm{sol}$. This implies summability of the right-hand side of \reftext{\eqref{est_n1v}} for sufficiently small ${\skliaustask v\skliaustasd
}_\mathrm{sol}$ and
since $m \ge2$, the desired
\begin{equation}\label{est_n1v}
{\skliaustask{\mathcal N}^{(1)}(v)\skliaustasd}_\mathrm{rhs}
\lesssim{\skliaustask v\skliaustasd}_\mathrm{sol}^2
\quad\mbox{for} \quad{\skliaustask v\skliaustasd}_\mathrm{sol} \ll1
\end{equation}
is proved.

We take $f := x^{1-m} \bigtimes_{j = 1}^m D^{a^{(j)}} v$, where $m$ and
$a^{(j)}$ meet \reftext{\eqref{simple_n1v_cond}}, in each of the addends in \reftext{\eqref{norm_rhs}}, i.e.,
\begin{eqnarray*}
{\skliaustask f\skliaustasd}_\mathrm{rhs}^2 &=& \int_I
\left( {\left\lVert f \right\rVert}_{k-2,-\delta-\frac12}^2 + {\left\lVert(D_x-1)f \right\rVert}_{\tilde
k-2,-\delta+\frac12}^2 \right) \mathrm{d} t \\
&&+\int_I\left( {\left\lVert\tilde q(D_x-1)(D_x-1)f
\right\rVert}_{\tilde k-2,\delta
+\frac12}^2 +  {\left\lVert\tilde q(D_x-1)(D_x-1)f \right
\rVert}_{\check k - 2,-\delta
+\frac32}^2\right)\mathrm{d} t \\
&&+ \int_I \left({\left\lVert(D_x-4)(D_x-3)\tilde
q(D_x-1)(D_x-1)f \right\rVert}_{\check
k-2,\delta+\frac32}^2 + {\left\lVert D_y \tilde q(D_x-1) (D_x-1)f \right\rVert
}_{\breve k-2, -\delta+\frac
52}^2\right)\mathrm{d} t,
\end{eqnarray*}
individually, where $f_1 = 0$ since
$f = O\left(x^2\right)$ as $x \searrow 0$
almost everywhere.
Distributing $\kappa$ derivatives $D$ on the factors $D^{a^{(j)}}
v$ demonstrates that it suffices to estimate a term of the form
\begin{subequations}\label{est_term_n1v}
\begin{equation}
\int_I  {\left\lVert x^{1-m} \bigtimes_{j = 1}^m
D^{b^{(j)}} v \right\rVert}_\rho^2 \mathrm{d} t.
\end{equation}
Here, without loss of generality $ {\left\lvert b^{(1)} \right
\rvert} = \max_j  {\left\lvert b^{(j)} \right\rvert}$, so that we have
\begin{equation}\label{est_term_n1v_cond_1}
\begin{aligned}
& 2 \le {\left\lvert b^{(1)} \right\rvert} \le\kappa+
4, \quad1 \le {\left\lvert b^{(j)} \right\rvert} \le
\left\lfloor\frac{\kappa+1}2\right\rfloor+ 2 \quad\mbox{for}
\quad j \in\{2,\ldots,m\}
, \\
& m \le\sum_{j = 1}^m  {\left\lvert b^{(j)} \right\rvert
} \le\kappa+ m + 3,
\end{aligned}
\end{equation}
with $(\kappa,\rho)$ as in the norms \reftext{\eqref{norm_rhs}}, i.e.,
\begin{eqnarray}\label{est_term_n1v_cond_2}
(\kappa,\rho)&\in&\Bigl\{\left(k - 2,-\delta-\tfrac12\right
),\
\left(\tilde k - 1,-\delta+\tfrac12\right
),\
\left(\tilde k + 3,\delta+\tfrac12\right),\ \left(\check k +
3,-\delta
+\tfrac32\right),\nonumber\\
&& \left(\check k + 5,\delta+\tfrac32\right),\ \left
(\breve k + 4,-\delta+\tfrac52\right)\Bigr\},
\end{eqnarray}
and for $j\in\{1,\ldots,m\}$ we have
\begin{equation}\label{est_term_n1v_cond_3}
b_y^{(j)} \le\kappa_y, \quad\mbox{where}\quad\kappa_y:=
\begin{cases}
k-2 & \mbox{for} \quad\rho= -\delta-\tfrac 12, \\
\tilde k-2 & \mbox{for} \quad\rho= \pm\delta+\tfrac 12, \\
\check k - 2 & \mbox{for} \quad\rho= \pm\delta+ \tfrac32,\\
\breve k - 1 & \mbox{for} \quad\rho= -\delta+\tfrac52.
\end{cases}
\end{equation}
\end{subequations}
Note that for the weight $\rho= - \delta+ \frac5 2$ two
$D_y$-derivatives appear in the respective term of the norm
${\skliaustask\cdot\skliaustasd}_\mathrm{rhs}$ (cf.~\reftext{\eqref{norm_rhs}}),
so that in that case we
have in \reftext{\eqref{est_term_n1v}} the further restrictions
%
%
\begin{equation}\label{d52_restrictions}
\sum_{j = 1}^m  {\left\lvert b^{(j)} \right\rvert} \ge
m+3 \quad\mbox{and} \quad\sum_{j =
1}^m b^{(j)}_y \ge 2.
\end{equation}
\textit{Estimates for the weights $\rho\in\left\{-\delta-\frac12, \pm\delta+
\frac1 2, - \delta+ \frac3 2\right\}$.}
We start by considering terms appearing in the norm
\reftext{\eqref{norm_rhs}}
that correspond to weights
$\rho \in \left\{-\delta-\frac12, \pm\delta+\frac12, -\delta +\frac 32\right\}$.

In view of \reftext{\eqref{simple_n1v_cond}} we can assume without loss of
generality that one of the following two cases applies:
\begin{enumerate}[(i)]
\item[(i)] $b_y^{(1)} \ge1$ and $ {\left
\lvert b^{(1)} \right\rvert} \ge2$,
\item[(ii)] $b_y^{(1)} = 0$, $b_y^{(2)} \ge1$, and
$b_x^{(1)} =  {\left\lvert b^{(1)} \right\rvert} \ge
 {\left\lvert b^{(2)} \right\rvert} \ge2$.
\end{enumerate}
Suppose case (i) is valid. For $\rho\in\left\{-\delta-\frac12, \pm \delta + \frac1 2, - \delta+ \frac3 2\right\}$
we can estimate in \reftext{\eqref{est_term_n1v}} according to
\begin{equation}\label{est_n1v_case_1}
\int_I  {\left\lVert x^{1-m} \bigtimes_{j = 1}^m
D^{b^{(j)}} v \right\rVert}_\rho^2 \le
\int_I  {\left\lVert D^{b^{(1)}} v \right\rVert}_\rho^2
\mathrm{d} t \times\prod_{j = 2}^m {\skliaustask x^{-1} D^{b^{(j)}}
v\skliaustasd}_{BC^0\left(I \times(0,\infty) \times\mathbb
{R}\right)}^2.
\end{equation}
Due to \reftext{\eqref{est_term_n1v_cond_1}}--\reftext{\eqref{est_term_n1v_cond_3}} and
(i) we have almost everywhere $D^{b^{(1)}} v = O\left(x^2\right)$ as $x
\searrow0$ for weights $\rho \in \left\{-\delta-\frac12, \pm\delta+ \frac1 2, - \delta+
\frac3 2\right\}$, and the elliptic-regularity result of \reftext{Lemma~\ref{lem:mainhardy}} yields
\begin{equation*}
 {\left\lVert D^{b^{(1)}} v \right\rVert}_\rho\lesssim
\begin{cases}  {\left\lVert v \right\rVert}_{k
+ 2,-\delta-\frac12} & \mbox{for}
\quad\rho= -\delta-\frac12, \\
{\left\lVert D_x v \right\rVert}_{\tilde k
+ 2,-\delta+\frac12} & \mbox{for}
\quad\rho= -\delta+\frac12, \\
 {\left\lVert \tilde q(D_x) D_x v \right\rVert}_{\tilde k + 2,\delta
+\frac12} & \mbox{for} \quad\rho
= \delta+\frac12, \\
 {\left\lVert \tilde q(D_x) D_x v \right\rVert}_{\check k +
2,-\delta+\frac32} & \mbox{for} \quad
\rho= -\delta+\frac32,
\end{cases}
\end{equation*}
so that in view of the definition of the norm ${\skliaustask\cdot
\skliaustasd}_\mathrm
{sol}$ we have
\begin{equation}\label{db1vrho}
\int_I  {\left\lVert D^{b^{(1)}} v \right\rVert}_\rho^2
\mathrm{d} t \le{\skliaustask v\skliaustasd}_\mathrm{sol}^2.
\end{equation}
Furthermore, for $j \ge2$ we can estimate by \reftext{Lemma~\ref{lem:bc0_bounds}}
\begin{eqnarray}\nonumber
{\skliaustask x^{-1} D^{b^{(j)}}
v\skliaustasd}_{BC^0\left(I \times(0,\infty) \times\mathbb
{R}\right)} &\lesssim& \max_{0 \le {\left\lvert
\ell\right\rvert} \le\left\lfloor\frac{\kappa+1}2\right
\rfloor+2} \left( {\skliaustask D^\ell v_x \skliaustasd
}_{BC^0\left(I \times (0,\infty)
\times \mathbb{R}\right)} +  {\skliaustask D^\ell v_y
\skliaustasd}_{BC^0\left(I \times (0,\infty)
\times\mathbb{R}\right)}\right) \nonumber\\
&\stackrel{\text{\reftext{\eqref{bc0_grad_sol}}}}{\lesssim}& {\skliaustask
v\skliaustasd}_\mathrm{sol},
\label{est_sup_n1v}
\end{eqnarray}
provided that, with the notation
\begin{subequations}\label{parameterss1}
\begin{equation}\label{parameters1}
\left\lfloor\frac{\kappa+1} 2\right\rfloor+ 2 =
\begin{cases} \left\lfloor\frac{k+1}{2}\right\rfloor+ 1 &
\mbox{for} \quad\rho=
-\delta-\tfrac12, \\
\left\lfloor\frac{\tilde k}{2}\right\rfloor+ 2 &
\mbox{for} \quad\rho=
-\delta+\tfrac12, \\
\left\lfloor\frac{\tilde k} 2\right\rfloor+ 4 & \mbox{for} \quad
\rho= \delta+\tfrac
12, \\
\left\lfloor\frac{\check k} 2\right\rfloor+ 4 & \mbox{for} \quad
\rho= -\delta+\tfrac32,
\end{cases}
\end{equation}
it holds
\begin{equation}\label{cond_4wise}
\left\lfloor\frac{\kappa+1} 2\right\rfloor+ 2\le\min\left\{ \tilde k - 2, \check k - 2\right\}.
\end{equation}
\end{subequations}
The bounds \reftext{\eqref{db1vrho}} and \reftext{\eqref{est_sup_n1v}} give
\begin{equation}\label{est_n1v_case_1_2}
\int_I  {\left\lVert x^{1-m} \bigtimes_{j = 1}^m
D^{b^{(j)}} v \right\rVert}_\rho^2 \le
\left(C {\skliaustask v\skliaustasd}_\mathrm{sol}\right)^{2m}
\end{equation}
for a constant $C$ as desired.

Now suppose that we are in case (ii). Consider first the
cases $\rho=-\delta\pm\frac12$. Because $b_x^{(1)} \ge1$ we have almost everywhere
$D^{b^{(1)}} v = O(x)$ as $x \searrow0$. Estimating as in \reftext{\eqref{est_n1v_case_1}} and using \reftext{Lemma~\ref{lem:mainhardy}} to absorb terms,
we find, similarly to case (i),
\begin{equation*}
 {\left\lVert D^{b^{(1)}} v \right\rVert}_\rho\lesssim
 \begin{cases}
 {\left\lVert v \right\rVert}_{k+2,-\delta
-\frac12} &\mbox{for} \quad\rho= -\delta-\tfrac12,\\\noalign{\vspace{2pt}}
{\left\lVert D_xv \right\rVert}_{\tilde k+2,-\delta
+\frac12} &\mbox{for} \quad\rho= -\delta+\tfrac12,
\end{cases}
\end{equation*}
after which we can argue as in the previous case.

Next, we treat the case
(ii) for weights $\rho\in\left\{\delta+ \frac1 2, -
\delta+ \frac3 2\right\}$.

In case $\rho= \delta+ \frac1 2$ we can write
\begin{subequations}\label{sum-subtract}
\begin{equation}
D^{b^{(1)}} v = D^{b^{(1)}_x}_x (v - v_0) = \underbrace{D^{b^{(1)}_x}_x
\left(v - v_0 - v_1 x\right)}_{=:w} + v_1 x,
\end{equation}
and in case $\rho=-\delta+\frac32$ we also sum and subtract terms and write
\begin{eqnarray}
D^{b^{(1)}} v &=& D^{b^{(1)}_x}_x (v - v_0) \nonumber\\
&=& \underbrace{D^{b^{(1)}_x}_x \left(v - v_0 - v_1 x - v_{1+\beta}
x^{1+\beta}\right)}_{=:w} + v_1 x + (1+\beta)^{b^{(1)}_x} v_{1+\beta}
x^{1+\beta}.
\end{eqnarray}
\end{subequations}
Note that $v_1 x$ and $v_{1+\beta} x^{1+\beta}$ are in the kernel of
$\tilde q(D_x) D_x$ (cf.~the definition \reftext{\eqref{norm_sol}} of the solution norm, in which this operator appears in the terms with weights $\rho$ as considered here).
We can thus apply \reftext{Lemma~\ref{lem:mainhardy}}, and have with the above notations
\begin{equation*}
 {\left\lVert w \right\rVert}_\rho\lesssim
\begin{cases}  {\left\lVert\tilde q(D_x)D_xv \right\rVert
}_{\tilde k+2,\delta+\frac12} &
\mbox{for} \quad\rho= \delta+\frac12, \\
 {\left\lVert\tilde q(D_x)D_xv \right\rVert}_{\check
k+2,-\delta+\frac32} & \mbox{for}
\quad\rho= -\delta+\tfrac32.
\end{cases}
\end{equation*}
Like in case~(i), together with \reftext{\eqref{bc0_grad_sol}} of
\reftext{Lemma~\ref{lem:bc0_bounds}}, this implies
\begin{equation*}
\int_I  {\left\lVert x^{1-m} w \prod_{j = 2}^m
D^{b^{(j)}} v \right\rVert}_\rho^2 \mathrm{d} t \le
\int_I  {\left\lVert w \right\rVert}_\rho^2 \mathrm{d}
t \times\prod_{j = 2}^m
 {\skliaustask x^{-1} D^{b^{(j)}} v \skliaustasd
}_{BC^0\left(I \times (0,\infty) \times \mathbb{R}\right)}^2
\le \left(C{\skliaustask v\skliaustasd}_\mathrm{sol}\right)^{2m}
\end{equation*}
for a constant $C$. Additionally, we have
\begin{align*}
\int_I  {\left\lVert x^{1-m} v_1 x \prod_{j = 2}^m
D^{b^{(j)}} v \right\rVert}_\rho^2 \mathrm{d} t
&\le {\left\lVert v_1 \right\rVert
}_{BC^0\left(I \times \mathbb{R}\right)}^2 \times\int
_I  {\left\lVert D^{b^{(2)}} v \right\rVert}_\rho^2
\mathrm{d} t \times\prod_{j = 3}^m
{\skliaustask x^{-1} D^{b^{(j)}} v \skliaustasd
}_{BC^0\left(I \times (0,\infty) \times \mathbb{R}
\right)}^2 \\
&\le \left(C{\skliaustask v\skliaustasd}_\mathrm{sol}\right)^{2m},
\end{align*}
for a sufficiently large constant $C$, by using estimate~\reftext{\eqref{bc0_grad_sol}} of \reftext{Lemma~\ref{lem:bc0_bounds}}
and the fact that almost everywhere $D^{b^{(2)}} v = O\left(x^2\right)$ as $x \searrow
0$, so that this term can be treated as $D^{b^{(1)}} v$ in case~(i). This completes the treatment of the weight $\rho= \delta+ \frac1 2$.

We next restrict ourselves to the weight
$\rho= - \delta+ \frac 3 2$ for the rest of the
proof step and have
\begin{align}\nonumber
&\int_I  {\left\lVert x^{1-m} v_{1+\beta}
x^{1+\beta} \prod_{j = 2}^m D^{b^{(j)}} v \right\rVert}_{-\delta
+\frac3 2}^2 \mathrm{d} t \\
&\quad\le {\left\lVert v_{1+\beta} \right\rVert
}_{L^2\left(I;BC^0(\mathbb{R})\right)}^2 \times
\sup_{t \in I}  {\left\lVert D^{b^{(2)}} v \right\rVert
}_{-\delta+\frac 3 2-\beta}^2
\times\prod_{j = 3}^m {\skliaustask
x^{-1} D^{b^{(j)}} v \skliaustasd}_{BC^0\left(I \times (0,\infty) \times
\mathbb{R}\right)}^2 \nonumber\\
&\quad \le \sup_{t \in I} {\left\lVert D^{b^{(2)}} v \right\rVert
}_{-\delta+\frac 3 2-\beta}^2
\times \left(C {\skliaustask v\skliaustasd}_\mathrm{sol}\right)^{2 (m-1)}
\le \left(C {\skliaustask v\skliaustasd}_\mathrm
{sol}\right)^{2m}, \label{est_n1v1b_2}
\end{align}
where $C$ is sufficiently large and in the last line we have used estimate~\reftext{\eqref{bc0_v1beta_v2}} of
\reftext{Lemma~\ref{lem:bc0_bounds}} and the bound
\reftext{\eqref{supl2_auxiliary}} of \reftext{Lemma~\ref{lem:bc0_auxiliary}}.

\medskip

In order to prove \reftext{\eqref{est_n1v1b_2}}, note that the elementary estimate
\begin{equation*}
{\left\lVert D^{b^{(2)}} v \right\rVert}_{-\delta+\frac 3 2-\beta}^2 \lesssim {\left\lVert D^{b^{(2)}} v \right\rVert}_\delta^2 + {\left\lVert D^{b^{(2)}} v \right\rVert}_{-\delta+1}^2 \lesssim {\left\lVert \tilde q(D_x) D_x v \right\rVert}_{\tilde k,\delta}^2 + {\left\lVert \tilde q(D_x) D_x v \right\rVert}_{\check k,-\delta+1}^2
\end{equation*}
holds true provided
\begin{equation}\label{parameters2}
\left\lfloor\frac{\check k}2\right\rfloor+4
\le\min\left\{\tilde k, \check k\right\},
\end{equation}
and the right-hand side of the inequality is uniformly bounded in time
$t$ by ${\skliaustask v\skliaustasd}_\mathrm{sol}^{2}$.
This completes the treatment of
case~(ii) for $\rho= - \delta+ \frac3 2$.

\medskip
\noindent\textit{Estimates for the weight $\rho= \delta+ \frac3 2$.}
Again, due to \reftext{\eqref{simple_n1v}}, we can focus on estimating the term
%
%
\begin{equation}\label{n1v_d32}
\int_I  {\left\lVert(D_x-4) (D_x-3) \tilde q(D_x-1)
(D_x-1) x^{1-m} \bigtimes_{j = 1}^m D^{a^{(j)}} v \right\rVert
}_{\check k - 2, \delta+ \frac3 2}^2 \mathrm{d} t,
\end{equation}
where the $a^{(j)}$ meet conditions \reftext{\eqref{simple_n1v_cond}}. In this
case, we note that the operator $\tilde q(D_x-1)$
(cf.~\reftext{\eqref{poly_q2}} for the definition of $\tilde q(\zeta)$) figures
the factor $D_x-2$. This operator cancels contributions $O(x^2)$ which
we could not exclude by using only the structure \reftext{\eqref{simple_n1v}}.
If we distribute $D$-derivatives coming from the norm $ {\left
\lVert\cdot\right\rVert}_{\check k - 2, \delta+ \frac3 2}$,
then the operator $D_x-2$ shifts to
$D_x-2-b_y$, where $b_y$ is the number of $D_y$-derivatives applied.
However, the presence of an $x$-factor in the
operator $D_y=x\partial_y$ leads to a term of order $O\left
(x^{2+b_y}\right)$ and therefore the $O\left(x^2\right)$ coefficient vanishes if
$b_y > 0$. If instead $b_y=0$, then we may apply
\reftext{Lemma~\ref{lem:mainhardy}}, and
therefore it suffices to estimate instead of \reftext{\eqref{n1v_d32}} a term of
the form
\begin{subequations}\label{n1v_d32_distr}
\begin{equation}\label{n1v_d32_term}
\int_I  {\left\lVert x^{1-m} \bigtimes_{j = 1}^m
D^{b^{(j)}} v - x^2 \left(x^{1-m} \bigtimes_{j = 1}^m D^{b^{(j)}}
v\right)_2 \right\rVert}_{\delta+ \frac32}^2 \mathrm{d} t,
\end{equation}
where the $b^{(j)}$ fulfill the conditions
\begin{equation}\label{n1v_d32_cond}
\begin{aligned}
& 2 \le {\left\lvert b^{(1)} \right\rvert} \le\check k
+ 9, \quad1 \le {\left\lvert b^{(j)} \right\rvert}
\le\left\lfloor\frac{\check k}2\right\rfloor+ 5 \quad\mbox
{for} \quad j \in\{2,\ldots
,m\}, \\
& m \le\sum_{j = 1}^m  {\left\lvert b^{(j)} \right\rvert
} \le\check k + m + 8, \quad\mbox
{and} \quad b_y^{(j)} \le\check k - 2\quad\mbox{for}\quad j\in\{
1,\ldots,m\}.
\end{aligned}
\end{equation}
\end{subequations}
Due to conditions~\reftext{\eqref{simple_n1v_cond}}, we are always in one of the following two cases:
\begin{enumerate}[(i)]
\item[(i)] $b_y^{(1)} \ge1$ and $ {\left
\lvert b^{(1)} \right\rvert} \ge2$,
\item[(ii)] $b_y^{(1)} = 0$, $b_y^{(2)} \ge1$, and
$b_x^{(1)} =  {\left\lvert b^{(1)} \right\rvert} \ge
 {\left\lvert b^{(2)} \right\rvert} \ge2$.
\end{enumerate}
If (i) holds true, then we infer that almost everywhere $D^{b^{(1)}} v =
O\left(x^2\right)$ and $D^{b^{(j)}} v = O(x)$ as $x \searrow0$
for $j \in\{2,\ldots,m\}$. This implies
\begin{equation*}
\left(x^{1-m} \bigtimes_{j = 1}^m D^{b^{(j)}} v\right)_2 = \left
(D^{b^{(1)}} v\right)_2 \times\bigtimes_{j = 2}^m \left(D^{b^{(j)}}
v\right)_1,
\end{equation*}
so that we may estimate the term \reftext{\eqref{n1v_d32_term}} according to
\begin{align}\nonumber
&\int_I  {\left\lVert x^{1-m} \bigtimes_{j =
1}^m D^{b^{(j)}} v - x^2 \left(x^{1-m} \bigtimes_{j = 1}^m
D^{b^{(j)}} v\right)_2 \right\rVert}_{\delta+
\frac3 2}^2 \mathrm{d} t \\
&\quad\lesssim \int_I  {\left\lVert\left(D^{b^{(1)}} v -
\left(D^{b^{(1)}} v\right)_2 x^2\right) \times x^{1-m} \bigtimes_{j
= 2}^m D^{b^{(j)}} v \right\rVert}_{\delta+\frac32}^2 \mathrm{d} t
\nonumber\\
&\quad\phantom{\lesssim} + \int_I  {\left\lVert\left(D^{b^{(1)}} v\right)_2
\times\left(x^{2-m} \bigtimes_{j = 2}^m D^{b^{(j)}} v - x \bigtimes
_{j = 2}^m \left(D^{b^{(j)}} v\right)_1\right) \right\rVert
}_{\delta+\frac12}^2 \mathrm{d} t. \label{n1v_d32_i}
\end{align}
We observe that for the second line of \reftext{\eqref{n1v_d32_i}} we have
\begin{align*}
&\int_I  {\left\lVert\left(D^{b^{(1)}} v -
\left(D^{b^{(1)}} v\right)_2 x^2\right) \times x^{1-m} \bigtimes
_{j = 2}^m D^{b^{(j)}} v \right\rVert}_{\delta
+\frac32}^2 \mathrm{d} t \\
&\quad\lesssim \int_I  {\left\lVert D^{b^{(1)}} v - \left
(D^{b^{(1)}} v\right)_2 x^2 \right\rVert}_{\delta+\frac3 2}^2
\mathrm{d} t \times\prod_{j = 2}^m {\skliaustask x^{-1} D^{b^{(j)}}
v\skliaustasd}_{BC^0\left(I \times(0,\infty) \times\mathbb
{R}\right)}^2 \\
&\quad \lesssim \int_I
{\left\lVert(D_x-3) (D_x-2) \tilde q(D_x) D_x v \right\rVert
}_{\check k + 2,\delta+\frac32}^2 \mathrm{d} t
\times \left(C {\skliaustask v\skliaustasd}_\mathrm{sol}\right)^{2(m-1)} \lesssim
\left(C{\skliaustask v\skliaustasd}_\mathrm{sol}\right)^{2m},
\end{align*}
where estimate~\reftext{\eqref{est_sup_n1v} has been used in the second-but-last step with a} constant $C$ sufficiently large under the additional constraint
\begin{equation}\label{constraint_bc0bounds}
\floor{\frac{\check k}{2}}+5 \le \min\left\{\tilde k-2, \check k - 2\right\}
\end{equation}
(see estimate~\reftext{\eqref{bc0_grad_init}} of \reftext{Lemma~\ref{lem:bc0_bounds}}) and
\reftext{Lemma~\ref{lem:mainhardy}} has been used
in the second-but-last step as well. The last line of \reftext{\eqref{n1v_d32_i}} can be estimated according to
\begin{align*}
& \int_I {\left\lVert\left(D^{b^{(1)}} v\right
)_2 \times\left(x^{2-m} \bigtimes_{j = 2}^m D^{b^{(j)}} v - x
\bigtimes_{j = 2}^m \left(D^{b^{(j)}} v\right)_1\right) \right
\rVert}_{\delta+\frac12}^2 \mathrm{d} t \\
& \quad\lesssim {\left\lVert\left(D^{b^{(1)}} v\right)_2
\right\rVert}_{L^2(I\times\mathbb{R})}^2
\times\sup_{t\in I, y\in \mathbb{R}}  {\left\lvert x^{2-m} \bigtimes_{j = 2}^m
D^{b^{(j)}} v - x \bigtimes_{j = 2}^m \left(D^{b^{(j)}} v\right)_1
\right\rvert}_{\delta+1}^2.
\end{align*}
For the first factor we get
\begin{equation*}
 {\left\lVert\left(D^{b^{(1)}} v\right)_2 \right\rVert}_{L^2(I \times\mathbb{R})}^2 =
\begin{cases}  {\left\lVert\partial_y v_1 \right\rVert
}_{L^2(I \times\mathbb{R})}^2 & \mbox{if}
\quad b_y^{(1)} = 1, \\
 {\left\lVert\partial_y^2v_0 \right\rVert}_{L^2(I
\times\mathbb{R})}^2 & \mbox{if} \quad
b_y^{(1)} = 2\mbox{ and }b_x^{(1)}=0,\\
0 & \mbox{else},
\end{cases}
\end{equation*}
so that this term is bounded by
${\skliaustask v\skliaustasd}_\mathrm{sol}^2$
in view of \reftext{\eqref{bc0_v1beta_v2}} of \reftext{Lemma~\ref{lem:bc0_bounds}}. For the second factor, by the triangle inequality we get
\begin{align*}
& \sup_{t\in I,y\in\mathbb{R}} {\left\lvert x^{2-m} \bigtimes_{j =
2}^m D^{b^{(j)}} v - x \bigtimes_{j = 2}^m \left(D^{b^{(j)}} v\right
)_1 \right\rvert}_{\delta+1}^2 \\
&\quad\lesssim \sup_{t\in I,y\in\mathbb{R}}  {\left\lvert D^{b^{(2)}} v - \left
(D^{b^{(2)}} v\right)_1 x \right\rvert}_{\delta+ 1}^2
 \times\prod_{j = 3}^m {\skliaustask x^{-1} D^{b^{(j)}}
v\skliaustasd}_{BC^0\left(I \times(0,\infty) \times\mathbb
{R}\right)}^2 \\
&\quad\phantom{\lesssim} +  {\left\lVert\left(D^{b^{(2)}} v\right)_1 \right
\rVert}_{BC^0(I\times\mathbb{R})}^2 \times
\sup_{t\in I, y\in\mathbb{R}} {\left\lvert x^{3-m} \bigtimes_{j = 3}^m
D^{b^{(j)}} v - x \bigtimes_{j = 3}^m \left(D^{b^{(j)}} v\right)_1
\right\rvert}_{\delta+1}^2.
\end{align*}
Then we note that, due to \reftext{\eqref{n1v_d32_cond}}, if $\left\lfloor
\frac{\check k}2\right\rfloor+5\le\check k$  (a condition which is already implied by \reftext{\eqref{constraint_bc0bounds}}), by using \reftext{Lemma~\ref{lem:mainhardy}} and \reftext{\eqref{supl2_auxiliary}} of \reftext{Lemma~\ref{lem:bc0_auxiliary}}, we find
\begin{eqnarray}
\sup_{t\in I, y\in\mathbb{R}}  {\left\lvert D^{b^{(2)}} v - \left(D^{b^{(2)}}
v\right)_1 x \right\rvert}_{\delta+ 1}^2&\lesssim& \sup_{t\in I, y\in\mathbb{R}}  {\left\lvert (D_x-1) D^{b^{(2)}} v \right\rvert}_{\delta+ 1}^2 \nonumber\\
&\stackrel{\text{\reftext{\eqref{supl2_auxiliary}}}}{\lesssim}&\sup_{t\in I} {\left\lVert (D_x-1) D^{b^{(2)}} v \right\rVert}_{1,\delta+ 1}^2 \nonumber\\
&\lesssim&\sup_{t\in I}  {\left\lVert (D_x-3)(D_x-2)\tilde q(D_x)D_x v \right\rVert}_{\check k, \delta+1}^2 \stackrel{\text{\reftext{\eqref{norm_sol}}}}{\lesssim} {\skliaustask
v\skliaustasd}_\mathrm{sol}^2, \label{sup_ty_b2}
\end{eqnarray}
where \reftext{Lemma~\ref{lem:mainhardy}} has been employed
in the forelast estimate, too. By \reftext{Lemma~\ref{lem:mainhardy}}
and using \reftext{\eqref{est_sup_n1v}}, the product
\begin{equation*}
\prod_{j = 3}^m {\skliaustask x^{-1} D^{b^{(j)}} v
\skliaustasd}_{BC^0\left(I \times (0,\infty) \times\mathbb{R}\right)}
\end{equation*}
is bounded by $\left(C{\skliaustask v\skliaustasd}_\mathrm{sol}\right)^{m-2}$ provided \reftext{\eqref{constraint_bc0bounds}} holds true,
where $C$ is a sufficiently large constant. Next, by
estimate~\reftext{\eqref{bc0_grad_sol}} of \reftext{Lemma~\ref{lem:bc0_bounds}}, we find
\begin{equation*}
 {\left\lVert\left(D^{b^{(2)}} v\right)_1 \right\rVert
}_{BC^0(I\times\mathbb{R})} \lesssim
 {\left\lVert(v_0)_y \right\rVert}_{BC^0(I \times
\mathbb{R})} +  {\left\lVert v_1 \right\rVert}_{BC^0(I
\times\mathbb{R})}
\lesssim{\skliaustask v\skliaustasd}_\mathrm{sol}.
\end{equation*}
Iterating the argument on the term
\begin{equation*}
\sup_{t\in I, y\in \mathbb{R}}   {\left\lvert x^{3-m} \bigtimes_{j = 3}^m
D^{b^{(j)}} v - x \bigtimes_{j = 3}^m \left(D^{b^{(j)}} v\right)_1
\right\rvert}_{\delta+1}^2,
\end{equation*}
we may upgrade \reftext{\eqref{n1v_d32_i}} (upon enlarging $C$) to
\begin{equation}\label{n1v_d32_final}
\int_I  {\left\lVert x^{1-m} \bigtimes_{j = 1}^m
D^{b^{(j)}} v - \left(x^{1-m} \bigtimes_{j = 1}^m D^{b^{(j)}}
v\right)_2 x^2 \right\rVert}_{\delta+
\frac3 2}^2 \mathrm{d} t \le\left(C {\skliaustask v\skliaustasd
}_\mathrm{sol}\right)^{2m}
\end{equation}
and this concludes the proof in case~(i).

Now suppose that we are in case~(ii). Here we have
\begin{equation*}
\left(x^{1-m} \bigtimes_{j = 1}^m D^{b^{(j)}} v\right)_2 = v_1
\times
\left(D^{b^{(2)}} v\right)_2 \times\bigtimes_{j = 3}^m \left
(D^{b^{(j)}} v\right)_1,
\end{equation*}
and we may estimate the term \reftext{\eqref{n1v_d32_term}} according to
\begin{align}\nonumber
&\int_I  {\left\lVert x^{1-m} \bigtimes_{j =
1}^m D^{b^{(j)}} v - x^2 \left(x^{1-m} \bigtimes_{j = 1}^m
D^{b^{(j)}} v\right)_2 \right\rVert}_{\delta+
\frac3 2}^2 \mathrm{d} t \\
&\quad\lesssim \int_I  {\left\lVert D^{b^{(1)}_x}_x (v - v_0 -
v_1x) \times x^{1-m} \bigtimes_{j = 2}^m D^{b^{(j)}} v \right\rVert
}_{\delta+ \frac3 2}^2 \mathrm{d} t \nonumber
\\
&\quad\phantom{\lesssim} + \int_I  {\left\lVert v_1 \left(x^{2-m} \bigtimes
_{j=2}^m D^{b^{(j)}} v - x^2 \left(D^{b^{(2)}} v\right)_2 \times
\bigtimes_{j = 3}^m \left(D^{b^{(j)}} v\right)_1\right) \right
\rVert}_{\delta+\frac3 2}^2 \mathrm{d} t. \label{n1v_d32_ii}
\end{align}

For the second line in \reftext{\eqref{n1v_d32_ii}} we use the decomposition
\begin{equation*}
D^{b^{(1)}_x}_x (v - v_0 - v_1 x) = D^{b^{(1)}_x}_x \underbrace{\left(v - v_0 -
v_1 x - v_{1+\beta} x^{1+\beta} - v_2 x^2\right)}_{=: w} + (1+\beta
)^{b_x^{(1)}} v_{1+\beta} x^{1+\beta} + 2^{b_x^{(1)}} v_2 x^2,
\end{equation*}
so that
\begin{align}
& \int_I  {\left\lVert D^{b^{(1)}_x}_x (v - v_0 -
v_1x) \times x^{1-m} \bigtimes_{j = 2}^m D^{b^{(j)}} v \right\rVert
}_{\delta+ \frac3 2}^2 \mathrm{d} t\nonumber\\
&\quad\lesssim \int_I  {\left\lVert D^{b^{(1)}_x}_x w \right
\rVert}_{\delta+\frac3 2}^2 \mathrm{d} t
\times\prod_{j = 2}^m {\skliaustask x^{-1} D^{b^{(j)}} v\skliaustasd
}_{BC^0\left(I
\times(0,\infty) \times\mathbb{R}\right)}^2 \nonumber\\
&\quad\phantom{\lesssim} + \int_I  {\left\lVert v_{1+\beta} D^{b^{(2)}} v
\right\rVert}_{\delta+ \frac3 2 -
\beta}^2 \mathrm{d} t \times\prod_{j = 3}^m {\skliaustask x^{-1}
D^{b^{(j)}} v\skliaustasd}_{BC^0\left(I \times(0,\infty) \times
\mathbb{R}\right)}^2 \nonumber\\
&\quad\phantom{\lesssim} + {\left\lVert v_2 \right\rVert
}_{L^2(I\times\mathbb{R})}^2 \times \sup_{t \in I, y \in \R}
 {\left\lvert D^{b^{(2)}} v \right\rvert}_{\delta+ 1}^2 \times\prod_{j = 3}^m
{\skliaustask x^{-1} D^{b^{(j)}} v\skliaustasd}_{BC^0\left(I\times
(0,\infty) \times\mathbb{R}\right)}^2. \label{bound_nv1_d+32_b}
\end{align}
The products
\begin{equation*}
\prod_{j = 2}^m {\skliaustask x^{-1} D^{b^{(j)}} v\skliaustasd
}_{BC^0\left(I \times
(0,\infty) \times\mathbb{R}\right)} \quad\mbox{and} \quad\prod
_{j = 3}^m
{\skliaustask x^{-1} D^{b^{(j)}} v\skliaustasd}_{BC^0\left(I \times
(0,\infty) \times\mathbb{R}
\right)}
\end{equation*}
are bounded by $\left(C{\skliaustask v\skliaustasd}_\mathrm{sol}\right)^{m-1}$
and $\left(C{\skliaustask v\skliaustasd}_\mathrm{sol}\right)^{m-2}$ with a sufficiently large constant $C$,
respectively, because of \reftext{\eqref{est_sup_n1v}}
requiring the constraint \reftext{\eqref{constraint_bc0bounds}}.
As in the context of \reftext{\eqref{est_n1v1b_2}} and \reftext{\eqref{parameters2}}, we may estimate
\begin{equation*}
\int_I  {\left\lVert v_{1+\beta} D^{b^{(2)}} v \right
\rVert}_{\delta+ \frac3 2 - \beta
}^2 \mathrm{d} t \lesssim{\skliaustask v\skliaustasd}_\mathrm{sol}^4,
\end{equation*}
under the additional constraints $\delta\in\left(0,\frac12\left
(\beta-\frac12\right)\right]$ and
\begin{equation}\label{parameters2_alt}
\left\lfloor\frac{\check k}{2}\right\rfloor+ 5 \le \min\left\{
\tilde k, \check k\right\},
\end{equation}
which is more restrictive than \reftext{\eqref{parameters2}}.
Finally, we have
${\left\lVert v_2 \right\rVert}_{L^2(I\times\mathbb
{R})}^2 \lesssim{\skliaustask v\skliaustasd}_\mathrm{sol}^2$
by estimate~\reftext{\eqref{bc0_v1beta_v2}} of
\reftext{Lemma~\ref{lem:bc0_bounds}} and $\sup_{t\in I, y \in \R}\left\lvert D^{b^{(2)}}v \right\rvert_{\delta+1}^2$ can be estimated as in \reftext{\eqref{sup_ty_b2}}, so that
\begin{equation}\label{bound_nv1_d+32_c}
\int_I  {\left\lVert D^{b^{(1)}_x}_x (v - v_0 - v_1x) \times
x^{1-m} \bigtimes_{j = 2}^m D^{b^{(j)}} v \right\rVert}_{\delta+
\frac3 2}^2 \mathrm{d} t \lesssim{\skliaustask v\skliaustasd
}_\mathrm{sol}^2.
\end{equation}

The last line in \reftext{\eqref{n1v_d32_ii}} can be bounded by
\begin{align*}
& \int_I  {\left\lVert v_1 \left(x^{2-m}
\bigtimes_{j=2}^m D^{b^{(j)}} v - x^2 \left(D^{b^{(2)}} v\right)_2
\times\bigtimes_{j = 3}^m \left(D^{b^{(j)}} v\right)_1\right)
\right\rVert}_{\delta+\frac3 2}^2 \mathrm{d} t \\
& \quad\lesssim {\left\lVert v_1 \right\rVert}_{BC^0(I
\times\mathbb{R})}^2 \times\int_I  {\left\lVert x^{2-m}
\bigtimes_{j=2}^m D^{b^{(j)}} v - x^2 \left(D^{b^{(2)}} v\right)_2
\times\bigtimes_{j = 3}^m \left(D^{b^{(j)}} v\right)_1 \right
\rVert}_{\delta
+\frac3 2}^2 \mathrm{d} t,
\end{align*}
where ${\skliaustask v\skliaustasd}_\mathrm{sol}$ bounds $
{\left\lVert v_1 \right\rVert}_{BC^0(I \times\mathbb{R}
)}$ by estimate~\reftext{\eqref{bc0_grad_sol}} of \reftext{Lemma~\ref{lem:bc0_bounds}} and
the term
\begin{equation*}
\int_I  {\left\lVert x^{2-m} \bigtimes_{j=2}^m
D^{b^{(j)}} v - x^2 \left(D^{b^{(2)}} v\right)_2 \times\bigtimes
_{j = 3}^m \left(D^{b^{(j)}} v\right)_1 \right\rVert}_{\delta
+\frac3 2}^2 \mathrm{d} t
\end{equation*}
can be treated as in case~(i).

Using these bounds in \reftext{\eqref{n1v_d32_ii}} concludes the proof of \reftext{\eqref{n1v_d32_final}} in case~(ii).

\medskip
\noindent\textit{Estimates for the weight $\rho= -\delta+ \frac5 2$.} 
In this last case of the proof of \reftext{\eqref{est_n1v}},
we start from equation~\reftext{\eqref{simple_n1v}},
after which we apply the $D$-operators coming from
the term $ {\left\lVert{D_y \tilde q(D_x-1) (D_x-1) 
\mathcal N}^{(1)}(v) \right\rVert}_{\breve k - 2,
-\delta+\frac52}$ which we need to estimate. After commuting $D_x$ and
$D_y$ derivatives with $x^{-1}$-factors, we find terms of the form\looseness=-1
\begin{equation*}
x^{1-m}\bigtimes_{j=1}^mD^{b^{(j)}}v,
\end{equation*}
which satisfy conditions~\reftext{\eqref{est_term_n1v_cond_1}}--\reftext{\eqref{est_term_n1v_cond_3}} and \reftext{\eqref{d52_restrictions}}.
Up to reordering terms, we are in one of the following two cases:

\begin{enumerate}[(i)]
\item[(i)] $b_y^{(1)} \ge 2$ and $\left\lvert b^{(1)}\right\rvert\ge 3$
\item[(ii)] $b_y^{(1)}\ge 1$, $\left\lvert b^{(1)}\right\rvert\ge 2$ and $b_y^{(2)}\ge 1$, $\left\lvert b^{(2)}\right\rvert\ge 2$.
\end{enumerate}
In case (i) we can estimate as follows
\begin{eqnarray*}
\int_I  {\left\lVert x^{1-m} \bigtimes_{j = 1}^m
D^{b^{(j)}} v \right\rVert}_{-\delta+\frac
52}^2 \mathrm{d} t &\le& \int_I  {\left\lVert D^{b^{(1)}}v
\right\rVert}_{-\delta+\frac52}^2 \mathrm{d} t
\times\prod_{j = 2}^m {\skliaustask x^{-1} D^{b^{(j)}} v\skliaustasd
}_{BC^0\left(I
\times(0,\infty) \times\mathbb{R}\right)}^2 \\
&\stackrel{\text{\reftext{\eqref{bc0_grad_sol}}}}{\le}& \int_I
{\left\lVert D_y \tilde q(D_x) D_x v \right\rVert}_{\breve k+ 2, -\delta+\frac52}^2
\mathrm{d} t \times \left(C{\skliaustask v\skliaustasd}_\mathrm
{sol}\right)^{2 (m-1)}
\le \left(C {\skliaustask v\skliaustasd
}_\mathrm{sol}\right)^{2m},
\end{eqnarray*}
where the constant $C$ is taken large enough and in the second estimate we have used \reftext{Lemma~\ref{lem:mainhardy}} and the fact that, due to conditions~\reftext{\eqref{est_term_n1v_cond_1}}--\reftext{\eqref{est_term_n1v_cond_3}}, the numbers of derivatives $D_x$ and $D_y$ in $b^{(1)}$ are dominated by the ones in $(1+b_x,2+b_y)$ for some $b=(b_x,b_y)\in \mathbb N_0^2$ with $\left\lvert b\right\rvert\le \breve k+1$.

\medskip

For case~(ii) we have
\begin{align*}
\int_I  {\left\lVert x^{1-m} \bigtimes_{j = 1}^m
D^{b^{(j)}} v \right\rVert}_{-\delta+\frac
52}^2 \mathrm{d} t &\le \int_I  {\left\lVert D^{b^{(1)}}v
\times D^{b^{(2)}}v \right\rVert}_{-\delta
+\frac72}^2 \mathrm{d} t \times\prod_{j = 3}^m {\skliaustask x^{-1}
D^{b^{(j)}} v\skliaustasd}_{BC^0\left(I \times(0,\infty) \times
\mathbb{R}\right)}^2 \\
&\le \skliaustask x^{-1}D^{b^{(1)}-(0,1)}v\skliaustasd_{BC^0(I\times(0,\infty)\times\mathbb{R})}\times\skliaustask x^{-1}D^{b^{(2)}-(0,1)}v\skliaustasd_{BC^0(I\times(0,\infty)\times\mathbb{R})}\\
&\phantom{\le} \times\int_I\left(\left\lVert D_yD^{b^{(1)}}v\right\rVert_{-\delta+\frac52}^2 +\left\lVert D_yD^{b^{(2)}}v\right\rVert_{-\delta+\frac52}^2\right)\mathrm{d}t \times \left(C \skliaustask v\skliaustasd_\mathrm{sol}\right)^{2(m-2)}\\
&\le \left(C\skliaustask v\skliaustasd_\mathrm{sol}\right)^{2m},
\end{align*}
where $C$ is a large enough constant and where for the second estimate we have used \reftext{\eqref{dy_dy}} of \reftext{Lemma~\ref{lem:interpolation_dy3}}.
Here, we have applied the bounds \reftext{\eqref{bc0_grad_sol}} of
\reftext{Lemma~\ref{lem:bc0_bounds}} for the first line,
which applies if $\left\lvert b^{(1)}\right\rvert-1\le\min\{\tilde k,\check k\}-2$,
and for the second line we have used conditions~\reftext{\eqref{est_term_n1v_cond_1}}--\reftext{\eqref{est_term_n1v_cond_3}}
together with the assumption that $\left\lvert b^{(j)}\right\rvert\ge 2$
for both $j=1,2$ and an application of \reftext{Lemma~\ref{lem:mainhardy}},
in order to estimate the terms by the contribution
$\int_I \left\lVert D_y \tilde q(D_x)D_x v\right\rVert_{\breve k+2,-\delta+\frac52}^{2} \d t$ from $\skliaustask v\skliaustasd_\mathrm{sol}^{2}$. 

\medskip

The conditions on indices under which the above bounds hold are
\begin{equation}\label{parameters_breve}
\breve k+6 \le \min\left\{\tilde k - 2,\check k -2\right\}.
\end{equation}
Under these conditions we find
\begin{equation}\label{est_n1v_case52}
\int_I  {\left\lVert x^{1-m} \bigtimes_{j = 1}^m
D^{b^{(j)}} v \right\rVert}_{-\delta+\frac
52}^2 \mathrm{d} t \le \left(C{\skliaustask v\skliaustasd}_\mathrm{sol}\right)^{2m},
\end{equation}
for a large enough constant $C$.

\medskip

\textit{Estimate of ${\skliaustask{\mathcal N}^{(2)}(v)\skliaustasd
}_\mathrm{rhs}$.}
We now desire to prove the analogue of \reftext{\eqref{est_n1v}} for $\mathcal N^{(2)}(v)$, namely that
\begin{equation}\label{est_n2v}
{\skliaustask{\mathcal N}^{(2)}(v)\skliaustasd}_\mathrm{rhs}
\lesssim{\skliaustask v\skliaustasd}_\mathrm
{sol}^2\quad\mbox{for} \quad{\skliaustask v\skliaustasd}_\mathrm
{sol} \ll1.
\end{equation}
In order to prove the above, for each of the four lines of
\reftext{\eqref{simple_n2_terms}}, we obtain terms with the basic structure \reftext{\eqref{simple_n1v}},
and we distribute derivatives and weights
as in each of the six addends from \reftext{\eqref{norm_rhs}}.

\medskip

\noindent\textit{Estimates for $\rho\in\left\{-\delta\pm\frac12\right\}$.}
For the first two terms in the sum
\begin{equation*}
 {\left\lVert {\mathcal N}^{(2)}(v) \right\rVert
}_{k-2,-\delta-\frac12}^2 + {\left\lVert(D_x-1) {\mathcal N}^{(2)}(v) \right\rVert
}_{\tilde k-2,-\delta+\frac12}^2
\end{equation*}
appearing in \reftext{\eqref{norm_rhs}}, we may perform the same discussion as
the one used in treating the analogous expressions
for ${\mathcal N}^{(1)}(v)$ in
those two cases. Note that the addends appearing in ${\skliaustask
{\mathcal N} ^{(1)}(v)\skliaustasd}_\mathrm{rhs}^{2}$
with weights $\rho= -\delta\pm\frac12$ have been
treated in the same way in both cases (i) and
(ii) from that discussion. By splitting as in \reftext{\eqref{est_n1v_case_1}},
the estimates thereafter, leading to the bound \reftext{\eqref{est_n1v_case_1_2}}, apply
and give the desired bound
\begin{equation}\label{est_n2_-d12}
\int_I \left( {\left\lVert {\mathcal N}^{(2)}(v) \right\rVert
}_{k-2,-\delta-\frac12}^2 + {\left\lVert(D_x-1) {\mathcal N}^{(2)}(v) \right\rVert
}_{\tilde k-2,-\delta+\frac12}^2 \right) \mathrm{d} t
\lesssim \left(C{\skliaustask v\skliaustasd}_\mathrm{sol}\right)^{2m}
\end{equation}
for a constant $C$.

\medskip

\textit{Estimate of ${\skliaustask{\mathcal N}^{(2)}(v)\skliaustasd
}_\mathrm{rhs}$ for $\rho
\in\left\{\delta+\frac12, \pm\delta+\frac32\right\}$.}
We next bound the parts of the norm
with weight $\rho\in\{\delta+\frac12,\pm\delta+\frac32\}$, namely
\begin{subequations}\label{remainstoest_N2}
\begin{align}
\int_I {\left\lVert\tilde q(D_x-1)(D_x-1){\mathcal
N}^{(2)}(v) \right\rVert}_{\tilde k-2,\delta
+\frac12}^2\mathrm{d} t, \label{remains_n2_1}
\eqncr
\int_I {\left\lVert\tilde q(D_x-1)(D_x-1){\mathcal
N}^{(2)}(v) \right\rVert}_{\check k-2, -\delta
+\frac32}^2\mathrm{d} t, \label{remains_n2_2}
\end{align}
and
%
%
\begin{equation}\label{remains_n2_3}
\int_I {\left\lVert(D_x-4)(D_x-3)\tilde q(D_x-1)(D_x-1)
{\mathcal N} ^{(2)}(v) \right\rVert}_{\check k -2,\delta+\frac
32}^2\mathrm{d} t.
\end{equation}
\end{subequations}
If in distributing derivatives from \reftext{\eqref{remainstoest_N2}}
at least one $D_y$-derivative acts on $\mathcal N^{(2)}(v)$,
then the resulting expression decays near $x=0$ at least like $O(x^2)$
and we see that the precise expansion as in \reftext{\eqref{simple_n2_terms}} is immaterial.
In view of definition~\reftext{\eqref{defn2}} for $\NN^{(2)}(v)$,
we recognize that
the structure \reftext{\eqref{simple_n1v_cond}} applies in this case.
Thus we may treat the ensuing terms as in the discussion
of the analogous ones from $\mathcal N^{(1)}(v)$,
allowing to bound each expression by $\left(C\skliaustask v\skliaustasd_\mathrm{sol}\right)^{2m}$,
where $C$ is a large enough constant like in \reftext{\eqref{est_n1v_case_1_2}}.

\medskip

We thus only need to discuss the contributions from \reftext{\eqref{remainstoest_N2}} in which only $D_x$-derivatives act on the expressions in the norms. For example, corresponding to \reftext{\eqref{remains_n2_1}}, we will consider the terms 
\begin{equation*}
\int_I \sum_{j = 0}^{\tilde k-2} \left\lVert D^j_x\tilde q(D_x-1)(D_x-1)\mathcal N^{(2)}(v)\right\rVert_{\delta+\frac12}^2\mathrm{d}t.
\end{equation*}
\textit{The $D_x$-derivative contributions from  \reftext{\eqref{remainstoest_N2}}}. We now consider separately the contributions coming from the four lines in \reftext{\eqref{simple_n2_terms}}.

\medskip

\textit{The first line of \reftext{\eqref{simple_n2_terms}} for
$\rho\in\left\{\delta+\frac12,\pm\delta+\frac32\right\}$.}
In this case we have an expression of the form
\begin{equation}\label{first_line_n2}
-\frac38(1+v_1)^{-3}\left(6v_1^2 + 8v_1^3 +3v_1^4
+2(v_0)_y^2+(v_0)_y^4\right)\ x,
\end{equation}
which is in the kernel of the operators appearing in \reftext{\eqref{remainstoest_N2}}, and thus the corresponding contribution disappears.

\medskip
\noindent\textit{The second line of \reftext{\eqref{simple_n2_terms}} for
$\rho\in\left\{\delta+\frac12,\pm\delta+\frac32\right\}$.}
Here, we need to consider \reftext{\eqref{remainstoest_N2} with $\NN^{(2)}(v)$ replaced by}
\begin{eqnarray*}
f&=&(1+v_1)^{-3}\left(4v_1+6v_1^2+4v_1^3+v_1^4 -2(v_0)_y^2 -
(v_0)_y^4\right)q(D_x)\phi\\
&\stackrel{\text{\reftext{\eqref{notation_nonlin}}}}{=}&(1+v_1)^{-4}\left
(4v_1+6v_1^2+4v_1^3+v_1^4 -2(v_0)_y^2 - (v_0)_y^4\right
)q(D_x)(v-v_0-v_1x).\qquad
\end{eqnarray*}
Notice that $x^{1+\beta}$ is in the kernel of $q(D_x)$ and
that the operator $\tilde q(D_x-1)(D_x-1)$
appearing in the norms \reftext{\eqref{remainstoest_N2}} figures the factor
$(D_x-2)$ (cf.~\reftext{\eqref{poly_q2}}) canceling the $x^2$-term of the
expansion of $v$ near $x=0$ (cf.~\reftext{\eqref{expansion_v}}).
Hence, we consider the following expression, in which the singular expansion up to and including the addend $\sim x^2$ vanishes:
\begin{equation}\label{second_line_n2_f}
f = (1+v_1)^{-4} \, \left(4v_1+6v_1^2+4v_1^3+v_1^4 -2(v_0)_y^2 -
(v_0)_y^4\right) \, q(D_x) (D_x-2) (v-v_0-v_1x).
\end{equation}
After distributing the remaining $D_x$-derivatives,
so that it suffices to estimate
\begin{subequations}\label{second_line_n2_g}
\begin{equation}
\int_I {\left\lVert g^{(2)} \right\rVert}_\rho^2\mathrm
{d} t,
\end{equation}
for $\rho\in\left\{\delta+\frac12,\pm\delta+\frac32\right\}$ and for
\begin{equation}\label{second_line_n2_defg}
g^{(2)} = \, (1+v_1)^{-4} \times w^m\, \times D^b_x q(D_x)(D_x-2) (v-v_0-v_1x),
\end{equation}
where $w\in\left\{v_1,\left(v_0\right)_y\right\}$,
$m\in\left\{1,2,3,4\right\}$,
and $b\in\N_0$ satisfying the following bounds:
\begin{equation}
b \le
\begin{cases} \tilde k+2 & \mbox{for}\quad\rho=\delta+\frac12, \\
\check k+2 & \mbox{for} \quad\rho=-\delta+\frac32,\\ \check k + 4 &
\mbox{for} \quad\rho=\delta+\frac32.
\end{cases}
\end{equation}
\end{subequations}
We derive bounds for the individual factors in \reftext{\eqref{second_line_n2_f}}
separately. 

We first concentrate on estimating the first two factors in \reftext{\eqref{second_line_n2_defg}}. Via \reftext{\eqref{bc0_grad_sol}} of \reftext{Lemma~\ref{lem:bc0_bounds}}, in case ${\skliaustask
v\skliaustasd}_\mathrm{sol} \le C^{-1}$  for a constant $C > 0$ only determined by
the implicit constant in \reftext{\eqref{bc0_grad_sol}}, we have
\begin{equation}\label{first_line_n2_1}
{\left\lVert(1+v_1)^{-4}\right\rVert}_{BC^0\left(I \times\mathbb{R}\right
)} \le \left(1 - \tfrac C 2 \skliaustask v\skliaustasd_\mathrm{sol}\right)^{-4} \lesssim 1.
\end{equation}
Next, via \reftext{\eqref{bc0_grad_sol}} of \reftext{Lemma~\ref{lem:bc0_bounds}} and upon enlarging $C$,
for $\skliaustask v\skliaustasd_\mathrm{sol}\le C^{-1}$
we have the bound
\begin{equation}\label{second_line_n2_2}
 {\left\lVert w^m \right\rVert}_{BC^0(I\times\mathbb R)}^2 \le \left\lVert w \right\rVert_{BC^0\left(I \times\mathbb{R}\right)}^{2m} \le \left(C {\skliaustask v\skliaustasd}_\mathrm{sol}\right)^{2m} \le {\skliaustask v\skliaustasd}_\mathrm{sol}^2.
\end{equation}
If $\skliaustask v \skliaustasd_\mathrm{sol} \le C^{-1}$ then \reftext{\eqref{first_line_n2_1}} and \reftext{\eqref{second_line_n2_2}} directly lead to
\begin{equation}\label{second_line_n2_1}
{\left\lVert(1+v_1)^{-4} w^m\right\rVert}_{BC^0\left(I  \times\mathbb{R}\right
)} \lesssim{\skliaustask v\skliaustasd}_\mathrm{sol}^2.
\end{equation}
For the last factor of \reftext{\eqref{second_line_n2_defg}} we use the fact that $0$, $1$, $1+\beta$, and $2$ are
zeros of $q(\zeta)(\zeta-2)$ (cf.~\reftext{\eqref{poly_q}}) and thus $q(D_x) (D_x-2)$ vanishes on $1$, $x$, $x^{1+\beta}$, and $x^2$. This leads
to
\begin{align} 
\int_I {\left\lVert D^b_x q(D_x)
(D_x-2) (v - v_0-v_1x) \right\rVert}_\rho^2 \mathrm{d} t
&\lesssim \begin{cases} \int_I {\left\lVert\tilde q(D_x)D_xv \right
\rVert}_{\tilde k +2,\delta+\frac
12}^2\mathrm{d} t & \mbox{for} \quad \rho = \delta + \frac12 \\
\int_I {\left\lVert\tilde q(D_x)D_xv \right
\rVert}_{\check k +2,\delta+\frac
32}^2\mathrm{d} t & \mbox{for} \quad \rho = -\delta + \frac32 \\
\int_I {\left\lVert(D_x-3)(D_x-2)\tilde q(D_x)D_x v \right
\rVert}_{\check k +2,\delta+\frac
32}^2\mathrm{d} t & \mbox{for} \quad \rho = \delta + \frac 32\end{cases}  \nonumber \\
&\lesssim {\skliaustask v\skliaustasd}_\mathrm{sol}^2, \label{second_line_n2_3}
\end{align}
which uses the bound on $b$
in \reftext{\eqref{second_line_n2_g}} as well
as \reftext{Lemma~\ref{lem:mainhardy}}.
Estimates~\reftext{\eqref{second_line_n2_1}} and \reftext{\eqref{second_line_n2_3}} allow to obtain for $\rho\in\left\{\delta+\frac12,\pm\delta+\frac32\right\}$, under the condition that $\skliaustask v \skliaustasd_\mathrm{sol} \le C^{-1}$,
\begin{equation}\label{second_line_n2_bound}
\int_I {\left\lVert g^{(2)} \right\rVert}_\rho^2\mathrm
{d} t \lesssim{\skliaustask v\skliaustasd}_\mathrm
{sol}^4
\end{equation}
for the terms \reftext{\eqref{second_line_n2_g}}.

\medskip
\noindent\textit{The third line of \reftext{\eqref{simple_n2_terms}} 
for $\rho\in\left\{\delta+\frac12,\pm\delta+\frac32\right\}$.}
We use $\psi\stackrel{\text{\reftext{\eqref{notation_nonlin}}}}{=} v-v_0$ in the third
line of \reftext{\eqref{simple_n2_terms}} as well as the commutation relation
$(D_x-2)D_y=D_y(D_x-1)$. Arguing analogously to the previous proof step, we end up considering terms of the form
%
%
\begin{equation}\label{third_line_n2_f}
(1+v_1)^{-3} \, \left((v_0)_y^3
\left(D_x + \tfrac12\right
)\left
(D_x^2 - \tfrac32D_x - \tfrac58\right)
+ (v_0)_y\left(D_x^2 - \tfrac
54 D_x
- \tfrac54\right)\right) \, D_y (D_x - 1)(v-v_0).
\end{equation}
Note that in $D_y(D_x-1) (v-v_0)$ the
contribution $v_1 x$ in the expansion \reftext{\eqref{expansion_v}} of $v$ near
$x = 0$ is cancelled and $D_y$ provides an extra $x$-factor.
The above term essentially requires the same treatment as the one
from the second line of \reftext{\eqref{simple_n2_terms}},
which we have just treated,
with the difference that $g^{(3)}$, which appears in the analogue of \reftext{\eqref{second_line_n2_g}}, has now the form%
\begin{subequations}
\begin{equation}\label{third_line_n2_defg}
\quad g^{(3)} = \, (1+v_1)^{-3} \, \left( v_0 \right)_y^m\, \times D_x^b D_y (D_x-1)(v-v_0),
\end{equation}
with $m\in\{1,3\}$ and
\begin{equation}\label{third_line_n2_newm'}
 b \le
\begin{cases} \tilde k + 5 & \mbox{for} \quad\rho=\delta+\frac12, \\
\check k + 5 & \mbox{for} \quad\rho=-\delta+\frac32, \\ \check k +7 &
\mbox{for} \quad\rho=\delta+\frac32.
\end{cases}
\end{equation}
\end{subequations}
The only other difference to the study of \reftext{\eqref{second_line_n2_g}} lies
in the fact that the last factor contains $D_y(D_x-1)(v-v_0)$. Therefore, we have to replace the bounds \reftext{\eqref{second_line_n2_3}} by
\begin{align} 
\int_I {\left\lVert D^b_xD_y(D_x-1)(v
- v_0) \right\rVert}_{\delta+\frac32}^2\mathrm{d} t &\lesssim \begin{cases} \int_I {\left\lVert\tilde q(D_x)D_xv \right
\rVert}_{\tilde k +2,\delta+\frac
12}^2\mathrm{d} t & \mbox{for} \quad \rho = \delta + \frac12, \\
\int_I {\left\lVert\tilde q(D_x)D_xv \right
\rVert}_{\check k +2,\delta+\frac
32}^2\mathrm{d} t & \mbox{for} \quad \rho = -\delta + \frac32, \\
\int_I {\left\lVert(D_x-3)(D_x-2)\tilde q(D_x)D_x v \right
\rVert}_{\check k +2,\delta+\frac
32}^2\mathrm{d} t & \mbox{for} \quad \rho = \delta + \frac 32\end{cases}  \nonumber \\
&\lesssim {\skliaustask v\skliaustasd}_\mathrm{sol}^2, \label{third_line_n2_3}
\end{align}
where \reftext{\eqref{third_line_n2_newm'}}
and \reftext{Lemma~\ref{lem:mainhardy}}
have been used.
By combining estimate~\reftext{\eqref{third_line_n2_3}}
with the same bounds as in \reftext{\eqref{second_line_n2_1}}
for the remaining factors from \reftext{\eqref{third_line_n2_defg}},
we find
\begin{equation}\label{third_line_n2_bound}
\int_I {\left\lVert g^{(3)} \right\rVert}_\rho^2
\lesssim
{\skliaustask v\skliaustasd}_\mathrm{sol}^4,
\end{equation}
valid for $\rho\in\{\delta+\tfrac12,\pm\delta+\tfrac32\}$ if
$\skliaustask v \skliaustasd_\mathrm{sol} \le C^{-1}$. 

\medskip
\noindent\textit{The last line in \reftext{\eqref{simple_n2_terms}} for
$\rho \in \left\{
\delta+\tfrac12,\pm\delta+\tfrac32\right\}$.}
Here we have to treat
\[
(1+v_1)^{-3}\sum_{(\mu,\nu,\tau)\in{\mathcal I}}{\mathcal
M}\left((\psi_y^{\mu
_j}(v_0)_y^{\nu_j}\phi_x^{\tau_j})_{j=1}^4\right)
\]
with functions $\psi$ and $\phi$ defined as in \reftext{\eqref{notation_nonlin}}.
Note that $(v_0)_y^{\nu_j}$ can be factored out,
since definition~\reftext{\eqref{def_4lin}} of ${\mathcal M}$
includes only derivatives $D_x$ and products commuting with $(v_0)_y^{\nu_j}$. Similarly, $\phi$ contains an additional
factor $(1+v_1)^{-1}$, which can be factored out as well.
Like in the previous cases,
and after introducing $m-1$ factors $x$ and commuting derivatives,
we find that we have to estimate
$\int _I {\left\lVert g^{(4)} \right\rVert}_\rho^2\mathrm{d} t$
with
\begin{subequations}\label{4th_line_n2_g}
\begin{equation}\label{4th_line_n2_g_def}
g^{(4)} = (1+v_1)^{-3-\lvert\tau\rvert} \, \left(v_0\right)_y^{\lvert\nu\rvert} x^{1-m} \, \bigtimes_{j=1}^m D^{b^{(j)}}_xw^{(j)},
\end{equation}
where $m=\lvert\mu\rvert+\lvert\tau\rvert\ge2$,
${\left\lvert\mu\right\rvert}+{\left\lvert\nu\right\rvert}\le 4$,
and we have
\begin{equation}\label{4th_line_n2_g1}
w^{(j)} \in \left\{D_x(v-v_0-v_1x), D_y(v-v_0)\right\} \quad \mbox{for} \quad 1\le j\le m.
\end{equation}
The indices appearing in \reftext{\eqref{4th_line_n2_g_def}} are
constrained as follows
\begin{equation}\label{4th_line_n2_indices}
b^{(1)} \le
\begin{cases}
\tilde k + 6 & \mbox{for} \quad\rho=\delta+\frac12,\\
\check k + 6 & \mbox{for} \quad\rho=-\delta+\frac32,\\
\check k + 8 &\mbox{for} \quad\rho=\delta+\frac32,
\end{cases}
 \quad \mbox{and} \quad  b^{(j)} \le
\begin{cases}
\left\lfloor\frac{\tilde k}2\right\rfloor+3 & \mbox{for} \quad
\rho=\delta+\frac12,\\
\left\lfloor\frac{\check k }2\right\rfloor+3 & \mbox{for} \quad
\rho=-\delta+\frac32,\\
\left\lfloor\frac{\check k }2\right\rfloor+4 &\mbox{for} \quad
\rho=\delta+\frac32,
\end{cases}
\quad \mbox{for} \quad j\ge 2.
\end{equation}
\end{subequations}
We estimate the factors $(1+v_1)^{-1}$ and $(v_0)_y$ from \reftext{\eqref{4th_line_n2_g_def}}
as in \reftext{\eqref{first_line_n2_1}}--\reftext{\eqref{second_line_n2_1}},
under the condition that
${\skliaustask v\skliaustasd}_\mathrm{sol} \le C^{-1}$. 
For the remaining product term in \reftext{\eqref{4th_line_n2_g_def}}, we first
estimate all but the first two terms in the $BC^0(I\times(0,\infty
)\times\mathbb{R})$-norm for $\rho\in\left\{\delta
+\frac12,\pm\delta+\frac32\right\}$ as follows:
\begin{eqnarray}\label{common_est_4th_line}
\int_I {\left\lVert x^{1-m} \bigtimes
_{j=1}^m D^{b^{(j)}}_x w^{(j)} \right\rVert}_\rho^2\mathrm
{d} t &\lesssim& \int_I {\left\lVert x^{-1}
D^{b^{(1)}}_xw^{(1)}\times D^{b^{(2)}}_xw^{(2)} \right\rVert
}_\rho^2\mathrm{d} t\times\prod_{j=3}^m {\skliaustask
x^{-1} D^{b^{(j)}}_xw^{(j)}\skliaustasd}_{BC^0(I\times(0,\infty)\times
\mathbb{R})}^2\nonumber\\
&\lesssim& \int_I {\left\lVert 
x^{-1} D^{b^{(1)}}_xw^{(1)}\times D^{b^{(2)}}_xw^{(2)} \right\rVert
}_\rho^2\mathrm{d} t \times \left(\tfrac C 2 {\skliaustask v\skliaustasd}_\mathrm{sol}\right)^{2(m-2)},
\end{eqnarray}
where $C > 0$ is chosen sufficiently large.
In order to justify the above estimate,
we observe that $\psi_y=v_y-(v_0)_y$ and
$(v-v_0-v_1x)_x = v_x - v_1$,
after which we note that the
number of $D_x$-derivatives acting on the $w^{(j)}$ with $j \in \{3,\ldots,m\}$
is bounded as in \reftext{\eqref{4th_line_n2_indices}}. Therefore,
we may use the bound \reftext{\eqref{bc0_grad_sol}}
of \reftext{Lemma~\ref{lem:bc0_bounds}} provided
\begin{equation*}
\max\left\{\left\lfloor\frac{\tilde k }2\right\rfloor+4,\left
\lfloor\frac{\check k }2\right\rfloor+5\right\} \le\min\left\{
\tilde k - 2, \check k - 2\right\},
\end{equation*}
which is already implied by \reftext{\eqref{parameterss1}} and \reftext{\eqref{constraint_bc0bounds}}.

\medskip

We next bound for $\rho\in\left\{\delta+\frac1 2,
\pm\delta+ \frac32\right\}$ as follows:
\begin{equation}\label{4th_line_n2_3}
\int_I {\left\lVert x^{-1} D^{b^{(1)}}_xw^{(1)}\times D^{b^{(2)}}_xw^{(2)} \right\rVert}_\rho^2\mathrm{d}
t\lesssim{\skliaustask v\skliaustasd}_\mathrm{sol}^4.
\end{equation}
Once \reftext{\eqref{4th_line_n2_3}} is proved,
combining it with \reftext{\eqref{common_est_4th_line}}
and with the estimates for the factors $(1+v_1)^{-1}$
and $(v_0)_y$, under the condition
$\skliaustask v\skliaustasd_\mathrm{sol}\le C^{-1}$,
we find the bound
\begin{equation}\label{4th_line_n2_final}
\int_I {\left\lVert g^{(4)} \right\rVert}_\rho^2 \mathrm{d} t
\lesssim {\skliaustask v\skliaustasd}^4_\mathrm{sol} \times \left(\tfrac C 2 {\skliaustask v\skliaustasd}_\mathrm{sol}\right)^{2(m-2)} \times \left( 1 - \tfrac C 2 {\skliaustask v\skliaustasd}_\mathrm{sol}\right)^{-2(\lvert\tau\rvert+3)} \lesssim
\skliaustask v\skliaustasd_\mathrm{sol}^4
\left(C{\skliaustask v\skliaustasd}_\mathrm{sol}\right)^{2 (m-2)}
\end{equation}
for the weights $\rho\in\left\{\delta+\tfrac12,\pm\delta+\tfrac32\right\}$,
where we have also used that ${\left\lvert\tau\right\rvert}\le m$.
In order to prove the bound \reftext{\eqref{4th_line_n2_3}} and to complete
the proof of \reftext{\eqref{4th_line_n2_final}}, we will separately consider
the three values of $\rho$.

\medskip

\noindent\textit{The bound \reftext{\eqref{4th_line_n2_3}} for $\rho=\delta
+\frac12$.}
We study at the same time the two possible values of $w^{(1)}$ given in \reftext{\eqref{4th_line_n2_g1}}. Note that the
first factor entering the norm in \reftext{\eqref{4th_line_n2_3}} fulfills either
$D^{b^{(1)}+1}_x (v-v_0-v_1x) = O\left(x^{1+\beta}\right)$ or
$D^{b^{(1)}}_x D_y(v-v_0)=O\left(x^2\right)$
as $x\searrow0$
almost everywhere. Since $2>1+\beta
>1+\delta$
we may estimate according to
\begin{align}
&\int_I{\left\lVert x^{-1}
D^{b^{(1)}}_xw^{(1)}{\times}\noxml{\,} D^{b^{(2)}}_x w^{(2)} \right\rVert
}_{\delta+\frac 1 2}^2\mathrm{d} t \nonumber\\
& \quad \lesssim \int_I\!{\left\lVert
D^{b^{(1)}}_xw^{(1)} \right\rVert}_{\delta+\frac
12}^2\mathrm{d} t{\times}\noxml{\,}{\skliaustask
x^{-1} D^{b^{(2)}}_xw^{(2)}\skliaustasd}_{BC^0(I\times(0,\infty)_x\times
\mathbb{R}_y)}^2 \nonumber\\
&\quad\lesssim
\begin{cases} \int_I {\left\lVert
D^{b^{(1)}+1}_x (v-v_0-v_1x) \right\rVert}_{\delta+\frac
12}^2\mathrm{d} t\times{\skliaustask v\skliaustasd}_\mathrm
{sol}^2&\mbox{for}\quad
w^{(1)}= D_x(v-v_0-v_1x)\\
\int_I{\left\lVert D^{b^{(1)}}_x D_y(v-v_0) \right\rVert
}_{\delta+\frac12}^2\mathrm{d} t\times{\skliaustask v\skliaustasd
}_\mathrm{sol}^2&\mbox{for}\quad w^{(1)}=D_y(v-v_0)
\end{cases}
\nonumber\\
&\quad\lesssim \int_I {\left\lVert\tilde q(D_x)D_xv \right
\rVert}_{\tilde k + 2,\delta+\frac
12}^2\mathrm{d} t\times{\skliaustask v\skliaustasd}_\mathrm
{sol}^2\le{\skliaustask v\skliaustasd}_\mathrm{sol}^4,\label{bound_n2_d+12}
\end{align}
where we have used estimate~\reftext{\eqref{bc0_grad_sol}} of
\reftext{Lemma~\ref{lem:bc0_bounds}}
in the third line and \reftext{Lemma~\ref{lem:mainhardy}} for
the final estimate.
This is justified if $\floor{\frac{\tilde k}{2}} + 4\le \min\left\{\tilde k -2,\check k -2\right\}$,
which is already implied by \reftext{\eqref{parameterss1}}.
This concludes the bound of the fourth line
of \reftext{\eqref{simple_n2}} for $\rho=\delta+\tfrac12$.

\medskip
\noindent\textit{The bound \reftext{\eqref{4th_line_n2_3}} for $\rho= -\delta
+ \frac32$.}
We first consider the case $w^{(1)}=D_x(v-v_0-v_1x)$ and write
\begin{equation}\label{sum_subtract_v1beta}
v-v_0-v_1x=\underbrace{(v-v_0-v_1x-v_{1+\beta}x^{1+\beta})}_{=:
w}+v_{1+\beta}x^{1+\beta}.
\end{equation}
Using the triangle inequality and bounds similar to the previous proof
step, we may estimate
\begin{align}\label{-d32_all}\nonumber
& \int_I {\left\lVert x^{-1} D^{b^{(1)}}_x w^{(1)}\times D^{b^{(2)}}_x w^{(2)} \right\rVert 
}_{-\delta+\frac32}^2\mathrm{d} t \\
& \quad \lesssim \int_I {\left\lVert D^{b^{(1)}+1}_x w
\right\rVert}_{-\delta+\frac32}^2\mathrm{d}
t\times{\skliaustask x^{-1}D^{b^{(2)}}_x w^{(2)}\skliaustasd
}_{BC^0(I\times(0,\infty)\times\mathbb{R})}^2 + \int_I {\left\lVert x^{-1}D^{b^{(1)}}_x v_{1+\beta} x^{1+\beta} \times  D^{b^{(2)}}_x w^{(2)}
\right\rVert}_{-\delta+\frac32}^2\mathrm{d} t\nonumber\\
&\quad\lesssim\int_I {\left\lVert D^{b^{(1)}+1}_x w \right
\rVert}_{-\delta+\frac32}^2\mathrm{d}
t\times{\skliaustask v\skliaustasd}_\mathrm{sol}^2 + \int
_I {\left\lVert v_{1+\beta} \times D^{b^{(2)}}_x w^{(2)} \right\rVert}_{-\delta+\frac32-\beta}^2\mathrm{d} t.
\end{align}
The first term above can be estimated by using \reftext{Lemma~\ref{lem:mainhardy}}, which applies because $w=O\left(x^2\right)$ as $x\searrow0$ almost everywhere,
leading to a finite $ {\left\lvert\cdot\right\rvert
}_{-\delta+2}$-norm in the definition
of $ {\left\lVert\cdot\right\rVert}_{-\delta+\frac
32}$. We also note that the
polynomial $\tilde q(\zeta)\zeta$ has roots $0$, $1$, and $1+\beta$
(cf.~\reftext{\eqref{poly_q2}}) and thus $\tilde q(D_x)D_xw=\tilde q(D_x)D_xv$.
As a result we have
\begin{equation}\label{-d32_1}
\int_I {\left\lVert D^{b^{(1)}+1}_x w \right\rVert
}_{-\delta+\frac32}^2\mathrm{d} t \lesssim \int
_I {\left\lVert\tilde q(D_x)D_xw \right\rVert}_{\check k
+2,-\delta+\frac32}^2\mathrm{d} t=\int
_I {\left\lVert\tilde q(D_x)D_xv \right\rVert}_{\check k
+2,-\delta+\frac32}^2\mathrm{d} t \le 
{\skliaustask v\skliaustasd}_\mathrm{sol}^2.
\end{equation}
For the second term in \reftext{\eqref{-d32_all}} we estimate similarly to \reftext{\eqref{est_n1v1b_2}} with $m=2$:
\begin{align}\label{-d32_21}
& \int_I {\left\lVert v_{1+\beta} D^{b^{(2)}}_xw^{(2)} \right
\rVert}_{-\delta+\frac32-\beta}^2\mathrm{d} t\nonumber
\\
&\quad\lesssim {\left\lVert v_{1+\beta} \right\rVert}^2_{L^2\left(I;BC^0(\mathbb{R})\right)} \times \sup_{t\in I}  {\left\lVert
D^{b^{(2)}}_x w^{(2)} \right\rVert}_{-\delta
+\frac32-\beta}^2\nonumber\\
&\quad\stackrel{\text{\reftext{\eqref{bc0_v1beta_v2}}}}{\lesssim} {\skliaustask
v\skliaustasd}_\mathrm{sol}^2 
\times
\begin{cases} 
\sup_{t\in I} {\left\lVert D^{b^{(2)}+1}_x (v-v_0-v_1x) \right\rVert}_{-\delta+\frac32-\beta}^2 & \mbox{for} \quad
w^{(2)}=D_x(v-v_0-v_1x),\\
\sup_{t\in I} {\left\lVert D^{b^{(2)}}_x D_y(v-v_0)
\right\rVert}_{-\delta+\frac32-\beta}^2 &\mbox{for}\quad w^{(2)}=D_y(v-v_0).
\end{cases}\nonumber\\
&\quad\lesssim{\skliaustask
v\skliaustasd}_\mathrm{sol}^2 
\times
\begin{cases} 
\sup_{t\in I} {\left\lVert D^{b^{(2)}}_x(D_x-1) v\right\rVert}_{-\delta+\frac32-\beta}^2 & \mbox{for} \quad
w^{(2)}=D_x(v-v_0-v_1x),\\
\sup_{t\in I} {\left\lVert D^{(b^{(2)}-1)_+}_x (D_x-1)D_yv
\right\rVert}_{-\delta+\frac32-\beta}^2 &\mbox{for}\quad w^{(2)}=D_y(v-v_0),
\end{cases}
\end{align}
where Lemma~\ref{lem:bc0_bounds} has been employed.
In order to estimate the remaining terms in \reftext{\eqref{-d32_21}},
we have applied \reftext{Lemma~\ref{lem:mainhardy}}
for $\underline{\gamma}=1$ and notice that
the operator $(D_x-1)$
together with the remaining derivative terms cancels the contributions $v_0$ and $v_1x$ or $(v_0)_y x$, respectively. Now, we may follow the argumentation
as in the estimates leading to \reftext{\eqref{est_n1v1b_2}}.
Notice that the conditions on $\tilde k$ and $\check k$ that \reftext{\eqref{-d32_21}} imposes, give, via \reftext{\eqref{4th_line_n2_indices}} the constraint analogous to \reftext{\eqref{parameters2}}, which again reads $\left\lfloor\frac{\check k}2\right\rfloor+4 \le \min\left\{\tilde k, \check k\right\}$ and is less restrictive than
\reftext{\eqref{constraint_bc0bounds}} imposed before.
This concludes our discussion for $w^{(1)} = D_x(v-v_0-v_1x)$.

\medskip

If instead $w^{(1)} = D_y(v-v_0)$, then we proceed like
in the case $\rho=\delta+\frac12$, the only difference being the
different contribution to ${\skliaustask v\skliaustasd}_\mathrm{sol}$
used in the final
estimate. We may estimate
\begin{eqnarray}\label{-d32_22}
\int_I {\left\lVert x^{-1}D^{b^{(1)}}_x D_y(v-v_0) \times
D^{b^{(2)}}_x w^{(2)} \right\rVert}_{-\delta+\frac32}^2\mathrm{d} t
&\lesssim&
\int_I {\left\lVert D^{b^{(1)}}_x D_y(v-v_0) \right\rVert
}_{-\delta+\frac32}^2\mathrm{d} t\times{\skliaustask v\skliaustasd
}_\mathrm{sol}^2
\nonumber\\
&\lesssim&\int_I {\left\lVert\tilde q(D_x)D_xv \right
\rVert}_{\check k + 2,-\delta+\frac
32}^2\mathrm{d} t\times{\skliaustask v\skliaustasd}_\mathrm
{sol}^2\nonumber\\
&\le&{\skliaustask v\skliaustasd}_\mathrm{sol}^4,
\end{eqnarray}
where estimate~\reftext{\eqref{bc0_grad_sol}} of
\reftext{Lemma~\ref{lem:bc0_bounds}} has been used, noting that the required constraint
on indices $\left\lfloor \frac{\check k}{2}\right\rfloor + 4\le\min\{\tilde k -2,\check k -2\}$
is already implied by \reftext{\eqref{constraint_bc0bounds}} imposed before.
This concludes the proof of \reftext{\eqref{4th_line_n2_3}}
for the case $\rho =-\delta+\tfrac32$.

\medskip
\noindent\textit{The bound \reftext{\eqref{4th_line_n2_3}} for $\rho=\delta
+\frac32$.}
We first consider the case $w^{(1)}=(v-v_0-v_1x)_x$ and write
\begin{equation*}
v-v_0-v_1x=\underbrace{\left(v-v_0-v_1x-v_{1+\beta}x^{1+\beta
}-v_2x^2\right)}_{=: w} + v_{1+\beta} x^{1+\beta} + v_2 x^2,
\end{equation*}
leading to
\begin{align}
\int_I {\left\lVert x^{-1}D^{b^{(1)}+1}_x (v-v_0-v_1x)
\times D^{b^{(2)}}_x w^{(2)} \right\rVert}_{\delta+\frac32}^2
\mathrm{d} t \nonumber &\lesssim \int_I {\left\lVert D^{b^{(1)}+1}_x w \right\rVert}_{\delta+\frac32}^2 \mathrm{d} t \times \left\lVert x^{-1} D^{b^{(2)}}_x w^{(2)} \right\rVert_{BC^0(I \times (0,\infty) \times \R)} \\
& \phantom{\lesssim} + \int_I {\left\lVert x^\beta v_{1+\beta} D^{b^{(2)}}_x w^{(2)} \right\rVert}_{\delta+\frac32}^2 \mathrm{d} t \nonumber\\
& \phantom{\lesssim} + \left\lVert v_2\right\rVert_{L^2(I \times \R)}^2  \times\sup_{t \in I, y\in\mathbb R} {\left\lvert D^{b^{(2)}}_x w^{(2)} \right\rvert}_{\delta+1}^2.
\label{d32_all}
\end{align}
We can now apply the same strategy used to pass from \reftext{\eqref{bound_nv1_d+32_b}} to \reftext{\eqref{bound_nv1_d+32_c}}, which concludes the discussion for the case $w^{(1)} = D_x (v-v_0-v_1x)$.

\medskip

For the case $w^{(1)} = D_y (v-v_0)$ we write
\begin{equation*}
v-v_0 = \underbrace{\left(v-v_0-v_1x\right)}_{=: w} + v_1 x,
\end{equation*}
and thus we have to bound the expression
\begin{align}\label{d32_21}
& \int_I {\left\lVert x^{-1}D^{b^{(1)}}_xD_y (v-v_0)\times D^{b^{(2)}}_x w^{(2)} \right\rVert}_{\delta+\frac32}^2\mathrm{d} t\nonumber\\
&\quad \lesssim \int_I\left\lVert x^{-1}D_x^{b^{(1)}}w\times D_x^{b^{(2)}} w^{(2)}\right\rVert_{\delta+\frac32}^2\mathrm{d}t + \left\lVert (v_1)_y\right\rVert_{BC^0\left(I;L^2(\R)\right)}^2 \times \int_I \sup_{y\in \mathbb R} \left\lvert D_x^{b^{(2)}} w^{(2)}\right\rvert_{\delta+1}^2 \mathrm d t.
\end{align}
Now we can estimate the first term in the sum of \reftext{\eqref{d32_21}} as
\begin{align*}
\int_I\left\lVert x^{-1}D_x^{b^{(1)}}w\times D_x^{b^{(2)}} w^{(2)}\right\rVert_{\delta+\frac32}^2\mathrm{d}t &\le \int_I \left\lVert D_y (D_x-1)^{b^{(1)}} (v - v_0) \right\rVert_{\delta+\frac 3 2}^2 \d t \times \skliaustask x^{-1} D_x^{b^{(2)}} w^{(2)} \skliaustasd_{BC^0\left(I \times (0,\infty) \times \R\right)}^2 \\
&\lesssim \int_I \left\lVert (D_x-3) (D_x-2) \tilde q(D_x) D_x v \right\rVert_{\check k + 2, \delta+\frac 3 2}^2 \d t \times \skliaustask v \skliaustasd_\mathrm{sol}^2 \lesssim \skliaustask v \skliaustasd_\mathrm{sol}^4,
\end{align*}
where \reftext{Lemma~\ref{lem:mainhardy}}
and estimate~\reftext{\eqref{bc0_grad_sol}}
of \reftext{Lemma~\ref{lem:bc0_bounds}}
have been employed and we have used the constraint
$\floor{\frac{\check k}2} + 5 \le
\min\left\{\tilde k - 2, \check k - 2\right\}$,
which is the same as condition~\reftext{\eqref{constraint_bc0bounds}} imposed before.
Furthermore, we have
$\left\lVert (v_1)_y\right\rVert_{BC^0\left(I;L^2(\R)\right)}
\lesssim \skliaustask v \skliaustasd_\mathrm{sol}$ by
estimate~\reftext{\eqref{bc0_v1beta_v2}} of
\reftext{Lemma~\ref{lem:bc0_bounds}} and by
interpolation we obtain
\begin{align*}
\sup_{y\in \mathbb R} \left\lvert D_x^{b^{(2)}} w^{(2)}\right\rvert_{\delta+1}^2
&\lesssim \left\lVert D_x^{b^{(2)}} w^{(2)} \right\rVert_{\delta + \frac 1 2}^2 + \left\lVert D_y D_x^{b^{(2)}} w^{(2)} \right\rVert_{\delta + \frac 3 2}^2 \\
&\lesssim \left\lVert \tilde q(D_x) D_x v \right\rVert_{\tilde k + 2, \delta + \frac 1 2}^2 + \left\lVert (D_x-3) (D_x-2) \tilde q(D_x) D_x v \right\rVert_{\check k + 2, \delta + \frac 3 2}^2
\end{align*}
under the assumption that
$\floor{\frac{\check k}2} + 5 \le \tilde k + 2$
(which is already true under the previous constraints on indices \reftext{\eqref{constraint_bc0bounds}}).
By integrating in time $t$ over $I$, we infer that the term
$\int_I \sup_{y\in \mathbb R} \left\lvert D_x^{b^{(2)}} w^{(2)}
\right\rvert_{\delta+1}^2 \mathrm d t$ can be bounded by
$\skliaustask v \skliaustasd_\mathrm{sol}^2$ as well,
and \reftext{\eqref{d32_21}} upgrades to
\begin{equation*}
\int_I {\left\lVert x^{-1}D^{b^{(1)}}_xD_y (v-v_0)\times D^{b^{(2)}}_x w^{(2)} \right\rVert}_{\delta+\frac32}^2\mathrm{d} t \lesssim \skliaustask v \skliaustasd_\mathrm{sol}^4.
\end{equation*}
This concludes the proof of \reftext{\eqref{4th_line_n2_3}} for the case $\rho=\delta+\frac32$.

\medskip

\noindent\textit{Estimates for $\mathcal N^{(2)}(v)$ for $\rho=-\delta +\frac52$.} Here the terms to discuss are the contributions coming from 
\[
\int_I{\left\lVert D_y\tilde q(D_x-1)(D_x-1)\mathcal N^{(2)}(v)\right\rVert}_{\breve k-2,-\delta+\frac52}^2\mathrm{d}t.
\]
Recall that from the definition of the norm in \reftext{\eqref{norm_2d_1}},
\[
{\left\lVert f\right\rVert}_{\breve k -2,-\delta+\frac52}^2 = \sum_{0\le j+j'\le \breve k-2}{\left\lVert D_y^{j'}D_x^jf\right\rVert}_{-\delta+\frac52}^2
\]
for a locally integrable $f \colon (0,\infty) \times \R \to \R$, and we note that for $j'\ge 1$ the corresponding terms can be treated like in the $\rho=-\delta+\frac52$ case of the bound for $\mathcal N^{(1)}(v)$. Therefore we may concentrate on the case $j'=0$, i.e.,
\begin{equation}\label{norm_delta_52}
\int_I \sum_{j = 0}^{\breve k - 2} {\left\lVert D_x^j D_y\tilde q(D_x-1)(D_x-1)\mathcal N^{(2)}(v)\right\rVert}_{-\delta+\frac52}^2\mathrm{d}t.
\end{equation}
Once more, we subdivide the discussion of the terms contributing to $\mathcal N^{(2)}(v)$ according to the subdivision given by the four lines of \reftext{\eqref{simple_n2_terms}}. 

\medskip
\noindent\textit{The first line in \reftext{\eqref{simple_n2_terms}} for $\rho=-\delta+\frac52$.} As before, in this case the operator $(D_x-1)$ in \reftext{\eqref{norm_delta_52}} cancels the first line of \reftext{\eqref{simple_n2_terms}}.

\medskip
\noindent\textit{The second line in \reftext{\eqref{simple_n2_terms}} for $\rho=-\delta+\frac52$.} Here, we observe that the $D_x$-derivatives appearing in \reftext{\eqref{norm_delta_52}} commute with the factors $v_1$ and $(v_0)_y$, the factor $(D_x-1)$ in \reftext{\eqref{norm_delta_52}} cancels $v_1x$ and $\tilde q(D_x)$ contains a $D_x$-factor, which vanishes on $v_0$. Thus by distributing the $D_y$-derivative in \reftext{\eqref{norm_delta_52}}, we have to estimate terms of the form
\begin{equation}\label{gprime_2}
\int_I{\left\lVert D_y(1+v_1)^{-4} w^m D_x^bq(D_x)(D_x-2)(D_x-1)v\right\rVert}_{-\delta+\frac52}^2\mathrm{d}t,
\end{equation}
where $w\in\{v_1,(v_0)_y\}$, $1 \le m \le 4$, and $0\le b\le \breve k+1$. We consider separately the cases in which
\begin{enumerate}[(i)]
\item the $D_y$-derivative in \reftext{\eqref{gprime_2}} acts on one of the factors containing $v_1$ and $(v_0)_y$, 
\item the $D_y$-derivative in \reftext{\eqref{gprime_2}} acts on the factor containing $v$.
\end{enumerate}
In case (i) we bound the resulting term from \reftext{\eqref{gprime_2}} by
\begin{equation}\label{gprime_2_i}
\left\lVert \partial_y\left((1+v_1)^{-4}w^m\right)\right\rVert^2_{BC^0(I;L^2(\R))} \times \int_I\sup_{y \in \mathbb R}\left\lvert D_x^b q(D_x)(D_x-2)(D_x-1)v\right\rvert_{-\delta+2}^2\mathrm{d}t.
\end{equation}
We now estimate the second factor in \reftext{\eqref{gprime_2_i}}:
\begin{align}
&\int_I\sup_{y\in\mathbb R}\left\lvert D_x^b q(D_x)(D_x-2)(D_x-1)v\right\rvert_{-\delta+2}^2\mathrm{d}t\nonumber\\
&\quad\lesssim \int_I\left\lVert D_x^b q(D_x)(D_x-2)(D_x-1)v\right\rVert_{-\delta+\frac32}^2\mathrm{d}t + \int_I\left\lVert D_y D_x^b q(D_x)(D_x-2)(D_x-1)v\right\rVert_{-\delta+\frac52}^2\mathrm{d}t\nonumber\\.
&\quad\lesssim\int_I\left\lVert \tilde q(D_x)D_xv\right\rVert_{\check k+2,-\delta+\frac32}^2\mathrm{d}t + \int_I\left\lVert D_y\tilde q(D_x)D_xv\right\rVert_{\breve k +2,-\delta+\frac52}^2\mathrm{d}t\nonumber\\
&\quad\lesssim\skliaustask v\skliaustasd_\mathrm{sol}^2,\label{gprime_2_i_bd1}
\end{align}
where we have used that $\tilde q(D_x)D_x=q(D_x)(D_x-1)$ and in order for \reftext{\eqref{gprime_2_i_bd1}} to hold we require the condition on indices $\breve k\le \check k$, which is, however, a consequence of \reftext{\eqref{parameters_breve}}, and thus can be omitted.

\medskip

The second factor in \reftext{\eqref{gprime_2_i}} can be estimated according to
\begin{align*}
\left\lVert \partial_y\left((1+v_1)^{-4}w^m\right)\right\rVert_{BC^0(I;L^2(\R))} &\lesssim \left\lVert (1+v_1)^{-4} w^{m-1}\right\rVert_{BC^0(I \times \R)} \left\lVert w_y \right\rVert_{BC^0(I;L^2(\R))} \\
&\phantom{\lesssim} + \left\lVert (1+v_1)^{-5} w^m\right\rVert_{BC^0(I \times \R)} \left\lVert (v_1)_y \right\rVert_{BC^0(I;L^2(\R))}.
\end{align*}
With help of \reftext{Lemma~\ref{lem:bc0_bounds}}, i.e., by applying the bounds \reftext{\eqref{bc0_grad_sol}} to $\left\lVert v_1\right\rVert_{BC^0(I \times \R)}$ and $\left\lVert (v_0)_y\right\rVert_{BC^0(I \times \R)}$, and the bounds \reftext{\eqref{bc0_v1beta_v2}} to $\left\lVert (v_1)_y \right\rVert_{BC^0(I;L^2(\R))}$ and $\left\lVert (v_0)_{yy} \right\rVert_{BC^0(I;L^2(\R))}$, we arrive at 
\begin{equation}\label{gprime_2_i_bd2}
\left\lVert \partial_y\left((1+v_1)^{-4}w^m\right) \right\rVert_{BC^0(I;L^2(\mathbb R))} \lesssim \skliaustask v \skliaustasd_\mathrm{sol},
\end{equation}
valid under the hypotheses that $\skliaustask v\skliaustasd_\mathrm{sol} \le C^{-1}$ with $C > 0$ sufficiently large. In this situation, estimates \reftext{\eqref{gprime_2_i_bd1}} and \reftext{\eqref{gprime_2_i_bd2}} complete the treatment of case (i).

For case (ii), we may estimate the factors $(1+v_1)^{-1}$ and $w$ as above, so that
\begin{align}\label{gprime_2_ii}
&\int_I{\left\lVert (1+v_1)^{-4} w^m D_y D_x^bq(D_x)(D_x-2)(D_x-1)v\right\rVert}_{-\delta+\frac52}^2\mathrm{d}t\nonumber\\
&\quad\lesssim
\left\lVert (1+v_1)^{-4} w^m \right\rVert_{BC^0(I \times \R)}^2
\times \int_I{\left\lVert D_y D_x^bq(D_x)(D_x-2)(D_x-1)v\right\rVert}_{-\delta+\frac52}^2\mathrm{d}t\nonumber\\
&\quad \lesssim \skliaustask v\skliaustasd_\mathrm{sol}^2 \times \int_I{\left\lVert D_y\tilde q(D_x) D_xv\right\rVert}_{\breve k+2,-\delta+\frac52}^2\mathrm{d}t \lesssim
\skliaustask v\skliaustasd_\mathrm{sol}^{4},
\end{align}
under the condition that $\skliaustask v \skliaustasd_\mathrm{sol}\le C^{-1}$  for $C > 0$ sufficiently large. Here, we have used \reftext{Lemma~\ref{lem:mainhardy}} in the forelast line, and we have employed the bound $b\le \breve k+1$. This concludes case (ii) and the discussion of the second line from \reftext{\eqref{simple_n2_terms}}, so that in summary we end up with
\begin{equation}\label{gprime_2_end}
\int_I{\left\lVert D_y(1+v_1)^{-4} w^m D_x^bq(D_x)(D_x-2)(D_x-1)v\right\rVert}_{-\delta+\frac52}^2\mathrm{d}t \lesssim \skliaustask v\skliaustasd_\mathrm{sol}^4.
\end{equation}

\medskip

\textit{The third line in \reftext{\eqref{simple_n2_terms}} for $\rho=-\delta+\frac52$.} In this case, by using the structure of this line, it suffices to treat terms of the form
\[
\int_I\left\lVert D_x^jD_y\tilde q(D_x-1)(D_x-1)\left((1+v_1)^{-3}(v_0)_y^m D_y (v-v_0)\right)\right\rVert_{-\delta+\frac52}^2\mathrm{d}t
\]
for $j\le \breve k+1$ and $1 \le m\le 3$. After distributing derivatives and noting that $\tilde q(D_x)$ contains a factor $D_x$, which vanishes on $v_0$, we reduce to discussing terms of the form
\begin{equation}\label{gprime_3}
\int_I{\left\lVert D_y(1+v_1)^{-3}(v_0)_y^mD_x^bD_y\tilde q(D_x)D_x v\right\rVert}_{-\delta+\frac52}^2\mathrm{d}t,
\end{equation}
where $0 \le b\le \breve k+1$. Again we consider separately the following cases
\begin{enumerate}[(i)]
\item the $D_y$-derivative in \reftext{\eqref{gprime_3}} acts on one of the factors containing $v_1$ and $(v_0)_y$, 
\item the $D_y$-derivative in \reftext{\eqref{gprime_3}} acts on the factor containing $v$.
\end{enumerate}
Case (i) is treated as before in \reftext{\eqref{gprime_2_i}}, and the only difference is that the new term corresponding to \reftext{\eqref{gprime_2_i_bd1}} has now a $D_y$-factor replacing a $(D_x-2)$-factor from \reftext{\eqref{gprime_2_i_bd1}}. Case (ii) is treated as in \reftext{\eqref{gprime_2_ii}}, thereby concluding the estimate of the third line of \reftext{\eqref{simple_n2_terms}}, which is bounded by the same quantity as in \reftext{\eqref{gprime_2_end}}.

\medskip

\noindent\textit{The fourth line in \reftext{\eqref{simple_n2_terms}} for $\rho=-\delta+\frac52$.} For $f$ given by the fourth line of \reftext{\eqref{simple_n2_terms}}, we have to estimate expressions of the form 
\begin{equation}\label{gprime_4}
\int_I\left\lVert D_x^jD_y\tilde q(D_x-1)(D_x-1)f\right\rVert_{-\delta+\frac52}^2\mathrm{d}t=\int_I\left\lVert \tilde g^{(4)} \right\rVert_{-\delta+\frac52}^2\mathrm{d}t
\end{equation}
with $0\le j\le \breve k - 2$, and where $\tilde g^{(4)}$ is given by one of the following expressions, depending on the factor on which the $D_y$-derivative in \reftext{\eqref{gprime_4}} falls:
\begin{subequations}\label{gprime4}
\begin{align}
\tilde g^{(4)} &= (3+\lvert\tau\rvert)(1+v_1)^{-4-{\lvert\tau\rvert}}(v_0)_y^{\lvert\nu\rvert}  \times (v_1)_y \times x^{2-m}\bigtimes_{j=1}^m D_x^{b^{(j)}} w^{(j)},\label{gprime41}\\
\tilde g^{(4)} &= \lvert\nu\rvert (1+v_1)^{-3-\lvert\tau\rvert}(v_0)_y^{\lvert\nu\rvert-1} \times (v_0)_{yy} \times x^{2-m}\bigtimes_{j=1}^mD_x^{b^{(j)}}w^{(j)},\label{gprime42}\\
\tilde g^{(4)} &= \lvert\nu\rvert(1+v_1)^{-3-\lvert\tau\rvert}(v_0)_y^{\lvert\nu\rvert} \times D_yD_x^{b^{(1)}}w^{(1)}\times \bigtimes_{j=2}^m x^{-1} D_x^{b^{(j)}}w^{(j)},\label{gprime43}\\
\tilde g^{(4)} &= \lvert\nu\rvert(1+v_1)^{-3-\lvert\tau\rvert}(v_0)_y^{\lvert\nu\rvert} \times x^{-1}D_x^{b^{(1)}}w^{(1)}\times D_yD_x^{b^{(2)}}w^{(2)}\times\bigtimes_{j=3}^mx^{-1} D_x^{b^{(j)}}w^{(j)},\label{gprime44}
\end{align}
where $m = \lvert\mu\rvert + \lvert\tau\rvert$, ${\left\lvert\mu\right\rvert}+{\left\lvert\tau\right\rvert}\le 4$,
$0 \le \left\lvert \nu \right\rvert \le 4$, and
\begin{equation}
w^{(j)}\in\left\{D_x(v-v_0-v_1x), D_y(v-v_0)\right\} \quad \mbox{with} \quad 0 \le b^{(1)} \le \breve k+6 \quad \mbox{and} \quad 0 \le b^{(j)} \le \left\lfloor\frac{\breve k}{2}\right\rfloor + 3\quad\mbox{for}\quad j\ge2.
\end{equation}
\end{subequations}
We first consider \reftext{\eqref{gprime_4}} with $\tilde g^{(4)}$ given by \reftext{\eqref{gprime41}}, where we use
\begin{equation}\label{tg4_1}
\begin{aligned}
\int_I \left\lVert \tilde g^{(4)} \right\rVert_{-\delta+\frac52}^2\mathrm{d}t &\lesssim (3+\lvert\tau\rvert) \left\lVert
(1+v_1)^{-4-\lvert\tau\rvert} (v_0)_y^{\lvert\nu\rvert} \right\rVert_{BC^0(I\times \mathbb R)}^2
\times \left\lVert (v_1)_y \right\rVert_{BC^0(I;L^2(\mathbb R))}^2 \\
&\phantom{\lesssim} \times\int_I\sup_{y\in \mathbb R}\left\lvert D_x^{b^{(1)}} w^{(1)} \times D_x^{b^{(2)}} w^{(2)} \right\rvert_{-\delta+3}^2\mathrm{d}t \times \prod_{j = 3}^m \skliaustask x^{-1} D_x^{b^{(j)}} w^{(j)} \skliaustasd_{BC^0(I \times (0,\infty) \times \R)}^2.
\end{aligned}
\end{equation}
Here, we estimate the term $\left\lVert
(1+v_1)^{-4-\lvert\tau\rvert} (v_0)_y^{\lvert\nu\rvert} \right\rVert_{BC^0(I\times \mathbb R)}$ according to
\begin{align*}
\left\lVert
(1+v_1)^{-4-\lvert\tau\rvert} (v_0)_y^{\lvert\nu\rvert} \right\rVert_{BC^0(I\times \mathbb R)} &\le \left(1 - \left\lVert v_1 \right\rVert_{BC^0(I \times \R)}\right)^{-4-\lvert\tau\rvert} \times \left\lVert (v_0)_y \right\rVert_{BC^0(I \times \R)}^{\lvert\nu\rvert} \\
&\lesssim \left(1 - \tfrac C 2 \skliaustask v \skliaustasd_\mathrm{sol}\right)^{-\lvert\tau\rvert} \times \left(C \skliaustask v \skliaustasd_\mathrm{sol}\right)^{\lvert\nu\rvert} \le \left(1 - \tfrac C 2 \skliaustask v \skliaustasd_\mathrm{sol}\right)^{-\lvert\tau\rvert}
\end{align*}
for $\skliaustask v \skliaustasd_\mathrm{sol} \le C^{-1}$ with $C > 0$ sufficiently large. Furthermore,
$\left\lVert (v_1)_y \right\rVert_{BC^0(I;L^2(\mathbb R))} \le C \skliaustask v \skliaustasd_\mathrm{sol}$
by estimate~\reftext{\eqref{bc0_v1beta_v2}} of \reftext{Lemma~\ref{lem:bc0_bounds}}
and $\prod_{j = 3}^m \skliaustask x^{-1} D_x^{b^{(j)}} w^{(j)} \skliaustasd_{BC^0(I \times (0,\infty) \times \R)}^2
\le \left(\tfrac C 2 \skliaustask v \skliaustasd_\mathrm{sol}\right)^{2 (m-2)}$
for a sufficiently large $C$ by estimate \reftext{\eqref{bc0_grad_sol}}
of \reftext{Lemma~\ref{lem:bc0_bounds}} under the condition that
$\floor{\frac{\breve k}{2}}+3\le \min\left\{\tilde k -2,\check k -2\right\}$,
which is already implied by condition \reftext{\eqref{parameters_breve}}.
For the remaining factor, we estimate according to
\begin{equation}\label{f3}
\begin{aligned}
\int_I\sup_{y\in\mathbb R}\left\lvert D_x^{b^{(1)}}w^{(1)} \times D_x^{b^{(2)}}w^{(2)}\right\rvert_{-\delta+3}^2\mathrm{d}t &\lesssim \int_I{\left\lVert D_x^{b^{(1)}}w^{(1)}\times D_x^{b^{(2)}}w^{(2)}\right\rVert}_{-\delta +\frac52}^2\mathrm{d}t \\
& \phantom{\lesssim} + \int_I\left\lVert D_yD_x^{b^{(1)}}w^{(1)}\times D_x^{b^{(2)}}w^{(2)}\right\rVert_{-\delta+\frac72}^2\mathrm{d}t \\
& \phantom{\lesssim} +\int_I\left\lVert D_x^{b^{(1)}}w^{(1)}\times D_yD_x^{b^{(2)}}w^{(2)}\right\rVert_{-\delta+\frac72}^2\mathrm{d}t
\end{aligned}
\end{equation}
The first summand on the right-hand side of \reftext{\eqref{f3}} can be treated as in \reftext{\eqref{-d32_all}}, under the condition that $\breve k\le \check k$, which is already implied by \reftext{\eqref{parameters_breve}}.
For the second summand, we write
\begin{align}\label{f3_2}
&\int_I\left\lVert D_yD_x^{b^{(1)}}w^{(1)}\times D_x^{b^{(2)}}w^{(2)}\right\rVert_{-\delta+\frac72}^2\mathrm{d}t \nonumber\\
&\quad\lesssim\int_I\left\lVert D_yD_x^{b^{(1)}}w^{(1)}\right\rVert_{-\delta+2}^2\mathrm{d}t\times\skliaustask x^{-\frac32} D_x^{b^{(2)}}w^{(2)}\skliaustasd_{BC^0(I\times(0,\infty)\times\mathbb R)}^2\nonumber\\
&\quad\lesssim\int_I\left\lVert D_y\tilde q(D_x) D_x^{(b^{(1)}-5)_+ +1}w^{(1)}\right\rVert_{-\delta+2}^2\mathrm{d}t\times\skliaustask x^{-\frac32} D_x^{b^{(2)}}w^{(2)}\skliaustasd_{BC^0(I\times(0,\infty)\times\mathbb R)}^2,
\end{align}
where we have used \reftext{Lemma~\ref{lem:mainhardy}} in the last line. The first factor in \reftext{\eqref{f3_2}} is controlled by the $\breve k$-contribution to $\skliaustask v\skliaustasd_\mathrm{sol}^{2}$, and the second one is treated by interpolation as follows:
\begin{align}\label{f3_2_2}
\left\lVert x^{-\frac32} D_x^{b^{(2)}}w^{(2)}\right\rVert_{BC^0(I\times(0,\infty)\times\mathbb R)}^2 &\lesssim \sup_{t\in I} \left\lVert D_x^{b^{(2)}}w^{(2)}
\right\rVert_{1,\delta+1} \times \sup_{t\in I} \left\lVert D_yD_x^{b^{(2)}}w^{(2)}
\right\rVert_{1,-\delta+2}\nonumber\\
&\lesssim\sup_{t\in I}\left\lVert D_x^{b^{(2)}}w^{(2)}
\right\rVert_{1,\delta+1}^2 + \sup_{t\in I}\left\lVert D_yD_x^{b^{(2)}}w^{(2)}
\right\rVert_{1,-\delta+2}^2 \nonumber\\
&\lesssim\sup_{t\in I}\left\lVert (D_x-3)(D_x-2)\tilde q(D_x)D_xv
\right\rVert_{\check k,\delta+1}^2 + \sup_{t\in I}\left\lVert D_y\tilde q(D_x)D_x v
\right\rVert_{\breve k,-\delta+2}^2,
\end{align}
where in the last line we have used \reftext{Lemma~\ref{lem:mainhardy}} and that the factor $(D_x-1)$, coming from $\tilde q(D_x)$, annihilates $v_1 x$ and $(v_0)_y x$. The conditions on indices required for \reftext{\eqref{f3_2_2}} to apply, reduce to $\floor{ \frac{\breve k}{2}} \le \check k +3$, which is already implied by \reftext{\eqref{parameters_breve}}. This concludes the bound for the second summand in \reftext{\eqref{f3}}.

\medskip

The third summand in \reftext{\eqref{f3}} can be estimated according to
\begin{align}\label{f3_3}
&\int_I\left\lVert D_x^{b^{(1)}}w^{(1)}\times D_yD_x^{b^{(2)}}w^{(2)}\right\rVert_{-\delta+\frac72}^2\mathrm{d}t\nonumber\\
&\quad\lesssim\int_I\left\lVert x^{-\frac32}D_x^{b^{(1)}}w^{(1)}\right\rVert_{BC^0((0,\infty)_x\times\mathbb R_y)}^2 \mathrm{d}t\times\sup_{t\in I}\left\lVert D_yD_x^{b^{(2)}}w^{(2)}\right\rVert_{-\delta+2}^2\nonumber\\
&\quad\lesssim\int_I\left\lVert x^{-\frac32}D_x^{b^{(1)}}w^{(1)}\right\rVert_{BC^0((0,\infty)_x\times\mathbb R_y)}^2 \mathrm{d}t\times\sup_{t\in I}\left\lVert D_y\tilde q(D_x)D_xv\right\rVert_{\breve k,-\delta+2}^2,
\end{align}
where we have used \reftext{Lemma~\ref{lem:mainhardy}} and the fact that $v_1 x$ and $(v_0)_y x$ are annihilated by $\tilde q(D_x)$. The above second factor directly appears in $\skliaustask v\skliaustasd_\mathrm{sol}^2$ and the first factor is treated as follows:
\begin{align}\label{f3_3_1}
&\int_I\left\lVert x^{-\frac32}D_x^{b^{(1)}}w^{(1)}\right\rVert_{BC^0((0,\infty)_x\times\mathbb R_y)}^2 \mathrm{d}t \nonumber\\
&\quad\lesssim\int_I\left\lVert D_x^{b^{(1)}}w^{(1)}\right\rVert_{1,-\delta+\frac32} \times \left\lVert D_yD_x^{b^{(1)}}w^{(1)}\right\rVert_{1,\delta+\frac32}\mathrm{d}t\nonumber\\
&\quad\lesssim\int_I\left\lVert \tilde q(D_x)D_xv\right\rVert_{\check k +2,-\delta+\frac32}^2\mathrm{d}t + \int_I\left\lVert(D_x-3)(D_x-2)\tilde q(D_x)D_xv\right\rVert_{\check k+2,\delta+\frac32}^2\mathrm{d}t,
\end{align}
where in the last line we have used \reftext{Lemma~\ref{lem:mainhardy}}
and that $v_1 x$ and $(v_0)_y x$ are in the kernel of $\tilde q(D_x)$.
Furthermore, we require constraints on indices that reduce to $\breve k +1\le \check k$,
which, however, is implied by \reftext{\eqref{parameters_breve}}.
The bound \reftext{\eqref{f3_3_1}} completes the treatment of \reftext{\eqref{f3}}
and thus of all factors in \reftext{\eqref{tg4_1}}.
Collecting all estimates obtained so far,
we conclude the treatment of $\tilde g^{(4)}$
given by \reftext{\eqref{gprime41}} and obtain
\begin{equation}\label{est_tg4_1}
\int_I\left\lVert \tilde g^{(4)}\right\rVert_{-\delta+\frac52}^2\mathrm{d}t\lesssim (3+|\tau|)\left(1 - \tfrac C 2 \skliaustask v\skliaustasd_\mathrm{sol}\right)^{-|\tau|} \left(\tfrac C 2 \skliaustask v\skliaustasd_\mathrm{sol}\right)^{2m} \lesssim m \left(C \skliaustask v\skliaustasd_\mathrm{sol}\right)^{2m}
\end{equation}
under the condition that $\skliaustask v\skliaustasd_\mathrm{sol}\le C^{-1}$
and using $\lvert\tau\rvert\le m$ and $\lvert\nu\rvert\le 4$. 

\medskip

The case where $\tilde g^{(4)}$ is given by \reftext{\eqref{gprime42}}
is treated in exactly the same way,
but with the role of the factor $(v_1)_y$ now taken over by $(v_0)_{yy}$.
For $\tilde g^{(4)}$ given by \reftext{\eqref{gprime43}},
we estimate according to
\begin{equation}\label{tg4_3}
\begin{aligned}
\int_I \left\lVert \tilde g^{(4)} \right\rVert_{-\delta+\frac52}^2\mathrm{d}t &\lesssim \lvert\nu\rvert \left\lVert
(1+v_1)^{-3-\lvert\tau\rvert} (v_0)_y^{\lvert\nu\rvert} \right\rVert_{BC^0(I\times \mathbb R)}^2
\times \int_I \left\lVert D_y D_x^{b^{(1)}} w^{(1)} \times D_x^{b^{(2)}} w^{(2)} \right\rVert_{-\delta+\frac 7 2}^2\mathrm{d}t \\
&\phantom{\lesssim} \times \prod_{j = 3}^m \skliaustask x^{-1} D_x^{b^{(j)}} w^{(j)} \skliaustasd_{BC^0(I \times (0,\infty) \times \R)}^2,
\end{aligned}
\end{equation}
where $\lvert\nu\rvert \left\lVert
(1+v_1)^{-3-\lvert\tau\rvert} (v_0)_y^{\lvert\nu\rvert} \right\rVert_{BC^0(I\times \mathbb R)}^2$ and $\prod_{j = 3}^m \skliaustask x^{-1} D_x^{b^{(j)}} w^{(j)} \skliaustasd_{BC^0(I \times (0,\infty) \times \R)}$ can be teated as in the case in which $\tilde g^{(4)}$ is given by \reftext{\eqref{gprime41}}. The remaining factor $\int_I\left\lVert D_y D_x^{b^{(1)}}w^{(1)}\times D_x^{b^{(2)}}w^{(2)}\right\rVert_{-\delta+\frac72}^2\mathrm{d}t$ was already estimated in \reftext{\eqref{f3_2}}, so that under the condition that $\skliaustask v\skliaustasd_\mathrm{sol} \le C^{-1}$ and using the fact that $|\tau|\le m$ and $|\nu|\le 4$, \reftext{\eqref{tg4_3}} upgrades to
\begin{equation}\label{est_tg4_3}
\int_I \left\lVert \tilde g^{(4)}\right\rVert_{-\delta+\frac52}^2\mathrm{d}t\lesssim\left(1-\tfrac C 2\skliaustask v\skliaustasd_\mathrm{sol}\right)^{-|\tau|} \left(\tfrac C 2 \skliaustask v\skliaustasd_\mathrm{sol}\right)^{2m}
\le \left(C \skliaustask v\skliaustasd_\mathrm{sol}\right)^{2m}.
\end{equation}
For $\tilde g^{(4)}$ given by \reftext{\eqref{gprime44}}, we have

\begin{equation}\label{tg4_4}
\begin{aligned}
\int_I \left\lVert \tilde g^{(4)} \right\rVert_{-\delta+\frac52}^2\mathrm{d}t &\lesssim \lvert\nu\rvert \left\lVert
(1+v_1)^{-3-\lvert\tau\rvert} (v_0)_y^{\lvert\nu\rvert} \right\rVert_{BC^0(I\times \mathbb R)}^2
\times \int_I \left\lVert D_x^{b^{(1)}} w^{(1)} \times D_y D_x^{b^{(2)}} w^{(2)} \right\rVert_{-\delta+\frac 7 2}^2\mathrm{d}t \\
&\phantom{\lesssim} \times \prod_{j = 3}^m \skliaustask x^{-1} D_x^{b^{(j)}} w^{(j)} \skliaustasd_{BC^0(I \times (0,\infty) \times \R)}^2,
\end{aligned}
\end{equation}
where the only difference compared to \eqref{tg4_3} is the factor $\int_I \left\lVert D_x^{b^{(1)}} w^{(1)} \times D_y D_x^{b^{(2)}} w^{(2)} \right\rVert_{-\delta+\frac 7 2}^2\mathrm{d}t$, which is now estimated as in \reftext{\eqref{f3_3}} rather than \reftext{\eqref{f3_2}}. As a result, under the condition that $\skliaustask v\skliaustasd_\mathrm{sol} \le C^{-1}$ and again using the fact that $|\tau|\le m$ and $|\nu|\le 4$, \reftext{\eqref{tg4_4}} upgrades to the same bound as in \reftext{\eqref{est_tg4_3}}.

\medskip

This concludes the treatment of \reftext{\eqref{gprime4}}, and of the term $\mathcal N^{(2)}(v)$ in the case $\rho=-\delta+\frac52$. Under the condition that $\skliaustask v\skliaustasd_\mathrm{sol} \le C^{-1}$, the bounds \reftext{\eqref{est_tg4_1}} and \reftext{\eqref{est_tg4_3}} give
\begin{equation}\label{est_n2_-d52}
\int_I {\left\lVert D_y\tilde q(D_x-1)(D_x-1){\mathcal N}^{(2)}(v) \right
\rVert}_{\breve k -2,-\delta+\frac52}^2\mathrm{d}
t\lesssim m \left(C {\skliaustask v\skliaustasd}_\mathrm{sol}\right)^{2m}.
\end{equation}
\noindent\textit{Conclusion of the estimates for $\skliaustask {\mathcal
N}^{(1)}(v)\skliaustasd_\mathrm{rhs}$ and $\skliaustask {\mathcal
N}^{(2)}(v)\skliaustasd_\mathrm{rhs}$.}
We now sum estimate~\reftext{\eqref{est_n2_-d12}},
the corresponding estimates of the terms in \reftext{\eqref{remainstoest_N2}},
and estimates~\reftext{\eqref{second_line_n2_bound}},
\reftext{\eqref{third_line_n2_bound}}, \reftext{\eqref{4th_line_n2_final}},
\reftext{\eqref{gprime_2_end}} and \reftext{\eqref{est_n2_-d52}}
together with the estimates coming from taking at least one $D_y$-derivative
of $\mathcal N^{(2)}(v)$ (see~\reftext{\eqref{remainstoest_N2}}
and the discussion thereafter) and from the estimation of the bounds \reftext{\eqref{est_n1v_terms}} for the terms
\reftext{\eqref{simple_n1v_terms}} from
$\skliaustask\mathcal N^{(1)}(v)\skliaustasd_\mathrm{rhs}$.
We thus find, under the condition that
$\skliaustask v\skliaustasd_\mathrm{sol} \le C^{-1}$ as above,
\begin{equation}\label{n2_bound_near0_almost_end}
{\skliaustask \mathcal N^{(1)}(v)\skliaustasd}_\mathrm{rhs}^2 + {\skliaustask \mathcal N^{(2)}(v)\skliaustasd}_\mathrm{rhs}^2\lesssim \skliaustask v\skliaustasd_\mathrm{sol}^4 \sum_{m = 2}^{\infty} c(m) \left(C \skliaustask v\skliaustasd_\mathrm{sol}\right)^{2(m-2)}\lesssim {\skliaustask v\skliaustasd}_\mathrm{sol}^4.
\end{equation}
Here the number $c(m)$ is a bound on the number of terms of fixed degree $m$ in $v$ appearing \reftext{\eqref{est_n1v_terms}}, and in the above mentioned other estimates for $\mathcal N^{(2)}(v)$. 

\medskip

In order to bound $c(m)$,
we note that due to the form of $\mathcal N^{(1)}(v)$
and $\mathcal N^{(2)}(v)$, we can estimate the number of terms
of the form \reftext{\eqref{est_n1v_terms}} by $Cm^C$, and similarly
the total number of terms to be discussed for $\mathcal N^{(2)}(v)$
is bounded by $Cm^C$. As both cases are treated in a similar fashion,
we discuss in detail the latter case, which is more involved.
First consider the index set
${\mathcal I}$ in \reftext{\eqref{simple_n2}},
in which the possibly unbounded indices are
$\tau_2$, $\tau_3$, and $\tau_4$,
for which we have the control $\lvert\tau\rvert\le m$.
Since the total number of partitions of $m$ into three integers is
$\frac 1 2 (m+1)(m+2)$,
the number of terms with a total of $m$
factors is growing like $C m^2$. The contributions to
$\mathcal N^{(2)}(v)$ from the remaining parts of
\reftext{\eqref{simple_n2_terms}} give at most $C$ terms.

\medskip

To each of these terms, we then distribute derivatives coming from the definition of the $\skliaustask\cdot\skliaustasd_\mathrm{rhs}$-norm, i.e.,
we apply $\kappa$ derivatives with $\kappa$ as in \reftext{\eqref{est_term_n1v_cond_2}}.
The sum of coefficients of the terms obtained after this operation is bounded by
\begin{equation}\label{boundc1}
c(m) \le m\binom{C m^2+\kappa}{\kappa} \lesssim
m^{C},
\end{equation}
where the extra factor $m$ comes from the the bound \reftext{\eqref{est_n2_-d52}}, and again $C \in \N$ only depends on $k$, $\tilde k$, $\check k$, $\breve k$, and $\delta$.

\medskip

Due to the bound \reftext{\eqref{boundc1}} on $c(m)$, the condition
${\skliaustask v\skliaustasd}_\mathrm
{sol} \le C^{-1}$ ensures that the series appearing in \reftext{\eqref{n2_bound_near0_almost_end}}
is convergent by the root test. This justifies the last inequality in \reftext{\eqref{n2_bound_near0_almost_end}} and thus we have
\begin{equation}\label{est_nv}
{\skliaustask{\mathcal N}(v)\skliaustasd}_\mathrm{rhs} \lesssim
{\skliaustask v\skliaustasd}_\mathrm{sol}^2 \quad
\mbox{for} \quad{\skliaustask v\skliaustasd}_\mathrm{sol} \ll1.
\end{equation}
\textit{Conclusion of the proof.}
In order to pass from \reftext{\eqref{est_nv}} to the general bound \reftext{\eqref{non_est_main}}, we consider the multilinear terms into which we have
decomposed the nonlinearity in the above arguments.
We note that for an
$m$-linear form
${\mathcal M}\left(w^{(1)},\ldots,w^{(m)}\right)$
we have the decomposition
\begin{align}
& {\mathcal M}\left(w^{(1,1)},\ldots,w^{(m
,1)}\right) - {\mathcal M}\left(w^{(1,2)},\ldots,w^{(m
,2)}\right) \nonumber\\
&\quad= \sum_{j=1}^m {\mathcal M}\left(w^{(1,1)},\ldots
,w^{(j-1,1)},w^{(j,1)}-w^{(j,2)},w^{(j+1,2)},\ldots,w^{(m
,2)}\right
). \label{decomp_nlinear}
\end{align}
This decomposition can be applied for $\mathcal M$ and $m$ corresponding to
\begin{enumerate}[(i)]
\item[(i)] the summands in \reftext{\eqref{simple_n1v_terms}}
that contribute to the expression of
${\skliaustask{\mathcal N}^{(1)}(v)\skliaustasd}_\mathrm{sol}^{2}$,
the terms of the sum contributing to ${\skliaustask\mathcal N^{(2)}(v)\skliaustasd}_\mathrm{sol}$ on which at least one $D_y$-derivative (coming from
definition~\reftext{\eqref{norm_2d_1}} of the norm
$\lVert\cdot\rVert_{\kappa,\rho}$) acts,
and the summands coming from
${\skliaustask{\mathcal N}^{(2)}(v)\skliaustasd}_\mathrm{sol}^{2}$ with weights
$\rho\in\left\{-\delta+\frac12,-\delta+\frac52\right\}$,
\item[(ii)] the terms \reftext{\eqref{second_line_n2_g}},
\reftext{\eqref{third_line_n2_defg}}, \reftext{\eqref{4th_line_n2_g}},
and \reftext{\eqref{gprime4}}, which give the
contributions $g^{(2)}$, $g^{(3)}$, $g^{(4)}$,
and $\tilde g^{(4)}$ to ${\mathcal N}^{(2)}(v)$
for the weights $\rho\in\left\{\delta+\frac12,\pm\delta+\frac32,-\delta+\frac52\right\}$.
The latter come from those terms in the norms $\lVert\cdot\rVert_{\kappa,\rho}$,
in which only $D_x$-derivatives are distributed.
\end{enumerate}
This increases the final coefficients in our estimates, such as $c(m)$ appearing in \reftext{\eqref{n2_bound_near0_almost_end}}, by a factor of at most $m$, which is still controllable with the same strategy explained after \eqref{n2_bound_near0_almost_end}.

\medskip

As an example of how to use the expression \reftext{\eqref{decomp_nlinear}},
we treat terms of the form $h:=(1+v_1)^{-4}v_1 q(D_x)(D_x-2)(v-v_0-v_1x)$,
which is a special case of \reftext{\eqref{second_line_n2_g}}, in more detail. The other cases are treated analogously. If we denote by $h^{(j)}$ the versions of the term $h$ with $v=v^{(j)}$ for $j=1,2$, and by $\tilde h^{(j)}$ the terms coming from the factor $q(D_x)(D_x-2)(v^{(j)}-v_0^{(j)}-v_1^{(j)}x)$ with $v=v^{(j)}$ for $j=1,2$, then
\begin{eqnarray*}
h^{(1)} - h^{(2)} &=& 
\left(1+v_1^{(1)}\right)^{-4} 
\left(1+v_1^{(2)}\right)^{-4} v_1^{(1)}\tilde h^{(1)} 
\left(\left(1+v_1^{(2)}\right)^4 - \left(1+v_1^{(1)}\right)^4\right) \\
&& + \left(1+v_1^{(2)}\right)^{-4}\left(v_1^{(1)}-v_1^{(2)}\right)\tilde h^{(1)}+\left(1+v_1^{(2)}\right)^{-4}v_1^{(2)} \left(\tilde h^{(1)} - \tilde h^{(2)}\right),
\end{eqnarray*}
where for the factors $\left(1+v^{(j)}\right)^{-4}$ and $v_1^{(j)}$ with
$j = 1,2$ and $\tilde h^{(j)}$ we
may use the same bounds as in \reftext{\eqref{first_line_n2_1}}, \reftext{\eqref{second_line_n2_2}}
and \reftext{\eqref{second_line_n2_3}}, respectively, while for
the differences
\begin{equation*}
\left(\left(1+v_1^{(2)}\right)^{4} - \left(1+v_1^{(1)}\right)^{4}\right), \quad \left(v_1^{(1)}-v_1^{(2)}\right)\quad
\mbox{and} \quad \left(w^{(1)}-w^{(2)}\right)
\end{equation*}
we use
\reftext{\eqref{decomp_nlinear}} and then employ bounds
analogous to those in \reftext{\eqref{first_line_n2_1}}, \reftext{\eqref{second_line_n2_2}} and \reftext{\eqref{second_line_n2_3}} again. From the
above we directly deduce a bound of the form
\reftext{\eqref{non_est_main}} in this case.

\medskip

Finally, we recall and simplify the conditions we have required
on the indices. We have assumed that \reftext{\eqref{parameterss1}} holds,
then we have imposed conditions~ \reftext{\eqref{constraint_bc0bounds}},
\reftext{\eqref{parameters_breve}}, and the conditions
$k\ge \max\left\{\tilde k+4,\check k+6,\breve k+5\right\}$,
$\tilde k \ge3$, $\check k \ge4$, and $\breve k\ge4$,
which summarize the requirements of
\reftext{Lemmata~\ref{lem:equivalence_init}, \ref{lem:equivalence_rhs},
\ref{lem:approx_init}, \ref{lem:approx_rhs}, and \ref{lem:approx_sol}}.
The combination of these constraints gives
\begin{align*}
\max\left\{\tilde k,\check k +2,\breve k + 1\right\} +4 &\le k,
& \left\lfloor\frac{k+1}{2}\right\rfloor +3 &\le\min\left\{\tilde k,\check k\right\},
& \left\lfloor\frac{\tilde k}{2}\right\rfloor + 7&\le\min\left\{\tilde k, \check k\right\},\\
\left\lfloor \frac{\check k}{2}\right\rfloor + 7&\le\min\left\{\tilde k,\check k\right\},
& \breve k + 8 &\le\min\left\{\tilde k,\check k\right\},
& 4 &\le \breve k.
\end{align*}
These conditions are equivalent to
\begin{equation*}
\max\left\{13, \left\lfloor\frac{k+1}2\right\rfloor+3,\breve k+8\right\}\le \min\left\{\tilde k, \check k\right\}, \quad \max\left\{\tilde k, \check k+2\right\}+4\le k, \quad 4 \le \breve k,
\end{equation*}
which are our assumed condition \reftext{\eqref{parameters}} of
\reftext{Assumptions~\ref{ass:parameters}}.
\end{proof}
%
\subsection{Well-posedness, a-priori estimates, and stability for the
nonlinear equation}\label{sec:well}

\begin{proof}[Proof of \reftext{Theorem~\ref{main_th}}]
Throughout the proof, all estimates and constants only depend on $k$, $\tilde k$, $\check k$, $\breve k$, and $\delta$.
We also use the equivalence of $\skliaustask \cdot \skliaustasd_\mathrm{Sol}$
and $\skliaustask \cdot \skliaustasd_\mathrm{sol}$ on the time interval
$I = [0,\infty)$ (cf.~\reftext{Lemma~\ref{lem:equivalence_sol}})
without further mention.

\medskip
\noindent\textit{Existence.}
We fix $\varepsilon > 0$ to be determined later, and consider a locally
integrable $v^{(0)}: (0,\infty)\times\mathbb{R}\to\mathbb{R}$ with
$ {\left\lVert v^{(0)} \right\rVert}_\mathrm
{init}\le\varepsilon$, like in the assumption of the theorem.
Then we will work in the spaces
\begin{subequations}
\begin{eqnarray}
S_{\varepsilon^\prime,v^{(0)}}&:=&\left\{v:(0,\infty)^2\times
\mathbb{R}\to\mathbb{R}\text{
locally integrable: }{\skliaustask v\skliaustasd}_\mathrm
{sol}\le\varepsilon^\prime,
v_{|t=0}=v^{(0)}\right\},\quad \label{def_sepsv}
\eqncr
R_{\varepsilon^\prime}&:=&\left\{f:(0,\infty)^2\times\mathbb
{R}\to\mathbb{R}\text{ locally
integrable: }{\skliaustask f\skliaustasd}_\mathrm{rhs}\le\varepsilon
^\prime\right\},
\end{eqnarray}
\end{subequations}
where $\varepsilon^\prime> 0$ will be fixed later.
In order to show that $S_{\varepsilon^\prime,v^{(0)}}$
is nonempty for $\varepsilon > 0$ small enough,
we first apply Proposition~\reftext{\ref{prop:maxregevolution}}
with $f=0$ and $v^{(0)}$ as above, and consider the so obtained solution $v$.
Due to estimate \reftext{\eqref{mr_interval}},
we obtain that ${\left\lVert v^{(0)} \right\rVert}_\mathrm{init}\le\varepsilon$
implies ${\skliaustask v\skliaustasd}_\mathrm{sol}\le C\varepsilon$,
which for $\varepsilon\le C^{-1}\varepsilon^\prime$ implies that such $v$
belongs to $S_{\varepsilon^\prime, v^{(0)}}$.
Below we impose this constraint on $\varepsilon$ throughout.

Consider the maps
\begin{equation*}
{\mathcal N}:S_{\varepsilon^\prime,v^{(0)}}\to R_{\varepsilon^\prime
}, \quad
T_{v^{(0)}}:R_{\varepsilon^\prime}\to S_{\varepsilon^\prime,v^{(0)}},
\end{equation*}
where to every $v\in S_{\varepsilon^\prime, v^{(0)}}$ the nonlinear operator
${\mathcal N}$ (explicitly defined in \reftext{\eqref{def_nonlinearity}})
associates the nonlinear right-hand side of our equation \reftext{\eqref{eq_transformed}},
and to $f\in R_{\varepsilon^\prime}$ the linear operator $T_{v^{(0)}}$ associates
the solution to the linearized equation \reftext{\eqref{eq_lin}}, as given by
\reftext{Proposition~\ref{prop:maxregevolution}}.

By \reftext{Proposition~\ref{prop:non_est}} for the case $v^{(1)}=v, v^{(2)}=0$,
we find that if $\varepsilon^\prime$ is chosen sufficiently small, we have
\begin{equation}\label{norm_control_N}
{\skliaustask{\mathcal N}(v)\skliaustasd}_\mathrm{rhs}\lesssim
{\skliaustask v\skliaustasd}_\mathrm{sol}^2,
\end{equation}
which shows that ${\mathcal N}$ maps $S_{\varepsilon^\prime,v^{(0)}}$
to $R_{\varepsilon^\prime
}$ if $0 < \varepsilon^\prime\ll 1$. Similarly, by the
maximal-regularity estimate \reftext{\eqref{mr_interval}} from
\reftext{Proposition~\ref{prop:maxregevolution}}, we find that for any
locally integrable $v^{(0)} \colon (0,\infty) \times \R \to \R$
such that $ {\left\lVert v^{(0)} \right\rVert}_\mathrm{init}<\infty$
and for any locally integrable $f \colon (0,\infty)^2 \times \R \to \R$
with $\skliaustask f \skliaustasd_\mathrm{rhs} < \infty$,
we have the bound
\begin{equation}\label{norm_control_T}
{\skliaustask T_{v^{(0)}}(f)\skliaustasd}_\mathrm{sol}\lesssim
 {\left\lVert v^{(0)} \right\rVert}_\mathrm
{init}+{\skliaustask f\skliaustasd}_\mathrm{rhs},
\end{equation}
and thus $T_{v^{(0)}}$ also maps $R_{\varepsilon^\prime}$ to
$S_{\varepsilon^\prime, v^{(0)}}$ if $\varepsilon^\prime\ll1$ and if
${\left\lVert v^{(0)} \right\rVert}_\mathrm{init} \le \varepsilon \ll \varepsilon^\prime$,
which can be ensured up to diminishing $\varepsilon^\prime$ and $\varepsilon$.
Under this last condition, we may define the self-map
\begin{equation}\label{def_T_v0}
\mathcal T_{v^{(0)}}:=T_{v^{(0)}}\circ{\mathcal N}:S_{\varepsilon
^\prime,v^{(0)}}\to
S_{\varepsilon^\prime,v^{(0)}},
\end{equation}
and in order to prove the existence part of our theorem it will suffice
to prove that, for $\varepsilon^\prime,\varepsilon > 0$ small enough, the operator $\mathcal T_{v^{(0)}}$ has a fixed point in $S_{\varepsilon
^\prime ,v^{(0)}}$ for each $v^{(0)}$ such that $ {\left\lVert
v^{(0)} \right\rVert}_\mathrm
{init}\le\varepsilon$. To this aim we note that, again by
\reftext{Propositions~\ref{prop:maxregevolution}} and \reftext{\eqref{non_est_main}},
we may write for any given $v^{(1)}, v^{(2)}\in S_{\varepsilon^\prime,v^{(0)}}$,
\begin{eqnarray*}
{\skliaustask\mathcal T_{v^{(0)}}\left(v^{(1)}\right)-\mathcal
T_{v^{(0)}}\left(v^{(1)}\right)\skliaustasd}_\mathrm{sol} &\stackrel{\text{\reftext{\eqref{norm_control_T}, \eqref{def_T_v0}}}}{\lesssim}&
{\skliaustask{\mathcal N}\left(v^{(1)}\right)
-{\mathcal N}\left(v^{(2)}\right)\skliaustasd}_\mathrm{rhs}\\
&\stackrel{\text{\reftext{\eqref{non_est_main}}}}{\lesssim}&
\left({\skliaustask v^{(1)}\skliaustasd}_\mathrm{sol}
+ {\skliaustask v^{(2)}\skliaustasd}_\mathrm{sol}\right)
{\skliaustask v^{(1)} - v^{(2)}\skliaustasd}_\mathrm{sol} \\
&\stackrel{\text{\reftext{\eqref{def_sepsv}}}}{\le}&
2\varepsilon^\prime{\skliaustask v^{(1)} - v^{(2)}\skliaustasd}_\mathrm{sol}.
\end{eqnarray*}
Thus for $\varepsilon^\prime> 0$ small enough there exists a choice of $\varepsilon > 0$
such that that $\mathcal T_{v^{(0)}}$ is indeed a contraction,
and therefore it allows for a fixed point. The property that
\begin{equation}\label{fixedpt}
\mathcal T_{v^{(0)}}(v)=v
\end{equation}
directly establishes the existence part of the theorem, in view of the
definition \reftext{\eqref{def_T_v0}}.

\medskip
\noindent\textit{A-priori bounds.}
We now use the above definitions and estimates to find
\begin{eqnarray*}
{\skliaustask v\skliaustasd}_\mathrm{sol}&\stackrel{\text{\reftext{\eqref{fixedpt}}}}{=}&{\skliaustask\mathcal T_{v^{(0)}}(v)\skliaustasd
}_\mathrm{sol}\stackrel{\text{\reftext{\eqref{def_T_v0}}}}{=}{\skliaustask
T_{v^{(0)}}\left({\mathcal N}(v)\right
)\skliaustasd}_\mathrm{sol}\\
&\stackrel{\text{\reftext{\eqref{norm_control_T}}}}{\le}& C_1 \left(
{\left\lVert v^{(0)} \right\rVert}_\mathrm{init} + {\skliaustask
{\mathcal N}(v)\skliaustasd}_\mathrm{rhs}\right)
\stackrel{\text{\reftext{\eqref{norm_control_N}}}}{\le} C_2 \left( {\left
\lVert v^{(0)} \right\rVert}_\mathrm{init} + {\skliaustask
v\skliaustasd}_\mathrm{sol}^2\right),
\end{eqnarray*}
with constants $C_j$.
Fixing $\varepsilon^\prime< \frac{1}{C_2}$ we can absorb the
last term above to the left obtaining
\begin{equation*}
{\skliaustask v\skliaustasd}_\mathrm{sol} \le\frac{1}{1-C_2
\varepsilon^\prime} {\left\lVert v^{(0)} \right\rVert
}_\mathrm{init},
\end{equation*}
which gives the a-priori bound \reftext{\eqref{main_estimate}}.

\medskip
\noindent\textit{Uniqueness.}
Assume that for a given $v^{(0)}$ as in the statement of the theorem,
there exist two different finite-norm solutions $v^{(1)}$ and $v^{(2)}$ to our
equation,
where we assume that $v^{(1)}$ is the previously
constructed one meeting the a-priori estimate
${\skliaustask v^{(1)}\skliaustasd}_\mathrm{sol} \lesssim{\left\lVert v^{(0)} \right\rVert
}_\mathrm{init} \le \varepsilon$. Then due to the continuity
in time as proved in
\reftext{Corollary~\ref{coroll:contnorm}} we find that there exists a maximal
time $t^*$ such
that $v^{(1)}(t,x,y)=v^{(2)}(t,x,y)$ on $[0,t^*]\times(0,\infty
)\times\mathbb{R}
$ and $v^{(1)}(t^*+t,\cdot,\cdot)\neq v^{(2)}(t^*+t,\cdot,\cdot)$ for
all $t > 0$ small enough. Then using the traces
$v^{(1)}|_{t=t^*}=v^{(2)}|_{t=t^*}$ as initial data for our equation on
$I_\tau:=[t^*,t^*+\tau)$ and using the norms ${\skliaustask\cdot
\skliaustasd}_{\mathrm
{sol},\tau}$ defined like in \reftext{\eqref{norm_sol}}, but where the interval
in \reftext{\eqref{norm_sol}} is chosen to be $I:=I_\tau$,
we have by \reftext{\eqref{prop:non_est}}, \reftext{\eqref{norm_control_T}},
and \reftext{\eqref{def_T_v0}}
\begin{eqnarray}\label{uniquenesseq}
{\skliaustask v^{(1)}-v^{(2)}\skliaustasd}_{\mathrm{sol},\tau} &=&
{\skliaustask\mathcal T_{v^{(0)}}\left(v^{(1)}\right)-\mathcal
T_{v^{(0)}}\left(v^{(2)}\right)\skliaustasd}_{\mathrm{sol},\tau
}\nonumber\\
&\lesssim&\left({\skliaustask v^{(1)}\skliaustasd}_{\mathrm
{sol},\tau} + {\skliaustask v^{(2)}\skliaustasd}_{\mathrm{sol},\tau
}\right) {\skliaustask v^{(1)} - v^{(2)}\skliaustasd}_{\mathrm
{sol},\tau}.
\end{eqnarray}
However, due to the last item of \reftext{Corollary~\ref{coroll:contnorm}},
we have for $j=1,2$ that
\begin{equation*}
\lim_{\tau \searrow t^*} {\skliaustask v^{(j)}\skliaustasd}_{\mathrm{sol},\tau} \stackrel{\text{\reftext{\eqref{norm_init},\eqref{norm_sol}}}}{=}
 {\left\lVert v^{(j)}|_{t=t^*} \right\rVert}_\mathrm{init}
 \stackrel{\text{\reftext{\eqref{main_estimate}}}}{\lesssim} {\left\lVert v^{(0)} \right\rVert}_\mathrm{init} \le \varepsilon,
\end{equation*}
which for $0 < \varepsilon\ll1$ contradicts \reftext{\eqref{uniquenesseq}},
concluding the proof of uniqueness.

\medskip
\noindent\textit{Stability.}
As a consequence of the bound
${\skliaustask v\skliaustasd}_\mathrm{sol}<\infty$ for the solution,
by using the first item of \reftext{Corollary~\ref{coroll:contnorm}},
the stability property $\lim_{t\to+\infty} {\left\lVert
v(t,\cdot,\cdot) \right\rVert}_\mathrm{init}=0$ directly follows. This concludes the proof of \reftext{Theorem~\ref{main_th}}.
\end{proof}



\appendix
\renewcommand{\thesection}{\Alph{section}}

\section{Coordinate transformations}\label{app:coordtrans}
\subsection{Transformation onto a fixed domain}\label{app:trafo_fix}
In this appendix, we transform problem~\reftext{\eqref{tfe_free}} via the change
of variables \reftext{\eqref{hodograph}}. First note that for a function $f =
f(t,y,z)$, the gradient of the transformed function $\tilde f(t,x,y):=
f\left(t,y,Z(t,x,y)\right)$ is given by:
%
%
\begin{equation}\label{transf_deriv}
\begin{pmatrix} \tilde f_t \\ \tilde f_x \\ \tilde f_y
\end{pmatrix}
=
\begin{pmatrix} f_t + f_z Z_t \\ f_z Z_x \\ f_y + f_z Z_y
\end{pmatrix}
=
\begin{pmatrix} 1 & 0 & Z_t\\ 0 & 0 & Z_x \\ 0 & 1 & Z_y
\end{pmatrix}
\cdot
\begin{pmatrix} f_t \\ f_y \\ f_z
\end{pmatrix}
.
\end{equation}
By inverting the matrix, we can read off the transformation of
derivatives from those of $f$ to those of $\tilde f$ as
%
%
\begin{subequations}\label{trafo_der}
%
%
\begin{align}
\partial_t &\mapsto\partial_t - Z_t F \partial_x, \quad\mbox{where}
\quad F := Z_x^{-1},
\eqncr
\partial_y &\mapsto\partial_y - G \partial_x, \quad\mbox{where}
\quad
G := Z_x^{-1} Z_y,
\eqncr
\partial_z &\mapsto F \partial_x.
\end{align}
\end{subequations}
Combining \reftext{\eqref{hodograph}} and \reftext{\eqref{trafo_der}} we find for \reftext{\eqref{tfe_higher}}
%
%
\begin{align}\label{tfe_transformed0}
(\partial_t - F Z_t \partial_x) x^{\frac3 2} + \left((\partial_y - G
\partial_x) x^3 (\partial_y - G \partial_x) + F \partial_x x^3 F
\partial_x \right) \left((\partial_y - G \partial_x)^2 + (F
\partial
_x)^2\right) x^{\frac3 2} = 0
\end{align}
for $t, x > 0$ and $y \in\mathbb{R}$. Now we may use that
\begin{equation*}
\left((\partial_y - G \partial_x)^2 + (F \partial_x)^2\right)
x^{\frac
3 2} = \frac{3}{2 x^{\frac1 2}} \left(- D_y G + G \left(D_x +
\tfrac12\right) G + F \left(D_x + \tfrac1 2\right) F\right),
\end{equation*}
with $D_x = x \partial_x$ and $D_y = x \partial_y$ (cf. \reftext{\eqref{log_der}}). Next, we observe that
\begin{align*}
& \left((\partial_y - G \partial_x) x^3 (\partial_y - G \partial
_x) + F
\partial_x x^3 F \partial_x \right) x^{- \frac1 2}= x^{\frac1 2}
\left
(D_y^2 - D_y G \left(D_x - \tfrac1 2\right) - G D_y \left(D_x +
\tfrac
3 2\right) \right. \\
& \qquad\qquad\left. + G \left(D_x + \tfrac3 2\right) G \left
(D_x -
\tfrac1 2\right) + F \left(D_x + \tfrac3 2\right) F \left(D_x -
\tfrac1 2\right)\right)
\end{align*}
and as a result equation~\reftext{\eqref{tfe_transformed0}} turns into \reftext{\eqref{tfe_transformed}}.

\subsection{Advection velocity and expansions near the contact
line}\label{app:advection}
We recall that, with the notations introduced in \reftext{\eqref{trafo_der}}, and
given the definition \reftext{\eqref{def_V}} of the velocity $V$ under which the
height $h \stackrel{\text{\reftext{\eqref{hodograph}}}}{=} x^{\frac3 2}$ of our fluid
is advected via \reftext{\eqref{tfe_higher}}, we obtain
%
%
\begin{eqnarray}\label{reexpr_v}
V&:=&h\nabla\Delta h = h
\begin{pmatrix} \partial_y \\ \partial_z
\end{pmatrix}
\left(\partial_y^2 + \partial_z^2\right)h = x^{\frac3 2}
\begin{pmatrix} \partial_y-G\partial_x\\
F\partial_x
\end{pmatrix}
\left(\left(\partial_y-G\partial_x\right)^2 + \left(F\partial
_x\right
)^2\right)x^{\frac3 2}\nonumber\\
&=&\frac3 2 x^{\frac1 2}
\begin{pmatrix} D_y-GD_x\\FD_x
\end{pmatrix}
x^{-\frac1 2}\left(-D_yG+G\left(D_x+\tfrac1 2\right)G +F\left
(D_x+\tfrac
1 2\right)F\right)\nonumber\\
&=&\frac3 2
\begin{pmatrix} D_y-G\left(D_x-\tfrac12\right)\\
F\left(D_x-\tfrac
12\right)
\end{pmatrix}
\underbrace{\left(-D_yG+G\left(D_x+\tfrac1 2\right)G +F\left
(D_x+\tfrac
1 2\right)F\right)}_{=: P}.
\end{eqnarray}
We note that the traveling-wave profile $Z_\mathrm{TW}=x-\tfrac38t$
meets $(Z_\mathrm{TW})_x = 1$ and $(Z_\mathrm{TW})_y = 0$. Using the
power series \reftext{\eqref{as_sol}} valid as $x\searrow0$, we expand the
expression of $V$ as $x\searrow0$ into powers of $x$ (cf.~\reftext{\eqref{asy_V}}), i.e., almost everywhere
%
%
\begin{equation}\label{asy_V_gen}
V(t,y,Z(t,x,y)) = V_0(t,y) + V_\beta(t,y)x^\beta+V_1(t,y)x + o\left
(x^{1+\delta}\right) \quad\mbox{as} \quad x\searrow0.
\end{equation}
Now we separate the terms based on formulas~\reftext{\eqref{trafo_der}} and \reftext{\eqref{def_v}}, i.e., $F=(1+v_x)^{-1}$ and $G=v_y(1+v_x)^{-1}$. With the
notation as in \reftext{\eqref{expansion_v}} or \reftext{\eqref{as_sol}}, we obtain almost
everywhere
%
%
\begin{subequations}\label{expansion_FG}
%
%
\begin{eqnarray}
\noxml{\hspace*{-26pt}}D^\ell F &=& D^\ell\left(\frac1{1+v_1} - \frac{(1+\beta)v_{1+\beta
}}{(1+v_1)^2}x^\beta- \frac{2 v_2}{(1+v_1)^2}x\right) +
o(x^{1+\delta
}),\label{expansion_F}
\eqncr
\noxml{\hspace*{-26pt}}D^\ell G &=& D^\ell\left(\frac{(v_0)_y}{1+v_1} - \frac{(1+\beta
)(v_0)_yv_{1+\beta}}{(1+v_1)^2}x^\beta+ \left(\frac{(v_1)_y}{1+v_1} -
\frac{2(v_0)_yv_2}{(1+v_1)^2}\right)x\right) + o(x^{1+\delta})
\label{expansion_G}
\end{eqnarray}
\end{subequations}
as $x \searrow 0$, where we need to allow for $ {\left\lvert\ell\right
\rvert} \le2$. With $P$ defined as in \reftext{\eqref{reexpr_v}} and $D^\ell P= D^\ell\bigl(P_0+P_\beta x^\beta+ P_1
x\bigr) + o(x^{1+\delta})$ for $ {\left\lvert\ell
\right\rvert}\le1$, a straight-forward
computation gives
%
%
\begin{subequations}\label{expansion_P}
%
%
\begin{eqnarray}
P_0 &=& \frac12\frac{(v_0)_y^2+1}{(1+v_1)^2},
\eqncr
P_\beta&=& - \frac{(1+\beta)^2v_{1+\beta}\left((v_0)_y^2+1\right
)}{(1+v_1)^3},
\eqncr
P_1 &=& -\frac{(v_0)_{yy}}{1+v_1}+3\frac
{(v_0)_y(v_1)_y}{(1+v_1)^2}-4\frac{\left((v_0)_y^2 +1\right)v_2}{(1+v_1)^3}.
\end{eqnarray}
\end{subequations}
We then find by a direct computation that \reftext{\eqref{asy_V_gen}} holds with
$V_\beta=0$, thus confirming \reftext{\eqref{asy_V}}. More precisely, we find the
following expressions for $V_0$, $V_\beta$, and $V_1$:
%
%
\begin{subequations}\label{coeff_V}
%
%
\begin{eqnarray}
V_0 &=& \frac38 \frac{1+(v_0)_y^2}{(1+v_1)^3}
\begin{pmatrix} (v_0)_y \\
-1
\end{pmatrix}
,\quad\quad
V_\beta=
\begin{pmatrix} 0 \\ 0
\end{pmatrix}
,\label{coeff_V0_Vbeta}
\eqncr
V_1 &=& \frac38
\begin{pmatrix} \frac{6(v_0)_y(v_0)_{yy}}{(1+v_1)^2} - \frac{\left
(6+9(v_0)_y^2\right)(v_1)_y}{(1+v_1)^3} + \frac{6\left((v_0)_y^3 +
(v_0)_y\right)v_2}{(1+v_1)^4}\\
-\frac{2(v_0)_{yy}}{(1+v_1)^2} + \frac{6(v_0)_y(v_1)_y}{(1+v_1)^3} -
\frac{6\left((v_0)_y^2 +1\right)v_2}{(1+v_1)^4}
\end{pmatrix}
.\label{coeff_V1}
\end{eqnarray}
\end{subequations}
Due to the advection equation \reftext{\eqref{def_V}}, if $\partial\{h > 0\}$ is at
time $t=0$ the graph
\begin{equation*}
\Gamma_0:= \left\{\left(y,Z_0^{(0)}(y)\right):\ y\in\mathbb
{R}\right\},
\end{equation*}
then the advected contact line at time $t$, which we denote $\Gamma
_t:=\partial\{h > 0\}$, can be parameterized by
%
%
\begin{equation}\label{contact_1}
\Gamma_t=\{(Y(t,y),Z(t,y)):\ y\in\mathbb{R}\},
\end{equation}
and the above parameterization evolves under the advection equation \reftext{\eqref{advection}}.

We claim that we may express $\Gamma_t$ as the graph of a regular
function, i.e., that there exists a regular function $Z_0:[0,\infty
)\times\mathbb{R}\to\mathbb{R}$ which satisfies $Z_0(0,y) =
Z_0^{(0)}(y)$ and
%
%
\begin{equation}\label{def_Z0}
\Gamma_t=\{(y,Z_0(t,y)):\ y\in\mathbb{R}\}\text{ for all }t\ge0.
\end{equation}
To obtain \reftext{\eqref{def_Z0}}, note that under the condition that
%
%
\begin{equation}\label{graphcondition}
\sup_{t\in I} {\left\lvert(v_0)_y \right\rvert
}_{BC^0(\mathbb{R}_y)}<\infty,
\end{equation}
the velocity $V_0$ is never perpendicular to the $z$-axis during our
time-evolution. As a consequence of the advection equation $\partial_t
h +\nabla_{y,z}\cdot(hV)=0$ we also find that $V_0$ remains always
orthogonal to $\partial\{h > 0\}$, and thus this set remains a graph at
all times. As a result of the regularity of $V_0$ we also find that
$\partial\{h > 0\}$ remains regular at all times. We may ensure
arbitrarily high regularity depending on the parameters $\tilde k$,
$\check k$, and $\breve k$ in our norms according to the bounds \reftext{\eqref{bc0_bounds_summ}}, proved in \reftext{Lemma~\ref{lem:bc0_bounds}}, though the
question of smoothing of the free boundary is not investigated here.
The above sufficient condition \reftext{\eqref{graphcondition}} is obtained from
$ {\left\lVert v^{(0)} \right\rVert}_\mathrm
{init}<\varepsilon$ which we assumed in \reftext{Theorem~\ref{main_th}} due to
our main estimate \reftext{\eqref{main_estimate}}, coupled with
the bound \reftext{\eqref{bc0_grad_init}} of \reftext{Lemma~\ref{lem:bc0_bounds}}.

In order to invert the hodograph transform and to pass from the
understanding of the function $v$ to that of $h$, we use the formulas
%
%
\begin{equation}\label{hodoeq}
\tilde z(t,x,y):=Z(t,x,y) - Z_0(t,y) = \int_0^xZ_x(t,x^\prime
,y)\mathrm{d}
x^\prime, \ \ Z_x=(Z_\mathrm{TW})_x+v_x = 1+v_x.
\end{equation}
From the expansion in powers of $x$ in \reftext{\eqref{as_sol}} we find almost everywhere
%
%
\begin{equation}\label{hodo2}
\tilde z = (1+v_1) x + v_{1+\beta} x^{1+\beta} + v_2x^2 +
o\left(x^{2+\delta
}\right) \quad\mbox{as} \quad x \searrow0,
\end{equation}
from which we infer that $\tilde z(t,x,y)$ is strictly
monotone in $x$ as
$x\searrow0$ and satisfies $\tilde z(t,x,y)\to0$ as \mbox{$x\searrow0$}.
Furthermore, we can invert \reftext{\eqref{hodo2}} for $x$ in a neighborhood of
$0$ and find almost everywhere
%
%
\begin{subequations}\label{hodo_inv}
%
%
\begin{eqnarray}
x&=&\frac{1}{1+v_1} \tilde z - \frac{v_{1+\beta}}{(1+v_1)^{2+\beta}}
\tilde z^{1+\beta} -\frac{v_2}{(1+v_1)^3}\tilde z^2+ o\left(\tilde
z^{2+\delta}\right)
\eqncr
&=&\frac{\tilde z}{1+v_1}\left(1-\frac{v_{1+\beta
}}{(1+v_1)^{1+\beta}}
\tilde z^\beta-\frac{v_2}{(1+v_1)^2}\tilde z+ o\left(\tilde
z^{1+\delta
}\right)\right)
\end{eqnarray}
\end{subequations}
as $\tilde z \searrow0$ and using this in the hodograph transform \reftext{\eqref{hodograph}}, we find via the Taylor expansion of $x^{\frac3 2}$
near $x=1$ that almost everywhere
%
%
\begin{align}\label{hodo_inv_2}
h(t,y,Z(t,x,y))=x^{\frac32}=\frac{\tilde z^{\frac32}}{(1+v_1)^{\frac
32}}\left(1-\frac32\frac{v_{1+\beta}}{(1+v_1)^{1+\beta}}\tilde
z^\beta
-\frac32\frac{v_2}{(1+v_1)^2}\tilde z+ o\left(\tilde z^{1+\delta
}\right
)\right)
\end{align}
as $\tilde z \searrow0$ holds true. Now \reftext{\eqref{def_Z0}}, \reftext{\eqref{hodoeq}},
and \reftext{\eqref{hodo_inv_2}} allow to re-write the equations
expressing the evolution of the moving interface $Z_0(t,y)$ starting at
$Z_0^ {(0)}(y)$ in the original coordinates, leading to \reftext{\eqref{hodo3}}.
We thus observe that up to factoring the traveling-wave profile $\tilde
z^{\frac32}$ the function $h$ has the same form as $v$ itself, and the
supremum norm bounds on $v_1$, $v_{1+\beta}$, and $v_2$ give bounds on
the coefficients of $h$ as well.

\section{Proofs of auxiliary results}\label{app:proofs}

\subsection{Proofs of the norm-equivalence lemmata}\label{app:equivnorm}

\begin{proof}[Proof of \reftext{Lemma~\ref{lem:equivalence_init}}]
Throughout the proof, estimates depend on $\tilde k$, $\check k$,
$\breve k$, and $\delta$. Then the condition $ {\left
\lVert v^{(0)} \right\rVert}_\mathrm
{init}<\infty$ allows to verify the conditions of \reftext{Lemma~\ref{lem:mainhardy}} and to use estimate~\reftext{\eqref{est_hardy}}
with $\rho= - \delta - \frac 1 2-\ell_y$ and $\overline\gamma=0$ in
\begin{eqnarray}
{\left\lVert D_y v^{(0)} \right\rVert}_{\tilde k - 1,-\delta}^{2}
&\sim& {\left\lVert\partial_y v^{(0)} \right\rVert}_{\tilde k - 1, -\delta-1}^{2}
= \sum_{0 \le \ell_x + \ell_y \le \tilde k - 1} \left\lVert \partial_y^{\ell_y+1} D_x^{\ell_x} v\right\rVert_{-\delta-1-\ell_y}^2 \nonumber\\
&\stackrel{\text{\reftext{\eqref{est_hardy}}}}{\lesssim}& \sum_{0 \le \ell_x + \ell_y \le \tilde k - 1} \left\lVert \partial_y^{\ell_y+1} D_x^{\ell_x+1} v\right\rVert_{-\delta-1-\ell_y}^2 \sim {\left\lVert\partial_y D_x v^{(0)} \right\rVert}_{\tilde k - 1, -\delta-1}^{2} \sim {\left\lVert D_y D_x v^{(0)} \right\rVert}_{\tilde k -1, -\delta}^{2} \nonumber\\
&\le& {\left\lVert D_x v^{(0)} \right\rVert}_{\tilde k, -\delta}^{2},\label{bd_xy}
\end{eqnarray}
where we have additionally used the fact that $\partial_y$ commutes
with $D_x$. In order to apply \reftext{Lemma~\ref{lem:mainhardy}},
we infer that
$\partial_y^{\ell_y+1} D_x^{\ell_x} v^{(0)}\left(\overline x_n\right) \to 0$
as $n \to \infty$ for all $(\ell_y,\ell_y)\in \N_0^2$
with $\ell_x + \ell_y \le \tilde k - 1$
and a sequence $\left(\overline x_n\right)_{n \in \N}$ with $\overline x_n \to \infty$
as $n \to \infty$, because
$\left\lVert \partial_y^{\ell_y+1} D_x^{\ell_x} v^{(0)} \right\rVert_{-\delta-2-\ell_y}
\lesssim \left\lVert v^{(0)} \right\rVert_{k,-\delta-1} \le \left\lVert v^{(0)} \right\rVert_\mathrm{init} <\infty$,
where we have used $k \ge \tilde k$.
Applying estimate~\reftext{\eqref{est_hardy}} of
We further obtain
\begin{eqnarray}\nonumber
 {\left\lVert D_y D_x v^{(0)} \right\rVert}_{\tilde k-1,\delta}^{2}
 &\sim&  \sum_{0 \le \ell_x + \ell_y \le \tilde k - 1} {\left\lVert\partial_y^{\ell_y+1} D_x^{\ell_x+1} v^{(0)} \right\rVert}_{\delta-1-\ell_y}^2 \\
 &\stackrel{\text{\reftext{\eqref{est_hardy}}}}{\lesssim}&
\sum_{0 \le \ell_x + \ell_y \le \tilde k - 1} {\left\lVert\partial_y^{\ell_y+1} \left(D_x-1-\beta\right) D_x^{\max\{\ell_x-1,0\}+1} v^{(0)} \right\rVert}_{\delta-1-\ell_y}^2 \nonumber\\
 &\stackrel{\text{\reftext{\eqref{est_hardy}}}}{\lesssim}&
\sum_{0 \le \ell_x + \ell_y \le \tilde k - 1} {\left\lVert\partial_y^{\ell_y+1} \left(D_x-1\right) \left(D_x-1-\beta\right) D_x^{\max\{\ell_x-2,0\}+1} v^{(0)} \right\rVert}_{\delta-1-\ell_y}^2 \nonumber\\
 &\stackrel{\text{\reftext{\eqref{est_hardy}}}}{\lesssim}&
\sum_{0 \le \ell_x + \ell_y \le \tilde k - 1} {\left\lVert\partial_y^{\ell_y+1} \left(D_x+\beta-\tfrac 12\right) \left(D_x-1\right) \left(D_x-1-\beta\right) D_x^{\max\{\ell_x-3,0\}+1} v^{(0)} \right\rVert}_{\delta-1-\ell_y}^2 \nonumber\\
&\stackrel{\text{\reftext{\eqref{est_hardy}}}}{\lesssim}&
\sum_{0 \le \ell_x + \ell_y \le \tilde k - 1} {\left\lVert\partial_y^{\ell_y+1} \tilde q(D_x) D_x^{\max\{\ell_x-4,0\}+1} v^{(0)} \right\rVert}_{\delta-1-\ell_y}^2
\sim {\left\lVert D_y\tilde q(D_x) D_x v^{(0)} \right
\rVert}_{\tilde k-1,\delta}^{2} \nonumber\\
&\le&
 {\left\lVert\tilde q(D_x)D_xv^{(0)} \right\rVert
}_{\tilde k,\delta}^{2} \le \left\lVert v^{(0)} \right\rVert_\mathrm{init}^2,\label{bd_xyqtilde}
\end{eqnarray}
where we have utilized \reftext{Lemma~\ref{lem:mainhardy}} with
$\rho= \delta - \frac 1 2 -\ell_y$ and the choices $\overline\gamma=1+\beta,1, -\beta+\frac12, - \frac 1 2$ for $\ell_y \ge 1$, and
$\overline\gamma=1+\beta,1,-\beta+\frac12$ and $\underline\gamma=-\frac12$ for $\ell_y = 0$. We first verify that, for a sequence $\left(\overline x_n\right)_{n \in \N}$ with $\overline x_n \to \infty$ as $n \to \infty$ there holds, almost everyhwhere in $y \in \R$,
\begin{equation*}
\overline x_n^{-1-\beta}\partial_y^{\ell_y+1} D_x^{\ell_x+1} v^{(0)}\left(\overline x_n,y\right) \to 0, \quad \overline x_n^{-1} \partial_y^{\ell_y+1} (D_x-1-\beta) D_x^{\max\{\ell_x-1,0\}+1} v^{(0)}\left(\overline x_n,y\right) \to 0,
\end{equation*}
and
\begin{equation*}
\overline x_n^{\beta - \frac 1 2} \partial_y^{\ell_y+1} \left(D_x-1\right) (D_x-1-\beta) D_x^{\max\{\ell_x-2,0\}+1} v^{(0)}\left(\overline x_n,y\right) \to 0 \quad \mbox{as} \quad n \to \infty
\end{equation*}
for all $\left(\ell_x,\ell_y\right) \in \N_0^2$ with $\ell_x+\ell_y \le \tilde k-1$,
as well as 
\begin{equation*}
\overline x_n^{\frac 1 2} \partial_y^{\ell_y+1} \left(D_x+\beta-\tfrac 12\right) \left(D_x-1\right) (D_x-1-\beta) D_x^{\max\{\ell_x-1,0\}+1} v^{(0)}\left(\overline x_n,y\right) \to 0 \quad \mbox{as} \quad n \to \infty
\end{equation*}
if $\ell_y \ge 1$. In order to prove these asymptotics, we use the bound
\begin{align*}
& {\left\lVert \partial_y^{\ell_y+1} D_x^{\ell_x+1} v^{(0)} \right\rVert}_{-\delta-2-\ell_y} + {\left\lVert \partial_y^{\ell_y+1} (D_x-1-\beta) D_x^{\max\{\ell_x-1,0\}+1} v^{(0)} \right\rVert}_{-\delta-2-\ell_y} \\
& \quad + {\left\lVert \partial_y^{\ell_y+1} \left(D_x-1\right) (D_x-1-\beta) D_x^{\max\{\ell_x-2,0\}+1} v^{(0)} \right\rVert}_{-\delta-2-\ell_y} \lesssim \left\lVert v^{(0)} \right\rVert_{k,-\delta-1}
\le \left\lVert v^{(0)} \right\rVert_\mathrm{init} <\infty
\end{align*}
and in the last case, for $\ell_y \ge 1$,
\begin{equation*}
{\left\lVert \partial_y^{\ell_y+1} \left(D_x+\beta-\tfrac 12\right) \left(D_x-1\right) \left(D_x-1-\beta\right) D_x^{\max\{\ell_x-3,0\}+1} v^{(0)} \right\rVert}_{-\delta-2-\ell_y} \lesssim \left\lVert v^{(0)} \right\rVert_{k,-\delta-1}
\le \left\lVert v^{(0)} \right\rVert_\mathrm{init} < \infty,
\end{equation*}
which requires $k \ge \tilde k + 4$. In order to verify the remaining condition
\begin{equation*}
\underline x_n^{\frac 1 2} w \left(\underline x_n,y\right) \to 0 \quad \mbox{as} \quad n \to \infty
\end{equation*}
for $w := \partial_y \left(D_x+\beta-\tfrac 12\right) \left(D_x-1\right)
\left(D_x-1-\beta\right) D_x^{\max\{\ell_x-3,0\}+1} v^{(0)}$, for all $\ell_x \in \N_0$ with $\ell_x \le \tilde k - 1$ and a sequence $\left(\underline x_n\right)_{n \in \N}$ with $\underline x_n \to 0$ as $n \to \infty$ almost everyhwere in $y \in \R$,
we note that
\begin{equation*}
\left\lVert \left(D_x+\tfrac 1 2\right) w \right\rVert_{\delta-1} \le \left\lVert \tilde q(D_x)D_xv^{(0)} \right\rVert_{\tilde k,\delta} \le \left\lVert v^{(0)} \right\rVert_\mathrm{init} < \infty,
\end{equation*}
which implies
\begin{equation*}
\underline x_n^{\frac 1 2} w\left(\underline x_n,y\right) = w_{-\frac 1 2}(y) + \int_0^{\underline x_n} x^{\frac 1 2} \left(\left(D_x+\tfrac12\right) w\right)\left(x,y\right) \, \frac{\d x}{x} \quad \to\quad w_{-\frac 1 2}(y) \quad \mbox{as} \quad n \to \infty,
\end{equation*}
because
\begin{equation*}
\left\lvert \underline \int_0^{\underline x_n} x^{\frac 1 2} \left(\left(D_x+\tfrac12\right) w\right)\left(x,y\right) \, \frac{\d x}{x} \right\rvert \le \left(\int_0^{\underline x_n} x^{4-2\delta} \, \frac{\d x}{x}\right)^{\frac 1 2} \left\lvert \left(D_x + \tfrac 1 2\right) v(\cdot,y) \right\rvert_{\delta-1} \to 0 \quad \mbox{as} \quad n \to \infty
\end{equation*}
for a sequence $\left(\underline x_n\right)_{n \in \N}$ such that $\underline x_n \to 0$ as $n \to \infty$ almost everywhere in $y \in \R$. On the other hand,
\begin{equation*}
w_{-\frac 1 2} = (\beta-1) \left(-\tfrac 3 2\right) \left(-\tfrac 3 2-\beta\right) \left(-\tfrac 1 2\right)^{\max\{\ell_x-3,0\}} \left(D_x v^{(0)}\right)_{-\frac 1 2}
\end{equation*}
and $\left\lVert D_x v^{(0)} \right\rVert_{\tilde k, - \delta} \le \left\lVert v^{(0)} \right\rVert_\mathrm{init} < \infty$, which ensures $\left(D_x v^{(0)}\right)_{-\frac 1 2} = 0$ almost everywhere, and therefore $\underline x_n^{\frac 1 2} w(x_n,y) \to 0$ as $n \to \infty$ almost everywhere in $y \in \R$.

\medskip

Similarly to \reftext{\eqref{bd_xyqtilde}} and \reftext{\eqref{bd_xy}}, we obtain the bounds
\begin{subequations}
\begin{align}
 {\left\lVert D_y^2 v^{(0)} \right\rVert}_{\check k -2,
-\delta+1}&\lesssim {\left\lVert D_yD_xv^{(0)} \right
\rVert}_{\check k -1,-\delta+1}\lesssim {\left\lVert
\tilde q(D_x)D_xv^{(0)} \right\rVert}_{\check k, -\delta+1},\label{bd1}\\
 {\left\lVert D_y^2 D_x v^{(0)} \right\rVert}_{\check k
-2,\delta+1}&\lesssim {\left\lVert D_y\tilde
q(D_x)D_xv^{(0)} \right\rVert}_{\check k -1,\delta+1}\lesssim
 {\left\lVert(D_x-3)(D_x-2)\tilde q(D_x)D_xv^{(0)} \right
\rVert}_{\check k,\delta+1},\label{bd2}\\
 {\left\lVert D_y^3 v^{(0)} \right\rVert}_{\breve k
-2,-\delta+2} &\lesssim {\left\lVert D_y^2 D_x v^{(0)} \right\rVert}_{\breve k -1,-\delta+2}\lesssim
 {\left\lVert D_y \tilde q(D_x) D_x v^{(0)} \right
\rVert}_{\breve k,-\delta+2}.\label{bd3}
\end{align}
\end{subequations}
Inequalities \reftext{\eqref{bd1}, \eqref{bd2}, \eqref{bd3}}, together with the definition \reftext{\eqref{norm_init_p}} directly allow to conclude. The proofs are briefly detailed below.

\medskip

\noindent\textit{Proof of \reftext{\eqref{bd1}}.} For the first estimate, we have
\begin{align*}
 {\left\lVert D_y^2 v^{(0)} \right\rVert}_{\check k-2,-\delta+1}^{2}
&\sim  \sum_{0\le \ell_x+\ell_y\le \check k -2}{\left\lVert\partial_y^{\ell_y+2} D_x^{\ell_x} v^{(0)} \right\rVert}_{-\delta-1-\ell_y}^{2}
\stackrel{\text{\reftext{\eqref{est_hardy}}}}{\lesssim}
\sum_{0\le\ell_x+\ell_y\le \check k -2} {\left\lVert\partial_y^{\ell_y+2} D_x^{\ell_x+1} v^{(0)}\right\rVert}_{-\delta-1-\ell_y}^{2} \\ 
&\sim {\left\lVert D_y^2D_x v^{(0)} \right\rVert}_{\check k-2,-\delta+1}^{2} \le {\left\lVert D_yD_xv^{(0)} \right\rVert}_{\check k-1,-\delta+1}^{2},
\end{align*}
where we used \reftext{\eqref{est_hardy} from Lemma~\ref{lem:mainhardy}}, with $\overline\gamma=0$ and $\rho=-\delta - \frac12-\ell_y$ for $\ell_y\ge 0$. The hypotheses of the lemma hold due to the bounds $\left\lVert \partial_y^2 v^{(0)}\right\rVert_{\check k-2,-\delta + 1}\lesssim \left\lVert v^{(0)}\right\rVert_{k,-\delta-1}\le \left\lVert v^{(0)}\right\rVert_\mathrm{init}<\infty$, for which we require $k \ge \check k$.

\medskip

For the second estimate in \reftext{\eqref{bd1}}, we start by writing
\begin{equation*}
\left\lVert D_yD_xv^{(0)}\right\rVert_{\check k-1,-\delta +1}^{2} \sim \sum_{0\le \ell_x+\ell_y\le \check k -1}\left\lVert \partial_y^{\ell_y+1}D_x^{\ell_x+1}v^{(0)}\right\rVert_{-\delta-\ell_y}^{2}, 
\end{equation*}
from which we proceed like for \reftext{\eqref{bd_xyqtilde}}, and this time we apply
\reftext{\eqref{est_hardy}} with $\rho=-\delta+\frac12-\ell_y$, and with the choices
$\overline\gamma=1+\beta,1,-\beta+\frac12,-\frac12$ for $\ell_y\ge 1$ and
$\overline\gamma=1+\beta,1$ and $\underline\gamma=-\beta+\frac12,-\frac 12$
for $\ell_y=0$. The justification of the hypotheses of
\reftext{Lemma \ref{lem:mainhardy}} are then anologous,
and lead to the condition $k \ge \check k + 4$.

\medskip

\noindent\textit{Proof of \reftext{\eqref{bd2}}.} For the first estimate, we obtain
\begin{align*}
 {\left\lVert D_y^2D_x v^{(0)} \right\rVert}_{\check k-2,\delta+1}^{2}
&\sim  \sum_{0\le\ell_x+\ell_y\le \check k-2}{\left\lVert\partial_y^{\ell_y+2} D_x^{\ell_x+1} v^{(0)}\right\rVert}_{\delta-1-\ell_y}^{2}
\stackrel{\text{\reftext{\eqref{est_hardy}}}}{\lesssim}
{\left\lVert\partial_y^{\ell_y+2}\tilde q (D_x)D_x v^{(0)} \right\rVert}_{\delta-1-\ell_y}^{2} \\
&\sim {\left\lVert D_y^2\tilde q(D_x)D_x v^{(0)} \right\rVert}_{\check k-2,\delta+1}^{2}
\le {\left\lVert D_y\tilde q(D_x)D_xv^{(0)} \right\rVert}_{\check k-1,\delta+1}^{2}.
\end{align*}
Here, we have applied \reftext{Lemma~\ref{lem:mainhardy}} with $\rho=\delta-\frac12-\ell_y$ and with the same subdivision of cases for $\underline\gamma$ and $\overline \gamma$ depending on $\ell_y$ as for \reftext{\eqref{bd_xyqtilde}}. In order to prove the hypotheses of decay at infinity required for \reftext{Lemma~\ref{lem:mainhardy}}, we again use the finiteness of $\left\lVert v^{(0)}\right\rVert_{k,-\delta-1}$, to do which we need the condition $k\ge \check k + 4$. 

\medskip

For the second estimate in \reftext{\eqref{bd2}}, we start from
\begin{equation*}
\left\lVert D_y\tilde q(D_x)D_xv^{(0)}\right\rVert_{\check k -1,\delta+1}^{2}
\sim \sum_{0\le \ell_x+\ell_y\le \check k-1}\left\lVert\partial_y^{\ell_y+1}\tilde q(D_x)D_x^{\ell_x+1}v^{(0)}\right\rVert_{\delta-\ell_y}^{2},
\end{equation*}
and then we apply again \reftext{Lemma~\ref{lem:mainhardy}}, this time with $\rho=\delta+\frac12-\ell_y$ and $\overline\gamma=2,3$. In order to justify the application of the lemma, now we need to require $k\ge \check k+6$.

\medskip

\noindent\textit{Proof of \reftext{\eqref{bd3}}.} For the first estimate we write 
\begin{align*}
 {\left\lVert D_y^3 v^{(0)} \right\rVert}_{\breve k-2,-\delta+2}^{2} 
 &\sim \sum_{0\le\ell_x+\ell_y\le \breve k-2}{\left\lVert\partial_y^{\ell_y+3}D_x^{\ell_x} v^{(0)} \right\rVert}_{-\delta-1-\ell_y}^{2}
 \stackrel{\text{\reftext{\eqref{est_hardy}}}}{\lesssim}
\sum_{0\le \ell_x+\ell_y\le \breve k-2} {\left\lVert\partial_y^{\ell_y+3} D_x^{\ell_x+1} v^{(0)} \right\rVert}_{-\delta-1-\ell_y}^{2}\\
&\sim {\left\lVert D_y^2D_x v^{(0)} \right\rVert}_{\breve k-1,-\delta+2}^{2},
\end{align*}
for which, as before, we have applied \reftext{Lemma~\ref{lem:mainhardy}} with $\overline\gamma=0$ and $\rho=-\delta-\frac 12-\ell_y$, to justify which, as above, we require $k \ge \breve k+2$. For the second estimate in \reftext{\eqref{bd2}} we write, along the same lines as above,
\begin{align*}
 {\left\lVert D_y^2D_x v^{(0)} \right\rVert}_{\breve k-1,-\delta+2}^{2} 
&\sim  \sum_{0\le\ell_x+\ell_y\le \breve k-1}{\left\lVert\partial_y^{\ell_y+2} D_x^{\ell_x+1} v^{(0)}\right\rVert}_{-\delta-\ell_y}^{2}
\stackrel{\text{\reftext{\eqref{est_hardy}}}}{\lesssim} 
\sum_{0\le\ell_x+\ell_y\le \breve k-1} {\left\lVert\partial_y^{\ell_y+2}\tilde q (D_x) D_x^{\ell_x+1} v^{(0)} \right\rVert}_{-\delta-\ell_y}^{2} \\
&\sim {\left\lVert D_y\tilde q(D_x)D_x v^{(0)} \right\rVert}_{\breve k,-\delta+2}^{2},
\end{align*}
where we applied \reftext{Lemma~\ref{lem:mainhardy}}
with $\rho=-\delta+\frac 12-\ell_y$ and with $\overline\gamma=1+\beta,1,-\beta+\frac12,-\frac12$ if $\ell_y\ge 1$ and $\overline\gamma=1+\beta,1$ and $\underline\gamma=-\beta+\frac12,-\frac12$ if $\ell_y=0$, which in both cases are justified similarly as for \reftext{\eqref{bd_xyqtilde}} if $k \ge \breve k+5$.

\end{proof}
\begin{proof}[Proof of \reftext{Lemma~\ref{lem:equivalence_rhs}}]
Throughout the proof, estimates depend on $\tilde k$, $\check k$,
$\breve k$, and $\delta$. The proof follows the same lines as the one
of \reftext{Lemma~\ref{lem:equivalence_init}}. By using the hypothesis
${\skliaustask f\skliaustasd}_\mathrm{rhs}<\infty$, we find that
almost everywhere in time $t \in
I$ the norms of the form $ {\left\lVert\cdot\right
\rVert}_{\kappa, \alpha}$ appearing in
the time intervals defining ${\skliaustask\cdot\skliaustasd}_\mathrm
{rhs}$ are all
finite. At such fixed $t$, we apply again repeatedly \reftext{Lemma~\ref{lem:mainhardy}}. In the present case, all roots $\underline \gamma$ and $\overline \gamma$ of polynomials in $D_x$ that appear during the proof are increased by $1$, while the weights $\rho$ to be used are increased by $\frac12$. Furthermore, all the indices $k$, $\tilde k$, $\check k$, and $\breve k$ are shifted by $-2$ compared to that case The subdivision into cases depending on $\ell_y$ presents terms that can be treated by the same strategy as in \reftext{Lemma~\ref{lem:equivalence_init}} and the relative bounds between indices do not change.
\begin{subequations}\label{est_f_equivalence}
\begin{align}
 {\left\lVert D_y f \right\rVert}_{\tilde k - 3, -\delta
+\frac12} &\lesssim {\left\lVert(D_x-1) f \right\rVert
}_{\tilde k-2, -\delta+\frac12},\label{l35a}\\
 {\left\lVert D_y(D_x-1)f \right\rVert}_{\tilde k
-3,\delta+\frac12}&\lesssim {\left\lVert\tilde
q(D_x-1)(D_x-1)f \right\rVert}_{\tilde k - 2, \delta+\frac12},\label{l35b}\\
 {\left\lVert D_y^2f \right\rVert}_{\check k -4, -\delta
+\frac32}&\lesssim {\left\lVert D_y(D_x-1)f \right
\rVert}_{\check k -3,-\delta+\frac32}\lesssim {\left
\lVert\tilde q(D_x-1)(D_x-1)f \right\rVert}_{\check k - 2,-\delta
+\frac32},\label{l35c}\\
 {\left\lVert D_y^2(D_x-1)f \right\rVert}_{\check k -
4,\delta+\frac32}&\lesssim {\left\lVert D_y\tilde
q(D_x-1)(D_x-1)f \right\rVert}_{\check k - 3, \delta+\frac32}\label{l35d}\\
&\lesssim {\left\lVert(D_x-4)(D_x-3)\tilde
q(D_x-1)(D_x-1)f \right\rVert}_{\check k - 2,
\delta+\frac32},\label{l35e}\\
{\left\lVert D_y^3 f \right\rVert}_{\breve k
- 4,-\delta + \frac 5 2} &\lesssim {\left\lVert D_y^2 (D_x-1) f \right\rVert}_{\breve k - 3,-\delta + \frac 5 2} \lesssim
 {\left\lVert D_y \tilde q(D_x-1) (D_x-1) f \right
\rVert}_{\breve k - 2,-\delta + \frac 5 2}.\label{l35f}
\end{align}
\end{subequations}
By comparing estimates~\reftext{\eqref{est_f_equivalence}} to the
definitions \reftext{\eqref{norm_rhs}} and \reftext{\eqref{norm_rhs_p}}
of the norms $\skliaustask \cdot \skliaustasd_\mathrm{rhs}$ and
$\skliaustask \cdot \skliaustasd_\mathrm{rhs}^\prime$, the thesis follows.
The justification of \reftext{\eqref{est_f_equivalence}} mainly follows
the proof \reftext{Lemma~\ref{lem:equivalence_init}}. We summarize
the case discussions as follows, where in each case we use the
norm $\left\lVert f\right\rVert_{k-2,-\delta-\frac12}$ to justify
the application of \reftext{Lemma~\ref{lem:mainhardy}}:
\begin{itemize}
\item For \reftext{\eqref{l35a}} we choose $\overline\gamma=1$ for all $\ell_y \ge 0$, where $\rho = - \delta - \ell_y$, and the hypotheses of \reftext{Lemma~\ref{lem:mainhardy}} are verified under the constraint on indices $k \ge \tilde k$.
\item For \reftext{\eqref{l35b}} we choose $\overline\gamma=2+\beta,2,-\beta+\frac32,\frac12$ for all $\ell_y\ge 0$, where $\rho = \delta - \ell_y$, under the requirement $k \ge \tilde k+4$.
\item For \reftext{\eqref{l35c}} we choose for the first estimate $\overline\gamma=1$ for all $\ell_y \ge 0$, where $\rho = - \delta - \ell_y$. For the second one we choose $\overline\gamma=2+\beta,2,-\beta+\frac32,\frac12$ for all $\ell_y\ge 1$, and $\overline\gamma=2+\beta,2$ and $\underline\gamma=-\beta+\frac32, \frac12$ for $\ell_y=0$, where $\rho = -\delta+1-\ell_y$. Here, we require the constraint $k \ge \check k+4$, and in case $\ell_y=0$ for the case $\underline\gamma$, we use finiteness of $\left\lVert (D_x-1)f\right\rVert_{\tilde k - 2,-\delta +\frac12}$ to justify the application of \reftext{Lemma~\ref{lem:mainhardy}}.
\item For \reftext{\eqref{l35d}} we choose $\overline\gamma=2+\beta,2,-\beta+\frac32,\frac12$ for all $\ell_y\ge 0$, where $\rho = \delta-\ell_y$, and we require the constraint on indices $k \ge \check k+4$.
\item For \reftext{\eqref{l35e}} we choose $\overline\gamma=3,4$ for $\rho = \delta + 1 - \ell_y$, and we require the constraint $k \ge \check k +6$.
\item For \reftext{\eqref{l35f}} we proceed with the same case subdivision and justifications as for \reftext{\eqref{l35c}}, with $\breve k$ replacing $\check k$ but with the same justification of the cases $\underline\gamma$ via the term $\left\lVert(D_x-1)f\right\rVert_{\tilde k -2,-\delta+\frac12}$. The requirement on indices which we need is $k \ge \breve k+5$, coming from the discussion of $\overline\gamma$.
\end{itemize}
\end{proof}
\begin{proof}[Proof of \reftext{Lemma~\ref{lem:equivalence_sol}}]
Throughout the proof, estimates depend on $\tilde k$, $\check k$,
$\breve k$, and $\delta$. First, we recall the definition \reftext{\eqref{not_3barnorm}} of the norms used to define the norm ${\skliaustask v\skliaustasd}_\mathrm{sol}$:
\begin{equation*}
{\skliaustask w\skliaustasd}_{\ell,\gamma}^2 = \sup_{t\in [0,\infty)} {\left\lVert w \right\rVert}_{\ell,\gamma
}^2 +\int_0^\infty {\left\lVert w \right\rVert}_{\ell-2,\gamma
-\frac12}^2\mathrm{d} t +\int_0^\infty {\left\lVert w \right
\rVert}_{\ell+2,\gamma+\frac12}^2\mathrm{d} t.
\end{equation*}
Thus, by applying \reftext{\eqref{traceest}} of \reftext{Lemma~\ref{lem:trace_estimate}} we
can absorb the supremum terms from the formula \reftext{\eqref{norm_sol}} into
the integral terms. Next, like in the proof of \reftext{Lemma~\ref{lem:equivalence_init}}, we use \reftext{Lemma~\ref{lem:mainhardy}} for the choices
\begin{equation*}
(w,\kappa,\alpha)\in\left\{(\partial_t v, 0,0),(v,4,1)\right\}.
\end{equation*}
In this case, we can prove the following bounds if the norms appearing on the right-hand side are finite:
\begin{subequations}\label{l36}
\begin{align}
 {\left\lVert D_yw \right\rVert}_{\tilde k -3 +\kappa,
-\delta-\frac12+\alpha}&\lesssim
 {\left\lVert D_xw \right\rVert}_{\tilde k - 2+\kappa
,-\delta-\frac12+\alpha},\label{l36a}\\
 {\left\lVert D_yD_xw \right\rVert}_{\tilde k - 3+\kappa
,\delta- \frac12+\alpha}&\lesssim
 {\left\lVert\tilde q(D_x)D_xw \right\rVert}_{\tilde k
-2+\kappa,\delta- \frac12+\alpha
},\label{l36b}\\
 {\left\lVert D_y^2w \right\rVert}_{\check k - 4+\kappa,
-\delta+\frac12+\alpha}&\lesssim
 {\left\lVert D_yD_xw \right\rVert}_{\check k - 3+\kappa
, -\delta+\frac12+\alpha}\lesssim
 {\left\lVert\tilde q(D_x)D_xw \right\rVert}_{\check k -
2+\kappa,-\delta+\frac12+\alpha
},\label{l36c}\\
 {\left\lVert D_y^2D_xw \right\rVert}_{\check k -4+\kappa
, \delta+\frac12+\alpha}&\lesssim
 {\left\lVert D_y\tilde q(D_x)D_xw \right\rVert}_{\check
k -3+\kappa,\delta+\frac
12+\alpha}\label{l36d}\\
&\lesssim {\left\lVert(D_x-3)(D_x-2)\tilde q(D_x)D_xw \right
\rVert}_{\check k -2+\kappa
,\delta+\frac12+\alpha},\label{l36e}\\
 {\left\lVert D_y^3 w \right\rVert}_{\breve k -4+\kappa
, -\delta+\frac32+\alpha}& \lesssim
 {\left\lVert D_y^2D_xw \right\rVert}_{\breve
k -3+\kappa,-\delta+\frac
32+\alpha} \lesssim {\left\lVert D_y \tilde q(D_x)D_xw \right
\rVert}_{\breve k -2+\kappa
,-\delta+\frac32+\alpha}.\label{l36f}
\end{align}
\end{subequations}
The above bounds can be justified like in \reftext{Lemma~\ref{lem:equivalence_init}}, with the weights $\rho$ shifted by $-\frac12+\alpha$, whereas the indices $k$, $\tilde k$, $\check k$, and $\breve k$ are all shifted by $+\kappa-2$. This modifies the partition of the roots of polynomials, underlying our polynomial operators in $D_x$, into $\underline\gamma,\overline\gamma$ for the application of the \reftext{Lemma~\ref{lem:mainhardy}}, but the justifications for the hypotheses of that lemma are easily adapted:
\begin{itemize}
\item For \reftext{\eqref{l36a}} we choose $\overline\gamma=1$,
where $\rho = -\delta-1+\alpha-\ell_y$ and where we
require the bound $k \ge \tilde k$.
\item For \reftext{\eqref{l36b}} we have $\rho = \delta-1+\alpha-\ell_y$ and we
choose $\overline\gamma=1+\beta,1,-\beta+\frac12,-\frac12$
for $\alpha=0$ and $\ell_y\ge 0$ as well as for $\alpha=1$ and $\ell_y\ge 1$, while for $\alpha=1$ and $\ell_y=0$ we use 
$\overline\gamma=1+\beta,1$ and $\underline\gamma = -\beta+\frac12,-\frac12$.
Here, we require $k \ge \tilde k+4$ and for the case of $\underline\gamma$
if $\alpha=1$, we use finiteness of
$\left\lVert D_xv\right\rVert_{\tilde k +2,-\delta+\frac12}$.
\item For \reftext{\eqref{l36c}} in the first estimate we choose $\overline\gamma=0$ in all cases since $\rho = -\delta-1+\alpha-\ell_y < \overline\gamma$, whereas for the second one we have $\rho = -\delta+\alpha-\ell_y$ and proceed as follows: We choose $\overline\gamma=1+\beta,1,-\beta+\frac12,-\frac12$ for $\alpha=0$ and $\ell_y\ge 1$ as well as for $\alpha=1$ and $\ell_y\ge 2$.
For the remaining cases, we take $\overline\gamma=1+\beta,1,\underline\gamma=-\beta+\frac12,-\frac12$.
In order to justify the application of \reftext{Lemma~\ref{lem:mainhardy}}
for the cases $\underline\gamma$ we use the finiteness of the
terms $\left\lVert \partial_tD_xv\right\rVert_{\tilde k-2,-\delta-\frac12}$ for $\alpha=0$
and $\left\lVert D_xv\right\rVert_{\tilde k +2,-\delta+\frac12}$ for $\alpha=1$.
\item For \reftext{\eqref{l36d}} we have $\rho = \delta-1+\alpha-\ell_y$ and are led to choose $\overline\gamma=1+\beta,1,-\beta+\frac12,-\frac12$
if $\alpha=0$ and $\ell_y\ge 0$ or if $\alpha=1$ and $\ell_y\ge 1$.
If $\alpha=1$ and $\ell_y=0$, then we take $\overline\gamma=1+\beta,1$ and
$\underline\gamma=-\beta+\frac12,-\frac12$. This requires having $k \ge \check k +4$ and the cases $\underline\gamma$ are justified precisely like the case $\alpha=1$ and $\ell_y=0$ in \reftext{\eqref{l36c}}.
\item For \reftext{\eqref{l36e}} we have $\rho = \delta+\alpha-\ell_y$ and can therefore choose $\overline \gamma=2,3$ in all cases, which requires $k \ge \check k +6$.
\item For both inequalities of \reftext{\eqref{l36f}} we have the same case subdivision and discussion as for \reftext{\eqref{l36c}}, with $\breve k$ replacing $\check k$
except for the justification of the cases $\underline\gamma$, which can be kept the same, and we require $k \ge \breve k +5$.
\end{itemize}
In particular, due to the finiteness hypothesis
${\skliaustask v\skliaustasd}_\mathrm{sol} < \infty$,
all the bounds \reftext{\eqref{l36}} hold almost everywhere in time
$t\in [0,\infty)$ for our choices of $(w,\kappa,\alpha)$,
where for the case $w=\partial_t v$ we also use the fact that
$\partial_t$-derivatives commute with $D_x$ and $D_y$ operators.
By inspecting the norm
$\skliaustask v\skliaustasd_\mathrm{sol}$ (cf.~\reftext{\eqref{norm_sol}}),
we see that this allows to absorb all the
integral terms into the ones appearing in the norm
${\skliaustask v\skliaustasd}_\mathrm{Sol}$ (cf.~\reftext{\eqref{norm_sol_simple}}),
as desired.
\end{proof}
\subsection{Proof of the embedding lemmata}\label{app:embedding}

We first prove an essential elliptic regularity result:
\begin{proof}[Proof of Lemma~\ref{lem:mainhardy}]
We have for $\gamma \in \left\{\underline \gamma, \overline \gamma\right\}$ and $f \in \left\{\underline f, \overline f\right\}$
\begin{equation*}
\left\lvert \left(D_x-\gamma\right) f \right\rvert_\rho^2 = \int_0^\infty x^{-2\rho+2\gamma+1} \left(\partial_x \left(x^{-\gamma} f\right)\right)^2 \d x \gtrsim_{\gamma,\rho} \int_0^\infty x^{-2\rho+2\gamma-1} \left(x^{-\gamma} f\right)^2 \d x = \left\lvert f \right\rvert_\rho^2,
\end{equation*}
where Hardy's inequality (cf.~\cite[Lemma~A.1]{gko.2008}) has been employed with
\begin{equation}\label{proof_hardy}
x_n^{-\gamma} f(x_n) \to 0 \quad \mbox{as} \quad n \to \infty,
\end{equation}
where $\left(x_n\right)_{n\in \N} = \left(\underline x_n\right)_{n \in \N}$ for $\gamma = \underline \gamma$ and $\left(x_n\right)_{n\in \N} = \left(\overline x_n\right)_{n \in \N}$ for $\gamma = \overline \gamma$. The proof is concluded by noting that
\begin{equation*}
\left\lvert f \right\rvert_{1,\rho} \lesssim \left\lvert f \right\rvert_\rho + \left\lvert D_x f \right\rvert_\rho \le (1+\gamma) \left\lvert f\right\rvert_\rho + \left\lvert (D_x-\gamma) f \right\rvert_\rho \stackrel{\text{\reftext{\eqref{proof_hardy}}}}{\lesssim_{\gamma,\rho}} \left\lvert (D_x-\gamma) f \right\rvert_\rho.
\end{equation*}
\end{proof}
We give two auxiliary lemmata.
\begin{lemma}\label{lem:bc0_auxiliary2}
For $w \in C^\infty\left((0,\infty)\right)$, $\tilde w\in\mathbb{R}$,
and $\alpha_1 < \gamma < \alpha_2$, we have
\begin{equation}\label{bc0_auxiliary2}
 {\left\lvert\tilde w \right\rvert
}\lesssim_{\alpha
_1,\alpha_2,\gamma}  {\left\lvert w \right\rvert}_{\alpha_1} +  {\left\lvert w+x^\gamma\tilde w \right
\rvert}_{\alpha_2}.
\end{equation}
\end{lemma}
\begin{proof}
By using the inequality
$ {\left\lvert x^\gamma\tilde w \right\rvert}^2\le
2\left( {\left\lvert w+x^\gamma\tilde w \right\rvert}^2
+  {\left\lvert w \right\rvert}^2\right)$, and the fact
that any power of $x$ is bounded by positive constants from
below and above for $x \in [1,2]$, we find
\begin{eqnarray*}
 {\left\lvert\tilde w \right\rvert
}^2_{BC^0(\mathbb{R}_y)} &\lesssim_{\gamma}& \int_1^2 {\left\lvert x^\gamma\tilde w \right\rvert}^2 \frac{\mathrm
{d} x}{x}\lesssim_{\alpha_1,\alpha_2} \int_1^2 x^{- 2\alpha_1}
{\left\lvert w \right\rvert}^2 \frac{\mathrm
{d} x}{x} + \int_1^2 x^{- 2\alpha_2}  {\left\lvert w+x^\gamma\tilde w \right\rvert}^2 \frac{\mathrm{d}
x}{x} \\
&\lesssim& {\left\lvert w \right\rvert}_{\alpha_1}^2+ {\left\lvert w+x^\gamma\tilde w \right\rvert
}_{\alpha_2}^2.
\end{eqnarray*}
\end{proof}
\begin{lemma}\label{lem:bc0_auxiliary}
For $w \in C_\mathrm{c}^0\left([0,\infty)\times\mathbb{R}\right)
\cap C^\infty\left((0,\infty) \times\mathbb{R}\right)$ and $\alpha_1< 0<\alpha_2$
we have with $w_0 := w_{|x = 0}$
\begin{align}\label{bc0_auxiliary}
 {\left\lVert w \right\rVert}_{BC^0\left((0,\infty)_x
\times\mathbb{R}_y\right)}
+ {\left\lVert w-w_0 \right\rVert}_{BC^0\left((0,\infty
)_x \times\mathbb{R}_y\right)}
+ {\left\lvert w_0 \right\rvert}_{BC^0(\mathbb
{R}_y)} \lesssim_{\alpha_1,\alpha_2} {\left\lVert D_x w \right\rVert
}_{1,\alpha_1} +
 {\left\lVert D_x w \right\rVert}_{1, \alpha_2}
\end{align}
and
\begin{align}\label{l2_auxiliary}
 \int_{-\infty}^\infty {\left\lvert w_0 \right\rvert}^2\mathrm{d} y \lesssim_{\alpha_1,\alpha_2} {\left\lVert D_x w \right\rVert
}_{\alpha_1-\frac12}^2 +
 {\left\lVert D_x w \right\rVert}_{\alpha_2-\frac12}^2.
\end{align}
Furthermore, for fixed $\alpha\in\mathbb{R}$ and
$w \in C^\infty\left((0,\infty) \times\mathbb{R}\right)$ we have
\begin{align}\label{supl2_auxiliary}
\sup_{y\in\mathbb{R}}\left\lvert w\right\rvert_\alpha \lesssim_\alpha\left\lVert w\right\rVert_{1,\alpha}.
\end{align}
\end{lemma}
\begin{proof}
We first prove \reftext{\eqref{supl2_auxiliary}}. By using standard embeddings, we find
\begin{eqnarray}
\sup_{y \in\mathbb{R}} {\left\lvert w \right
\rvert}_{\alpha}^2&\le& \int_0^\infty x^{-2 \alpha} \sup_{y \in\mathbb
{R}} w^2
\frac{\mathrm{d} x}{x} \nonumber\\
&\lesssim& \int_0^\infty\left(x^{-2 \alpha-1} \int_\mathbb{R}
w^2 \mathrm{d}
y\right)^{\frac1 2} \left(x^{-2\alpha-1} \int_\mathbb{R}(D_y w)^2 \mathrm{d}
y\right)^{\frac1 2} \frac{\mathrm{d} x}{x} \nonumber\\
&\le& \frac12\int_\mathbb{R}\int_0^\infty x^{-2\alpha-1} w^2
\frac{\mathrm{d}
x}{x} \, \mathrm{d} y + \frac12\int_\mathbb{R}\int_0^\infty
x^{-2\alpha-1} (D_y w)^2 \frac{\mathrm{d} x}{x} \, \mathrm{d} y \nonumber\\
&=& \frac12 {\left\lVert w \right\rVert}_{\alpha}^2 
+ \frac12 {\left\lVert D_y w \right\rVert}_{\alpha}^2.\label{auxil_2}
\end{eqnarray}
By definition of the norms ${\left\lVert\cdot\right
\rVert}_{k,\alpha}$ in \reftext{\eqref{norm_2d_1}} trivially $ {\left
\lVert D_y^j w \right\rVert
}_{\alpha} \le {\left\lVert w \right\rVert
}_{1,\alpha}$ for $j \in \{0,1\}$, which concludes the proof of \reftext{\eqref{supl2_auxiliary}}.

\medskip

Next, reasoning like in \cite[Estimate~(8.4)]{ggko.2014}, we
use the Sobolev embedding in $s=\log x$ and \reftext{Lemma~\ref{lem:mainhardy}}
to reduce to norms with $D_x w$ only. We choose a smooth cut-off
function $\chi$ such that $\chi(x)=1$ for $x\le1$ and $\chi(x)=0$ for
$x\ge2$ and obtain
\begin{align}\nonumber
{\left\lvert w_0 \right\rvert}^2 +
{\left\lvert w \right\rvert}_{BC^0((0,\infty)_x)}^2 +
{\left\lvert w-w_0 \right\rvert}_{BC^0((0,\infty)_x)}^2 &\lesssim  {\left\lvert w_0 \right\rvert}^2 +
 {\left\lvert(1-\chi) w \right\rvert}_{BC^0((0,\infty
)_x)}^2 +
 {\left\lvert\chi(w-w_0) \right\rvert}_{BC^0((0,\infty
)_x)}^2\nonumber\\
&\lesssim  {\left\lvert w_0 \right\rvert}^2 +
 {\left\lvert(1-\chi) w \right\rvert}_{W^{1,2}(\mathbb
{R}_s)}^2 +  {\left\lvert\chi(w-w_0) \right\rvert
}_{W^{1,2}(\mathbb{R}_s)}^2\nonumber\\
&\lesssim  {\left\lvert w_0 \right\rvert}^2 +
 {\left\lvert(1-\chi) w \right\rvert}_{1,\alpha_1}^2 +
 {\left\lvert\chi(w-w_0) \right\rvert}_{1,\alpha_2}^2 \nonumber\\
&\lesssim  {\left\lvert w \right\rvert}_{1,\alpha_1}^2
+  {\left\lvert w-w_0 \right\rvert}_{1,\alpha_2}^2
\stackrel{\text{\reftext{\eqref{est_hardy}}}}{\lesssim} {\left\lvert D_x w
\right\rvert}_{\alpha_1}^2 +  {\left\lvert D_x w \right
\rvert}_{\alpha_2}^2.\label{auxil_1}
\end{align}
Here, we have first used the triangle inequality and subsequently the standard
Sobolev embedding on the real line. For the third estimate we have
introduced weights which on the support of the functions that appear in
the norms are in both cases bounded from below by a positive constant.
The second but last estimate follows from
\reftext{Lemma~\ref{lem:bc0_auxiliary2}}.
Finally, in the last estimate we have
used \reftext{Lemma~\ref{lem:mainhardy}}.

\medskip 

By taking the supremum in $y\in\mathbb R$ of \reftext{\eqref{auxil_1}} and summing estimate~\reftext{\eqref{auxil_2}} for $\alpha=\alpha_1$ and $\alpha = \alpha_2$ and for $D_xw$ rather than $w$, we complete the proof of \reftext{\eqref{bc0_auxiliary}}.

\medskip

In order to prove \reftext{\eqref{l2_auxiliary}}, we integrate \reftext{\eqref{auxil_1}} in $y$ and obtain in particular
\begin{equation*}
\int_{-\infty}^\infty \left\lvert w_0\right\rvert^2\mathrm{d} y\lesssim\int_{-\infty}^\infty \left\lvert D_x w\right\rvert^2_{\alpha_1}\mathrm{d}y + \int_{-\infty}^\infty \left\lvert D_x w\right\rvert^2_{\alpha_2}\mathrm{d}y \\
=\left\lVert D_xw\right\rVert_{\alpha_1-\frac12}^2 + \left\lVert D_xw\right\rVert_{\alpha_2-\frac12}^2,
\end{equation*}
which proves \reftext{\eqref{l2_auxiliary}} and concludes the proof.
\end{proof}
\begin{proof}[Proof of \reftext{Lemma~\ref{lem:bc0_bounds}}]
Throughout the proof, estimates depend only on the
number of derivatives of the functions that we need to estimate,
+with implicit constants which depend only on $\tilde k$, $\check k$,
and $\delta$. By an approximation argument using
\reftext{Lemma~\ref{lem:approx_init}},
\reftext{Lemma~\ref{lem:approx_rhs}}, and
\reftext{Lemma~\ref{lem:approx_sol}},
we may assume that all functions are smooth,
compactly supported in $x \in [0,\infty)$ and $y\in\mathbb R$,
and meet expansions \eqref{expansion_v}, \eqref{expansion_f}, and \eqref{expansion_v0} with smooth coefficients and no remainder. 

\medskip
\noindent\textit{Proof of estimates~\reftext{\eqref{bc0_grad_init}} and~\reftext{\eqref{bc0_grad_sol}}.}
Note that \reftext{\eqref{bc0_grad_sol}} follows from
\reftext{\eqref{bc0_grad_init}} by taking the supremum
in $t\in I$ and therefore we only prove the latter.

We write $v_x^{(0)}=x^{-1}D_xv^{(0)}$ and use
estimate~\reftext{\eqref{bc0_auxiliary}} of
\reftext{Lemma~\ref{lem:bc0_auxiliary}}
for $w:=D^\ell v^{(0)}_x$, giving
\begin{align}
& {\left\lVert D^\ell v^{(0)}_x \right\rVert}_{BC^0\left((0,\infty) \times\mathbb{R}\right)} \nonumber\\
& \quad \lesssim  {\left\lVert
D_xD^\ell x^{-1} D_x v^{(0)} \right\rVert}_{1,-\delta} +
{\left\lVert D_xD^\ell x^{-1} D_x v^{(0)} \right\rVert}_{1,\delta}
\nonumber\\
&\quad\lesssim {\left\lVert x^{-1}D_y^{\ell_y}\left(D_x+\ell
_y-1\right)\left(D_x-1\right)^{\ell_x}D_x v^{(0)} \right\rVert
}_{1,-\delta} + {\left\lVert x^{-1}D_y^{\ell_y}\left
(D_x+\ell_y-1\right)\left(D_x-1\right)^{\ell_x}D_x v^{(0)} \right
\rVert}_{1,\delta} \nonumber\\
&\quad\lesssim {\left\lVert D_y^{\ell_y}\left(D_x+\ell
_y-1\right)\left(D_x-1\right)^{\ell_x}D_x v^{(0)} \right\rVert
}_{1,-\delta+1} +  {\left\lVert D_y^{\ell
_y}\left(D_x+\ell_y-1\right)\left(D_x-1\right)^{\ell_x}D_x
v^{(0)} \right\rVert}_{1,\delta+1}. \label{bc0x_bound1}
\end{align}
Now we consider two cases. In case of $\ell_x \ge1$ or $\ell_y=0$, we
have a term
\begin{equation*}
(D_x-1)D_xv^{(0)}=(D_x-1)D_x\left(v^{(0)}-v^{(0)}_0-v^{(0)}_1x\right)
\end{equation*}
appearing in the last two lines of \reftext{\eqref{bc0x_bound1}}. Using the
triangle inequality, we may bound these terms by
\begin{align}\nonumber
& \sum_{0\le {\left\lvert\tilde\ell\right
\rvert}\le {\left\lvert\ell\right\rvert}}
{\left\lVert D^{\tilde\ell} (D_x-1) D_x\left
(v^{(0)}-v^{(0)}_0-v^{(0)}_1x\right) \right\rVert}_{1,-\delta
+1} + \sum_{0\le {\left\lvert\tilde\ell\right
\rvert}\le {\left\lvert\ell\right\rvert}}
{\left\lVert D^{\tilde\ell} (D_x-1) D_x\left
(v^{(0)}-v^{(0)}_0-v^{(0)}_1x\right) \right\rVert}_{1,\delta
+1} \nonumber\\
&\quad\lesssim  {\left\lVert\tilde q(D_x)D_xv^{(0)} \right
\rVert}_{\check k,-\delta+
1} + {\left\lVert(D_x-3)(D_x-2)\tilde q(D_x)D_xv^{(0)}
\right\rVert}_{\check k, \delta
+1}, \label{bc0x_bound2}
\end{align}
for $ {\left\lvert\tilde\ell\right\rvert}\le
 {\left\lvert\ell\right\rvert}\le\check k-2$, and
where in the
last estimate we have applied \reftext{Lemma~\ref{lem:mainhardy}} together with
the fact that $\tilde q(D_x)$ figures the factor $D_x-1$ (cf.~\reftext{\eqref{poly_q2}}). Due to the definition of the initial data norm
${\skliaustask\cdot\skliaustasd}_\mathrm{init}$ (cf.~\reftext{\eqref{norm_init}}), this allows to bound
the first term in \reftext{\eqref{bc0x_bound1}}.

If instead $\ell_y \ge1$ and $\ell_x = 0$, we find in both terms of
the last two lines of \reftext{\eqref{bc0x_bound1}} the quantity
\begin{equation*}
D_y^{\ell_y}(D_x+\ell_y-1)D_xv^{(0)}=D_y^{\ell_y}(D_x+\ell
_y-1)D_x\left
(v^{(0)}-v_0^{(0)}\right),
\end{equation*}
and we recognize that this quantity is $O\left(x^2\right)$ as $x\searrow0$, due to one
$x$-factor coming from the $D_y$-operators and a second one coming from
$v^{(0)}-v_0^{(0)}$. This allows to proceed as above by again applying
\reftext{Lemma~\ref{lem:mainhardy}}. Therefore, we find that in either case

\begin{align}\label{bc0x_bound3}
 {\left\lVert D^\ell v^{(0)}_x \right\rVert}_{BC^0\left
((0,\infty) \times\mathbb{R}\right)}
\lesssim {\left\lVert\tilde q(D_x)D_xv^{(0)} \right
\rVert}_{\check k,-\delta+
1} + {\left\lVert(D_x-3)(D_x-2)\tilde q(D_x)D_xv^{(0)}
\right\rVert}_{\check k, \delta+1}.
\end{align}
This verifies \reftext{\eqref{bc0_grad_init}} for the fist term on the left-hand
side provided $ {\left\lvert\ell\right\rvert}\le\check
k - 2$. For verifying the estimate
for the third term on the left-hand side of \reftext{\eqref{bc0_grad_init}}, we use \reftext{Lemma~\ref{lem:bc0_auxiliary}} and
write analogously to \reftext{\eqref{bc0x_bound1}}

\begin{equation*}
 {\left\lVert D^\ell v^{(0)}_y \right\rVert}_{BC^0\left
((0,\infty) \times\mathbb{R}\right)}
\stackrel{\text{\reftext{\eqref{bc0_auxiliary}}}}{\lesssim}  {\left\lVert
D_y^{\ell_y+1} (D_x+\ell_y)D_x^{\ell_x}v^{(0)} \right\rVert
}_{1,-\delta+1} +  {\left\lVert D_y^{\ell_y+1}
(D_x+\ell_y)D_x^{\ell_x}v^{(0)} \right\rVert}_{1,\delta+1}.
\end{equation*}
Now if $\ell_x \ge1$ or $\ell_y=0$, then at least one $D_x$-derivative
acts on $v^{(0)}$ and cancels the contribution $v^{(0)}_0$.
This, together with the $x$-factor coming from the $D_y$-operator leads
to the expression $O\left(x^2\right)$ as $x\searrow0$.
If instead $\ell_x=0$ and $\ell_y \ge1$,
then we have at least two factors $x$ coming from
the $D_y$-operators, and again the terms are
$O\left(x^2\right)$ as $x\searrow0$. Therefore, we can proceed as in
\reftext{\eqref{bc0x_bound2}} and find the bound
\reftext{\eqref{bc0x_bound3}} for this term as well.

For the remaining terms in \reftext{\eqref{bc0_grad_init}} we proceed by
using \reftext{\eqref{bc0_auxiliary}} of \reftext{Lemma~\ref{lem:bc0_auxiliary}} and then
\reftext{Lemma~\ref{lem:mainhardy}}, to find, first with $w=v^{(0)}$, $w_0=v^{(0)}_0$,
\begin{equation*}
 {\left\lvert v^{(0)}_0 \right\rvert}_{BC^0(\mathbb{R})}
\lesssim  {\left\lVert D_xv^{(0)}\right\rVert
}_{1, -\delta} + {\left\lVert D_x v^{(0)}\right\rVert}_{1,\delta} \lesssim {\left\lVert D_x v^{(0)} \right
\rVert}_{1,-\delta} +
 {\left\lVert\tilde q(D_x)D_x v^{(0)} \right\rVert
}_{1,\delta}
\end{equation*}
next, with $w=\left(v^{(0)}\right)_y$, $w_0=\left(v^{(0)}_0\right)_y$
\begin{eqnarray*}
 {\left\lvert \left(v^{(0)}_0\right)_y \right\rvert}_{BC^0(\mathbb{R})}
&\lesssim&  {\left\lVert D_x w\right\rVert}_{1, -\delta} 
+ {\left\lVert D_xw\right\rVert}_{1,\delta} = {\left\lVert D_yD_x v^{(0)} \right
\rVert}_{1,-\delta+1} +  {\left\lVert D_yD_x v^{(0)} \right\rVert}_{1,\delta+1}\\
&\lesssim&{\left\lVert \tilde q(D_x)D_xv^{(0)}\right\rVert}_{2, -\delta+1} 
+ {\left\lVert (D_x-3)(D_x-2)\tilde q(D_x)D_xv^{(0)}\right\rVert}_{2,\delta+1}
\end{eqnarray*}
and finally, with $w=\left(v^{(0)}\right)_x$, $w_0=v^{(0)}_1$, by commuting derivatives
\begin{eqnarray*}
 {\left\lvert v^{(0)}_1\right\rvert}_{BC^0(\mathbb{R})} 
&\lesssim&  {\left\lVert D_x w\right\rVert}_{1, -\delta} + {\left\lVert D_xw\right\rVert}_{1,\delta} = {\left\lVert (D_x-1)D_x v^{(0)} \right
\rVert}_{1,-\delta+1} +  {\left\lVert (D_x-1)D_x v^{(0)} \right\rVert}_{1,\delta+1}\\
&\lesssim&{\left\lVert \tilde q(D_x)D_xv^{(0)}\right\rVert}_{2, -\delta+1} 
+ {\left\lVert (D_x-3)(D_x-2)\tilde q(D_x)D_xv^{(0)}\right\rVert}_{2,\delta+1}.
\end{eqnarray*}
This completes the proof of \reftext{\eqref{bc0_grad_init}}, by bounding the
second and fourth terms on the left-hand side, provided $\ell^\prime
\le
\min\left\{\tilde k - 1, \check k - 1\right\}$ and $\ell^{\prime
}
\le\min\left\{\tilde k - 2, \check k - 2\right\}$ (cf.~\reftext{\eqref{norm_init}}).

\medskip
\noindent\textit{Proof of estimate~\reftext{\eqref{bc0_v1beta_v2}}.}
We now use \reftext{\eqref{l2_auxiliary}} of 
\reftext{Lemma~\ref{lem:bc0_auxiliary}}, in combination with Lemma~\reftext{\ref{lem:mainhardy}},
to obtain the following bounds, which work similarly as
the previous estimate. 
For the choices $w=v$ or $w = x^{-1} D_x v$ or $w = x^{-1} D_y v$,
$\alpha_1=-\delta$, $\alpha_2=\delta$,
and $\alpha_2 = - \delta + \frac 1 2$
in \reftext{Lemma~\ref{lem:bc0_auxiliary}},
we have, respectively,
\begin{align*}
{\left\lvert v_0\right\rvert}_{L^2\left(\mathbb{R}_y\right)} &\lesssim {\left\lVert D_xv\right\rVert}_{-\delta-\frac12} +{\left\lVert D_x v\right\rVert}_{-\delta+\frac12} \lesssim {\left\lVert v\right\rVert}_{1,-\delta-\frac12} + {\left\lVert D_x v\right\rVert}_{-\delta+\frac12} \\
 {\left\lvert v_1 \right
\rvert}_{L^2\left(\mathbb{R}_y\right)} &\lesssim
{\left\lVert (D_x-1)D_xv\right\rVert}_{-\delta + \frac12}  
+ {\left\lVert (D_x-1)D_xv \right\rVert}_{\delta+\frac12} \lesssim {\left\lVert D_x v \right\rVert}_{1,-\delta+ \frac12} 
+ {\left\lVert\tilde q(D_x)D_x v \right\rVert}_{1,\delta+\frac1 2},\\
 {\left\lvert \left(v_0\right)_y \right
\rvert}_{L^2\left(\mathbb{R}_y\right)} &\lesssim
{\left\lVert D_yD_xv\right\rVert}_{-\delta + \frac12} 
+ {\left\lVert D_yD_xv \right\rVert}_{\delta+\frac12} \lesssim  {\left\lVert D_x v \right\rVert}_{1,-\delta+ \frac12} 
+ {\left\lVert\tilde q(D_x)D_x v \right\rVert}_{1,\delta+\frac1 2}.
\end{align*}
Next, with $w=v_{yy}=x^{-2}D_y^2v$ or $w=v_{xy}=x^{-2}D_yD_xv$ and the same $\alpha_1,\alpha_2$, we find
\begin{align*}
 {\left\lvert (v_0)_{yy} \right\rvert}_{L^2\left(\mathbb{R}_y\right)} &\lesssim
{\left\lVert D_y^2D_x v \right\rVert}_{-\delta + \frac3 2} +
 {\left\lVert D_y^2D_xv \right\rVert}_{\delta+\frac32} \\
&\lesssim  {\left\lVert \tilde q(D_x)D_x v \right\rVert}_{2,-\delta+ \frac3 2} 
+ {\left\lVert (D_x-3)(D_x-2)\tilde q(D_x) D_x v \right\rVert}_{2,\delta+\frac32},\\
{\left\lvert (v_1)_y\right\rvert}_{L^2\left(\mathbb{R}_y\right)} &\lesssim \left\lVert D_y(D_x-1)D_xv\right\rVert_{-\delta+\frac32}+\left\lVert D_y(D_x-1)D_xv\right\rVert_{\delta+\frac32}\\
&\lesssim \left\lVert\tilde q(D_x)D_xv\right\rVert_{1,-\delta+\frac32} + \left\lVert(D_x-3)(D_x-2)\tilde q(D_x)D_xv\right\rVert_{1,\delta+\frac32}.
\end{align*}

\medskip

For $w=x^{-1-\beta} (D_x-1)D_xv$ and $\alpha_1=\delta-\beta$, $\alpha_2=-\delta-\beta+1$, we find
\begin{align*}
 {\left\lvert v_{1+\beta}\right\rvert}_{L^2\left(\mathbb{R}_y\right)} &\lesssim
 {\left\lVert D_xw\right\rVert}_{\delta-\beta-\frac12}
 +{\left\lVert D_xw\right\rVert}_{-\delta-\beta+\frac12} \\
&\lesssim {\left\lVert (D_x-\beta-1)(D_x-1)D_x v \right\rVert}_{\delta + \frac12} 
+ {\left\lVert (D_x-\beta-1)(D_x-1)D_xv \right\rVert}_{-\delta+\frac32} \\
&\lesssim  {\left\lVert \tilde q(D_x)D_x v \right\rVert}_{\delta+ \frac1 2} 
+ {\left\lVert \tilde q(D_x) D_x v \right\rVert}_{-\delta+\frac32},
\end{align*}
and in the same way
\begin{equation*}
 {\left\lvert (v_{1+\beta})_y \right\rvert}_{L^2\left(\mathbb{R}_y\right)} \lesssim
 {\left\lVert \tilde q(D_x)D_x v_y \right\rVert}_{\delta+ \frac1 2} 
+ {\left\lVert \tilde q(D_x) D_x v_y \right\rVert}_{-\delta+\frac32} \sim {\left\lVert D_y \tilde q(D_x)D_x v \right\rVert}_{\delta+ \frac3 2} 
+ {\left\lVert D_y \tilde q(D_x) D_x v \right\rVert}_{-\delta+\frac52}.
\end{equation*}
If we choose $w=x^{-2}(D_x-\beta-1)(D_x-1)D_xv$ and $\alpha_1=-\delta, \alpha_2=\delta$, we find
\begin{align*}
 {\left\lvert v_2 \right\rvert}_{L^2\left(\mathbb{R}_y\right)} &\lesssim
{\left\lVert(D_x-2)(D_x-\beta-1)(D_x-1)D_x v \right\rVert}_{-\delta + \frac3 2} + {\left\lVert (D_x-2)(D_x-\beta-1)(D_x-1)D_xv \right\rVert}_{\delta+\frac32} \\
&\lesssim  {\left\lVert \tilde q(D_x)D_x v \right\rVert}_{-\delta+ \frac3 2} 
+ {\left\lVert (D_x-3)(D_x-2)\tilde q(D_x) D_x v \right\rVert}_{\delta+\frac32}.
\end{align*}
After taking the integral in $t$ over $I$ of the above bounds,
we control $\left\lvert v_{1+\beta} \right\rvert_{BC^0(\R_y)}$ by
$\left\lvert v_{1+\beta} \right\rvert_{H^1(\R_y)}$
using the Sobolev embedding, and we note
that the resulting terms are bounded by
${\skliaustask v\skliaustasd}_\mathrm{sol}^2$.

\medskip

Furthermore, by applying \reftext{Lemma~\ref{lem:bc0_auxiliary}} to
$w=x^{-2}D_yD_xv$ and $\alpha_1=\delta-\tfrac12$ and
$\alpha_2=-\delta+\tfrac12$, we find
\begin{align*}
{\left\lvert\left(v_1\right)_y\right\rvert}_{L^2(\mathbb R_y)}&\lesssim{\left\lVert D_y(D_x-1)D_xv\right\rVert}_{\delta+1} + {\left\lVert D_y(D_x-1)D_xv\right\rVert}_{-\delta+2}\\
&\lesssim{\left\lVert D_y \tilde q(D_x) D_xv\right\rVert}_{\delta+1} + {\left\lVert D_y \tilde q(D_x)D_xv\right\rVert}_{-\delta+2}\\
&\le{\left\lVert D_y\tilde q(D_x)D_xv\right\rVert}_{\check k -1, \delta+1}+ {\left\lVert D_y\tilde q(D_x)D_xv\right\rVert}_{\breve k,-\delta+2},
\end{align*}
and by taking the supremum in $t$ over $I$
we see that the resulting terms are bounded by
${\skliaustask v\skliaustasd}_\mathrm{sol}^2$.

\medskip

Similarly, for $(v_0)_y$ we take $w=x^{-2}D_y^2v$,
$\alpha_1=\delta-\tfrac12$, and $\alpha_2=-\delta+\tfrac12$,
so that we find
\[
{\left\lvert \left(v_0\right)_{yy}\right\rvert}_{L^2(\mathbb R_y)} \lesssim {\left\lVert D_x x^{-2} D_y^2v\right\rVert}_{\delta-1} + {\left\lVert D_x x^{-2}D_y^2 v\right\rVert}_{-\delta} = {\left\lVert D_y^2D_x v\right\rVert}_{\check k -2, \delta +1} + {\left\lVert D_y^2 D_xv\right\rVert}_{\breve k -1, -\delta+2}.
\]
The supremum in $t$ over $I$ of these terms can be controlled as above by
${\skliaustask v\skliaustasd}_\mathrm{sol}^2$, which gives the bound for
$\left\lVert \left(v_1\right)_y\right\rVert_{BC^0(I;L^2(\mathbb R))}$
and $\left\lVert \left(v_0\right)_{yy}\right\rVert_{BC^0(I;L^2(\mathbb R))}$. 

\medskip

In order to prove the bounds for the remaining terms $\left\lVert v_1\right\rVert_{BC^0(I;L^2(\mathbb R))}$ and $\left\lVert \left(v_0\right)_y\right\rVert_{BC^0(I;L^2(\mathbb R))}$,
we now take $w=x^{-1} D_x v$ or $w= x^{-1} D_y v$,
as well as $\alpha_1=\delta-\frac12$ and $\alpha_2=-\delta+\frac12$,
respectively, and we find
\begin{align*}
 {\left\lvert v_1 \right
\rvert}_{L^2\left(\mathbb{R}_y\right)} &\lesssim
{\left\lVert (D_x-1)D_xv\right\rVert}_{\delta}  
+ {\left\lVert (D_x-1)D_xv \right\rVert}_{-\delta+1} \lesssim {\left\lVert \tilde q(D_x)D_x v \right\rVert}_{\delta} 
+ {\left\lVert\tilde q(D_x)D_x v \right\rVert}_{-\delta+1},\\
 {\left\lvert \left(v_0\right)_y \right
\rvert}_{L^2\left(\mathbb{R}_y\right)} &\lesssim
{\left\lVert D_yD_xv\right\rVert}_{\delta}
+ {\left\lVert D_yD_xv \right\rVert}_{-\delta+1},
\end{align*}
after which we may take the supremum in $t$ over $I$ and control the ensuing terms by $\skliaustask v\skliaustasd_\mathrm{sol}$.

\medskip

\noindent\textit{Proof of estimate~\reftext{\eqref{bcint_f1_f2}}.}
We apply \reftext{\eqref{l2_auxiliary}} for $\alpha_1=-\delta,\alpha_2=\delta$, 
for the two different choices $w=x^{-1}f$ and $w=x^{-2}(D_x-1)f$, 
followed by Lemma~\reftext{\ref{lem:mainhardy}}, thus obtaining the two following bounds:
\begin{align*}
\left\lvert f_1\right\rvert_{L^2\left(\mathbb R_y\right)} &\lesssim
\left\lVert(D_x-1)f\right\rVert_{-\delta+\frac12}+\left\lVert(D_x-1)f\right\rVert_{\delta+\frac12} \lesssim \left\lVert(D_x-1)f\right\rVert_{-\delta+\frac12}+\left\lVert\tilde q(D_x-1)(D_x-1)f\right\rVert_{\delta+\frac12},\\
\left\lvert f_2\right\rvert_{L^2\left(\mathbb R_y\right)} &\lesssim
\left\lVert(D_x-2)(D_x-1)f\right\rVert_{-\delta+\frac12}+\left\lVert(D_x-2)(D_x-1)f\right\rVert_{\delta+\frac32}\\
&\lesssim \left\lVert\tilde q(D_x-1)(D_x-1)f\right\rVert_{-\delta+\frac32}
+\left\lVert(D_x-4)(D_x-3)\tilde q(D_x-1)(D_x-1)f\right\rVert_{\delta+\frac32}.
\end{align*}
We may then integrate the two above bounds in $t$ over $I$ and bound
the right-hand sides by ${\skliaustask f\skliaustasd}_\mathrm{rhs}^2$
under the conditions
$0\le\ell\le\tilde k - 3$ and $0\le\ell^\prime\le\check k - 3$ (cf.~\reftext{\eqref{norm_rhs}}).
\end{proof}
%

\subsection{Proofs of the approximation lemmata}\label{app:approx}

\begin{proof}[Proof of \reftext{Lemma~\ref{lem:approx_init}}]
Throughout the proof, estimates depend on $k$, $\tilde k$, $\check k$,
$\breve k$, and $\delta$.

\medskip
\noindent\textit{Introducing a cut-off in the coordinates $s$ and
$\eta$.}
We have to approximate $v^{(0)}$ contemporarily in all the addends
appearing in $ {\left\lVert v^{(0)} \right\rVert}_\mathrm
{init}^2$ (cf.~\reftext{\eqref{norm_init}}). In order to adapt to the standard
theory of Sobolev
spaces, we introduce the Fourier coordinate $\eta\in
\mathbb{R}$ in
the $y$-direction and the variable $s=\log x$ in the $x$-direction.
Through the use of \reftext{\eqref{norm_2d}}, we can rewrite the norms by
effectively replacing $D_y$ by $e^s\eta$ and $D_x$ by $\partial_s$. We
write for example
%
%
\begin{eqnarray}\label{approx_init_example}
 {\left\lVert\tilde q(D_x)D_x v^{(0)} \right\rVert
}_{\tilde k, \delta}^2&=&\sum_{0\le
j+j^\prime\le\tilde k}\int_\mathbb{R}\eta^{2j} {\left
\lvert x^{j-\frac12}D_x^{j^\prime}\tilde q(D_x)D_xv^{(0)} \right
\rvert}_{\delta}^2\mathrm{d}\eta\nonumber\\
&=& \sum_{0\le j+j^\prime\le\tilde k} {\left\lVert\eta
^j\ e^{\left(-\delta-\frac12+j\right)s}\ \tilde q(\partial_s)\partial
_s^{j^\prime+1}v^{(0)} \right\rVert}^2_{L^2(\mathbb{R}
_s\times\mathbb{R}_\eta)},\nobreakpostdisplay
\end{eqnarray}
and a similar identity is valid for the other norms contributing to
${\left\lVert v^{(0)} \right\rVert}_\mathrm{init}$.

We now introduce a cut-off function $\chi_n$ defined as
$\chi_n(s,\eta) := \phi\left(\frac s n\right) \psi\left(\frac\eta n\right)$,
where $\phi$ and $\psi$ are smooth and satisfy
\begin{subequations}\label{approx_init_cut-off_s_eta}
\begin{equation}\label{phi_s}
\phi(s) = 1 \quad \text{for} \quad s \in (-\infty, 1] \quad
\mbox{and} \quad \phi(s) = 0 \quad \text{for} \quad s \in[2,\infty),
\end{equation}
as well as
\begin{equation}\label{psi_eta}
\psi(\eta) = 1 \quad \text{for} \quad \eta \in [-1,1] \quad
\mbox{and} \quad \psi(\eta) = 0 \quad \mbox{for} \quad \eta \notin [-2,2].
\end{equation}
\end{subequations}
Note that in the norms contributing to $ {\left\lVert
v^{(0)} \right\rVert}_\mathrm{init}$
the difference $v^{(0)}-\chi_n v^{(0)}$ produces contributions which
tend to zero as $n\to\infty$: For each term in the norm $
{\left\lVert v^{(0)} \right\rVert}_\mathrm{init}$ we may distribute
derivatives. If some of the
factors from $\tilde q(\partial_s)\partial_s$ fall on the cut-off, then
we lose the structure $\tilde q(\partial_s)\partial_s$. However, the
resulting term can be controlled by using the term $\left\lVert v^{(0)} \right\rVert_{k,-\delta-1}$, provided $k\ge \tilde k+4$. For example for the term \reftext{\eqref{approx_init_example}} we obtain
\begin{align*}
&\sum_{0\le j+j^\prime\le\tilde k} {\left\lVert\eta^j
e^{\left(-\delta-\frac12+j\right)s}\ \tilde q(\partial_s)\partial_s^{j^\prime
+1}(1-\chi_n)v^{(0)} \right\rVert}^2_{L^2(\mathbb{R}
_s\times\mathbb{R}_\eta)}\\
& \quad \lesssim \sum_{0\le j+j^\prime\le\tilde k} {\left\lVert\eta^j
e^{\left(-\delta-\frac12+j\right)s}\ (1-\chi_n) \ \tilde q(\partial_s)\partial_s^{j^\prime
+1} v^{(0)} \right\rVert}^2_{L^2\left(\mathbb{R}
_s\times\mathbb{R}_\eta\right)} \\
& \quad \phantom{\lesssim} + \sum_{\substack{0 \le j \le \tilde k\\ 0 \le j + j^\prime\le\tilde k+4}}
{\left\lVert\eta^j
e^{\left(-\delta+\frac12+j\right)s} \partial_s^{j^\prime} v^{(0)} \right\rVert}^2_{L^2\left([n,2n]\times\mathbb{R}_\eta\right)},
\end{align*}
where the last two terms converge to zero as $n\to\infty$ by
dominated convergence provided $k \ge \tilde k + 4$. A similar reasoning
for the remaining norms contributing to
$ {\left\lVert v^{(0)} \right\rVert}_\mathrm{init}$
shows that $ {\left\lVert v^{(0)}-\chi_nv^{(0)} \right
\rVert}_\mathrm{init}\to0$ as $n\to\infty$.

The approximants $\chi_nv^{(0)}$ already satisfy the property
$(G_\infty
)$ and the smoothness in $y$ since smoothness in $y$ is equivalent to
decay in $\eta$.

\medskip
\noindent\textit{Truncation in $y$.}
At this point we truncate our function in $y$ to ensure smoothness
in $\eta$ (decay in $\eta$ is already valid due to the previous proof step),
and up to passing through a diagonal argument, due to \reftext{\eqref{phi_s}},
we may assume that $v^{(0)}$ is already zero for $x\gg 1$.
We then introduce a cut-off function $\chi\in C^\infty(\mathbb R)$
such that $\chi(y)=1$ for $y\in[-1,1]$ and $\chi(y)=0$ for $y\notin[-2,2]$.
Then we define $\chi_n(y)=\chi\left(\frac y n\right)$ and we claim that
$\left\lVert v^{(0)}-\chi_n v^{(0)}\right\rVert_\mathrm{init}\to 0$ as $n\to\infty$.
This can be proved by considering separately the different terms in
the definition of $\left\lVert\cdot\right\rVert_\mathrm{init}$ and
showing that they all vanish as $n\to\infty$.
We again consider only the term corresponding to \reftext{\eqref{approx_init_example}},
all the others being treated by the same reasoning.
If $D_y$-operators appear in our terms,
then they change the weight by a power of $x$.
However, as we have assumed that $v^{(0)}(x,y)=0$ for $x\gg 1$,
all $x$-factors are uniformly bounded. Thus in all cases we can write
\begin{align*}
\left\lVert \tilde q(D_x)D_x \left(1 - \chi_n\right) v^{(0)}\right\rVert_{\tilde k,\delta}^2
&=\sum_{0\le j+j^\prime\le \tilde k}\int_0^\infty x^{-2\delta-1+2j}\int_{-\infty}^\infty
\left(\partial_y^j \tilde q(D_x)D_x^{j^\prime+1} \left(1-\chi_n(y)\right)v^{(0)}(x,y)\right)^2
\mathrm{d}y \frac{\mathrm{d}x}{x}\\
&\lesssim\sum_{0\le j+j^\prime\le \tilde k}\int_0^\infty x^{-2\delta-1+2j}
\int_{(-\infty,-n]\cup[n,\infty)}\left(\partial_y^j\tilde q(D_x)D_x^{j^\prime+1}
v^{(0)}\right)^2\mathrm{d}y\frac{\mathrm{d}x}{x},
\end{align*}
and each term in the above sum converges to zero as $n\to\infty$ by dominated convergence. 

\medskip
\noindent\textit{Mollification in $s=\log x$.}
As a consequence of the previous step, we may without loss of generality
restrict ourselves to functions $v^{(0)}$ such that $v^{(0)}(x,\eta) =
0$ for $(x,y)$ with $x \gg1$ or $|y|\gg 1$, and which thus are smooth in $\eta$.
As a next step, we prove that we can also approximate by functions 
that are smooth in both $(x,\eta)$.
To this aim, we mollify the function $v^{(0)}$ in the
variable $s=\log x$ with
a mollifier $\chi_\varepsilon$, where $\varepsilon> 0$, $\chi
_\varepsilon(s) = \varepsilon^{-1} \chi_1\left
(\frac s \varepsilon\right)$, and $\chi_1 \in
C^\infty_0(\mathbb{R}^2)$
with $\int\chi_1(s) \, \mathrm{d} s = 1$. Then the
mollification commutes with derivatives $\partial_s$. Furthermore, we
have that
\begin{equation*}
e^{-\rho s} \chi_\varepsilon * v^{(0)} (s,\eta)= \int_{-\infty
}^\infty e^{-\rho\left(s - s^\prime\right)}
\chi_\varepsilon\left(s-s^\prime\right) e^{-\rho s^\prime} v^{(0)}\left(s^\prime,\eta\right) \mathrm{d} s^\prime
\mathrm{d} r^\prime =
\underline\chi_\varepsilon * \underline v^{(0)}(s,\eta),
\end{equation*}
where $\underline\chi_\varepsilon(s) := e^{-\rho s} \chi
_\varepsilon(s)$
and $\underline v^{(0)}(s,\eta) := e^{-\rho s} v^{(0)}(s,\eta)$.
Hence, upon multiplying the mollifier with an appropriate weight, it
can be pulled out of all expressions of the form \reftext{\eqref{approx_init_example}} for the norms contributing to $ {\left
\lVert v^{(0)} \right\rVert}_\mathrm{init}$. Then $\underline\chi
_\varepsilon\to\delta_0$ in
$\mathcal D^\prime$ as $\varepsilon \searrow 0$
regardless of the value of $\rho$ and $\nu$ and the
proof step is concluded.

\medskip
\noindent\textit{Decay conditions as $x\searrow0$.}
By the previous step, we can assume without loss of generality that
$v^{(0)} \in C^\infty\left((0,\infty)_x \times \mathbb{R}_\eta\right)$
with $v^{(0)}(x,y) = 0$ for $(x,y)$ such that $x \gg1$ or
${\left\lvert y\right\rvert} \gg1$. Hence, also
$v^{(0)} = v^{(0)}(x,\eta)$ is smooth.
It therefore remains to discuss the asymptotics of $v^{(0)}$ as $x
\searrow0$.

To verify the decay conditions on $v^{(0)}$, we iteratively apply the
following basic reasoning for $\gamma\in\{0, 1\}$:

If $w_1,w_2$ satisfy $(D_x-\gamma)w^{(1)}=w^{(2)}$ and if $
{\left\lvert w^{(2)} \right\rvert}_\rho < \infty$ for some $\rho
\in\mathbb{R}\setminus\{\gamma\}$
then we can write
\begin{equation}\label{gensolw1}
w^{(1)}(x)=\tilde w(x) + w_\gamma x^\gamma,
\end{equation}
for some $w_\gamma\in\mathbb R$, where the following explicit formulas
for $\tilde w$ hold true:
\begin{equation}\label{explicitw}
\tilde w(x)=\left\{
\begin{array}{ll}
x^{\gamma}\int_0^x\tilde z^{-\gamma}w^{(2)}(\tilde z) \frac{\mathrm
{d}\tilde
z}{\tilde z} & \text{ if }\gamma<\rho,\\[5pt]
-x^{\gamma}\int_x^\infty\tilde z^{-\gamma}w^{(2)}(\tilde z) \frac
{\mathrm{d}
\tilde z}{\tilde z} & \text{ if }\gamma>\rho.
\end{array}
\right.
\end{equation}
Now note that for the part $\tilde w$ the asymptotics required for
applying \reftext{Lemma~\ref{lem:mainhardy}} follow directly from the assumption
$ {\left\lvert w^{(2)} \right\rvert}_{1,\rho}<\infty$
and from the explicit formulas \reftext{\eqref{explicitw}}.

By iteratively using expressions \reftext{\eqref{gensolw1}} and \reftext{\eqref{explicitw}}, we obtain that the function $v^{(0)}$ has the form
\begin{equation*}
v^{(0)}(x,\eta) = v^{(0)}_0(\eta) + v^{(0)}_1(\eta) x + R(x,\eta)
\quad
\mbox{as} \quad x \searrow0,
\end{equation*}
with a remainder term $R$. More precisely, $R$ can be characterized by
noting that by \reftext{Lemma~\ref{lem:mainhardy}} and due to the fact that the
operator $\tilde q(D_x)D_x$ cancels the terms $v_0^{(0)}$ and
$v_1^{(0)}x$ (cf.~\reftext{\eqref{poly_q2}}), we have
\begin{eqnarray*}
 {\left\lVert R \right\rVert}_{\tilde k,\delta
+1}&\lesssim& {\left\lVert(D_x-3)(D_x-2)\tilde
q(D_x)D_xR \right\rVert}_{\check k,\delta+1}\\
&=& {\left\lVert(D_x-3)(D_x-2)\tilde q(D_x)D_x v^{(0)}
\right\rVert}_{\check k,\delta
+1}\le {\left\lVert v^{(0)} \right\rVert}_\mathrm
{init}<\infty,
\end{eqnarray*}
from which it follows that we have $R(x,\eta) = o\left(x^{\frac3 2+\delta}\right)$
as $x\searrow0$ almost everywhere in $\eta\in\mathbb{R}$. Because of
the representation \eqref{explicitw}, the
coefficients $\left(v^{(0)}_0\right)_y$ and $v^{(0)}_1$ are smooth
in $\eta$. Now we define
\begin{equation*}
v^{(0,n)} := v^{(0)}_0 + v^{(0)}_1 x + \chi_n R,
\end{equation*}
where $\chi_n(s) := \chi_1\left(\frac s n\right)$ and $\chi_1$ is a
smooth cut-off function such that
$\chi_1(s) = 1$ for $s \in [-1,\infty)$
and $\chi_1(s) = 0$ for $s \in (-\infty,-2]$. Then like in the
first proof-step, $v^{(0,n)} = v^{(0)}$ for $s \ge-n$, and by
distributing derivatives, by using \reftext{Lemma~\ref{lem:mainhardy}}, and by
dominated convergence due to the finiteness of $ {\left
\lVert v^{(0)} \right\rVert}_\mathrm
{init}$, we deduce $ {\left\lVert v^{(0)} - v^{(0,n)}
\right\rVert}_\mathrm{init} \to0$
as $n \to\infty$. On the other hand, for $s\le-2n$ we have $v^{(0,n)}
:= v^{(0)}_0 + v^{(0)}_1 x$, which concludes the proof.
\end{proof}
\begin{proof}[Proof of \reftext{Lemma~\ref{lem:approx_rhs}}]
Throughout the proof, estimates depend on $k$, $\tilde k$, $\check k$,
$\breve k$, and $\delta$.

\medskip
\noindent\textit{Truncation in $s=\log x$, $\eta$, and $t$.}
Almost everywhere in $t\in I$, the norms appearing in the time integrals from
${\skliaustask f\skliaustasd}_\mathrm{rhs}^2$ (cf.~\reftext{\eqref{norm_rhs}})
are finite, and
therefore we can proceed along the lines of the proof of \reftext{Lemma~\ref{lem:approx_init}}. We work in the variables $t$, $s:=\log x$, and $\eta
$, and again introduce a smooth cut-off function $\chi_n:I\times
[0,\infty)\times\mathbb{R}\to\mathbb{R}$ with values in the
interval $[0,1]$, where
$\chi_n(t,s,\eta)=\sigma_n(t)\chi\left(\frac s n,\tfrac\eta n\right)$, for
all $t\in I$ the function $\chi$ satisfies \reftext{\eqref{approx_init_cut-off_s_eta}}, and moreover
$\sigma_n \in C_\mathrm{c}^\infty\big(\mathring I\big)$ with
$\sigma_n(t) = 1$ for $t \in I_n$, where $I_n \nearrow I$ is a
sequence of compact
intervals exhausting $I$. By the same reasoning as in the first
step of the proof of \reftext{Lemma~\ref{lem:approx_init}}, after integration
over $I$ as well, we find ${\skliaustask f - \chi_nf\skliaustasd
}_\mathrm{rhs}\to0$ as
$n\to\infty$.

\medskip
\noindent\textit{Mollification in $s=\log x$, and $t$, truncation in $y$.}
By the previous step we may assume without loss of generality that
$f(t,x,\eta)$ is zero if $x\gg1$ or $ {\left\lvert\eta
\right\rvert}\gg1$ or $t\notin I'$
for some compact interval $I'\Subset I$. We then perform a truncation in $y$ and a
mollification in the variable $s=\log x$ as in the
proof of \reftext{Lemma~\ref{lem:approx_init}}, which allows to approximate $f$
by functions that are smooth in $x$ and $y$ at fixed~$t$. Then we
perform a further mollification in $t$ and obtain approximants that are
smooth in $t$, $x$, $\eta$, and~$y$.

\medskip
\noindent\textit{Decay conditions as $x\searrow0$.}
For $\gamma\in\{1,2\}$ we apply the reasoning of the corresponding step
of \reftext{Lemma~\ref{lem:approx_init}} to the solution of
$(D_x-\gamma)w^{(1)}=w^{(2)}$ via formulas \reftext{\eqref{gensolw1}} and \reftext{\eqref{explicitw}}. This time we find that $f$ has the form
\begin{equation*}
f(t,x,\eta)=f_1(t,\eta)x + f_2(t,\eta)x^2 + R(t,x,\eta)\quad
\text{as}\quad x\searrow0.
\end{equation*}
Now applying \reftext{Lemma~\ref{lem:mainhardy}} and using the fact that $\tilde
q(D_x-1)(D_x-1)$ cancels the terms $f_1x$ and $f_2x^2$ above, we find
that at fixed $t$ it holds
\begin{eqnarray*}
 {\left\lVert R \right\rVert}_{\check k-2,\delta+\frac
32}&\lesssim& {\left\lVert(D_x-4)(D_x-3)\tilde
q(D_x-1)(D_x-1)R \right\rVert}_{\check k -2,\delta+\frac32}\\
&=& {\left\lVert(D_x-4)(D_x-3)\tilde q(D_x-1)(D_x-1)f
\right\rVert}_{\check k -2,\delta
+\frac32}<\infty.
\end{eqnarray*}
Therefore, $R(t,x,\eta)=o\left(x^{2+\delta}\right)$ as $x\searrow0$. As in the proof of
\reftext{Lemma~\ref{lem:approx_init}}, we infer that
the coefficients $f_1$ and $f_2$ are
smooth in $\eta$. Now we define
\begin{equation*}
f^{(n)} := f_1x + f_2x^2 + \chi_n R,
\end{equation*}
where $\chi_n(s) := \chi_1\left(\frac s n\right)$ and $\chi_1$ is a
smooth cut-off function such that $\chi_1(s) = 1$ for $s \in [-1,\infty)$
and $\chi_1(s) = 0$ for $s \in (-\infty,-2]$. Then like in the
last step of \reftext{Lemma~\ref{lem:approx_init}} we may prove that
${\skliaustask f-f^{(n)}\skliaustasd}_\mathrm{rhs}\to0$ as $n\to
\infty$, and at the same time we
have $f^{(n)} := f_1 x+f_2x^2$ for $x \ll_n 1$, which finalizes the proof.
\end{proof}
\begin{proof}[Proof of \reftext{Lemma~\ref{lem:approx_sol}}]
Throughout the proof, estimates depend on $k$, $\tilde k$, $\check k$,
$\breve k$, and $\delta$.

By using \reftext{Lemma~\ref{lem:equivalence_sol}}, we may approximate in the
norm ${\skliaustask\cdot\skliaustasd}_{\mathrm{Sol}}$ (cf.~\reftext{\eqref{norm_sol_simple}})
which is equivalent to the norm ${\skliaustask\cdot\skliaustasd
}_\mathrm{sol}$
(cf.~\reftext{\eqref{norm_sol}}) for the choice of the time interval to be
$I=[0,\infty)$. The proof is mainly the same as the one of
\reftext{Lemma~\ref{lem:approx_rhs}}, with differences appearing only while
treating the remainder term $R$ and the cut-off in time $t$. In order
to control the effect of the terms in the norm ${\skliaustask\cdot
\skliaustasd}_\mathrm
{Sol}$ involving $\partial_t v$, we use the cut-off function $\chi
(t,s,\eta) = \sigma\left(\tfrac{t}{n}\right)\chi\left(\tfrac
{s}{n},\tfrac{\eta}{n}\right)$, where $\chi$ is a smooth function which
satisfies \reftext{\eqref{approx_init_cut-off_s_eta}} as in the previous proof
and $\sigma:(0,\infty)\to[0,1]$ satisfies
%
%
\begin{equation}\label{approx_cut-off_t}
\sigma(t)=1 \quad\mbox{for}\quad t\in(0,1], \qquad\sigma(t) = 0
\quad
\mbox{for} \quad t \in[2,\infty).
\end{equation}
A $\partial_t$-derivative acting on $\chi_n$ produces an extra factor
bounded in the $BC^0$-norm (in fact the bound of the extra factor also
decays like $O\left(\tfrac1n\right)$, but we do not need to use this)
and with support in the interval $[n,2n]$, which allows to control the
extra terms as $n\to\infty$ by dominated convergence. Concerning the
expansion near $x = 0$, in the last step of the proof we find the expression
\begin{equation*}
v(t,x,\eta)=v_0(t,\eta) + v_1(t,\eta)x +v_{1+\beta}(t,\eta
)x^{1+\beta}+
v_2(t,\eta)x^2 + R(t,x,\eta)\quad\text{ as }\quad x\searrow0.
\end{equation*}
For estimating $R$ we use \reftext{Lemma~\ref{lem:mainhardy}}, and the fact that
the operator $(D_x-2)\tilde q(D_x)D_x$ cancels all the terms different
than $R$ on the right. Then we may estimate for $t\in
I$ as follows
\begin{eqnarray*}
 {\left\lVert R \right\rVert}_{\check k+2,\delta+\frac
32}&\lesssim& {\left\lVert(D_x-3) (D_x-2) \tilde q(D_x) D_x
R \right\rVert}_{\check k + 2, \delta+ \frac3 2}\\
&=& {\left\lVert(D_x-3) (D_x-2) \tilde q(D_x) D_x v \right
\rVert}_{\check k + 2, \delta+
\frac3 2}<\infty,
\end{eqnarray*}
which implies that $R(t,x,\eta)=o\left(x^{2+\delta}\right)$ as $x\searrow0$.
Now we may use the same arguments as in the proof of \reftext{Lemma~\ref{lem:approx_init}} in order
to obtain smoothness in $\eta$ of the above coefficients $v_0$, $v_1$,
$v_{1+\beta}$, and $v_2$. The rest of the
proof, including the construction of $v^{(n)}$ by truncating the
remainder $R$ in the $t$-variable, is like in the proof of \reftext{Lemma~\ref{lem:approx_rhs}}.
\end{proof}
\begin{proof}[Proof of \reftext{Corollary~\ref{coroll:contnorm}}]
On one hand the approximants constructed in \reftext{Lemma~\ref{lem:approx_sol}}
have bounded support in time, i.e., $v_n(t,\cdot,\cdot)=0$ for $t>2n$,
and on the other hand convergence in the ${\skliaustask\cdot
\skliaustasd}_\mathrm
{sol}$-norm implies convergence of $(v_0)_y$ in the supremum norm due to \reftext{\eqref{bc0_grad_sol}} of \reftext{Lemma~\ref{lem:bc0_bounds}}. Moreover, since
$ {\left\lVert v(t,\cdot,\cdot) \right\rVert}_\mathrm
{init}\le{\skliaustask v\skliaustasd}_\mathrm{sol}$, the
above-mentioned support properties of $v_n$ imply $ {\left
\lVert v(t,\cdot,\cdot) \right\rVert}_\mathrm{init} \to0$ as $t\to\infty$.
These considerations allow to prove the
first item of the corollary.

The second item follows due to the fact that the required property
holds for the approximants constructed in \reftext{Lemma~\ref{lem:approx_sol}},
again by using the bound \reftext{\eqref{bc0_grad_sol}} of \reftext{Lemma~\ref{lem:bc0_bounds}}.

Finally, note that the supremum part of the norm \reftext{\eqref{norm_sol}} can
be rewritten as $\sup_{t\in I}\!{\left\lVert v(t,\cdot
,\cdot) \right\rVert}_\mathrm{init}$.
From the smoothness properties of the approximants $v^{(n)}$ from
\reftext{Lemma~\ref{lem:approx_sol}}, we find that $t\mapsto {\left
\lVert v^{(n)}(t,\cdot,\cdot) \right\rVert}_\mathrm{init}$ are
continuous, and from the
approximation property ${\skliaustask v^{(n)}-v\skliaustasd}_\mathrm
{sol}\to0$ we find
that these functions are also uniformly convergent and bounded, so that
the function $t\mapsto {\left\lVert v(t,\cdot,\cdot)
\right\rVert}_\mathrm{init}$ is also
continuous. Now note that the integral terms in the definition of
${\skliaustask v\skliaustasd}_{\mathrm{sol},\tau}$ vanish as $\tau
\searrow0$ by
dominated convergence. Thus
\begin{equation*}
\lim_{\tau\searrow0}{\skliaustask v\skliaustasd}_{\mathrm
{sol},\tau}=\lim_{\tau\searrow
0}\sup_{t\in I_\tau} {\left\lVert v(t,\cdot,\cdot)
\right\rVert}_\mathrm{init}= {\left\lVert v^{(0)} \right
\rVert}_\mathrm{init},
\end{equation*}
which concludes the proof of the last item as well as of the corollary.
\end{proof}
%
%
%

\bibliography{navier_higher_revision}
\bibliographystyle{plain}
\end{document}